\documentclass[10pt]{article}

\usepackage{latexsym,amsfonts,amsmath,amsthm,amssymb, epsfig}
\usepackage[a4paper,total={170mm,250mm},vcentering,dvips]{geometry}
\usepackage{verbatim}

\usepackage{dsfont}
\usepackage{enumitem}
\usepackage{csquotes}
\usepackage{bm}
\usepackage[pdftex,dvipsnames]{xcolor}
\usepackage{float}
\usepackage{caption}
\usepackage{parskip} 

\usepackage{cancel}
\usepackage{algpseudocode}
\usepackage{algorithm}
\usepackage{subcaption}
\usepackage{accents}

\usepackage{thmtools} 
\usepackage[hidelinks]{hyperref}
\usepackage[noabbrev,capitalize]{cleveref}

\newtheorem{thm}{Theorem}[section]
\newtheorem{cor}[thm]{Corollary}
\newtheorem{lem}[thm]{Lemma}

\newtheorem{defi}[thm]{Definition}

\newtheorem{rem}[thm]{Remark}
\newtheorem{dis}[thm]{Discussion}
\newtheorem{assumptions}[thm]{Assumption}
\newtheorem{main_result}{Main Result}

\DeclareMathOperator{\supp}{supp}
\DeclareMathOperator{\Lip}{\textup{Lip}}
\DeclareMathOperator*{\esssup}{ess\,sup}

\DeclareMathOperator*{\argmin}{arg\,min}

\newcommand{\R}{\mathbb{R}}
\newcommand{\N}{\mathbb{N}}
\newcommand{\V}{\mathbb{V}}

\newcommand{\diam}{\textup{diam}}
\newcommand{\calH}{\mathcal{H}}
\newcommand{\calF}{\mathcal{F}}
\newcommand{\calT}{\mathcal{T}}
\newcommand{\calM}{\mathcal{M}}
\newcommand{\calN}{\mathcal{N}}
\newcommand{\calL}{\mathcal{L}}
\newcommand{\calP}{\mathcal{P}}
\newcommand{\calE}{\mathcal{E}}
\newcommand{\calI}{\mathcal{I}}
\newcommand{\calB}{\mathcal{B}}
\newcommand{\curlybinom}[2]{\genfrac{\{}{\}}{0pt}{0}{#1}{#2}}
\newcommand{\LipProp}{\textup{LipProp}}
\newcommand{\Enl}{\mathfrak{L}}

\usepackage[normalem]{ulem}

\begin{document}

\title{Approximation classes for adaptive anisotropic space-time finite element methods. An almost characterization.}
\author{Pedro Morin\footnote{Universidad Nacional del Litoral and CONICET, Departamento de Matem\'a{}tica, Facultad de Ingenier\'i{}a Qu\'i{}mica, S3000AOM Santa Fe, Argentina. Email: \href{mailto:pmorin@fiq.unl.edu.ar}{pmorin@fiq.unl.edu.ar}}, Cornelia Schneider\footnote{Friedrich-Alexander-Universit\"at Erlangen-N\"urnberg, Angewandte Mathematik III, Cauerstr. 11, 91058 Erlangen, Germany. Email: \href{mailto:cornelia.schneider@fau.de}{cornelia.schneider@fau.de}}
, and Nick Schneider\footnote{Friedrich-Alexander-Universit\"at Erlangen-N\"urnberg, Angewandte Mathematik III, Cauerstr. 11, 91058 Erlangen, Germany. Email: \href{mailto:nick.schneider@fau.de}{nick.schneider@fau.de}} \footnote{corresponding author}}
\maketitle
\begin{abstract}
We study the approximation of $L_p$-functions, $p\in  (0,\infty]$, on cylindrical space-time domains $\Omega_T:=[0,T]\times \Omega$ with respect to continuous anisotropic space-time finite elements on prismatic meshes. In particular, we propose a suitable refinement technique which creates (locally refined) prismatic meshes with sufficient smoothness and the desired anisotropy, and prove complexity estimates. Furthermore, we define a (quasi-)interpolation operator on this type of meshes and use it to characterize the corresponding approximation classes by showing direct and inverse estimates in terms of anisotropic Besov norms. \\ \\
\\
\noindent{\em Key Words: anisotropic Besov spaces; approximation classes; mesh refinement; space-time finite elements; (quasi-)interpolation}. \\
{\em MSC2020 Math Subject Classifications: Primary 41A25, 65D05, 65M50; Secondary 65M60.} 
\end{abstract}

\tableofcontents

\section{Introduction}\label{sect:Introduction}

The main goal of this article is to characterize functions that can be approximated at a given convergence rate by adaptive space-time anisotropic discretizations with continuous finite elements. 
This study provides an insight into the approximation classes for optimal adaptive space-time finite element methods for instationary problems; it is in the spirit of the results presented in~\cite{BDDP02,GM14} for the stationary case, and those from~\cite{AMS23, AGMSS25} for time-stepping discretizations of instationary problems.
We show that a characterization is possible with respect to anisotropic Besov spaces, in particular, with the variant of those spaces introduced in \cite[Sect.~3]{MSS26}. The characterization is only \enquote*{almost} possible in terms of these spaces, i.e., requires a notion of generalized anisotropic Besov spaces to derive inverse estimates, as is the case in previous works for stationary problems~\cite{BDDP02,GM14}. This is mainly due to the fact that the approximants themselves naturally lack higher order Besov regularity.

To put our aim into perspective, let us briefly recall the role that approximation classes play in the theory of adaptive finite element methods (AFEM). For stationary problems, the optimality of AFEM is by now well understood~\cite{BDD04,Ste08,NSV09}, and the results proved there are of the following type. Let $u$ denote the exact solution being approximated in a norm $\|\cdot\|$, and assume that there exist $s>0$ and a constant $C$ such that, for each $N\in\N$, there is a mesh $\calP$ obtained from the initial mesh $\calP_0$ with $\#\calP- \#\calP_0 \le N$ such that
\[
\inf_{v_\calP}\|u-v_\calP\|\le C N^{-s},
\]
where the infimum is taken over all finite element functions $v_\calP$ over the mesh $\calP$. Then the adaptive loop generates a sequence $\{(\calP_k,u_k)\}_{k\in\N}$ of meshes and corresponding discrete solutions such that $\|u-u_k\|\le \hat C\,(\#\calP_k-\#\calP_0)^{-s}$ for all $k\in\N$. In other words, if the exact solution can be ideally approximated with complexity $\mathcal{O}(N^{-s})$, then the AFEM produces approximations converging with the same rate. The set of functions which can be ideally approximated with complexity $\mathcal{O}(N^{-s})$ forms precisely the \emph{approximation class} $\mathbb{A}_s$ (see \cref{subsubsect:Approximation classes} for a precise definition), and its characterization in terms of classical smoothness spaces is what turns such an abstract optimality statement into a concrete convergence rate for a given objective function. This is the question we address in this article for continuous finite elements on adaptively refined anisotropic prismatic space-time meshes. A space-time finite element method on locally refined prismatic meshes, whose error is bounded by a multiple of the best approximation error from the discrete space, has recently been presented in~\cite{DSS25}; our results describe which functions such a method can approximate at a given rate, provided the meshes are (quasi-)optimally chosen.

We stress that the target of study is thus merely a question of approximation theory, and not of regularity theory for partial differential equations: we identify the functions that can be approximated at a given rate by continuous finite elements on adaptively refined anisotropic prismatic meshes, regardless of whether these functions arise as solutions of a PDE or not. In particular, our results are not tied to parabolic problems. The anisotropy is encoded in a pair of smoothness parameters $(s_1,s_2)$ for the temporal and spatial directions, respectively, and both the refinement algorithm and the estimates are valid for any such pair; the scaling tailored to parabolic problems used in~\cite{DS22,DSS25} is one instance. 
The same framework applies, for example, to functions on a cylinder $\Omega_T$ which are smooth along its axis and singular across it, as is the case for edge singularities of elliptic problems on such domains, for which anisotropic meshes are a classical tool, see, e.g., \mbox{\cite{Ape99,Li18}} and the references therein; the \enquote*{time} variable then plays no dynamical role and merely denotes the distinguished direction. 
The complementary question, namely under which conditions on the data the solution of a given evolution problem, say the heat equation on a polygonal domain, belongs to the anisotropic Besov spaces appearing in our characterization, is of independent interest and falls outside the scope of this article.  For stationary elliptic problems the analogous question of Besov regularity of e.g. the Poisson problem is by now classical, see e.g. \cite{DD97}; for evolution problems on nonsmooth domains much less is known. This will be studied in a forthcoming paper and we regard the present characterization as the key step in making such regularity results usable  for the analysis of adaptive space-time methods.

Finite element methods based on simultaneous space-time discretizations have gained much attention mainly due to their capacity for parallel implementation and efficiency, and also due to the simplicity for handling singularities via local mesh refinement; numerical experiments also show an excellent performance~\cite{VW21,LS22,DVHB23}. 
A finite element method for the heat equation that computes quasi-optimal approximations with respect to natural norms has been presented in~\cite{DSS25}. Similarly to what is done in~\cite{DS22}, they also incorporate local mesh refinements in space-time using a refinement strategy that guarantees an anisotropy tailored to parabolic problems. Such local refinement involves prismatic meshes where the elements are cartesian products of simplicial elements in space and time intervals. They allow for a simple implementation of anisotropy without producing small or large angles when the space elements are non-degenerate.
In some sense, our work complements the results from~\cite{DS22,DSS25} because it characterizes the regularity of the functions that can be optimally approximated with given rates and different polynomial degrees for the finite elements.
Wavelet-based space-time methods have also been used for solving instationary problems (see~\cite{SS09} and the references therein), but here we focus on finite element methods and cite only the references strictly related to our work.

In order to achieve our results, which are suitable for a variety of different time and space scalings we propose a local refinement method for space-time prismatic anisotropic partitions and prove complexity estimates. A similar refinement method is slightly hinted in~\cite{DS22}, and briefly described in~\cite[Sect.~4.1]{DSS25} for a parabolic anisotropy. Moreover, locally refined prismatic meshes with parabolic anisotropy have also been used in~\cite{GvV22} for an adaptive space-time boundary element method for the heat equation, there on meshes of the lateral boundary of the space-time cylinder.
Our refinement algorithm is thus not only a tool for performing the analysis of space-time finite element methods for instationary problems, but also an essential building block for the implementation of such methods for different problems; this article thus  extends and complements the algorithms presented in~\cite{DS22,DSS25,GS24}.

The first approximation theoretic investigation of the stationary, i.e., time-independent, adaptive (isotropic) finite element method was given by Binev, Dahmen, DeVore, and Petrushev \cite{BDDP02}, where they considered linear FEM, triangular coverings of bounded Lipschitz domains in $\R^2$, and an  algorithm based on the newest-vertex bisection as refinement technique. The latter, in particular, preserves conformity and fulfills a complexity estimate as shown by the first three aforementioned authors in \cite{BDD04}. 
The generalization of these investigations to adaptive stationary FEM of arbitrary polynomial order and space dimension was published by Gaspoz and Morin \cite{GM14}. Here, a generalization of the refinement algorithm from \cite{BDD04}, based on work by Stevenson \cite{Ste08}, was used. The latter itself relies on the generalized Maubach-Traxler \cite{Mau95,Tra97} bisection routine for simplices. In particular, this high-dimensional bisection routine avoided potential degeneration of the minimal head angle of the refined simplices. Furthermore, the approach to inverse estimates using generalized (isotropic) Besov spaces was  introduced in \cite{GM14} and more thoroughly studied in \cite{Gan17}.

The nonstationary angle of this topic remained largely open due to its considerable complexity. A first step toward its analysis was taken by the authors of this article and the collaborators Actis and Gaspoz \cite{AMS23, AGMSS25} with respect to adaptive time-stepping FEM. A key tool in these investigations was the consideration of Banach space-valued Besov spaces. However, no inverse estimates are provided there. Instead, in \cite{MSS26}, we turned towards anisotropic space-time FEM and the corresponding anisotropic Besov spaces. There, a prismatic tensor product mesh was used as an initial covering together with an atomic refinement method for prisms  which consists of a Maubach-Traxler bisection in the spatial and an appropriate number of classical bisections in the temporal direction. The number of temporal bisections was chosen such that the refinement correctly accounted for the anisotropy of the function, i.e., different smoothness in time and space. Since this method preserves the mesh geometry on the individual level of the prisms, direct estimates for discontinuous anisotropic space-time FEM could be obtained in the end of this article. Still, the refinement technique used there does not guarantee any global mesh regularity, e.g. a space-time analogue of a property such as the conformity of a simplicial triangulation in the stationary case. Therefore, it was essentially limited to the study of adaptive anisotropic approximation with discontinuous finite elements. Thus, an objective of this manuscript now is to improve our refinement technique from \cite{MSS26} and prove corresponding results for continuous finite elements.

The main results of our work can be roughly summarized as follows:

\begin{description}
\item[Direct result]
If a function belongs to the anisotropic Besov space $B^{s_1, s_2}_{q,q}(\Omega_T)$, then it can be approximated in $L_p(\Omega_T)$ by continuous space-time finite elements of order $(r_1, r_2)$ on (anisotropic) prismatic meshes $\calP$ with an approximation error of order $(\# \calP)^{-\frac{1}{\frac{1}{s_1}+\frac{d}{s_2}}}$.
If $\alpha_1>0$ and $\alpha_2>0$, the same result holds for functions in $B^{\alpha_1+s_1, \alpha_2+s_2}_{q,q}(\Omega_T)$, when measuring the error in $B_{p,p}^{\alpha_1, \alpha_2}(\Omega_T)$.

\item[Inverse result.]
If a function can be approximated in $L_p(\Omega_T)$ by continuous space-time finite elements of order $(r_1, r_2)$ on prismatic meshes $\calP$ with an error decay of $(\# \calP)^{-\frac{1}{\frac{1}{s_1}+\frac{d}{s_2}}}$, then it belongs to the \emph{generalized Besov space} $\widehat{B}^{s_1(1-\epsilon), s_2(1-\epsilon)}_{q,q}(\Omega_T)$ (for all $0< \epsilon < 1$).
If $\alpha_1>0$ and $\alpha_2>0$, and the error is measured in $\hat B_{p,p}^{\alpha_1, \alpha_2}(\Omega_T)$, then such an error decay is observed only if the function belongs to $\widehat B^{\alpha_1+s_1, \alpha_2+s_2}_{q,q}(\Omega_T)$.

\end{description}
A more precise description of our results can be found in \cref{subsect:main_results}, where all the assumptions on the coefficients are clearly stated, after some necessary definitions have been established.

In some sense, what we present in this paper is an extension of the results from \cite{MSS26} to the case of continuous anisotropic space-time FEM. This is not just a simple extension though; several important difficulties arise due to the fact that the basis functions have larger support, and can be very complicated if the meshes are allowed to have arbitrary hanging nodes.
Hence the refinement algorithm is not straightforward, since we will enforce the meshes to be conforming in space and allow only for 1-irregularities in time (see \cref{subsect:Patch_refine} below); therefore, a careful analysis of its complexity is  necessary. 

Finally, we point out that the characterization proposed here carries over Leisner’s cube-based result~\cite{Lei00} in the setting of wavelet approximation to finite elements and general domains. Despite the conceptual similarity, the methods employed in this paper are entirely different.

The paper is organized as follows. In  \cref{sect:preliminaries}, in order to be self-contained, we first discuss some preliminary results, mainly from \cite{MSS26} that we require throughout the paper. Moreover, we state our main results regarding the almost characterization of approximation classes with respect to adaptive, anisotropic space-time finite elements. Afterwards, in \cref{sect:space_time_partition_and_refinement}, we present an algorithm based on the above mentioned atomic refinement technique in order to generate locally refined prismatic meshes that are regular enough and suitable for approximation with our continuous anisotropic space-time FEM while simultaneously allowing to prove a complexity result. In \cref{sect:investigation of mesh geometry and Quasi-interpolation} we analyze in more detail the lattice geometry of the space-time partitions generated in \cref{sect:space_time_partition_and_refinement} in order to construct a quasi-interpolation operator from the Lebesgue spaces $L_p(\Omega_T)$, $p\in (0,\infty]$, to the space of space-time finite elements. In particular, we show that the operator satisfies a local stability property and globally fulfills a near-best approximation condition. In \cref{sect:Almost_characterization}, we first introduce the aforementioned generalized anisotropic Besov spaces and study when they coincide with the classical ones. Further, we use this knowledge to show a Sobolev-type embedding theorem in the classical anisotropic Besov scale. Finally, we apply the previously derived results to prove our main results, i.e., direct and inverse estimate results for the approximation of (generalized) Besov functions with continuous anisotropic space-time finite elements. 

An interesting finding is that a comparison of the direct estimate result for continuous finite elements and the one for discontinuous ones previously shown in \cite[Thm.~1.5]{MSS26} yields that the involved constants will only differ quantitatively, not qualitatively, i.e., the constants depend on the same parameters, see \cref{disc:comparison_continuous_discontinuous}.

\section{Preliminaries and main results}\label{sect:preliminaries}

In this section, we first briefly introduce the notation and function spaces that will be used throughout the article and are necessary to state our main results, which we do immediately afterwards.

\subsection{Preliminaries}

\subsubsection{General setting}\label{subsubsect:Prelim_General_Setting_FEM_Spaces}

We will always consider $T\in (0,\infty)$, $\Omega\subset \R^d$, $d\in\N$, a bounded polyhedral Lipschitz domain, and \mbox{$\Omega_T:=[0,T]\times \Omega$}. Furthermore, we will call a non-overlapping covering of $\Omega_T$ by \textbf{prisms} of the form $I\times S$, $I$ an interval and $S$ a $d$-dimensional simplex, a \textbf{space-time partition}. For those and $s_1, s_2 \in (0,\infty)$, we define the quantity $a(\calP)$ which describes the \textbf{(maximal) anisotropy of the space-time partition} as in \cite[Sect.~1]{MSS26}, i.e.,
\begin{align*}
    a(\calP, s_1, s_2) :=  a(\calP) := \max\limits_{I\times S\in \calP}\left(\max\left(\frac{|I|}{|S|^{\frac{s_2}{s_1d}}},\frac{|S|^{\frac{s_2}{s_1d}}}{|I|}\right)\right).
\end{align*}
Here and in the remainder of this article, $|\cdot|$ applied to sets always represents the Lebesgue measure. Whether the $1$-, $d$-, or $(d+1)$-dimensional version is meant, will always be clear from the context. Correspondingly, denoting $B_{\rho}(x_0)$ as the closed ball with radius $\rho$ around $x_0$, we adopt the notation
\begin{align*}
    \kappa_\calP:=\sup\limits_{J\times S\in \calP} \kappa_S :=\sup\limits_{J\times S\in \calP} \frac{\diam(S)}{\rho_S} :=\sup\limits_{J\times S\in \calP}\frac{\diam(S)}{\sup\{\rho\in(0,\infty)\mid B_{\rho}(x_0)\subset S\text{ for some } x_0\in S\}},
\end{align*}
\begin{figure}[H]
	\begin{center}
		\includegraphics[height=4cm]{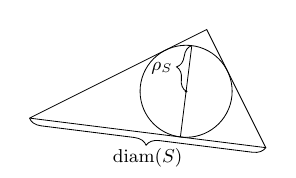}
	\end{center}
\end{figure}
which is the \textbf{shape-regularity constant of the simplices} in $\calP$, in the sense that $\kappa_\calP\rightarrow \infty$ if it contains simplices $S$ whose minimum head angle goes to zero. Of course, for any bounded set $E$ in any Euclidean space, $\diam(E)$ denotes the diameter of $E$, i.e., $\diam(E):=\sup_{x,y\in E}|x-y|$. 

On these space-time partitions, we will study approximation with respect to \textbf{(dis)continuous anisotropic finite elements} on $\calP$, i.e., 
\begin{align*}
    \mathbb{V}^{r_1, r_2}_{\calP, \textup{DC}}:=\left\{F:\Omega_T\rightarrow \R\mid F_{|I\times S}\in \Pi^{r_1, r_2}_{t,\bm{x}}(I\times S), \, I\times S \in \calP\right\}\quad\text{and}\quad\mathbb{V}^{r_1, r_2}_{\calP}:=\left\{F\in \mathbb{V}^{r_1, r_2}_{\calP, \textup{DC}}\mid F \text{ continuous}\right\},
\end{align*}
where $r_1, r_2\in \N_{\ge 2}$ and $\Pi^{r_1, r_2}_{t,\bm{x}}$ corresponds to the space of \textbf{anisotropic polynomials} on $I\times S\in \calP$ of temporal order $r_1$ and spatial order $r_2$,
\begin{align}\label{aniso-poly-tensor}
    \Pi_{t,\bm{x}}^{r_1,r_2}(I\times S):=\Pi^{r_1}(I) \otimes \Pi^{r_2}(S):= \left\{P:I\times S\rightarrow \R\:\bigg|\: P(t,\bm{x}):= \sum\limits_{i=0}^{r_1-1} \sum\limits_{|\alpha|<r_2} c_{i, \alpha} t^i \bm{x}^\alpha,\,\, c_{i, \alpha}\in\mathbb{R} \right\}.
\end{align}
The sum over $\alpha$ is understood over the set of multi-indices $\alpha \in \N_0^d$ with $|\alpha|:=\alpha_1+\alpha_2+\dots+\alpha_d < r_2$ and $\bm{x}^\alpha = x_1^{\alpha_1}x_2^{\alpha_2}\dots x_d^{\alpha_d}$.
Of course, this definition can be easily extended to $\Pi^{r_1, r_2}_{t,\bm{x}}(D)$ for non-prismatic space-time domains.

\subsubsection{Temporal and spatial moduli of smoothness for anisotropic Besov spaces}\label{subsubsect:Besov spaces}

The version of anisotropic Besov spaces used in this article and the forerunner \cite{MSS26} is suited to space-time finite element analysis and does not necessarily coincide with previous approaches to those spaces (see \cite[Sect.~1]{MSS26} for more information on their distinctions). It is based on temporal and spatial difference operators
    \begin{align*}
		&\Delta_{h,t}f(t,\bm{x}):=f(t+h, \bm{x}) - f(t,\bm{x}), \quad\,\, \Delta_{h,t}^{r_1}=\Delta_{h,t}^{r_1-1}\circ \Delta_{h,t}\,\:\text{ for }r_1\ge 2,
		\\\text{and}\quad &\Delta_{\bm{h},\bm{x}}f(t,\bm{x}):=f(t, x+\bm{h}) - f(t,\bm{x}), \quad \Delta_{\bm{h},\bm{x}}^{r_2}=\Delta_{\bm{h},\bm{x}}^{r_2-1}\circ \Delta_{\bm{h},\bm{x}}\text{ for }r_2\ge 2,
	\end{align*}
with $h\in \R$ and $\bm{h}\in \R^d$. Those can be used to define temporal and spatial moduli of smoothness \mbox{$\omega_{r_1, t}(f, J\times R, \cdot)_p$}, \mbox{$\omega_{r_2, \bm{x}}(f, J\times R, \cdot)_p:[0,\infty)\rightarrow [0,\infty)$} for given $p\in (0,\infty]$, $r_1, r_2\in \N$, intervals $J\subset [0,T]$, domains $R\subset \Omega$, and corresponding $f\in L_p(J\times R)$, as in the isotropic case, i.e.,
\begin{align*}
    \omega_{r_1, t}(f, J\times R, \delta)_p:=\sup\limits_{|h|\le \delta } \|\Delta_{h,t}^{r_1}f \|_{L_p(J_{r_1,h}\times R)}\quad \text{and}\quad \omega_{r_2, \bm{x}}(f, J\times R, \delta)_p:=\sup\limits_{|\bm{h}|\le \delta } \|\Delta_{\bm{h},\bm{x}}^{r_2}f \|_{L_p(J\times R_{r_2,\bm{h}})},
\end{align*}
where $J_{r_1,h}:=\{t\in J\mid t+ih\in J \ \ \text{for all}\ \ i=1,\dots, r_1 \}$, \mbox{$R_{r_2,\bm{h}}:=\{\bm{x}\in R\mid x+i\bm{h}\in R \ \ \text{for all}\ \ i\in [0,r_2] \} $}, and $L_p(J\times R)$ is of course the Lebesgue space of real-valued $p$-integrable (for $p<\infty$) and essentially bounded (for $p=\infty$) functions on $J\times R$. Further, there are averaged versions of the above moduli, i.e.,
\begin{align*}
        &\mathrm{w}_{r_1, t}(f,J\times R,\delta)_p:= \left(\frac{1}{2\delta}\int\limits_{|h|\le \delta}\|\Delta_{h,t}^{r_1}f\|^p_{L_p\left(J_{r_1, h}\times R\right)}\,dh\right)^\frac{1}{p}
        \\\text{and}\quad &\mathrm{w}_{r_2, \bm{x}}(f,J\times R,\delta)_p:= \left(\frac{1}{(2\delta)^d}\int\limits_{|h|\le \delta}\|\Delta_{\bm{h},\bm{x}}^{r_2}f\|^p_{L_p\left(J\times R_{r_2, \bm{h}}\right)}\:d\bm{h}\right)^\frac{1}{p},
    \end{align*}
for $\delta, p\in (0,\infty)$. For $p=\infty$, they coincide with their corresponding supremum versions. Additionally, one sets \mbox{$\mathrm{w}_{r_i, \mathrm{dir}}(f,J\times R,0)_p:=0$} for any $(r_i, \mathrm{dir})\in \{(r_1, t),(r_2, \bm{x})\}$ and $p\in (0,\infty]$. In particular, these moduli admit the following properties.
\begin{rem}\label{Rem:Moduli of smoothness}
    Let $f\in L_p(J \times R)$, $(r_i, \mathrm{dir})\in \{(r_1, t),(r_2, \bm{x})\}$, $k_i\in \N_0$ with $k_i\le r_i$, $\delta\in[0,\infty)$, $m\in \N$, and $l\in [0,\infty)$. Then it holds
    \begin{enumerate}[label=(\roman*)]
        \item $\omega_{r_i, \mathrm{dir}}(f, J\times R, \cdot)_p$ is monotonically increasing, \label{Rem:Moduli of smoothness - monotonicity}
        \item $\omega_{r_i, \mathrm{dir}}(f, J\times R, \delta)_p^{\min(p,1)}\le 2^{r_i-k_i}\omega_{k_i, \mathrm{dir}}(f, J\times R, \delta)_p^{\min(p,1)}$, \label{Rem:Moduli of smoothness - higher order lesser order}
        \item $\omega_{r_i, \mathrm{dir}}(f, J\times R, l\delta)_p^{\min(p,1)} \le (l+1)^{r_i}\omega_{r_i, \mathrm{dir}}(f, J\times R, \delta)_p^{\min(p,1)}$, and \label{Rem:Moduli of smoothness - scaling item}
        \item $\mathrm{w}_{r_i, \mathrm{dir}}(f, \cdot, \delta)_p^p$ is super-additive, i.e., for $p<\infty$ and measurable sets $D_1,\ldots, D_m\subset D\subset J\times R$ such that every point of $D$ lies in at most $M$ of the sets $D_i$, it holds $\sum_{i=1}^m\mathrm{w}_{r_i,\mathrm{dir}}(f,D_i,\delta)_p^p \leq M \mathrm{w}_{r_i,\mathrm{dir}}(f,D,\delta)_p^p$.\label{Rem:Moduli of smoothness - additivity}
    \end{enumerate}
\end{rem}
\begin{proof}
    See \cite[Rem.~2.4]{MSS26}. Particularly, the requirement given there that $R$ has to be Lipschitz, is not necessary.
\end{proof}
If readers are interested in learning more about the properties of these moduli, such as the equivalence between the supremum and averaged versions, we refer them to \cite[Sect.~2]{MSS26}. As mentioned before, the moduli allow for the definition of \textbf{anisotropic Besov spaces}\footnote{This approach to anisotropic Besov spaces is different to the one used in \cite{Lei00}. The latter extends the one given in \cite[Ch.~4]{BIN79} to parameters $p,q\in (0,\infty]$. If and for which choice of parameters and domains these different approaches coincide, is a relevant question for further research.}, i.e.,
\begin{align*}
    B^{s_1, s_2}_{p,q}(J\times R):=\{f\in L_p(J\times R)\mid \|f\|_{B^{s_1, s_2}_{p,q}(J\times R)}:=\|f\|_{L_p(J\times R)}+|f|_{B^{s_1, s_2}_{p,q}(J\times R)}<\infty\},
\end{align*}
where $p,q\in (0,\infty]$, $s_1, s_2\in (0,\infty)$, $r_i:=\lfloor s_i \rfloor +1 $, $i=1,2$, and
\begin{align}\label{Eq:def-quasi-seminorm}
    |f|^{}_{B^{s_1, s_2}_{p,q}(J\times R)}
    &:=\left(\int\limits_0^\infty \left[\delta^{-s_1}\omega_{r_1, t}(f, J\times R, \delta)_p\right]^q\frac{d\delta}{\delta} 
    + \int\limits_0^\infty \left[\delta^{-s_2}\omega_{r_2, \bm{x}}(f, J\times R, \delta)_p\right]^q\frac{d\delta}{\delta}\right)^\frac1q,
    \quad\text{if }q<\infty,\notag
    \\
    \text{and}\quad |f|_{B^{s_1, s_2}_{p,\infty}(J\times R)}
    &:=\esssup_{\delta \in (0,\infty)}  \delta^{-s_1}\omega_{r_1, t}(f, J\times R, \delta)_p
    + \delta^{-s_2}\omega_{r_2, \bm{x}}(f, J\times R, \delta)_p, \quad \text{if }q=\infty,
\end{align}
correspondingly, represent the \textbf{anisotropic Besov (quasi-)seminorm}. In particular, as stated in \cite[Rem.~3.2]{MSS26}\footnote{These basic definitions and results also do not require $R$ to be Lipschitz, though only stated for the Lipschitz case in \cite{MSS26}.}, it has an equivalent discrete representation, given by
\begin{align}\label{Eq:equivalence quasi-seminorms}
    |f|^*_{B^{s_1, s_2}_{p,q}(J\times R)}&:=\left(\sum\limits_{n=0}^\infty 2^{n\frac{s_2}{d}q}\left[\omega_{r_1,t}\left(f,J\times R, 2^{-\frac{ns_2}{s_1 d}} \right)_{p}^q+ \omega_{r_2, \bm{x}}\left(f, J\times R, 2^{-\frac{n}{d}}\right)_{p}^q\right]\right)^\frac{1}{q}, 
    \quad\text{if }q<\infty,\notag
    \\
    \text{and}\quad |f|^*_{B^{s_1, s_2}_{p,\infty}(J\times R)}
    &:=\sup_{n\in \N_0} 2^{n\frac{s_2}{d}}\left[\omega_{r_1,t}\left(f,J\times R, 2^{-\frac{ns_2}{s_1 d}} \right)_{p}+ \omega_{r_2, \bm{x}}\left(f, J\times R, 2^{-\frac{n}{d}}\right)_{p}\right] ,  
    \quad \text{if }q=\infty,
\end{align}
with $|f|_{B^{s_1, s_2}_{p,q}(J\times R)}^*\sim_{\,\substack{p,q,s_1, s_2,\\|J|, \diam(R)}}|f|_{B^{s_1, s_2}_{p,q}(J\times R)}$.
In particular, $\|f\|^*_{B^{s_1, s_2}_{p,q}(J\times R)}:=|f|^*_{B^{s_1, s_2}_{p,q}(J\times R)}+ \|f\|_{L_p(J\times R)}$ \mbox{$\sim_{\,p,q,s_1, s_2}\|f\|_{B^{s_1, s_2}_{p,q}(J\times R)}$}, which can be shown as in the isotropic case in \cite[Thm.~10.1(i) together with Eq.~(10.5) of Ch.~2.10]{DL93}. Furthermore, if we \textit{assume $R$ to be Lipschitz} and replace the moduli of order $r_1$ and $r_2$, by arbitrary $\tilde{r}_i\in \N_{\ge 2}$ with $\tilde{r}_i>r_i$ in \eqref{Eq:def-quasi-seminorm} or \eqref{Eq:equivalence quasi-seminorms}, then the analogue of $\|\cdot\|_{B^{s_1, s_2}_{p,q}(J\times R)}$ or $\|\cdot\|_{B^{s_1, s_2}_{p,q}(J\times R)}^*$ is an equivalent \mbox{(quasi-)}norm\footnote{We will indicate these (quasi-)(semi)norms via superscript, i.e., $|\cdot|_{B^{s_1, s_2}_{p,q}(J\times R)}^{(\tilde{r}_1, \tilde{r}_2)}$, $\|\cdot\|_{B^{s_1, s_2}_{p,q}(J\times R)}^{(\tilde{r}_1, \tilde{r}_2)}$, $|\cdot|_{B^{s_1, s_2}_{p,q}(J\times R)}^{*,(\tilde{r}_1, \tilde{r}_2)}$, or \mbox{$\|\cdot\|_{B^{s_1, s_2}_{p,q}(J\times R)}^{*, (\tilde{r}_1, \tilde{r}_2)}$}.} on $B^{s_1, s_2}_{p,q}(J\times R)$, where the equivalence constants only depend on $d,p,q,s_1, s_2, \tilde{r}_1, \tilde{r}_2$, and the Lipschitz properties $\LipProp(R)$ of $R$ (in the sense of \cite[Def.~1.6]{MSS26})\footnote{With a little abuse of notation, we will indicate that the Lipschitz properties of $R$ only depend on a set of $parameters$ via $\LipProp(R)\sim_{\, parameters} 1$.}. This can be obtained using the technique given in \cite[Thm.~10.1(ii) of Ch.~2.10]{DL93} together with the Marchaud inequalities from \cite[Thm.~2.9 and Thm.~2.12]{MSS26}. 

Although the order $r_i$ of the differences may look like a smoothness parameter, the only relevant requirement is $\tilde{r}_i>s_i$. Indeed, by \cref{Rem:Moduli of smoothness}\ref{Rem:Moduli of smoothness - higher order lesser order} the modulus of order $\tilde{r}_i$ is bounded by the one of order $r_i$, while the Marchaud inequality bounds the latter by a weighted average of the former (plus a norm term); the weight $\delta^{-s_i}$ in \eqref{Eq:def-quasi-seminorm} then leads to the same (quasi-)norm whenever $s_i<r_i\le \tilde{r}_i$. Roughly speaking, a modulus of order $\tilde{r}_i$ can detect smoothness up to order $\tilde{r}_i$, in the sense that $\omega_{\tilde{r}_i,\mathrm{dir}}(f,J\times R,\delta)_p\lesssim\delta^{\tilde{r}_i}$ for smooth $f$, and no faster decay can be observed; unless $f$ is a polynomial of order $\tilde{r}_i$ in the respective direction (then the modulus vanishes). For a function with Besov smoothness $s_i<\tilde{r}_i$, the behavior of $\omega_{\tilde{r}_i,\mathrm{dir}}(f,J\times R,\delta)_p$ as $\delta\to 0$ measured by \eqref{Eq:def-quasi-seminorm}, namely a decay of order $\delta^{s_i}$ (in the $L_q(\frac{d\delta}{\delta})$-sense for $q<\infty$, and uniformly for $q=\infty$), lies below this threshold and is therefore captured independently of the particular value of $\tilde{r}_i$. This freedom in the choice of the order will be used later, since the polynomial orders of our finite elements are at least $2$, and thus may exceed $\lfloor s_i\rfloor +1$ when $s_i\in(0,1)$.

In particular, these spaces have proven useful in the context of approximation with anisotropic polynomials and (discontinuous) finite elements, a setting that we further investigate in this article. Again, we refer the interested reader to {\cite[Sect.~3-4]{MSS26}}. 

\subsubsection{Approximation classes}\label{subsubsect:Approximation classes}

The main results of this article can be expressed in terms of approximation classes. Since our refinement methods will differ for different (anisotropy) parameters, we will require a more general definition of those classes than in the isotropic case  \cite[Sect.~3]{BDDP02} as well as in \cite[Sect.~2]{GM14} and the time-stepping setting case \cite[Sect.~4.1]{AMS23}, respectively.

Let $X\subset \{f \mid f:D\rightarrow \R\}$ be a (quasi-)Banach space of real-valued functions on $D\subset \Omega_T$ and $Y\subset X$ be a linear subspace. Then, the \textbf{minimal approximation error} of $f$ in $Y$ with respect to the (quasi-)norm $\|\cdot\|_X$ of $X$, is defined as $E(f, Y, X):=\inf\limits_{F\in Y}\|f-F\|_X$. In particular, if $X=L_p(D)$, we write $E(f, Y, D)_p:=E(f, Y, L_p(D))$. 

Additionally, let $\calP_0$ be an initial space-time partition and let $\textup{REF}(\calP_0)$ denote the set of sub-space-time partitions created from $\calP_0$ with an arbitrary but fixed refinement technique.
Further, let $Y_{\calP}\subset X$ be a linear subspace for every $\calP\in\textup{REF}(\calP_0)$. Then, the \textbf{$N$-term approximation error of $f$}, $N\in \N_0$, with respect to $X$ and $Y_\bullet$ is given by
\begin{align*}
    \sigma_N(f, X, Y_\bullet, \textup{REF}):=\inf\limits_{\substack{\calP\in \textup{REF}(\calP_0),\\ \#\calP-\#\calP_0\le N}} E(f, Y_\calP, X) = \inf\limits_{\substack{\calP\in \textup{REF}(\calP_0),\\ \#\calP-\#\calP_0\le N}}\inf\limits_{F\in Y_\calP}\|f-F\|_X.
\end{align*}
Now, $f$ lies in the \textbf{approximation class} $\mathbb{A}_{s,q}(X, Y_\bullet, \textup{REF})$ with respect to $s\in (0,\infty)$ and $q \in (0,\infty]$, if and only if 
\begin{align}\label{eq:approximation_class_quasi-_norm}
    \|f\|_{\mathbb{A}_{s,q}(X, Y_\bullet, \textup{REF})} := \|f\|_{X} + |f|_{\mathbb{A}_{s,q}(X, Y_\bullet, \textup{REF})}:=\|f\|_{X} + \|(2^{Ns}\sigma_{2^{N}}(f, X, Y_\bullet, \textup{REF}))_{N\in \N_0}\|_{\ell^q(\N_0)}<\infty,
\end{align}
i.e., $f\in X$ and $(2^{Ns}\sigma_{2^{N}}(f, X, Y_\bullet, \textup{REF}))_{N\in \N_0}\in \ell^{q}(\N_0)$.\footnote{For given $q\in (0,\infty)$, the sequence space $\ell^q(\N_0)$ consists of all the real-valued sequences $a:\N_0\rightarrow\R$ for which the (quasi-)norm $\|(a_n)_{n\in \N_0}\|_{\ell^q(\N_0)}:=\left(\sum\limits_{n=0}^\infty |a_{n}|^q\right)^\frac{1}{q}$ is finite. The case $q=\infty$ works with the usual modification. It is well known that $\ell^{q_1}(\N_0)\hookrightarrow\ell^{q_2}(\N_0)$, if $q_1\le q_2$.} $\mathbb{A}_{s,q}(X, Y_\bullet, \textup{REF})$ is a quasi-normed space with respect to the quasi-norm from \eqref{eq:approximation_class_quasi-_norm} according to \cite[Sect.~1.4]{Pie81}, if the sequence of sets
\begin{align}\label{eq:necessary_condition_on_approximation_spaces}
    \mathfrak{A}_N(\textup{REF}):=\{f\in Y_\calP\mid \calP\in \textup{REF}(\calP_0), ~ \#\calP-\#\calP_0 \le N\},\quad N\in \N_0,
\end{align}
fulfills the conditions from \cite[Sect.~1.2]{Pie81}.

\begin{rem}\label{rem:indeed_approximation_classes_are_quasi-Banach_teaser}
    In \cref{cor:patch_refine_fulfills_necessary_condition_on_approximation_spaces}, we will indeed show that the corresponding spaces from \eqref{eq:necessary_condition_on_approximation_spaces} fulfill the above mentioned conditions with respect to anisotropic finite elements on meshes created by the refinement technique \textup{PATCH}\_\textup{REFINE}, which we propose for the local refinement of prismatic meshes. In particular, this means that the approximation classes in our formulations of the main results of this manuscript, presented in the next section, are indeed all quasi-normed.
\end{rem}

\subsection{Main results}\label{subsect:main_results}

We assume that the initial space-time partition $\calP_0$ has tensor product structure and fulfills some labeling conditions explained in more detail in the beginning of \cref{sect:space_time_partition_and_refinement}. For its refinement, we employ the atomic refinement method $\textup{ATOMIC}\_\textup{SPLIT}$ from \cite[Sect.~4]{MSS26} which, for the sake of completeness, we present in \cref{subsect:Atomic_ref}. Using this atomic refinement method, we construct the method $\textup{PATCH}\_\textup{REFINE}$ in \cref{subsect:Patch_refine}, which allows for local prismatic refinement while preserving the mesh regularity.

Here and in the sequel, we employ $\R^+:=(0,\infty)$ as well as $\R^+_0:=[0,\infty)$. Further, for $(\alpha_1, \alpha_2), (s_1, s_2)\in \R^2$ and $R\in \{\R^+, \R_0^+\}$, we will denote $(\alpha_1, \alpha_2)\in R(s_1, s_2)$ if there is a $\lambda \in R$ such that $(\alpha_1, \alpha_2) = \lambda(s_1, s_2) $.

Our first main result is the following embedding of anisotropic Besov spaces into the corresponding approximation classes. For the reader's convenience, we state a slightly streamlined version here,  which is a little vague in view of the dependence of the embedding constants. The precise dependence on those constants is stated in \cref{thm:direct_estimates}.

\begin{main_result}[Direct estimates]\label{Main_res:direct}
        Let $p,q\in (0,\infty]$, $s_1, s_2\in (0,\infty)$ with $\frac{1}{\frac{1}{s_1}+\frac{d}{s_2}}-\frac{1}{q}+\frac{1}{p}>0$, $(\alpha_1, \alpha_2)\in \R^+_0(s_1, s_2)$ with $\alpha_i<1+\frac1p$, $i=1,2$, and $r_1, r_2\in \N_{\ge 2}$ with $r_i>\alpha_i+s_i$, $i=1,2$. Then the embeddings
    \begin{align*}
        &B^{s_1, s_2}_{q,q}(\Omega_T)\hookrightarrow\mathbb{A}_{\frac{1}{\frac{1}{s_1}+\frac{d}{s_2}}, \infty}\Big(L_p(\Omega_T), \V^{r_1, r_2}_{\bullet}, \textup{PATCH}\_\textup{REFINE}(\cdot, \cdot, d, s_1, s_2)\Big), \quad\text{if}\quad (\alpha_1, \alpha_2)=(0,0), \quad \text{and}
        \\ & B^{\alpha_1+s_1, \alpha_2+s_2}_{q,q}(\Omega_T)\hookrightarrow\mathbb{A}_{\frac{1}{\frac{1}{s_1}+\frac{d}{s_2}}, \infty}\Big(B_{p,p}^{\alpha_1, \alpha_2}(\Omega_T), \V^{r_1, r_2}_{\bullet}, \textup{PATCH}\_\textup{REFINE}(\cdot, \cdot, d, s_1, s_2)\Big), \quad\text{if}\quad (\alpha_1, \alpha_2)\neq(0,0), 
    \end{align*}
    are continuous, where the embedding constant only depends on the parameters, the dimension and measure of $\Omega_T$, $\LipProp(\Omega)$, as well as the size and mesh geometry of $\calP_0$.
\end{main_result}

\begin{rem}
For $(\alpha_1, \alpha_2)=(0,0)$, this essentially means that, for the above choice of parameters, it is possible to approximate a function in $B^{s_1, s_2}_{q,q}(\Omega_T)$ in $L_p(\Omega_T)$ by space-time finite elements of order $(r_1, r_2)$ with a decay rate of the approximation error of $(\# \calP)^{-\frac{1}{\frac{1}{s_1}+\frac{d}{s_2}}}$ on a prismatic mesh $\calP$ if the refinement algorithm $\textup{PATCH}\_\textup{REFINE}(\cdot, \cdot, d, s_1, s_2)$ is used. 
If $\alpha_1>0$ and $\alpha_2>0$, the same result holds for functions in $B^{\alpha_1+s_1, \alpha_2+s_2}_{q,q}(\Omega_T)$, when measuring the error in $B_{p,p}^{\alpha_1, \alpha_2}(\Omega_T)$.
\end{rem}

In order to prove a corresponding result in the opposite direction, we require the notion of \textbf{generalized anisotropic Besov spaces} $\widehat{B}^{s_1, s_2}_{p,q}(\Omega_T)$, $\alpha_1, \alpha_2\in \R$, which will be introduced in \cref{subsect:Multiscale_decomp} and are defined with respect to a multiscale decomposition. These spaces partially coincide with the classical anisotropic spaces from \cref{subsubsect:Besov spaces} which is proven in \cref{cor:to_thm:Embedding_generalized_into_classical_aniso_spaces}.

For our second main result, we again provide a condensed version here, which captures the main features of the inverse result, while a more precise formulation is given in \cref{thm:Inverse_estimates}.

\begin{main_result}[Inverse estimates]\label{Main_res:invers}
    Let $p \in (0,\infty]$, $q\in (0,\infty)$, $s_1, s_2\in (0,\infty)$, $(\alpha_1, \alpha_2), (s_1', s_2')\in \R_0^+(s_1, s_2)$ with $0<s_i'<s_i$, $i=1,2$,  $r_1, r_2\in \N_{\ge 2}$, and $\frac{1}{\frac{1}{s_1}+\frac{d}{s_2}}+\frac{1}{p}-\frac{1}{q}=0$. Then,
    \begin{align*}
        &\mathbb{A}_{\frac{1}{\frac{1}{s_1}+\frac{d}{s_2}}, q}\Big(L_p(\Omega_T), \V^{r_1, r_2}_{\bullet}, \textup{PATCH}\_\textup{REFINE}(\cdot, \cdot, d, s_1, s_2)\Big)\hookrightarrow \widehat{B}^{s_1', s_2'}_{q,q}(\Omega_T),
         \quad\text{if}\quad (\alpha_1, \alpha_2)=(0,0), \quad \text{and}
        \\
        &\mathbb{A}_{\frac{1}{\frac{1}{s_1}+\frac{d}{s_2}}, q}\Big(\widehat{B}^{\alpha_1, \alpha_2}_{p,p}(\Omega_T), \V^{r_1, r_2}_{\bullet}, \textup{PATCH}\_\textup{REFINE}(\cdot, \cdot, d, s_1, s_2)\Big)\hookrightarrow \widehat{B}^{\alpha_1+s_1, \alpha_2+s_2}_{q, q}(\Omega_T) ,
        \quad\text{if}\quad (\alpha_1, \alpha_2)\neq(0,0),
    \end{align*}
    hold continuously with an embedding constant that only depends on the involved parameters, the dimension, as well as the initial mesh geometry and size.
\end{main_result}

\begin{rem}
This result is the opposite to \cref{Main_res:direct}. If $(\alpha_1, \alpha_2)=(0,0)$ is considered, it shows that a function, which can be approximated in the $L_p(\Omega_T)$-(quasi-)norm on prismatic meshes $\calP$ created by iterative applications of the method $\textup{PATCH}\_\textup{REFINE}(\cdot, \cdot, d, s_1, s_2)$, with an error decay of $(\# \calP)^{-\frac{1}{\frac{1}{s_1}+\frac{d}{s_2}}}$, has to possess a corresponding (generalized) anisotropic Besov regularity, i.e., lie in the space $\widehat{B}^{s_1', s_2'}_{q,q}(\Omega_T)$;  similarly for the case $\alpha_1>0$ and $\alpha_2>0$.
\end{rem}

\section{Space-time partition and mesh refinement}\label{sect:space_time_partition_and_refinement}

Our initial covering will be given by a space-time partition $\calP_0$ of $\Omega_T$, which has tensor product structure, i.e., $\calP_0:=\calI_0\otimes \calT_0:=\{I_0\times S_0\mid I_0\in \calI_0, \ S_0 \in \calT_0\}$. Hereby, $\calI_0$ is a disjoint partition of $[0,T]$ given by $\calI_0:=\{[t_i, t_{i+1}) \mid i=0,\dots, N-2 \} \cup \{[t_{N-1}, t_N]\}$, 
with $0=t_0<t_1 \dots < t_N=T$, and $\calT_0$ is a non-overlapping, simplicial, and conforming\footnote{A simplicial triangulation $\calT$ of $\Omega$ is called \textit{conforming}, if for any $S, S'\in \calT$, either $S=S'$, $S\cap S' = \emptyset$, or $S\cap S'$ is an $r$-dimensional common face of $S$ \textbf{and} $S'$ with $r\in \{0,\dots, d-1\}$.} triangulation of $\Omega$, where we assume its elements to be closed.

Since we want to iteratively refine prisms $I\times S$, i.e., $I\subset [0,T]$ an interval, and $S\subset\Omega$ a $d$-dimensional simplex, by the aforementioned method $\textup{ATOMIC}\_\textup{SPLIT}$ (see \cref{subsect:Atomic_ref} below), we need to introduce the $d$-dimensional bisection routine, initially proposed by Maubach \cite{Mau95} and Traxler \cite{Tra97}, $\textup{BISECT}(\cdot,d)$ for simplices which extends the classical, one-dimensional bisection routine $\textup{BISECT}(\cdot, 1)$ for intervals. The interested reader may find a good summary of the method in \cite[Sect.~2 and the beginning of Sect.~4]{Ste08}. For the method to work, the simplices have to be well-labeled, which is why we additionally assume $\calT_0$ to be a \textbf{well-labeled} triangulation in the sense of \cite{DGS26} or \cite[cond.~b) in Sect.~4]{Ste08}. 

For $(\dim, R) \in \{(1, I), (d, S)\}$, we will denote the set of created intervals or simplices (via bisection) by $\textup{BISECT}(R, \dim)$, correspondingly. Iteratively, for any $n\in \N$, we will denote $R'' \in \textup{BISECT}(R,\dim)^{\otimes n}$, if $R''\in \textup{BISECT}(R', \dim) $ for some $R'\in \textup{BISECT}(R, \dim)^{\otimes(n-1)}$, whereas $\textup{BISECT}(R,\dim)^{\otimes 0}:=\{R\}$. Further, we will call $R$ of \textbf{level} $n\in \N_0$ and write $\ell(R)=n$, if there exists some $R_0\in \calI_0$ or $R_0\in \calT_0$, respectively, such that $R\in \textup{BISECT}(R_0,\dim)^{\otimes n}$.

\subsection{Atomic refinement of prisms}\label{subsect:Atomic_ref}

We will use the refinement technique $\textup{ATOMIC}\_\textup{SPLIT}(\cdot , d, s_1, s_2)$ from \cite[Sect.~4]{MSS26} with respect to $d$ and given parameters $s_1, s_2\in (0,\infty)$ iteratively in order to create refined space-time partitions $\calP$, starting with the tensor product mesh $\calP_0$. Moreover, we apply the notation for $\textit{levels}$ of elements of a space-time partition given there, which extends the one given above for $\textup{BISECT}$ to $\textup{ATOMIC}\_\textup{SPLIT}(\cdot , d, s_1, s_2)$. In this sense, the level of a prism will directly determine the level of its involved temporal and spatial parts, see \cref{Remark_Temporal_Bisection_Active_Triangulation}\ref{Remark_Temporal_Bisection_Active_Triangulation_3} below.

In order to be self-contained, we include this algorithm for the atomic refinement of prisms $I\times S$ here.

\begin{algorithm}[h]
    \caption{$\textup{ATOMIC}\_\textup{SPLIT}(I\times S, d, s_1, s_2)$}
    \label{Algorithm_ANISOTROPIC_BISECT}
    \begin{algorithmic}
    	\State $n\gets \ell(S)+1$ 
    	\State $m\gets \left\lceil \frac{ns_2}{s_1d}\right\rceil - \left\lceil\frac{(n-1)s_2}{s_1d}\right\rceil$
    	\State $S_{children} \gets \textup{BISECT}(S, d)$
    	\State $I_{children}\gets \textup{BISECT}(I, 1)\phantom{}^{\otimes m }$
    	\State \Return $I_{children}\otimes S_{children}:=\{I\times S\mid I\in I_{children},~S\in S_{children}\}$
    \end{algorithmic}
\end{algorithm}

 \begin{center}
     \begin{figure}[H]
         \begin{center}
             \includegraphics[height = 6cm]{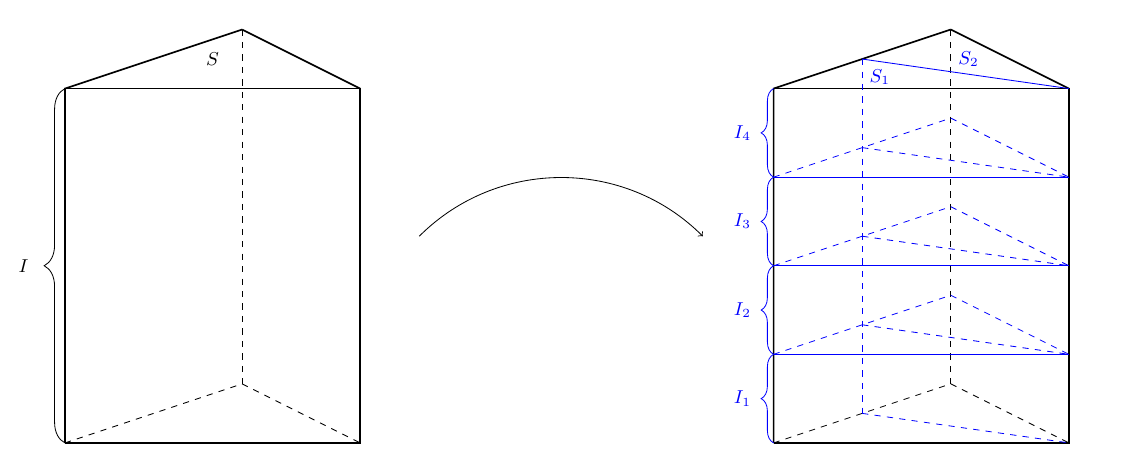}
             \captionof*{figure}{$\textup{ATOMIC}\_\textup{SPLIT}$ for $d=2$ and $s_2=4s_1=2s_1d$.}
         \end{center}
     \end{figure}
 \end{center}

 \begin{rem}\label{Remark_Temporal_Bisection_Active_Triangulation}

    \begin{enumerate}[label=(\roman*)]
    
        \item This technique has some nice properties, namely that for such $\calP$, it holds $\kappa_{\calP}\lesssim_{\,d}\kappa_{\calP_0}$ (due to the fact that the spatial $\textup{BISECT}(\cdot, d)$ method used bounds the minimum head angle of the spatial simplices from below\footnote{In \cite[App.~A]{DGS26}, an explicit bound has been established for a parameter related to $\kappa_{\calP}$. }) and $a(\calP)\lesssim a(\calP_0)$ as shown in \cite[Lem.~4.1]{MSS26}. This is relevant, since $\kappa_{\calP}$ and $a(\calP)$ are incorporated into the inequality constants of Whitney's inequality for approximation with discontinuous anisotropic finite elements and the corresponding direct estimate result, see \cite[Thm.~1.2 and Thm.~1.5]{MSS26}. Generalizations of this result will be proven and employed in this article, which also require a control of  mentioned constants.\label{Remark_Temporal_Bisection_Active_Triangulation_1}
        
        \item Further, we assume that $\textup{BISECT}(\cdot, 1)$ splits half-open intervals in two half-open intervals and closed intervals in one half-open interval and one closed interval. This is to guarantee that any temporal partition of $[0,T]$ created from $\calI_0$ by iterative applications of $\textup{BISECT}(\cdot,1)$ again only contains pairwise disjoint elements. This guarantees that for $t\in [0,T]$, the \textbf{active spatial triangulation at time $t$}, given by 
        \begin{align*}
            \calT(t, \calP):=\{S\mid I\times S\in \calP, \ t\in I\} 
        \end{align*}
        is well-defined and indeed a non-overlapping triangulation of $\Omega$. Similarly, $\textup{BISECT}(S, d)$ shall denote a set of two closed $d$-dimensional simplices if $S$ is a $d$-dimensional closed simplex.\label{Remark_Temporal_Bisection_Active_Triangulation_2}

        \item \label{Remark_Temporal_Bisection_Active_Triangulation_3}Additionally, according to \cite[end of page p.~31]{MSS26}, we know that for a prism  $I\times S$,
        \begin{align*}
            \ell(I\times S)= k\in \N_0 \quad \text{if and only if }\quad \ell(I)=\left\lceil \frac{ks_2}{s_1d}\right\rceil \:\:\text{and}\:\: \ell(S)=k.
        \end{align*}
        \item We finally observe that always, $0 \le m \le \lceil \frac{s_2}{s_1 d}\rceil+2$.
    \end{enumerate}
\end{rem}

In the next lemma we see that the size of an element can be estimated by its level and vice-versa, which corresponds to a similar result for the $\textup{BISECT}$ method given in the beginning of \cite[Sect.~4]{Ste08}.

\begin{lem}\label{lemma relationship diameter level - BISECT}
    Let $\curlybinom{I}{S}$ be a $\curlybinom{1}{d}$-dimensional simplex derived from an element $\curlybinom{I_0}{S_0}\in\curlybinom{\calI_0}{\calT_0}$ by finitely many applications of the method $\curlybinom{\textup{BISECT}(\cdot,1)}{\textup{BISECT}(\cdot,d)}$. Then
    \begin{align*}
        |I|=\diam(I)\sim_{\, \mu_1(\calP_0), \mu_2(\calP_0)} 2^{-\ell(I)}\quad\text{and}\quad \diam(S)\sim_{\,d, \kappa_{\calP_0}, \mu_3(\calP_0), \mu_4(\calP_0)} 2^{-\frac{\ell(S)}{d}},
    \end{align*}
    where we have used 
    \begin{align*}
        \mu_1(\calP_0):=\max\limits_{I\times S\in \calP_0}|I|,\quad
        \mu_2(\calP_0):=\min\limits_{I\times S\in \calP_0}|I|,\quad
        \mu_3(\calP_0):=\max\limits_{I\times S\in \calP_0}|S|,\quad \text{and}\quad \mu_4(\calP_0):=\min\limits_{I\times S\in \calP_0}|S|.
    \end{align*}
\end{lem}
\begin{rem}
    From now on, we will further use the abbreviations
    \begin{align*}
        \mu_t(\calP_0):=\{\mu_1(\calP_0), \mu_2(\calP_0)\},\quad \mu_{\bm{x}}(\calP_0):=\{\mu_3(\calP_0), \mu_4(\calP_0)\},\quad\text{and}\quad \mu(\calP_0):=\{\mu_1(\calP_0), \mu_2(\calP_0),\mu_3(\calP_0), \mu_4(\calP_0)\}.
    \end{align*}
\end{rem}
\begin{proof}
    For the temporal part, this can directly be seen by $\diam(I)=|I|=2^{-\ell(I)}|I_0|\sim_{\,\mu_t(\calP_0)} 2^{-\ell(I)}$. Since the $\textup{BISECT}(\cdot,d)$ method controls the minimum head angle of the simplices, we first have $\diam(S)\sim_{\,d ,\kappa_{\calP_0}}|S|^\frac{1}{d}$. Together with \cite[Eq.~(4.1)]{Ste08} and its surrounding text, we can derive the existence of two positive constants $C_1$ and $C_2$ such that 
    \begin{align*}
        C_1 \, 2^{-\frac{\ell(S)}{d}}\le |S|^\frac{1}{d}\sim_{\,d, \kappa_{\calP_0}} \diam(S)\le C_2 \, 2^{-\frac{\ell(S)}{d}}.
    \end{align*}
    This paragraph at the beginning of \cite[Sect.~4]{Ste08}, together with \cite[Thm.~2.1 and the explanations given below]{Ste08}, make it clear that $C_1$ and $C_2$ only depend on $d, \kappa_{\calP_0}$, and $\mu_{\bm{x}}(\calP_0)$.
\end{proof}

Now we can show the corresponding result for the $\textup{ATOMIC}\_\textup{SPLIT}$ method with respect to the \textbf{anisotropic metric} $\xi$,\footnote{$\xi$ is indeed a metric, since $|\cdot|^\gamma$ is subadditive for $\gamma\in (0,1]$. In case of $\frac{s_2}{s_1} = 2$, $\xi$ is equivalent (up to a constant) to the \textbf{parabolic distance}, used in \cite{AGI10, AG12} in the context of the heat equation.} defined via
\begin{align*}
    \xi\big((t,\bm{x}),(s,\bm{y})\big):=\left(|t-s|^{2\min\left(1,\frac{s_1}{s_2}\right)}+|\bm{x}-\bm{y}|^{2\min\left(1,\frac{s_2}{s_1}\right)}\right)^\frac12.
\end{align*}
Accordingly, we will denote the diameter of a set with respect to the above metric via $\diam_\xi$.

\begin{cor}\label{corollary relationship diameter level - ATOMIC_SPLIT}
    Let $I\times S$ be created from an element $I_0\times S_0\in \calP_0$ by finitely many applications of the method $\textup{ATOMIC}\_\textup{SPLIT}(\cdot, d, s_1, s_2)$. Then
    \begin{align*}
        \diam(I)^{\min\left(\frac{s_1}{s_2},1\right)}\sim_{\,d,s_1, s_2, \mu(\calP_0)} 2^{-\min\left(\frac{s_2}{s_1},1\right)\frac{\ell(I\times S)}{d}}\sim_{\,d,s_1, s_2,\kappa_{\calP_0}, \mu(\calP_0)}\diam(S)^{\min\left(\frac{s_2}{s_1},1\right)}
    \end{align*}
    and
    \begin{align*}
        \diam_{\xi}(I\times S)\sim_{\, d, s_1, s_2 ,\kappa_{\calP_0}, \mu(\calP_0)} 2^{-\min\left(\frac{s_2}{s_1},1\right)\frac{\ell(I\times S)}{d}}.
    \end{align*}
\end{cor}
\begin{proof}
    According to \cref{Remark_Temporal_Bisection_Active_Triangulation}\ref{Remark_Temporal_Bisection_Active_Triangulation_3}, the method generates prisms with $\ell(I)=\left\lceil\frac{\ell(I\times S)s_2}{s_1d}\right\rceil\sim_{\,d,s_1, s_2 } \ell(I\times S)$ and $\ell(S)=\ell(I\times S)$. Therefore, together with \cref{lemma relationship diameter level - BISECT}, we obtain
    \begin{align*}
        \diam(I)^{\min\left(\frac{s_1}{s_2},1\right)}\sim_{\,d,s_1, s_2,\mu(\calP_0)}2^{-\min\left(\frac{s_1}{s_2},1\right)\frac{s_2\ell(I\times S)}{s_1d}}=2^{-\min\left(\frac{s_2}{s_1},1\right)\frac{\ell(I\times S)}{d}}\sim_{\,d,s_1, s_2,\kappa_{\calP_0}, \mu(\calP_0)} \diam(S)^{\min\left(\frac{s_2}{s_1},1\right)}
    \end{align*}
    and thus
    \begin{align*}
        \diam_{\xi}(I\times S)&= \left(\diam(I)^{2\min\left(\frac{s_1}{s_2},1\right)}+\diam(S)^{2\min\left(\frac{s_2}{s_1},1\right)}\right)^\frac{1}{2}\sim_{\, d,s_1, s_2,\kappa_{\calP_0}, \mu(\calP_0)}  2^{-\min\left(\frac{s_2}{s_1},1\right)\frac{\ell(I\times S)}{d}}. \qedhere
    \end{align*}
\end{proof}

Note that this means that $\xi$ is explicitly chosen such that the temporal and spatial diameters of $I\times S$ are comparable with respect to it, which is false for $s_1\neq s_2$ with respect to the Euclidean metric, though both induce the same topology.

\subsection{Refinement of patches preserving mesh regularity}\label{subsect:Patch_refine}

As mentioned above, our atomic refinement method preserves the shape regularity of the prisms on the level of an individual prism. Nonetheless, arbitrary atomic refinements of prisms in a space-time partition could lead to other undesirable results. Therefore, we have to construct an algorithm that refines a given element and a certain number of surrounding elements in order to maintain desired properties. Our goal is to create partitions that guarantee the following \textbf{regularity conditions}:

\begin{enumerate}[label=(\roman*)]
    \item \label{Enumeration desired properties of the mesh 1} \textbf{Conformity in space}: $\calT(t,\calP)$ is conforming for every $t\in [0,T]$, 
    \item\label{Enumeration desired properties of the mesh 2} Validity of a \textbf{$1$-irregular rule in time}: This means that if two elements of $\calP$, let us call them $I\times S$ and $I'\times S'$, touch in time, i.e., their closures share a $d$-dimensional hyperface in $\R^{d+1}$ and $\# (\overline{I}\cap \overline{I'}) = 1$, then their levels differ by at most one. 
\end{enumerate}
\begin{rem}
    If $d=1$, the conformity condition in space is always fulfilled. To avoid arbitrary differences of levels of elements that touch in space, we instead require a $1$-irregular rule in space in this case as well.
\end{rem}

In order to create such partitions, we adapt the idea of the refinement strategy for spatial triangulations presented in \cite[Sect.~4-5]{Ste08} that preserves initial conformity to our needs. 

If we want to refine an element of a space-time partition, we have to consider the elements surrounding the element that have to be refined as well in order to maintain conformity in space and the 1-irregular rule in time (and space if $d=1$). This leads to the observation that whenever we want to refine an element of $\calP$ we have to consider all the elements that have to be refined as well in order  to not violate our desired properties.

Therefore, in the case $d\ge 2$, let $\textup{RE}(S)$ be the \textbf{refinement edge} of $S$, i.e., the edge of $S$ that gets bisected in case of an application of $\textup{BISECT}(S, d)$, for $I\times S \in \calP$. For the sake of notation, we will put $\textup{RE}(S):=\emptyset$ in case of $d=1$.

\begin{defi}\label{Def - Necessary Refinement of neighbors}
    We define the set of \textbf{necessary refinements of neighbors} of $I\times S\in \calP$ as follows
    \begin{align*}
        \calN(I\times S, \calP):=\calN_t(I\times S, \calP)\cup \calN_{\bm{x}}(I\times S, \calP).
    \end{align*}
    Here, 
    \begin{align*}
        \calN_t(I\times S, \calP):=\left\{I'\times S'\in \calP\setminus\{I\times S\}\mid \# \left(\overline{I}\cap \overline{I'}\right)=1, \  \textup{dim}(S\cap S')=d,\  \text{and}\  \ell(I'\times S')=\ell(I\times S)-1\right\}
    \end{align*}
    and
    \begin{align*}
        \calN_{\bm{x}}(I\times S, \calP)&:=
            \left\{I'\times S'\in \calP\setminus\{I\times S\}\mid  |I\cap I'|>0, \  1\le\textup{dim}(S\cap S')\le d-1, \  \text{and} \  \textup{RE}(S)\subset S'\right\}, \text{ if }d\ge 2,
            \\
        \calN_{\bm{x}}(I\times S, \calP)&:=
            \left\{I'\times S'\in \calP\setminus\{I\times S\}\mid  |I\cap I'|>0, \  \# (S\cap S')=1, \  \text{and} \  \ell(I'\times S')=\ell(I\times S)-1 \right\}, \text{ if $d=1$}.
    \end{align*}
\end{defi}

\begin{rem}
    Note that $\calN_t(I\times S, \calP)$ and $\calN_{\bm{x}}(I\times S, \calP)$ are always disjoint sets.
\end{rem}

If the space-time partition fulfills the mentioned regularity conditions, the set of necessary refinements can be characterized even further. The following lemma extends the result from \cite[Cor.~4.6]{Ste08} to a time-dependent setting.

\begin{lem}\label{Lemma_Properties_of_Neighboring_Space_Time_Elements}
    Let $\calP_0$ satisfy the assumptions stated at the beginning of \cref{sect:space_time_partition_and_refinement} and assume that 
    $\calP$ is a partition obtained from $\calP_0$ by a finite number of applications of $\textup{ATOMIC}\_\textup{SPLIT}$ and it is conforming in space, additionally $1$-irregular in space if $d=1$, and $1$-irregular in time. Then for \mbox{$I\times S\in\calP$}, the prism $I'\times S'\in \calN_{\bm{x}}(I\times S, \calP)$ fulfills exactly one of the following properties:
	\begin{itemize}
		\item $\ell(I'\times S')=\ell(I\times S)$, $I=I'$, and $\textup{RE}(S)=\textup{RE}(S')$.
		\item $\ell(I'\times S')=\ell(I\times S)-1$, $I\subset I'$, and there exists
		$I''\times S'' \in \textup{ATOMIC}\_\textup{SPLIT}(I'\times S', d, s_1, s_2)$
		with $\textup{RE}(S)=\textup{RE}(S'')$ and $I=I''$.
	\end{itemize}
    Further, if $I'\times S'\in \calN_{t}(I\times S, \calP)$, then $S\in \textup{BISECT}(S',d)$ is fulfilled.
\end{lem}

\begin{figure}[H]
    \begin{subfigure}{0.49\linewidth} 
        \includegraphics[width=\linewidth]{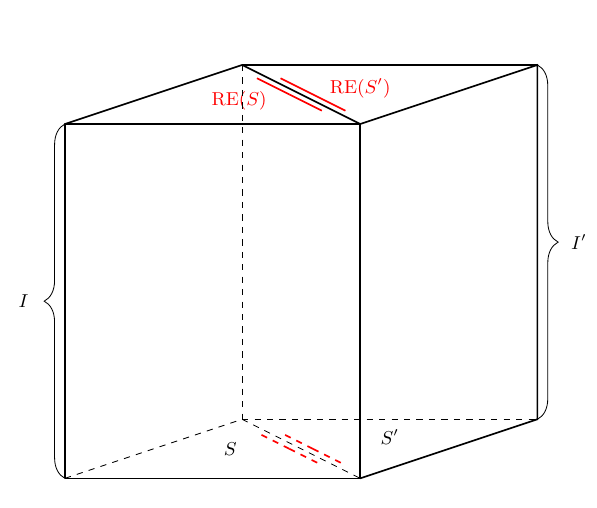}
    \end{subfigure}
    \hfill
    \begin{subfigure}{0.49\linewidth} 
        \includegraphics[width=\linewidth]{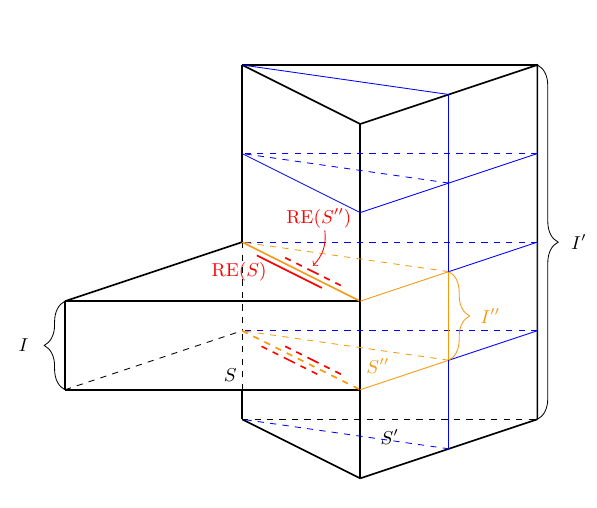}
    \end{subfigure}
    \captionof*{figure}{Illustration of the two possible cases for $I'\times S'\in \calN_{\bm{x}}(I\times S, \calP)$ for $d = 2$ and $s_2 = 4s_1 = 2s_1d$.}
\end{figure}

\begin{figure}[H]
	    \begin{center}
		    \includegraphics[height=0.42\linewidth]{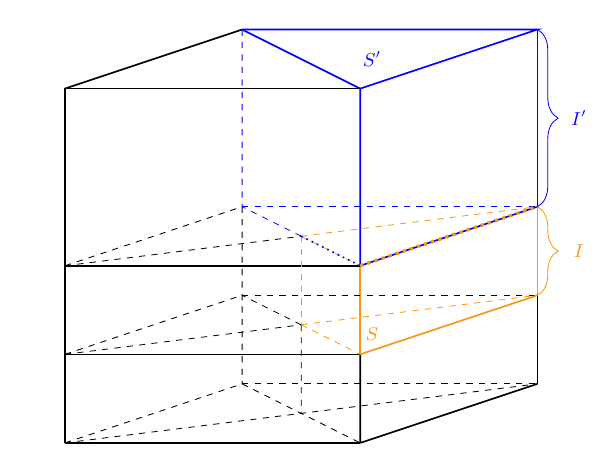}
	    \end{center}
        \captionof*{figure}{Illustration of $I'\times S'\in \calN_{t}(I\times S, \calP)$ for $d = 2$ and $s_2 = 2s_1 = s_1d$.}
    \end{figure}

\begin{proof}
    First, we consider the case of $I'\times S'\in \calN_{\bm{x}}(I\times S, \calP)$. Let $t_0\in I\cap I'$. Then $S,S'\in \calT(t_0,\calP)$ with RE$(S)\subset S'$. Thus \cite[Cor.~4.6]{Ste08} for $d\ge 2$ and the $1$-irregularity in space for $d=1$, respectively, imply that either 
	\begin{align*}
		\ell(I'\times S')=\ell(S')=\ell(S)=\ell(I\times S) \quad \text{or} \quad \ell(I'\times S')=\ell(S')=\ell(S)-1=\ell(I\times S)-1.
	\end{align*}
	Since $|I\cap I'|>0$, and $\calP$ has been derived from $\calP_0$, which has tensor product structure, by finitely many applications of $\textup{ATOMIC}\_\textup{SPLIT}(\cdot, d, s_1, s_2)$, we know that there is $I_0\in \calI_0$ such that \mbox{$I \in \textup{BISECT}(I_0,1)^{\otimes \ell(I)}$} and \mbox{$I' \in \textup{BISECT}(I_0,1)^{\otimes \ell(I')}$}. On the one hand, assume $\ell(I'\times S')=\ell(I\times S)$, which yields $\ell(I') = \ell(I)$ and, in particular, implies $I=I'$. Furthermore, in this case $\ell(S)=\ell(S')$ and thus \cite[Cor.~4.6]{Ste08} shows $\textup{RE}(S)=\textup{RE}(S')$. Whereas on the other hand consider $\ell(I'\times S')=\ell(I\times S)-1$, which gives $\ell(I') \le \ell(I)$, and thus leads to $I\subset I'$. Further, for any $I''\times S''\in \textup{ATOMIC}\_\textup{SPLIT}( I'\times S', d, s_1, s_2)$, $\ell(I''\times S'')=\ell(I\times S)$, i.e., $\ell(I'') = \ell(I)$ and $I, I''\in\textup{BISECT}(I',1)^{\otimes \ell(I) - \ell(I')}$. Among those $I''$, there must be one with $I=I''$. Further, in this case $\ell(S')=\ell(S)-1$ holds true, which implies the existence of $S''\in\textup{BISECT}(S',d)$ with $\textup{RE}(S)=\textup{RE}(S'')$ according to \cite[Cor.~4.6]{Ste08}, since $S, S'\in \calT(t_0,\calP)$.
    
    Lastly, we will consider $I'\times S'\in \calN_{t}(I\times S, \calP)$. The initial tensor product structure and $\ell(S')=\ell(I'\times S')= \ell(I\times S)-1= \ell(S)-1$ imply the asserted $S\in \textup{BISECT}(S',d)$. Thus, the proof is complete.
\end{proof}

\begin{cor}\label{Corollary_to_Lemma_Properties_of_Neighboring_Space_Time_Elements}
    Under the assumptions of the above lemma, it holds that, $\#\calN(I\times S, \calP)\lesssim_{\, d, s_1, s_2, \kappa_{\calP_0}}1$.
\end{cor}
\begin{proof}
    This is a direct consequence of the above lemma combined with $\kappa_\calP\sim_{\,d}\kappa_{\calP_0}$ due to \cite[Lem.~4.1]{MSS26}.
\end{proof}

Now we propose a recursive refinement strategy that exploits the results of \cref{Lemma_Properties_of_Neighboring_Space_Time_Elements}. This algorithm follows the idea of the \textbf{refine}-algorithm proposed in the beginning of \cite[Sect.~5]{Ste08}. 

\begin{algorithm}[H]
\caption{$\textup{PATCH}\_\textup{REFINE}(\calP, I\times S, d, s_1, s_2)$}
\begin{algorithmic}
	\State $K\gets \emptyset$ \ \% elements of $\calP$ that should be refined in the end
	\State $F\gets \{I\times S\}$ \ \% elements that have to be considered for refinement to ensure desired properties
    
	\While{$F\ne \emptyset$}
		\State $F_{new} \gets \emptyset$
		\For{$I'\times S'\in F$}
			\For{$I'' \times S'' \in \calN(I'\times S',\calP)\setminus(K\cup F)$}
				\If{$\ell(I'\times S')=\ell(I''\times S'')$}	
					\State $F_{new}\gets F_{new}\cup \{I''\times S''\}$
				\Else 
                    \State $\calP_{old}\gets \calP$ 
					\State $\calP\gets\textup{PATCH}\_\textup{REFINE}(\calP, I''\times S'', d, s_1, s_2)$
                    \If{$I'' \times S'' \in \calN_{\bm{x}}(I'\times S',\calP_{old})$}
    					\For{$I'''\times S'''\in \textup{ATOMIC}\_\textup{SPLIT}( I''\times S'', d, s_1, s_2$)
                        }
    						\If{$I'''\times S'''\in \calN(I'\times S',\calP)$}
    							\State $F_{new}\gets F_{new}\cup \{I'''\times S'''\}$
    						\EndIf
    					\EndFor			
                    \EndIf
				\EndIf
			\EndFor
		\EndFor
		\State $K\gets K\cup F$
		\State $F\gets F_{new}$
	\EndWhile
	\For{$\tilde{I}\times \tilde{S} \in K$}
		\State $\calP\gets \calP \cup\textup{ATOMIC}\_\textup{SPLIT}(\tilde{I}\times \tilde{S}, d, s_1, s_2)\setminus\{\tilde{I}\times \tilde{S}\}$ 
	\EndFor
	\State \Return $\calP$
\end{algorithmic}
\end{algorithm}

\begin{thm}\label{Theorem - conformity 1-reg Patch refine}
    Let $\calP$ fulfill the requirements from \cref{Lemma_Properties_of_Neighboring_Space_Time_Elements} and $I\times S\in \calP$. Then, the space-time partition $\calP':=\textup{PATCH}\_\textup{REFINE}(\calP, I\times S, d, s_1, s_2)$ does as well and the minimum amount of applications of $ \textup{ATOMIC}\_\textup{SPLIT}( \cdot, d, s_1, s_2)$ has been used in order to atomically refine $I\times S$ and simultaneously maintain these properties. In particular, the algorithm terminates. Further, any newly created element has a level less or equal to $\ell(I\times S)+1$.
\end{thm}
\begin{rem}
    This result is the analogous to \cite[Thm.~5.1]{Ste08}, but for space-time prismatic meshes. 
\end{rem}

\begin{proof}
    We will prove this by induction over $\ell(I\times S)$. In general, we will denote all the sets in the algorithm with index $k\in \N$ to indicate that they are created in the $k$-th loop of the \textbf{while}-procedure. (In fact, only $k=1,2$ will be necessary.)
    
    First, consider the case $\ell(I\times S)=0$. Then $I_1'\times S_1'=I\times S$ and $I_1''\times S_1''\in \calN(I\times S,\calP)$. \cref{Lemma_Properties_of_Neighboring_Space_Time_Elements} now tells us that $\ell(I''_1\times S''_1)=\ell(I\times S)=0$, since the other case $\ell(I''_1\times S''_1)=\ell(I\times S)-1 = -1$ is impossible. Thus, the \textbf{else}-clause is never triggered in this loop and it terminates due to the boundedness of $\#\calN(I\times S, \calP)$, as stated in \cref{Corollary_to_Lemma_Properties_of_Neighboring_Space_Time_Elements}, with $K_1=\{I\times S\}$ and $F_1=\calN(I\times S, \calP)=\calN_{\bm{x}}(I\times S, \calP)$. 
    \\For $d=1$, this set is already empty, for $d=2$ it might be if $S$ is at the boundary of $\Omega$. Nevertheless, we will proceed as if a second loop is always entered, without update of $K$ and $F$, so we can avoid to distinguish cases. In the second loop (we may now assume $d\ge 2$ temporarily), $I_2'\times S_2'\in F_1=\calN(I\times S, \calP)$ which in particular implies $\ell(I_2'\times S_2')=0$ as above. Furthermore, \cref{Lemma_Properties_of_Neighboring_Space_Time_Elements} now even guarantees that $I_2'=I$ and $\textup{RE}(S_2')=\textup{RE}(S)$. We will show that no $I_2''\times S_2''$ can be chosen, since $\calN(I_2'\times S_2',\calP)\setminus (K_1\cup F_1)=\emptyset $. Suppose $I_2''\times S_2''\in \calN(I_2'\times S_2',\calP)\setminus (K_1\cup F_1) $ exists. Similar to before, \cref{Lemma_Properties_of_Neighboring_Space_Time_Elements} implies $I_2''=I_2'=I$ and $\textup{RE}(S_2'')=\textup{RE}(S_2')=\textup{RE}(S)$, therefore $I_2''\times S_2''\in \calN(I\times S, \calP)\subset K_1 \cup F_1$, which is a contradiction. Thus, the loop terminates, again due to \cref{Corollary_to_Lemma_Properties_of_Neighboring_Space_Time_Elements}, with $K_2=\calN(I\times S, \calP)\cup\{I\times S\}$ and $F_2=0$, therefore the whole \textbf{while}-clause, and thus the whole algorithm, terminates. Finally for any $d\in \N$, $K=K_2=\calN(I\times S, \calP)\cup \{I\times S\}$ and for all $\tilde{I}\times \tilde{S}\in K$, $\tilde{I}=I$ and $\textup{RE}(\tilde{S})=\textup{RE}\left(S\right)$ hold true and $\calN(\tilde{I}\times \tilde{S}, \calP)\subset K $. Therefore, conformity in space as well as 1-regularity in time (and space if $d=1$), are fulfilled by $\calP'$. During the procedure, every newly created prism has been created from an element of $K$ by $\textup{ATOMIC}\_\textup{SPLIT}(\cdot, d, s_1, s_2)$, and all the elements of $K$ are of level $\ell(I\times S)=0$, thus newly created elements are all of level $\ell(I\times S)+1=1$ as claimed.

    In fact, the situation $\ell(I\times S)=n\in \N$ works very similar, where we assume that the assertion has already been shown for $n-1$. Then again, $I_1'\times S_1'=I\times S$ and $I_1''\times S_1''\in \calN(I\times S,\calP)$. Without loss of generality, we first consider all the cases, where $\ell(I_1''\times S_1'')\ne \ell(I_1'\times S_1') = \ell(I\times S)$, i.e, the situation where the \textbf{else}-case is entered. Then, in particular $\ell(I_1''\times S_1'') = \ell(I\times S)-1=n-1$ due to \cref{Lemma_Properties_of_Neighboring_Space_Time_Elements}, therefore by the induction assumption, $\textup{PATCH}\_\textup{REFINE}(\calP, \cdot, d, s_1, s_2)$ terminates with a space-time partition that fulfills the required grid properties. The space-time partition that has been created after all these calls of the content of the \textbf{else}-clause, will be called $\calP_1$. According to \cref{Lemma_Properties_of_Neighboring_Space_Time_Elements}, some newly created elements in $\calP_1$ now share the level, time-interval, and refinement edge of $I\times S$ (if $d\ge 2$), therefore they are added to $F_{new,1}$ to avoid non-conformity in space in the final space-time partition $\calP'$. Now, the elements $I_1''\times S_1''$ with $\ell(I_1''\times S_1'')= \ell(I\times S)$, which are in $\calN_{\bm{x}}(I\times S, \calP)\cap\calN_{\bm{x}}(I\times S, \calP_1)$, in particular, are considered. According to \cref{Lemma_Properties_of_Neighboring_Space_Time_Elements}, they also share interval and refinement edge with $I\times S$, such that the first \textbf{while}-loop terminates with $K_1=\{I\times S\}$ and $F_1=\calN(I\times S, \calP_1)=\calN_{\bm{x}}(I\times S, \calP_1)$. In particular, note that a call of the atomic refinement routine in the \textbf{else}-case only refines elements of order strictly less than $n=\ell(I\times S)$. Therefore, none of those elements that have been added to $F_{new,1}$ by the \textbf{if}-case have already been refined by the possibly recursive calls of the atomic refinement routine in the \textbf{else}-case.
    \\Again, there are cases, where the algorithm now directly terminates, for example if $d=1$. For technical reasons, we will again always assume that a second loop is entered (where applicable without update of $K$ and $F$). Nevertheless, we can now temporarily assume $d\ge 2$ without loss of generality. Now the second loop requires the choice of $I_2'\times S_2'\in \calN(I\times S, \calP_1)$ which implies $\ell(I_2'\times S_2')=\ell(I\times S)$ by the construction of $\calP_1$ and therefore $I_2' = I$ and $\textup{RE}(S_2')=\textup{RE}(S)$ due to \cref{Lemma_Properties_of_Neighboring_Space_Time_Elements}. Further, $I_2''\times S_2''\in \calN(I_2'\times S_2', \calP_1)\setminus(F_1\cup K_1)$ has to be chosen. Since $I_2''\times S_2''\in \calN_{\bm{x}}(I_2'\times S_2', \calP_1)$ implies $0<|I_2''\cap I_2'|=|I_2''\cap I|$ and $S_2''\supset \textup{RE}(S_2')=\textup{RE}(S)$, which yields $I_2''\times S_2''\in \calN_{\bm{x}}(I\times S, \calP_1)\subset K_1\cup F_1$, we know that necessarily $I_2''\times S_2''\in \calN_t(I_2'\times S_2', \calP_1)$. As above, these elements have $\ell(I_2''\times S_2'')=\ell(I_2'\times S_2')-1=\ell(I\times S)-1= n-1$, therefore $\textup{PATCH}\_\textup{REFINE}(\calP, \cdot, d, s_1, s_2)$ preserves the required conditions by induction assumption. After all these calls of the mentioned method in the \textbf{else}-case, which is always triggered, we obtain a space-time partition $\calP_2$ fulfilling the required properties and $\calN(I_2'\times S_2', \calP_2)\subset K_1\cup F_1$ for all $I_2'\times S_2'\in F_1\cup K_1$. Since the second loop now also terminates with $K_2=F_1 \cup K_1$ and $F_2=\emptyset$, the \textbf{while}-clause terminates, and thus the algorithm. Due to the above mentioned property, the last update step for $\calP'$ guarantees exactly the required properties. In this case, every new element has been either created by a call of $\textup{PATCH}\_\textup{REFINE}(\cdot, I''\times S'', d, s_1, s_2)$ to some $I''\times S''$ with $\ell(I''\times S'')=\ell(I\times S)-1$, which inductively only creates elements of level at most $\ell(I''\times S'')+1=\ell(I\times S)$, or from an element of $K$ by $\textup{ATOMIC}\_\textup{SPLIT}(\cdot, d, s_1, s_2)$ which, as in the case $n=0$, implies that they have a level of $\ell(I\times S)+1$.
    
    The minimality property is clearly fulfilled due to the construction of the algorithm, which concludes the proof.
\end{proof}

\begin{rem}
    During the course of this work, we found that illustrations play an important role in developing an intuitive understanding of how the algorithm $\textup{PATCH}\_\textup{REFINE}(\mathcal{P}, I \times S, d, s_1, s_2)$ operates in different situations. We therefore present several corresponding sketches, which we hope will also be helpful to the reader. In particular, we distinguish between the cases $d \ge 2$ and $d = 1$.
    
    \underline{\textup{Case $d\ge 2$}:}
    \begin{center}
        \begin{figure}[H]
            \begin{center}
                \includegraphics[width=\linewidth]{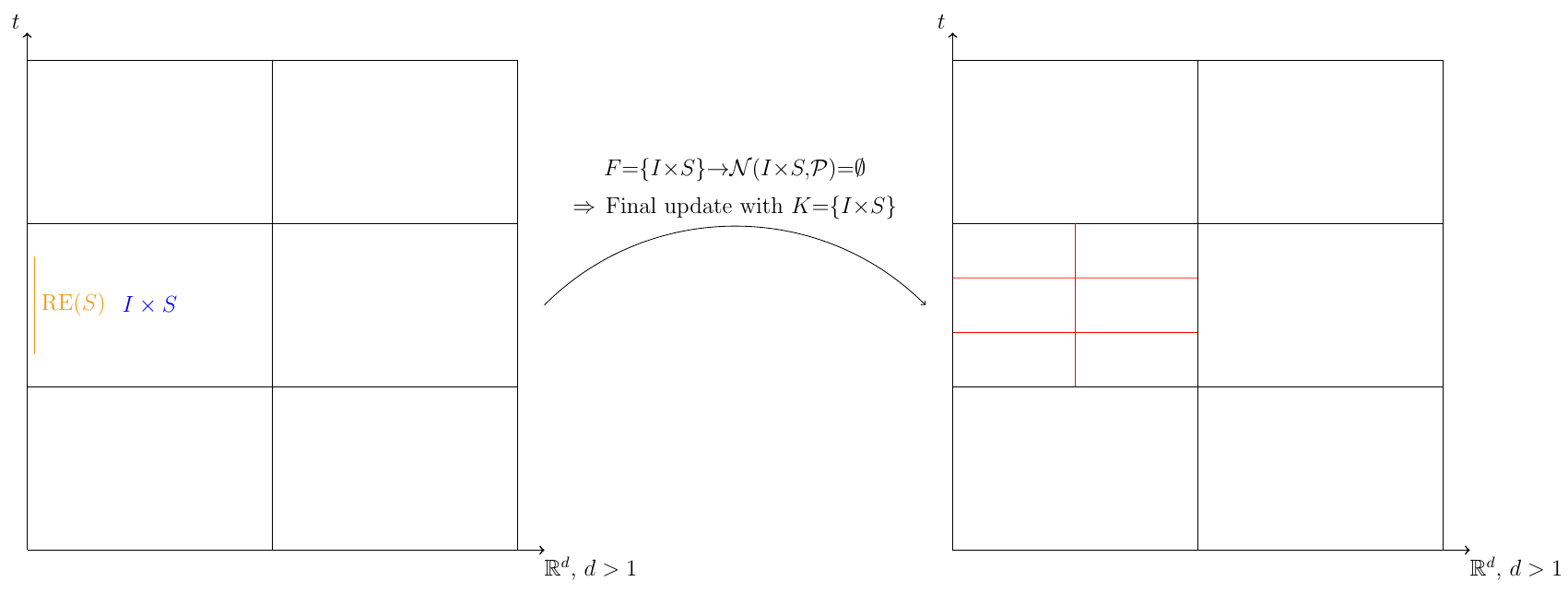}
            \end{center}
        \end{figure}
    \end{center}
    \begin{center}
        \begin{figure}[H]
            \begin{center}
                \includegraphics[width=\linewidth]{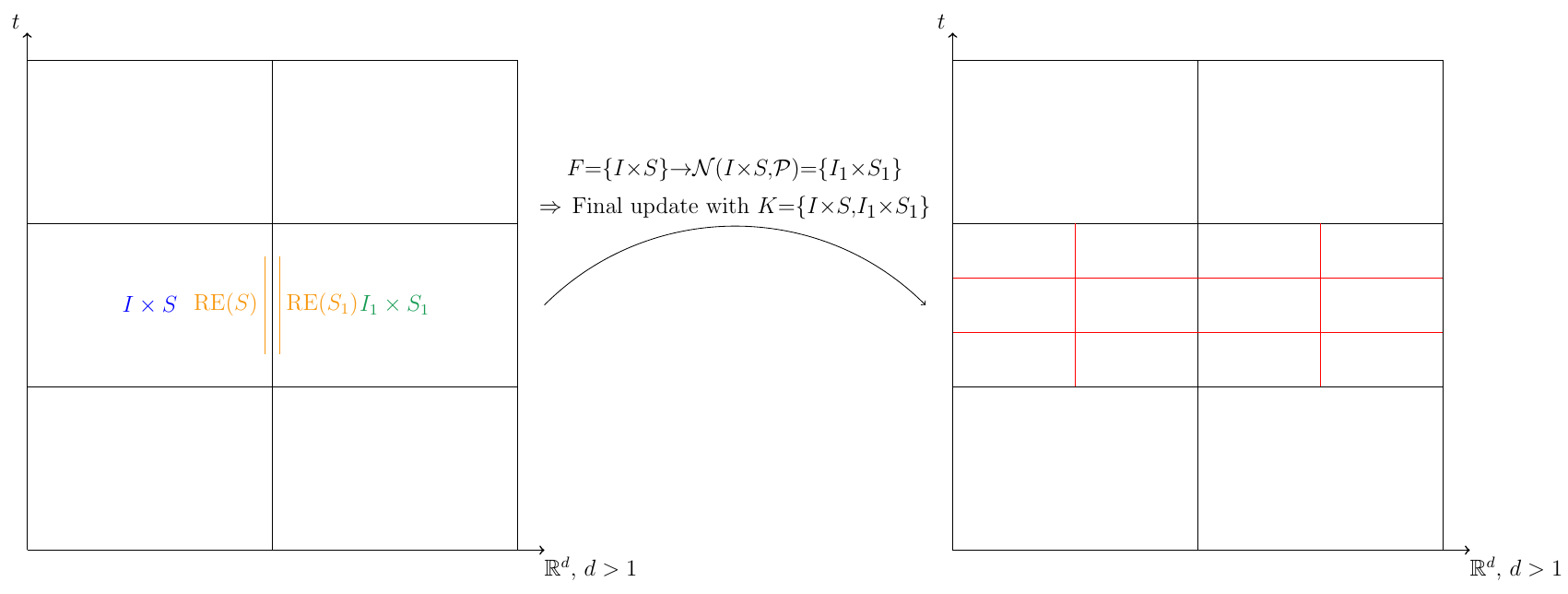}
            \end{center}
        \end{figure}
    \end{center}
    \begin{center}
        \begin{figure}[H]
            \begin{center}
                \includegraphics[width=\linewidth]{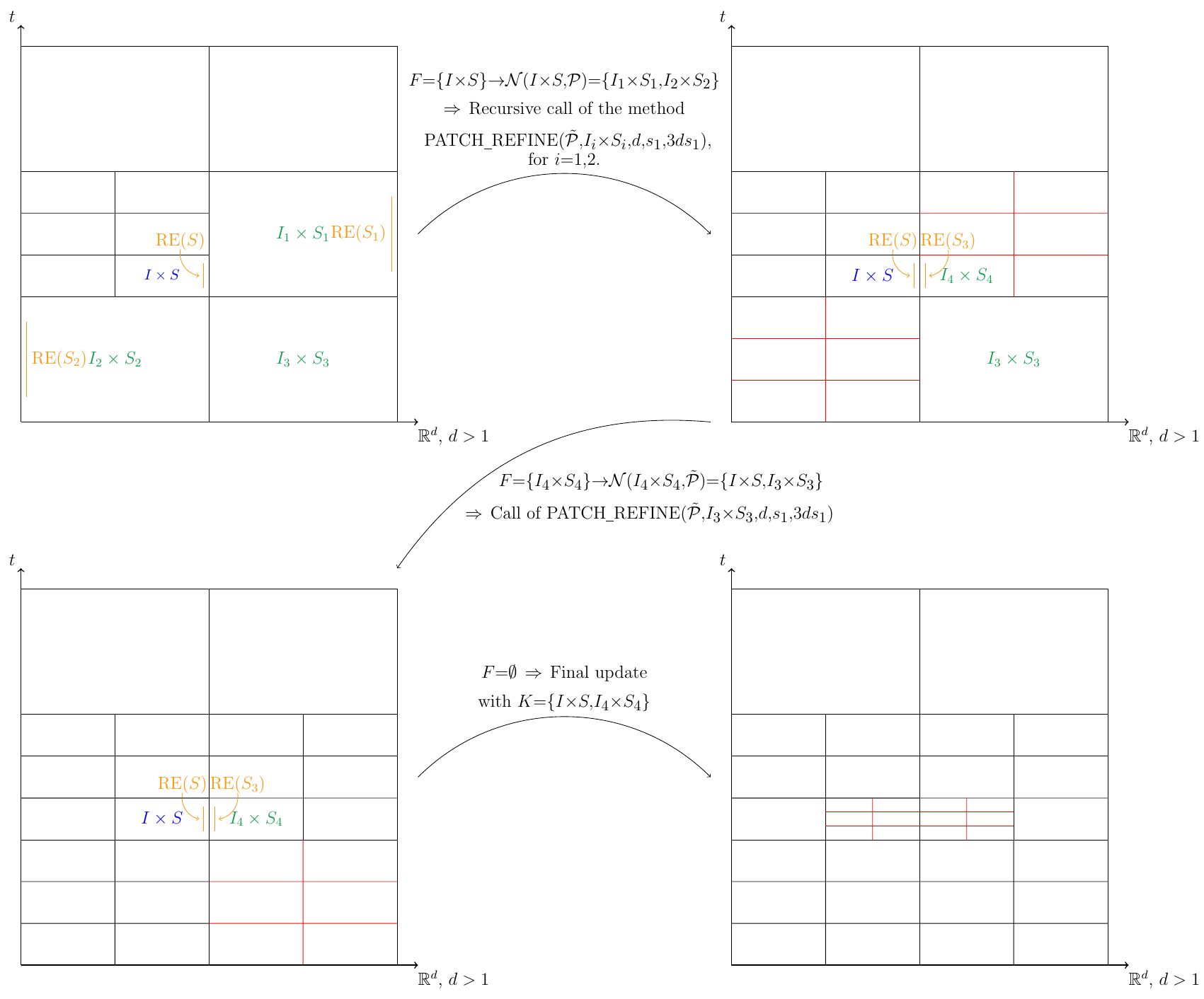}
            \end{center}
        \end{figure}
    \end{center}
   
    \underline{\textup{Case $d=1$:}}
    \begin{center}
        \begin{figure}[H]
            \begin{center}
                \includegraphics[width=\linewidth]{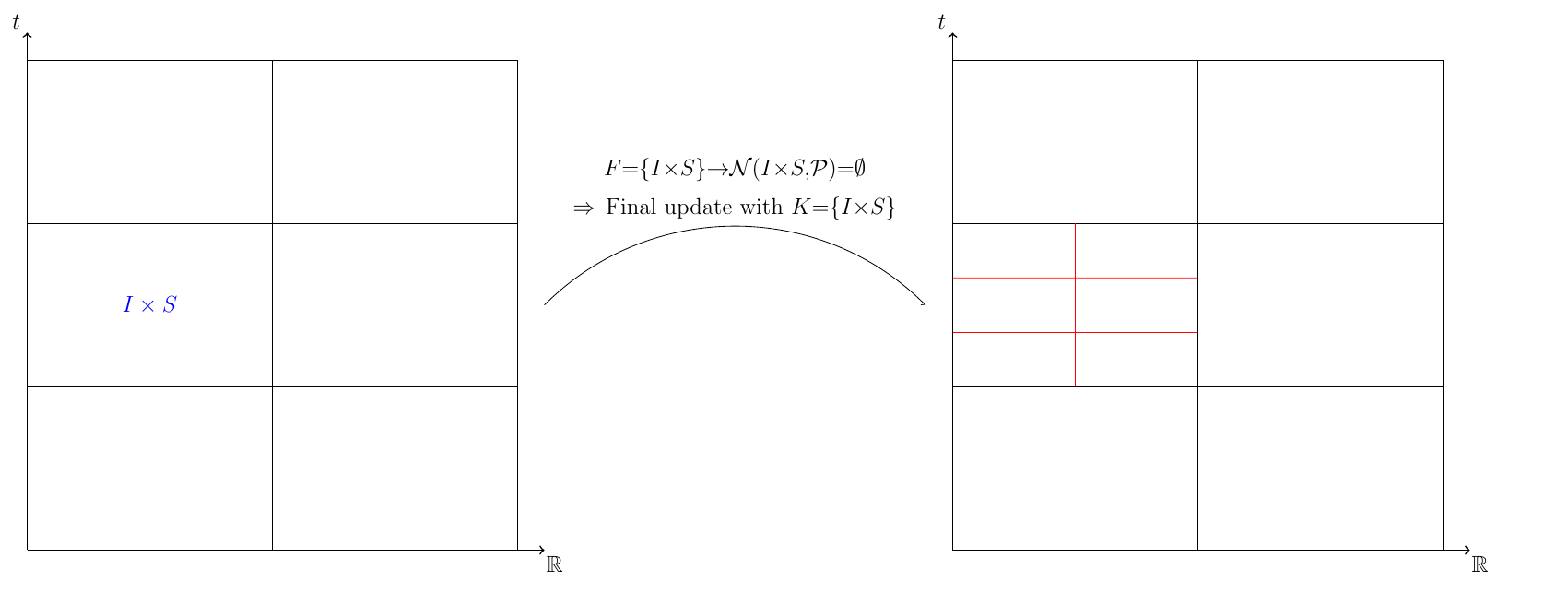}
            \end{center}
        \end{figure}
        \begin{figure}[H]
            \begin{center}
                \includegraphics[width=\linewidth]{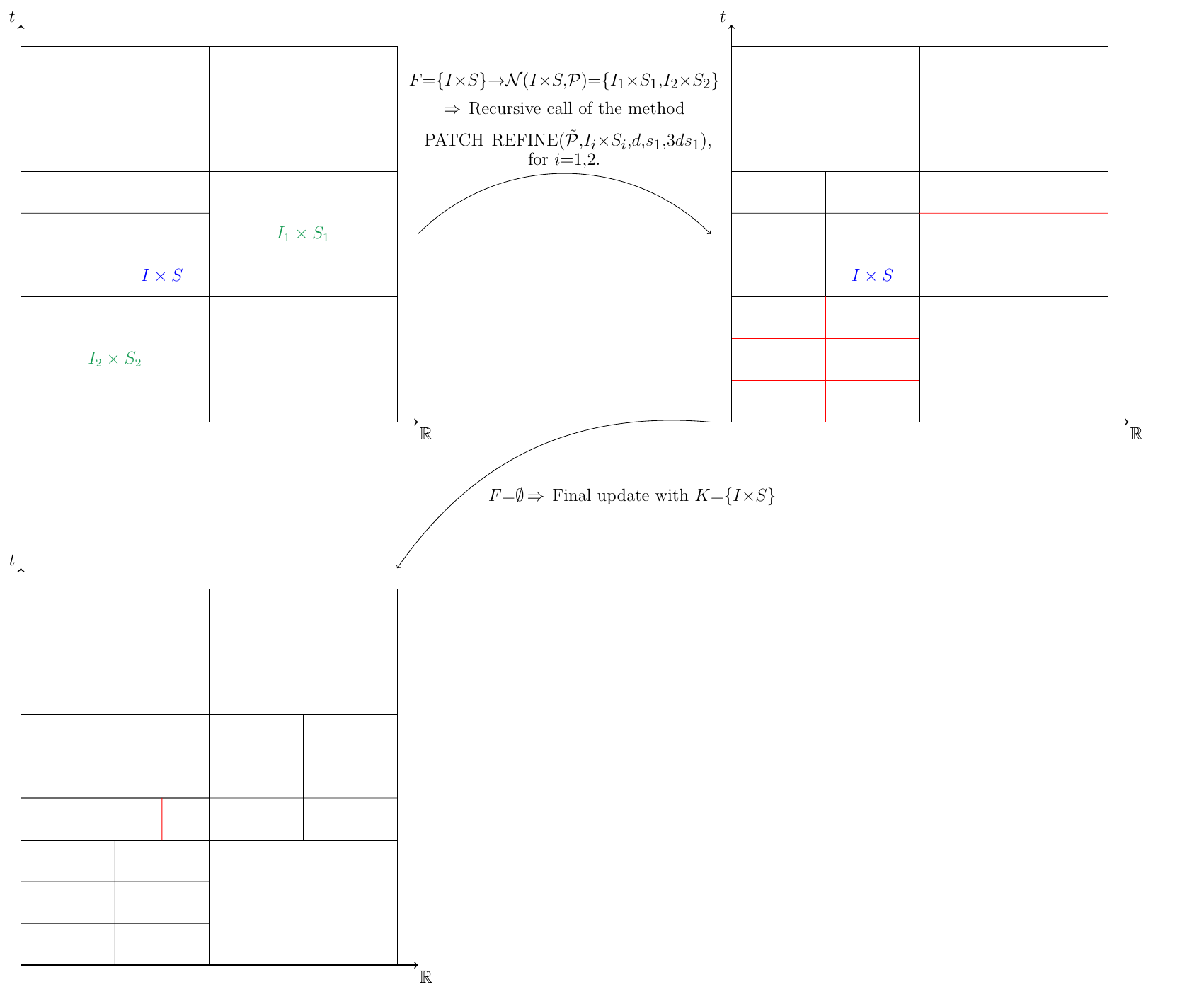}
            \end{center}
        \end{figure}
    \end{center}
    
\end{rem}

\begin{lem}\label{Enumeration further properties of the mesh}

Meshes $\calP$ created from $\calP_0$ by iterative applications of the patch refinement method presented above will also fulfill the following properties:

\begin{enumerate}[label=(\roman*)]

    \item \label{Enumeration further properties of the mesh - enum 1}\textit{Hierarchical structure: }For any $I\times S, I'\times S'\in \calP$, either $I\times S=I'\times S'$, $\left(\overline{I\times S}\right)\cap\left(\overline{I'\times S'}\right)= \emptyset$, or $\left(\overline{I\times S}\right)\cap\left(\overline{I'\times S'}\right)$ is an $r$-dimensional face of $I\times S$ \textbf{or} $I'\times S'$ with $r\in \{0,\dots, d\}$.\footnote{This is (only!) semantically very similar to conformity, where the $\textbf{or}$ would become an \textbf{and}.}
    
    \item \label{Enumeration further properties of the mesh - enum 2} \textit{Local grading}: For $I\times S\in \calP$, we define the \textbf{neighborhood of $I\times S$}, via 
    \begin{align*}
        \omega^{1}_\calP(I\times S):=\omega_\calP(I\times S):=\{I'\times S'\in \calP\mid \left(\overline{I\times S}\right)\cap\left(\overline{I'\times S'}\right)\neq \emptyset \}.
    \end{align*}
    Further, we define the \textbf{extended neighborhood of degree $j$ of $I\times S$} inductively via 
    \begin{align*}
        \omega^{j}_\calP(I\times S):=\{I'\times S'\in \calP\mid I'\times S' \in \omega_\calP(I''\times S''), \: I''\times S''\in \omega^{j-1}_\calP(I\times S)\}, \quad j\in \N_{\ge 2}.
    \end{align*}

    Then $|I\times S|\sim_{\,j, d, s_1, s_2, \kappa_{\calP_0}, \mu(\calP_0)} |\omega^{j}_\calP(I\times S)|:=\Big|\bigcup\limits_{I'\times S'\in \omega^{j}_\calP(I\times S)} I'\times S'\Big|$ for $j\in \N$ and $I\times S\in \calP$. Additionally, for such $I\times S$ and $j$, $|I'\times S'|\sim_{\, j,d,s_1, s_2, \kappa_{\calP_0}, \mu(\calP_0)}|I\times S|$ for any $I'\times S'\in \omega^{j}_\calP(I\times S)$. Further, for such $j$ and $I\times S$, we define the $\widetilde{\omega^{j}_\calP}(I\times S)$, the \textbf{cylindric closure of the (extended) neighborhood of $I\times S$}, via
    \begin{align*}
        \widetilde{\omega^{j}_\calP}(I\times S):=\left\{I'\times S'\in \calP \: \Bigg| \: I'\times S' \subset \left(\bigcup\limits_{I''\times S''\in \omega^{j}_\calP(I\times S)}I''\right) \times \left(\bigcup\limits_{I''\times S''\in \omega^{j}_\calP(I\times S)}S''\right)\right\}, \quad j\in \N.
    \end{align*}
    Similarly, we observe $|I\times S|\sim_{\,j, d, s_1, s_2, \kappa_{\calP_0}, \mu(\calP_0)} |\widetilde{\omega_\calP^{j}}(I\times S)|$, where the latter expression is defined correspondingly to $|\omega^{j}_\calP(I\times S)|$. In particular, we employ the notation 
    $\widetilde{\omega^{}_\calP}(I\times S):=\widetilde{\omega^{1}_\calP}(I\times S)$.\footnote{Note that we use $\omega$ for both the smoothness moduli and the (extended) neighborhood domains, as this corresponds to the usual notation. However, the subscript and the context always clearly indicate which object is meant.}
    
    \item \label{Enumeration further properties of the mesh - enum 3} 
    For every $j\in \N$, it holds
    \begin{align*}
        \#\omega^{j}_\calP(I\times S)\lesssim_{\, j,d,s_1, s_2, \kappa_{\calP_0}} 1.
    \end{align*} 
    For $j=1$, this in particular implies that the closure of every simplex in $\calP$ only intersects a uniformly bounded number of other elements of $\calP$ with its boundary.
    
\end{enumerate}

\end{lem}

\begin{proof}
    We will only provide a detailed proof for $d\ge 2$. The case of $d=1$ is easier and works correspondingly with minor changes due to the fact that the conformity in space is then replaced by a 1-irregular rule in space.
    
    \ref{Enumeration further properties of the mesh - enum 1} is a consequence of conformity in space and the $1$-irregular rule in time together with the conservation of the shape regularity of the partition by the atomic refinement as stated in \cite[Lem.~4.1]{MSS26}. 
    
    Now we go over to prove \ref{Enumeration further properties of the mesh - enum 2} and \ref{Enumeration further properties of the mesh - enum 3} for given $j\in \N$. First, we will recall a notion about conforming spatial triangulations $\calT$ derived from an initial, conforming, simplicial, and well-labeled triangulation via finitely many applications of $\textup{BISECT}(\cdot, d)$. Note that our initial spatial triangulation $\calT_0$ indeed fulfills the latter conditions according to the explanations given in the beginning of \cref{sect:space_time_partition_and_refinement}. Now if $S_1, S_2\in\calT$ are closed and $S_1\cap S_2\neq \emptyset$, then $|\ell(S_1)-\ell(S_2)|\le l_0$ for some $l_0\in \N$, which only depends on the triangulation itself, not on the particular simplices $S_1$ and $S_2$. This can be seen as follows:
    
    The levels of $S_1$ and $S_2$ will only differ by at most one, if one contains the refinement edge of the other one, according to \cite[Cor.~4.6]{Ste08}. Additionally the fact that after $d$ spatial bisections, all the edges of a $d$-dimensional simplex have been bisected once as stated in \cite[Problem~20]{NSV09}, implies that $|\ell(S_1) - \ell(S_2)|\lesssim_{\,d}1$ if they share a facet, i.e., $\dim(S_1\cap S_2)=d-1$. If they only share a lower dimensional face, they have a common vertex. The number of simplices which contain this vertex is bounded, depending on $d$ and the minimal head angle of the involved simplices, i.e., $\kappa_{\calP}$. Furthermore, there is a sequence of simplices, which contain the vertex and consecutively share a common facet starting with $S_1$ and ending with $S_2$. Together, this guarantees the existence of the aforementioned constant $l_0$ which only depends on $d$ and $\kappa_{\calP}$. Since the latter can be estimated in terms of $\kappa_{\calP_0}$ according to \mbox{\cref{Remark_Temporal_Bisection_Active_Triangulation}\ref{Remark_Temporal_Bisection_Active_Triangulation_1}}, we obtain that $l_0$ can be chosen only depending on $d$ and $\kappa_{\calP}\lesssim_{\,d}\kappa_{\calP_0}$, i.e., $l_0=l_0(d, \kappa_{\calP_0})$.

    Now we can consider $I'\times S'\in \omega_\calP(I\times S)$ for some $I\times S\in \calP$. If on the one hand, $\#(\overline{I}\cap\overline{I'})>1$, then there exists $t\in I\cap I'$ with $S, S'\in \calT(t,\calP)$. Since the latter is conforming according to \cref{Theorem - conformity 1-reg Patch refine}, due to the above argumentation and \cref{Remark_Temporal_Bisection_Active_Triangulation}\ref{Remark_Temporal_Bisection_Active_Triangulation_3}, this means that 
    \begin{align*}
        |\ell(I\times S)-\ell(I'\times S')|=|\ell(S)-\ell(S')|\le l_0(d,\kappa_{\calP_0})
    \end{align*}
    in this case. If, on the other hand, $\#(\overline{I}\cap\overline{I'})=1$, then there is another element $I''\times S''\in \omega_\calP(I\times S)$ which touches $I'\times S'$ in time (see \ref{Enumeration desired properties of the mesh 2} in the beginning of \cref{subsect:Patch_refine}) and fulfills $|S'\cap S''|>0$. For this prism, $\#(\overline{I}\cap\overline{I''})>1$ has to hold. According to the above argumentation, this yields $|\ell(I\times S)-\ell(I''\times S'')|\le l_0(d,\kappa_{\calP_0})$. Due to the 1-irregular rule in time, which $\calP$ also fulfills due to \cref{Theorem - conformity 1-reg Patch refine}, we now obtain
    \begin{align}\label{Enumeration further properties of the mesh - proof eq 1}
        |\ell(I\times S)-\ell(I'\times S')| \le |\ell(I\times S)-\ell(I''\times S'')| +1\le l_0(d,\kappa_{\calP_0})+1=:l(1, d,\kappa_{\calP_0}).
    \end{align}
    It is clear that this iteratively extends to the fact that for every $j\in \N$, there exists some $l=l(j, d,\kappa_{\calP_0})\in \N$ such that $|\ell(I\times S)-\ell(I'\times S')| \le l$ if $I'\times S'\in \omega^j_\calP(I\times S)$ and $I\times S\in \calP$. This directly yields $|I'\times S'|\sim_{\, j,d,s_1, s_2, \kappa_{\calP_0}, \mu(\calP_0)}|I\times S|$. Further, due to the level bound (for $j=1$), \cref{Remark_Temporal_Bisection_Active_Triangulation}\ref{Remark_Temporal_Bisection_Active_Triangulation_3} and \cref{lemma relationship diameter level - BISECT}, there is a lower bound $|I'|\gtrsim_{\,d,s_1, s_2, \kappa_{\calP_0}}|I| $ on the temporal length $|I'|$ of $I'\times S'\in \omega_\calP(I\times S)$ with $|I\cap I'|>0$. This, together with the 1-irregular rule in time and the fact that the number of elements of every active triangulation $\calT(t,\calP)$, $t\in \overline{I}$, which touch $S$ is bounded, yields \mbox{\ref{Enumeration further properties of the mesh - enum 3}} for $j=1$. For $j\ge 2$, this follows iteratively. This allows us to calculate
    \begin{align*}
        |\omega^{j}_\calP(I\times S)|\le\Big|\widetilde{\omega^{j}_\calP}(I\times S)\Big|&\le \left[\left(\sup\limits_{I'\times S'\in \omega^{j}_\calP(I\times S)}|I'| \right)\cdot \#\omega^{j}_\calP(I\times S)\right]\cdot \left[\left(\sup\limits_{I'\times S'\in \omega^{j}_\calP(I\times S)}|S'| \right)\cdot \#\omega^{j}_\calP(I\times S)\right]
        \\ &\lesssim_{\, j,d, s_1, s_2, \kappa_{\calP_0}} \left(\sup\limits_{I'\times S'\in \omega^{j}_\calP(I\times S)}|I'| \right)\cdot \left(\sup\limits_{I'\times S'\in \omega^{j}_\calP(I\times S)}|S'| \right)
        \\ &\le \frac{\mu_1(\calP_0)}{\mu_2(\calP_0)} \cdot 2^{\left\lceil l\cdot \frac{s_2}{s_1d}\right\rceil} |I|\cdot\frac{\mu_3(\calP_0)}{\mu_4(\calP_0)}\cdot 2^{l} |S|\lesssim_{\,j, d, s_1, s_2, \kappa_{\calP_0}, \mu(\calP_0)} |I\times S|.
    \end{align*}
    The opposite direction follows easily from $\Big|\widetilde{\omega^{j}_\calP}(I\times S)\Big|\ge |\omega^{j}_\calP(I\times S)| \ge |I\times S|$.
\end{proof}

In particular, similar properties have shown to be useful to define projection operators on finite element spaces in \cite[Sect.~6]{SS23}. 

\begin{rem}\label{rem:further comments on our meshes}
        It is worth observing the following facts:
        \begin{itemize}
            \item For the finite element analysis on appropriately refined meshes in the following sections, we would like to employ the Whitney-type estimate from \cite[Thm.~1.2]{MSS26} to domains of the type $\omega^{j}_\calP(I\times S)$, $I\times S\in \calP$, $j\in \N$. But this is in general not possible, since it might not be a cylindrical space-time domain. Therefore, we use the above defined $\widetilde{\omega^{j}_\calP}(I\times S)$ to work around this problem. We will study them in more detail in the following remark.

            \item In the temporal direction, non-conformity, i.e., hanging-nodes, cannot be avoided unless one refines the whole space-time partition uniformly.
            \item A bound on the cardinality of $\# \widetilde{\omega_\calP}(I\times S)$, $I\times S\in \calP$, cannot be given for $1$-irregular meshes in time (and space if $d=1$) that are conforming in space. As an example look at $d=1$, $T=1$, $\Omega=[0,1]$, $\calP_0:=\left\{[0,1]^2\right\}$, and $\calP_{k}:=\textup{PATCH}\_\textup{REFINE}\left(\calP_{k-1}, \left[1-2^{-(k-1)}, 1\right]^2, 1,1,1\right)$ for $k\in \N.$ Then $\#\widetilde{\omega_{\calP_k}}\left([0, \frac{1}{2}]^2\right)= \# \calP_k = 3k+1$ for $k\in\N$, and therefore $\#\widetilde{\omega_{\calP_k}}\left([0, \frac{1}{2}]^2\right)\xrightarrow{k\rightarrow\infty}\infty$. Below is an illustration of this particular mesh refinement and its consequences.
            \begin{center}
                \begin{figure}[H]
                    \begin{center}
                        \includegraphics[width=\linewidth]{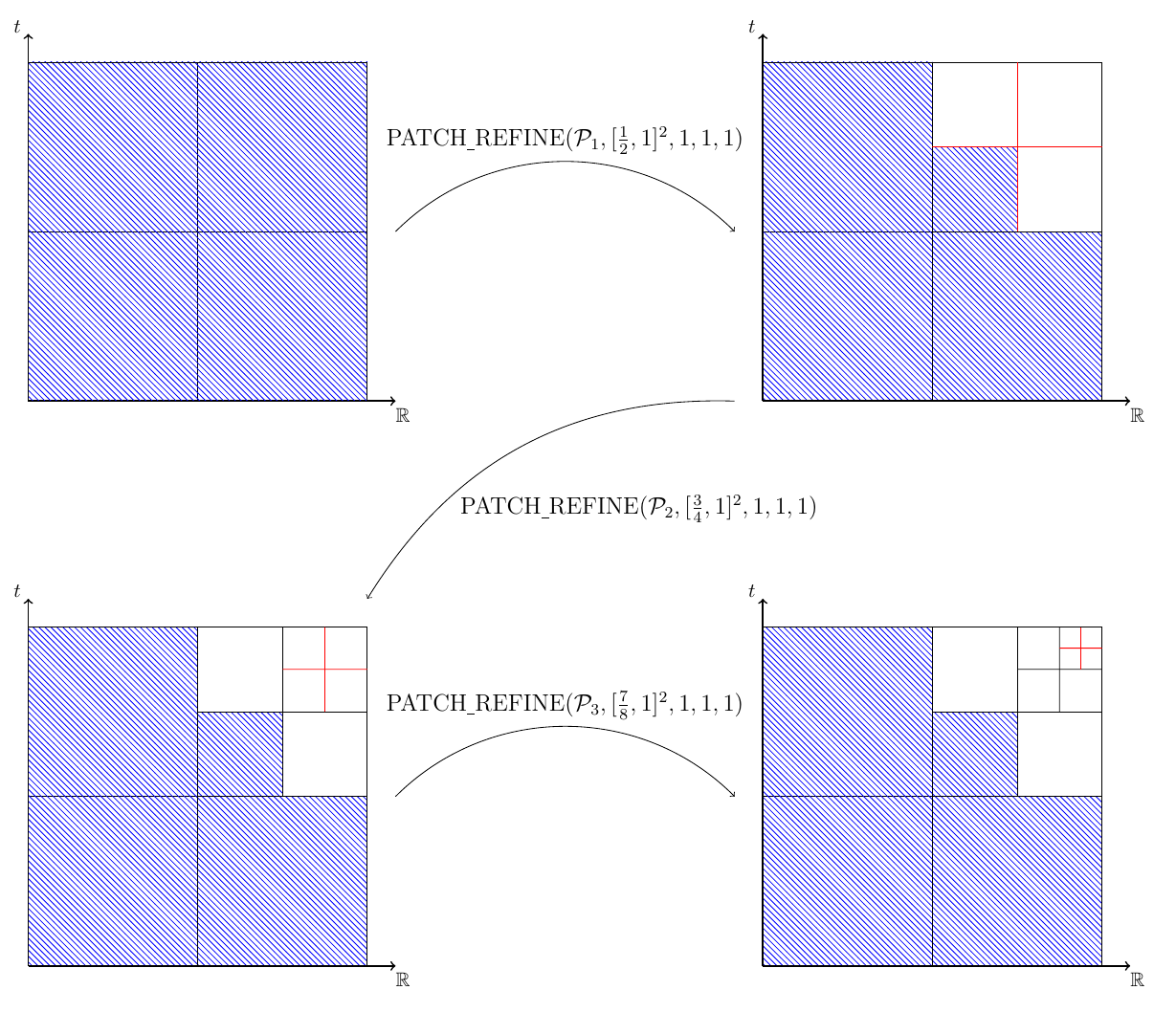}
                        \captionof*{figure}{The shaded part corresponds to $\omega_{\calP_k}([0,\frac12]^2)$. Its cardinality is bounded in contrast to $\#\widetilde{\omega_{\calP_k}}([0,\frac12]^2)$.}
                    \end{center}
                \end{figure}
            \end{center}
        \end{itemize}
\end{rem}

Since the cylindric closures of extended neighborhoods will serve as the local domains on which anisotropic Besov (quasi-)seminorms are measured, we record here what we need to know about their shape.
\begin{rem}\label{rem:Neighborhoods_generally_not_Lipschitz}
    \begin{enumerate}[label=(\roman*)]

        \item \label{rem:Neighborhoods_generally_not_Lipschitz - General S_w''}Let $\calP$ be created from $\calP_0$ by iterative applications of the patch refinement method. Then, for every $j\in \N$ and $I\times S\in \calP$, it is clear from the definition of $\widetilde{\omega_\calP^j}(I\times S)$ that the union of its elements has cylindrical shape, i.e., there is an interval $J\subset \R$ and $R\subset \R^d$ such that
        \begin{align*}
            \widetilde{\omega_\calP^j}(I\times S):=J\times R := \left(\bigcup\limits_{I'\times S'\in \omega^{j}_\calP(I\times S)}I'\right) \times \left(\bigcup\limits_{I'\times S'\in \omega^{j}_\calP(I\times S)}S'\right).
        \end{align*}
        As aforementioned, with a little abuse of notation, we will often just denote $J\times R=\widetilde{\omega^j_\calP}(I\times S)$. From context, it should always be clear if $\widetilde{\omega^j_\calP}(I\times S)$ should be interpreted as a set of prisms or subset in $\R^{d+1}$. 
        
        That $|J|\sim_{\,j,d,s_1, s_2,\kappa_{\calP_0},\mu(\calP_0)} |I|$ clearly holds because of the bounded level difference of elements of $\omega_\calP^j(I\times S)$ stated (in and) below \mbox{\eqref{Enumeration further properties of the mesh - proof eq 1}}. We now want to study, how $R$ can be represented. Let $\calT_n$, $n\in \N$, be the $n$-th uniform refinement of $\calT_0$ via $\textup{BISECT}(\cdot, d)$. Then, $\calT_n$ is conforming due to \cite[Thm.~4.3]{Ste08}. Further, again because of the statement below \mbox{\eqref{Enumeration further properties of the mesh - proof eq 1}}, and with the notation for the level difference $l$ from there, we know that $\widetilde{\calT}:=\calT_{\ell(I\times S)+l}$ locally(!) is a common refinement of all the active triangulations corresponding to $J\times R$, i.e., $\calT(t,\calP)$, $t\in J$. Therefore, $R$ can be covered by finitely many non-overlapping simplices $S_1'', \dots, S_k''\in \widetilde{\calT}$, $k\lesssim_{\,j,d, \kappa_{\calP_0}}1$ and the level bound $l$ between these simplices and $S$ essentially tells us that for all $w=1,\dots, k$,
        \begin{align}\label{rem:Neighborhoods_generally_not_Lipschitz - equation 2}
            \diam(S_w'')\sim_{\,\substack{j,d,\kappa_{\calP_0},\\\mu(\calP_0)}} \diam(S)\sim_{\,\substack{j,d,\kappa_{\calP_0},\\\mu(\calP_0)}}\diam(R)\quad\text{and} \quad |S_w''|\sim_{\,\substack{j,d,\kappa_{\calP_0},\\\mu(\calP_0)}} |S|\sim_{\,\substack{j,d,\kappa_{\calP_0},\\\mu(\calP_0)}}|R|.
        \end{align}
        Recalling \cref{Remark_Temporal_Bisection_Active_Triangulation}\ref{Remark_Temporal_Bisection_Active_Triangulation_3}, we also obtain a sequence $I_1'', \dots, I_m''\in \calI_{\left\lceil\frac{(\ell(I\times S)+l)s_2}{s_1d}\right\rceil}$, $m\lesssim_{\,j,d,s_1, s_2}1$, where $\calI_{\left\lceil\frac{(\ell(I\times S)+l)s_2}{s_1d}\right\rceil}$ is the $\left\lceil\frac{(\ell(I\times S)+l)s_2}{s_1d}\right\rceil$-th uniform refinement of the initial temporal partition $\calI_0$ with respect to $\textup{BISECT}(\cdot, 1)$, such that $\{I_1'', \dots, I_m''\}$ is a local refinement of $\{I'\mid I'\times S'\in \omega_\calP^j(I\times S) \}$ with
        \begin{align}\label{rem:Neighborhoods_generally_not_Lipschitz - equation 3}
            |I_v''|\sim_{\,j,d, s_1, s_2,\kappa_{\calP_0},\mu(\calP_0)}|I|\sim_{\,j,d, s_1, s_2,\kappa_{\calP_0},\mu(\calP_0)}|J|  
        \end{align}
        for $v=1,\dots, m$. So the covering of $J\times R$ given by $\calP_{J\times R}:=\{I_v''\times S_w''\mid v=1,\dots, m,~ w=1,\dots, k\}$, is a subset of $\calP_{\ell(I\times S)+l}$, i.e., the $(\ell(I\times S)+l)$-th uniform refinement of $\calP_0$ with respect to the $\textup{ATOMIC}\_\textup{SPLIT}$-method, since $\calP_{\ell(I\times S)+l}$ clearly is a tensor product mesh with $\calP_{\ell(I\times S)+l}=\calI_{\left\lceil\frac{(\ell(I\times S)+l)s_2}{s_1d}\right\rceil}\otimes\calT_{\ell(I\times S)+l}$. It fulfills $\#\calP_{J\times R}\lesssim_{\,j,d,s_1, s_2, \kappa_{\calP_0}}1$ and $ a(\calP_{J\times R})\lesssim a(\calP_0)$. In particular, it holds
        \begin{align}\label{rem:Neighborhoods_generally_not_Lipschitz - equation 3.5 - anisotropy of J x R}
            \frac{|J|^{s_1}}{|R|^{\frac{s_2}{d}}}\sim_{\,j,d,s_1, s_2, \kappa_{\calP_0}, \mu(\calP_0)} \frac{|I|^{s_1}}{|S|^{\frac{s_2}{d}}}\sim_{\,a(\calP_0)}1.
        \end{align}

        Furthermore, in general, there is no guarantee that the cylindrical closures will have a finite overlap. However, this holds for those closures of a fixed level, i.e.,
        \begin{align}\label{rem:Neighborhoods_generally_not_Lipschitz - equation 4}
            \#\left\{I'\times S'\in \calP\,\Big|\, I\times S\in \widetilde{\omega_{\calP}^j}(I'\times S'),~ \ell(I'\times S')=n \right\}\lesssim_{\,j,d,s_1, s_2, \kappa_{\calP_0},\mu(\calP_0)} 1, \:\: I\times S\in \calP, \:\: n\in \N_0,
        \end{align}
        because of \eqref{rem:Neighborhoods_generally_not_Lipschitz - equation 2} and \eqref{rem:Neighborhoods_generally_not_Lipschitz - equation 3}.

        \item \label{rem:Neighborhoods_generally_not_Lipschitz - reference domains} In some proofs of the article, we will need the \textbf{reference domains} $[0,1]\times R^{ref}$, which are the  image of the bijective, \textbf{scaling operator} $\Phi_{J\times R}:J\times R\rightarrow[0,1]\times R^{ref}$, $\Phi_{J\times R}(t,\bm{x}):=\left(\frac{t-a}{b-a}, |R|^{-\frac1d}\bm{x}\right)$, where $J=[a,b]$, $a<b$. 

        Of course, $\widetilde{\calP}_{[0,1]\times R^{ref}}:=\Phi_{J\times R}(\calP_{J\times R})$ represents a space-time partitioning of $[0,1]\times R^{ref}$. Together with \eqref{rem:Neighborhoods_generally_not_Lipschitz - equation 2}, this directly implies that, $\diam(R^{ref})\sim_{\,j,d,\kappa_{\calP_0},\mu(\calP_0)}1$. Further, the parameters in $\mu(\calP_{J\times R})$ are essentially bounded by $1$ from above and below depending on $j,d,s_1, s_2$ as well as $\mu(\calP_0)$ and $\kappa_{\widetilde{\calP}_{[0,1]\times R^{ref}}}= \kappa_{\calP_{J\times R}}\lesssim_{\,d}\kappa_{\calP_0}$. Additionally, $a(\widetilde{\calP}_{[0,1]\times R^{ref}})\sim_{\, j,d,s_1, s_2, \mu(\calP_0), a(\calP_0)}1$.

        \item \label{rem:Neighborhoods_generally_not_Lipschitz - Uniform case S_w''} Further, if $\calP=\calP_n$, $n\in \N$, is the $n$-th uniform refinement of $\calP_0$ with respect to the $\textup{ATOMIC}\_\textup{SPLIT}$-method from \cref{subsect:Atomic_ref}, then for all active triangulations $\calT(t,\calP_n)=\calT_n$, $t\in [0,T]$, holds. Thus, $\calT_n$ is already a common (trivial) global refinement of the active triangulations, and the $S''_v$ from above can be chosen as elements of $\calT_n$. In particular, this means that in this case 
        \begin{align*}
            \omega_{\calP_n}^j(I\times S)=\widetilde{\omega_{\calP_n}^j}(I\times S)=J\times R=J\times \omega^j_{\calT_n}(S)=J\times \left(\bigcup\limits_{v=1}^k S_w''\right)
        \end{align*}
        for any $I\times S\in \calP_n$ and $n\in \N_0$.\footnote{For a simplicial $\calT$ triangulation of $\Omega$, we will generally denote $\omega^1_\calT(S):=\omega_\calT(S):=\{S'\in \calT\mid S\cap S'\neq \emptyset\}$ for $S\in \calT$ and, correspondingly, $\omega^j_\calT(S):=\omega_\calT(S):=\{S'\in \calT\mid S'\in \omega_\calT(S''), S''\in \omega^{j-1}_\calT(S)\}$ for $j\in \N_{\ge 2}$.} Also, the patches generally have finite overlap, i.e., 
        \begin{align}\label{rem:Neighborhoods_generally_not_Lipschitz - equation 5}
            \#\left\{I'\times S'\in \calP\,\big|\, I\times S\in \omega_{\calP_n}^j(I'\times S') \right\}\lesssim_{\,j,d,\kappa_{\calP_0}} 1, \quad I\times S\in \calP_n,\quad n\in \N_0.
        \end{align}

        \item \label{rem:Neighborhoods_generally_not_Lipschitz - counterexample}Now, we encounter a problem: As already mentioned before in \cref{rem:further comments on our meshes}, we would like to apply the known theory from \cite{MSS26} for anisotropic polynomial approximation and anisotropic Besov spaces. But this theory is only developed for Lipschitz domains. And although one might expect $R$ to be Lipschitz, this is \textbf{generally not} the case. In particular, this already fails for $j=1$ in the uniform case as the following example shows.\\[-0.2cm]
        \begin{minipage}{0.5\textwidth}  
        Since the boundary of the (spatial) neighborhood domain around the upper vertex cannot be represented by a Lipschitz graph, this domain clearly is not Lipschitz. Therefore, we have to give new proofs for the aforementioned theorems and lemmata corresponding to those from \cite{MSS26} for (spatial) Lipschitz domains, which we need for our analysis of direct and inverse estimates in \cref{sect:Almost_characterization}. Since these proofs are very technical in nature and might not be of interest to every reader, while the results are most relevant for the proofs of our main results in \cref{sect:Almost_characterization}, we will just state them in the following and refer those readers who are interested in their proofs to \cref{Appendix:Proofs_for_omega_tilde_domains}, where we give detailed justifications. 
        \end{minipage}\hfill 
        \begin{minipage}{0.4\textwidth}  
        \begin{center}
            \begin{figure}[H]
                \begin{center}
                    \includegraphics[width=4.5cm]{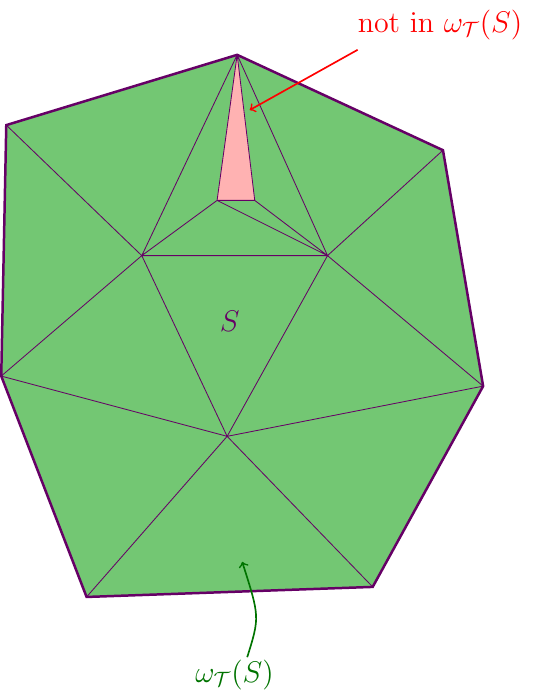}
                \end{center}
            \end{figure}
        \end{center}
        \end{minipage}\\
       
        \item \label{rem:Neighborhoods_generally_not_Lipschitz - remark on literature} We are aware of the fact that in the forerunner \cite{GM14}, this false statement that $\omega_{\calT}(S)$ was always Lipschitz, has unfortunately been stated and exploited on several occasions, explicitly in the proof of Lemma 4.4 and implicitly several times throughout the article. If one now disregards the temporal component, one can look upon our results in the following theorem and the related details in \cref{Appendix:Proofs_for_omega_tilde_domains} together with their application in \cref{sect:Almost_characterization} as a (retrospective) justification for the validity of the main results in \cite{GM14} for any initial (spatial) triangulation, which is conforming and well-labeled. 
        
        For potential initial triangulations, which guarantee that all $\omega_\calT^j(T)$, $j\in \N$, $T\in \calT$, $\calT$ conforming and adaptively derived from $\calT_0$ via $\textup{BISECT}(\cdot, d)$, are Lipschitz domains, the proofs in question from \cite{GM14} would of course work verbatim. But though we think that the above case is somewhat pathologic and thus that there are (potentially mild) conditions on $\calT_0$ which guarantee this property, we are not aware of any, especially in case of $d\ge 3$. This is a topic for further research.

        \item \label{rem:Neighborhoods_generally_not_Lipschitz - enlargement}Moreover, we define the \textbf{enlarged cylindric closure} $\Enl^i(\widetilde{\omega_\calP^{j}}(I\times S))$ for $i\in \{1,2\}$, $j\in \N$, and $I\times S\in \calP$, via
        \begin{align*}
            \Enl^i(\widetilde{\omega_\calP^{j}}(I\times S)):=J\times \left(\bigcup_{w=1}^k \omega_{\tilde{\calT}}^i(S_w'')\right)=J\times \left(\bigcup_{w=1}^k \bigcup_{S'''\in\omega_{\tilde{\calT}}^i(S_w'')}S'''\right)
        \end{align*}
        where we use the notation from \ref{rem:Neighborhoods_generally_not_Lipschitz - General S_w''} above. In case of $i=1$, we simply write $\Enl:=\Enl^1$. For these sets, the result on the local grading from \cref{Enumeration further properties of the mesh}\ref{Enumeration further properties of the mesh - enum 2} as well as \eqref{rem:Neighborhoods_generally_not_Lipschitz - equation 2}, \eqref{rem:Neighborhoods_generally_not_Lipschitz - equation 3}, \eqref{rem:Neighborhoods_generally_not_Lipschitz - equation 3.5 - anisotropy of J x R}, and \eqref{rem:Neighborhoods_generally_not_Lipschitz - equation 4} hold analogously with an additional dependence on $i$. Further, we extend the notion of reference domains from \ref{rem:Neighborhoods_generally_not_Lipschitz - reference domains} to this setting and denote the correspondingly scaled version of $\Enl(\widetilde{\omega_\calP^{j}}(I\times S))$ by $\Enl([0,1]\times R^{ref})$. 

        The reason why we need this definition and use these domains at the end of \cref{subsect:Multiscale_decomp} and \cref{subsect:direct estimates} is the following: At some point in the calculations we have to sum over Besov \mbox{(quasi-)}seminorms which is usually done by switching to the averaged moduli of smoothness. But the equivalence of the Besov \mbox{(quasi-)}seminorms with averaged and supremum moduli is only known to hold in the case where the spatial part is a Lipschitz domain. This might not be true for $\widetilde{\omega_\calP^j}(I\times S)$, even in the uniform case. Therefore, we need to enlarge the domain on one side of the estimate, see \cref{thm:Proofs_omega_tilde_domain}\ref{thm:Proofs_omega_tilde_domain - Quasi-equivalence of suprema and averaged moduli of smoothness} below. 
    \end{enumerate}
\end{rem}

\begin{thm}\label{thm:Proofs_omega_tilde_domain}
    Let $\calP$ be derived from $\calP_0$ by finitely many applications of $\textup{PATCH}\_\textup{REFINE}(\cdot, \cdot, d, s_1, s_2)$, $s_1, s_2\in (0,\infty)$, $I\times S\in \calP$, and $j\in \N$. The following results, similar to those from \cite{MSS26} for Lipschitz domains, hold for the domain $J\times R:=\widetilde{\omega_\calP^j}(I\times S)$, in detail described in \cref{rem:Neighborhoods_generally_not_Lipschitz}\ref{rem:Neighborhoods_generally_not_Lipschitz - General S_w''}.
    \begin{enumerate}[label=(\roman*)]
        \item \label{thm:Proofs_omega_tilde_domain - item Jackson}\textbf{Jackson inequality:} Let $p\in (0,\infty]$, $f\in L_p(J\times R)$, and $r_1, r_2\in \N$. Then, there exists an anisotropic polynomial $P\in \Pi^{r_1, r_2}_{t,\bm{x}}(J\times R)$ such that
        \begin{align*}
            \|f-P\|_{L_p\left(J\times R\right)}\lesssim_{\,j,d,p, r_1, r_2, \kappa_{\calP_0}, \mu(\calP_0)}\omega_{r_1, t}\left(f,J\times R, |J|\right)_p +\omega_{r_2, \bm{x}}\left(f,J\times R, \diam(R)\right)_p.
        \end{align*}
        The same holds true for the reference domains $[0,1]\times R^{ref}$ from \cref{rem:Neighborhoods_generally_not_Lipschitz}\ref{rem:Neighborhoods_generally_not_Lipschitz - reference domains} with $|J|$ and $\diam(R)$ replaced by $1$ and $\diam(R^{ref})$, respectively. 
        
        Further, the same is possible for $\Enl(J\times R)$ and $\Enl([0,1]\times R^{ref})$ from \cref{rem:Neighborhoods_generally_not_Lipschitz}\ref{rem:Neighborhoods_generally_not_Lipschitz - enlargement}, where the corresponding spatial diameters have to be replaced by $\diam(\Enl(R))$ and $\diam(\Enl(R^{ref}))$, respectively. Here we use $J\times\Enl(R)=\Enl(J\times R)$ and $[0,1]\times\Enl(R^{ref})=\Enl([0,1]\times R^{ref})$ with a slight abuse of notation.

        \item \label{thm:Proofs_omega_tilde_domain - item moduli of smoothness equal zero IFF anisotropic polynomials} \textbf{Relationship between temporal and spatial moduli and anisotropic polynomials:}
        	Let $p\in (0,\infty]$, $f\in L_p(J\times R)$, and $r_1, r_2\in \N$. Then,
            \begin{enumerate}[label=(\alph*)]
                \item $f\in \Pi_{t,\bm{x}}^{r_1,r_2}(J\times R)$ implies $\omega_{r_1,t}(f, J\times R, \delta_1)_p=\omega_{r_2,\bm{x}}(f, J\times R, \delta_2)_p=0$ for any $\delta_1,\delta_2 \in [0,\infty)$.\label{thm:Proofs_omega_tilde_domain - item moduli of smoothness equal zero IFF anisotropic polynomials - item 1}
                \item $\omega_{r_1,t}(f, J\times R, \delta_1)_p=0=\omega_{r_2,\bm{x}}(f, J\times R, \delta_2)_p$ for some $\delta_1, \delta_2\in(0,\infty)$ implies $f\in \Pi_{t,\bm{x}}^{r_1,r_2}(J\times R)$.\footnote{This is of course an abbreviation. Technically, it means that $f=P$ almost everywhere in $J\times R$, where $P\in \Pi^{r_1,r_2}_{t,\bm{x}}(J\times R)$.}
                \label{thm:Proofs_omega_tilde_domain - item moduli of smoothness equal zero IFF anisotropic polynomials - item 2}
            \end{enumerate}
         The same holds true for the reference domains $[0,1]\times R^{ref}$ from \cref{rem:Neighborhoods_generally_not_Lipschitz}\ref{rem:Neighborhoods_generally_not_Lipschitz - reference domains}.
         
        \item \label{thm:Proofs_omega_tilde_domain - Quasi-equivalence of suprema and averaged moduli of smoothness}\textbf{Quasi-equivalence of suprema and averaged moduli of smoothness: }
        \begin{enumerate}[label=(\alph*)]
            \item \label{thm:Proofs_omega_tilde_domain - Quasi-equivalence of suprema and averaged moduli of smoothness - uniform part} Let $f\in L_p(\Omega_T)$, $p\in (0,\infty]$, and $r_1, r_2\in \N$. Further, let $\calP=\calP_n$, $n\in \N_0$, and observe that $\omega_{\calP_n}^{j}(I\times S)=J\times \omega^{j}_{\calT_n}(S)$ as in \mbox{\cref{rem:Neighborhoods_generally_not_Lipschitz}\ref{rem:Neighborhoods_generally_not_Lipschitz - Uniform case S_w''}}. Then,
            \begin{equation*}
            \omega_{r_1,t}\left(f, \omega^j_{\calP_n}(I\times S), \delta\right)_p \lesssim_{\,p,r_1}\mathrm{w}_{r_1,t}\left(f, J\times \omega^{j+1}_{\calT_n}(S), \delta\right)_p \le \mathrm{w}_{r_1,t}\left(f, \omega^{j+1}_{\calP_n}(I\times S), \delta\right)_p 
            \end{equation*}
            for all $\delta\in[0,C_t\,|I|]$ with positive $C_t=C_t(j,d,s_1, s_2, r_1, \kappa_{\calP_0}, \mu(\calP_0))$. Furthermore, there is another positive constant $C_{\bm{x}}=C_{\bm{x}}(j,d,r_2,\kappa_{\calP_0},\mu(\calP_0),\LipProp(\Omega))$, such that
            \begin{align*}
                \omega_{r_2,\bm{x}}\left(f, \omega^j_{\calP_n}(I\times S), \delta\right)_p \lesssim_{\,\substack{d,p,r_2, \kappa_{\calP_0}, \mu(\calP_0),\\\LipProp(\Omega)}}\mathrm{w}_{r_2,\bm{x}}\left(f, J\times \omega^{j+1}_{\calT_n}(S), \delta\right)_p \le\mathrm{w}_{r_2,\bm{x}}\left(f, \omega^{j+1}_{\calP_n}(I\times S), \delta\right)_p 
            \end{align*}
            for $\delta\in[0,C_{\bm{x}}\,\diam(S)]$.
            \item \label{thm:Proofs_omega_tilde_domain - Quasi-equivalence of suprema and averaged moduli of smoothness - other part} For general $\calP$ and $I\times S\in \calP$, the left estimates in the temporal and spatial direction, carry over in the sense that
            \begin{align*}
                \omega_{r_i,\mathrm{dir}}\left(f, J\times R, \delta\right)_p \lesssim_{\,d,p,r_1, r_2, \kappa_{\calP_0}, \mu(\calP_0),\LipProp(\Omega)} \mathrm{w}_{r_i,\mathrm{dir}}\left(f, \Enl(J\times R), \delta\right)_p
            \end{align*}
            for corresponding $(r_i, \mathrm{dir})\in \{(r_1, t),(r_2, \bm{x})\}$, $\delta$, $C_t$ and $C_{\bm{x}}$. The same applies, with $J\times R$ and $\Enl(J\times R)$ replaced by $\Enl(J\times R)$ and $\Enl^2(J\times R)$, respectively.
        \end{enumerate}

        \textbf{Note} that this is not a classical result about the equivalence of supremum and averaged moduli of smoothness as e.g. in \cite[Lem.~5.1 of Ch.~6.5]{DL93}, \cite[Lem.~4.1 and Cor.~4.2]{GM14}, or \cite[Lem.~2.7 and Lem.~2.8]{MSS26}, since the domain on the right side is larger than the one on the left. Thus, our result is weaker. It is not clear to us whether a result of the former type holds in full generality in our setting, in particular for every dimension $d\in \N$. The domains involved are in general not Lipschitz and, especially for $d\ge 3$, may exhibit further irregularities.\footnote{We would like to point out that for the stationary case, there is a statement of exactly such a strong, i.e., classical result, for $j=1$ in \cite[Cor.~4.3]{GM14}, but we suspect the proof to be incorrect: in our example of a non-Lipschitz neighborhood domain from \cref{rem:Neighborhoods_generally_not_Lipschitz}\ref{rem:Neighborhoods_generally_not_Lipschitz - counterexample}, the pinching vertex admits a cone in each of the two simplices meeting there, but no single cone serves an entire neighborhood of it, so the uniform cone condition required in the corresponding proof, necessary to justify an application of \cite[Lem.~4.1]{GM14}, does not hold in this point. This is in agreement with the fact that the authors operated under the assumption that the neighborhood domains were Lipschitz; see also \cref{rem:Neighborhoods_generally_not_Lipschitz}\ref{rem:Neighborhoods_generally_not_Lipschitz - remark on literature}.}
        
        \item \textbf{Summation of Besov (quasi-)seminorms:}\label{thm:Proofs_omega_tilde_domain - summing of Besvo norms} Let $p,\alpha_1,\alpha_2\in (0,\infty)$ and $f\in B^{\alpha_1, \alpha_2}_{p,p}(\Omega_T)$. 
        \begin{enumerate}[label=(\alph*)]
            \item \label{thm:Proofs_omega_tilde_domain - summing of Besvo norms - uniform part} If $\calP=\calP_n$, $n\in \N_0$, as in \cref{rem:Neighborhoods_generally_not_Lipschitz}\ref{rem:Neighborhoods_generally_not_Lipschitz - Uniform case S_w''}, then
                \begin{align*}
                    \sum\limits_{I\times S\in \calP_n}|f|^p_{B^{\alpha_1, \alpha_2}_{p,p}\left(\omega_\calP^j(I\times S)\right)}\lesssim_{\,j,d,p,s_1, s_2,\alpha_1, \alpha_2, \kappa_{\calP_0},\mu(\calP_0), \LipProp(\Omega)}|f|^p_{B^{\alpha_1, \alpha_2}_{p,p}(\Omega_T)}
                \end{align*}
            holds for any $n\in \N_0$. 
            \item \label{thm:Proofs_omega_tilde_domain - summing of Besvo norms - general part} Further, for general $\calP$ and $n\in\N_0$,
            \begin{align*}
                \sum\limits_{\substack{I\times S\in \calP,\\\ell(I\times S)=n}}|f|^p_{B^{\alpha_1, \alpha_2}_{p,p}\left(\widetilde{\omega_\calP^j}(I\times S)\right)}\lesssim_{\,j,d,p,s_1, s_2,\alpha_1, \alpha_2, \kappa_{\calP_0},\mu(\calP_0), \LipProp(\Omega)}|f|^p_{B^{\alpha_1, \alpha_2}_{p,p}(\Omega_T)}
            \end{align*}
            holds true. The same applies if we replace $\widetilde{\omega_\calP^j}(I\times S)$ by $\Enl(\widetilde{\omega_\calP^j}(I\times S))$. 
            \item \label{thm:Proofs_omega_tilde_domain - summing of Besvo norms - R^ref part} For general $\calP$ and the reference domains $R^{ref}$ with space-time covering $\tilde{\calP}_{[0,1]\times R^{ref}}$ from \cref{rem:Neighborhoods_generally_not_Lipschitz}\ref{rem:Neighborhoods_generally_not_Lipschitz - reference domains}, we denote its $n$-th, $n\in \N_0$, uniform refinement with respect to $\textup{ATOMIC}\_\textup{SPLIT}(\cdot, d, s_1, s_2)$, as $\mathcal{C}_n:=\left(\tilde{\calP}_{[0,1]\times R^{ref}}\right)_n$. With the enlarged reference domain from \cref{rem:Neighborhoods_generally_not_Lipschitz}\ref{rem:Neighborhoods_generally_not_Lipschitz - enlargement}, we then have, for $f$ rescaled to the reference domain as in \cref{rem:Neighborhoods_generally_not_Lipschitz}\ref{rem:Neighborhoods_generally_not_Lipschitz - reference domains},
            \begin{align*}
                \sum\limits_{I\times S\in \mathcal{C}_n}|f|^p_{B^{\alpha_1, \alpha_2}_{p,p}\left(\widetilde{\omega_{\mathcal{C}_n}^j}(I\times S)\right)}\lesssim_{\,j,d,p,s_1, s_2,\alpha_1, \alpha_2, \kappa_{\calP_0},\mu(\calP_0), \LipProp(\Omega)}|f|^p_{B^{\alpha_1, \alpha_2}_{p,p}(\Enl([0,1]\times {R^{ref}}))}.
            \end{align*}
        \end{enumerate}
        For $p=\infty$, the corresponding results with maxima instead of sums are also true. Further, as long as $r_i\in \N_{\ge 2}$ with $r_i>\alpha_i$, $i=1,2$, the corresponding versions of the above statements with $|\cdot|^{(r_1, r_2)}_{B^{\alpha_1, \alpha_2}_{p,p}(\cdot)}$ instead of $|\cdot|_{B^{\alpha_1, \alpha_2}_{p,p}(\cdot)}$ also hold with an additional dependence of the constants on $r_1$ and $r_2$.

        \item \label{thm:Proofs_omega_tilde_domain - Embeddings between Besov and Lebesgue scale}\textbf{Embeddings between Besov and Lebesgue scale}: Let $p,q\in(0,\infty]$, $s_1, s_2\in (0,\infty)$, and $r_1, r_2\in \N_{\ge 2}$ with $r_i>s_i$, $i=1,2$. Further, let $\frac{1}{\frac{1}{s_1}+\frac{d}{s_2}}-\frac1q+\frac1p>0$, $[0,1]\times R^{ref}$ be the reference domain from \cref{rem:Neighborhoods_generally_not_Lipschitz}\ref{rem:Neighborhoods_generally_not_Lipschitz - reference domains} and $f\in B^{s_1, s_2}_{q,q}([0,1]\times R^{ref})$. Then the embeddings
        \begin{align}\label{thm:Proofs_omega_tilde_domain - Embeddings between Besov and Lebesgue scale - equation 1}
            &\left(B^{s_1, s_2}_{q,q}([0,1]\times R^{ref}), \|\cdot\|_{B^{s_1, s_2}_{q,q}([0,1]\times R^{ref})}\right)\hookrightarrow L_p([0,1]\times R^{ref})\notag
            \\\text{and }&\left(B^{s_1, s_2}_{q,q}([0,1]\times R^{ref}), \|\cdot\|^{(r_1, r_2)}_{B^{s_1, s_2}_{q,q}([0,1]\times R^{ref})}\right)\hookrightarrow L_p([0,1]\times R^{ref})
        \end{align}
        hold continuously with embedding constants only depending on $d,p,q,s_1, s_2, r_1, r_2, \kappa_{\calP_0}, \mu(\calP_0)$, and $a(\calP_0)$. Further, 
        \begin{align}\label{thm:Proofs_omega_tilde_domain - Embeddings between Besov and Lebesgue scale - equation 2}
            \left(B^{s_1, s_2}_{p,q}([0,1]\times R^{ref}), \|\cdot\|^{(r_1, r_2)}_{B^{s_1, s_2}_{p,q}([0,1]\times R^{ref})}\right)=\left(B^{s_1, s_2}_{p,q}([0,1]\times R^{ref}), \|\cdot\|^{*,(r_1, r_2)}_{B^{s_1, s_2}_{p,q}([0,1]\times R^{ref})}\right)
        \end{align}
        holds true, for any $p,q\in (0,\infty]$, in the sense of equivalent \mbox{(quasi-)}norms, with constants only depending on $d,p,q,s_1, s_2,r_1,r_2, \kappa_{\calP_0}$, and $\mu(\calP_0)$.

        \textbf{Note: } We do not necessarily require that 
        $$\left(B^{s_1, s_2}_{p,q}([0,1]\times R^{ref}), \|\cdot\|_{B^{s_1, s_2}_{p,q}([0,1]\times R^{ref})}\right)= \left(B^{s_1, s_2}_{p,q}([0,1]\times R^{ref}), \|\cdot\|^{(r_1, r_2)}_{B^{s_1, s_2}_{p,q}([0,1]\times R^{ref})}\right)$$
        in the sense of equivalent (quasi-)norms. This might not generally hold, since $R^{ref}$ is not necessarily Lipschitz.
    \end{enumerate}
\end{thm}
\begin{proof}
    See \cref{Appendix:Proofs_for_omega_tilde_domains}.
\end{proof}

\subsection{Algorithmic complexity}\label{subsect:Algorithmic_complexity}

We will proceed similarly to \cite[Thm.~5.2~and~Sect.~6]{Ste08}. In this section we want to answer the question how large a space-time partition can get, if we apply the following (potentially non-terminating) algorithm. Note that, in contrast to $\textup{PATCH}\_\textup{REFINE}(\cdot, \cdot, d, s_1, s_2)$, which performs a single refinement step, $\textup{MARK}\_\textup{AND}\_\textup{REFINE}(\calP_0, d, s_1, s_2)$ describes the whole adaptive loop: starting from the initial partition $\calP_0$, it repeatedly marks a set of elements of the current partition by means of a routine $\textup{MARK}$, and refines all marked elements with $\textup{PATCH}\_\textup{REFINE}(\cdot, \cdot, d, s_1, s_2)$, until no elements are marked anymore. Here, $\textup{MARK}$ is a placeholder for an arbitrary marking routine, which takes the current partition $\calP$ as input and returns a subset $\calM\subset\calP$ of marked elements; in applications it will typically also depend on the function to be approximated and on a tolerance (see e.g. \cref{subsect:direct estimates}). The complexity estimate in \cref{Thm:complexity results} below holds independently of the particular choice of $\textup{MARK}$.

\begin{algorithm}[H]
\caption{$\textup{MARK}\_\textup{AND}\_\textup{REFINE}(\calP_0, d, s_1, s_2)$}
\begin{algorithmic}
    \State $\calP\gets \calP_0$
    \While{True}
        \State $\calM\gets \textup{MARK}(\calP)$
        \If{$\calM=\emptyset$}
            \State \textbf{break}
        \EndIf
        \For{$I\times S\in \calM$}
            \If{$I\times S\in \calP$} \% if not, it has already been refined in a previous loop of this \textbf{for}-clause
                \State $\calP\gets \textup{PATCH}\_\textup{REFINE}(\calP, I\times S, d, s_1, s_2) $
            \EndIf
        \EndFor
    \EndWhile
\end{algorithmic}
\end{algorithm}

We will denote the space-time partition at the end of the $k$-th loop of the \textbf{while}-clause by $\calP_k$ and the set of marked elements in $\calP_k$ as $\calM_k$, i.e., $\calM_k=\textup{MARK}(\calP_k)$ and $\calP_{k+1}$ is obtained from $\calP_k$ by refining the elements of $\calM_k$ in the $(k+1)$-th loop. The goal of this section is to show

\begin{thm}\label{Thm:complexity results}
    Let $k\in \N_0$, then it holds
    \begin{align*}
        \#\calP_k -  \#\calP_0 \lesssim_{\, d, s_1, s_2, \kappa_{\calP_0}, \mu(\calP_0)} \sum\limits_{i=1}^k \#\calM_{i-1}.
    \end{align*}
\end{thm}

In order to prove this theorem, we will need a lemma that bounds the anisotropic distance $\textup{dist}_\xi(I'\times S', I\times S)$, with respect to $\xi$, which has been defined above \cref{corollary relationship diameter level - ATOMIC_SPLIT}, of an element $I\times S\in \calP$ to an element $I'\times S'\in \calP'$ that has been newly created by a call of $\calP':=\textup{PATCH}\_\textup{REFINE}(\calP, I\times S, d, s_1, s_2)$. Here, we assume $\calP$ to fulfill the usual properties from 
\cref{Lemma_Properties_of_Neighboring_Space_Time_Elements}, i.e., conformity in space as well as 1-irregularity in time (and space, if $d=1$). In the stationary setting this corresponds to \cite[Thm.~5.2]{Ste08}, the proof of which we have modified to suit our setting.

\begin{lem}\label{Lemma distance newly created element}
    Let $I\times S$, $I'\times S'$, $\calP$, and $\calP'$ be as above. Then there exists a constant $C$ only depending on $d, s_1, s_2, \kappa_{\calP_0}$, and $\mu(\calP_0)$ such that
    \begin{align*}
        \textup{dist}_\xi(I'\times S', I\times S)\le C\, 2^{\frac{1}{d}\min\left(\frac{s_2}{s_1},1\right)}\sum\limits_{k=\ell(I'\times S')}^{\ell(I\times S)} 2^{-\frac{k}{d}\min\left(\frac{s_2}{s_1},1\right)}.
    \end{align*}
    Here, $C$ is the constant from \cref{corollary relationship diameter level - ATOMIC_SPLIT} corresponding to the upper bound.
\end{lem}
\begin{proof}
    We will prove the assertion by induction over $\ell(I\times S)$.\footnote{This proof relies heavily on the one of \cref{Theorem - conformity 1-reg Patch refine}, therefore, we will use the results found there without explicit mentioning them.} If $\ell(I\times S)=0$, then $\ell(I'\times S')=1$ and $I\times S$ touches $I'\times S'$ in space, i.e., $\textup{dist}_\xi(I'\times S',I\times S)=0$, which corresponds to the assertion. In particular, if $d=1$, then $\ell(I\times S)=0$ even implies that $I'\times S'\in \textup{ATOMIC}\_\textup{SPLIT}(I\times S, d, s_1, s_2)$. Now assume the lemma holds true for levels strictly less than $l\in \N$ and $\ell(I\times S)=l$. Then there are two possibilities. If $I'\times S'$ has been created by the application of $\textup{ATOMIC}\_\textup{SPLIT}(\cdot, d, s_1, s_2)$ to an element of the set $K$ from the algorithm, then again $\textup{dist}_\xi(I'\times S',I\times S)=0$ and nothing is left to be shown. Otherwise, $I'\times S'$ is created by a recursive call of $\textup{PATCH}\_\textup{REFINE}(\calP, I''\times S'', d, s_1, s_2)$ with $\ell(I''\times S'')=\ell(I\times S)-1=l-1$. Then the fact that $(\overline{I\times S})\cap(\overline{I''\times S''})\ne \emptyset$ yields the first step below, the induction assumption together with \cref{corollary relationship diameter level - ATOMIC_SPLIT} the second one, and $\ell(I''\times S'')=\ell(I\times S)-1$ the last one:
    \begin{align*}
        \textup{dist}_\xi(I'\times S', I\times S)&\le \textup{dist}_\xi(I'\times S', I''\times S'') + \diam_{\xi}(I''\times S'')
        \\&\le C\, 2^{\frac{1}{d}\min\left(\frac{s_2}{s_1},1\right)}\sum\limits_{k=\ell(I'\times S')}^{\ell(I''\times S'')} 2^{-\frac{k}{d}\min\left(\frac{s_2}{s_1},1\right)} + C \, 2^{-\frac{\ell(I''\times S'')}{d}\min\left(\frac{s_2}{s_1},1\right)}
        \\&=C \, 2^{\frac{1}{d}\min\left(\frac{s_2}{s_1},1\right)}\sum\limits_{k=\ell(I'\times S')}^{\ell(I\times S)} 2^{-\frac{k}{d}\min\left(\frac{s_2}{s_1},1\right)}.
        \qedhere
    \end{align*}
\end{proof}

Now we can go over to the proof of \cref{Thm:complexity results} which is very similar to the one of the time-independent analogue from \cite[Thm.~6.1]{Ste08}. The latter itself is based on the proof of \cite[Thm.~2.4]{BDD04}.

\begin{proof}[Proof of \cref{Thm:complexity results}]
    First, we define the positive sequences $a:\N_0\cup \{-1\}\rightarrow \R^+$ with $a(p):=(p+2)^{-2}$ and $b:\N_0\rightarrow \R^+$ with $b(p):=2^{\frac{p}{d+1}\cdot \min\left(\frac{s_2}{s_1},1\right)}$. Those have the following properties, which can be shown with elementary techniques from calculus:
    \begin{enumerate}[label=(\roman*)]
        \item $\sum\limits_{p=-1}^\infty a(p):=C_1<\infty$, \label{Proof of Thm:complexity results - prop initial sequences - item 1}
        \item $\sum\limits_{p=0}^\infty b(p)2^{-\frac{p}{d}\cdot \min\left(\frac{s_2}{s_1},1\right)} :=C_2(d, s_1, s_2)<\infty$, and \label{Proof of Thm:complexity results - prop initial sequences - item 2}
        \item $\inf\limits_{p\in \N_0} a(p)b(p):=C_3(d, s_1, s_2) >0$. \label{Proof of Thm:complexity results - prop initial sequences - item 3}
    \end{enumerate}
    We will need the constant $A:=C\left(\frac{2^{\frac{1}{d}\min\left(\frac{s_2}{s_1}, 1\right)}}{1-2^{-\frac{1}{d}\min\left(\frac{s_2}{s_1}, 1\right)}}+1\right)C_2$, where $C$ is the constant from the upper bound of \cref{corollary relationship diameter level - ATOMIC_SPLIT} and \cref{Lemma distance newly created element}. Moreover we shall need the properties \ref{Proof of Thm:complexity results - prop initial sequences - item 1}-\ref{Proof of Thm:complexity results - prop initial sequences - item 3} throughout the proof. Now let $k\in \N$ (for $k=0$ there is nothing to be shown), $\tilde{\calP}$ any intermediate partition between $\calP_0$ and $\calP:=\calP_k$ during the call of $\textup{MARK}\_\textup{AND}\_\textup{REFINE}(\calP_0, d, s_1, s_2)$, and $\calM\subset\bigcup\limits_{i=1}^{k}\calM_{i-1}$ its subset of marked elements for which the subroutine $\textup{PATCH}\_\textup{REFINE}(\tilde{\calP}, \cdot, d, s_1, s_2)$ has been called. Further, we define the following function $\lambda:\calP\times \calM\rightarrow \R$ via
    \begin{align*}
        \lambda(I'&\times S', I\times S)
        \\&:=\begin{cases}
            a(\ell(I\times S) - \ell (I'\times S')), &\text{if } \textup{dist}_\xi(I'\times S', I\times S)<A\,2^{-\frac{\min\left(\frac{s_2}{s_1}, 1\right)}{d}\ell(I'\times S')}\text{ and }\ell(I'\times S')\le \ell(I\times S)+1,
            \\ 0, &\text{otherwise.} 
        \end{cases}
    \end{align*}
    Now let $I\times S\in \calM$ and $l\in \N_0$ with $l\le\ell(I\times S)+1$. Due to the spatial shape-regularity of the prisms (cf. \mbox{\cref{Remark_Temporal_Bisection_Active_Triangulation}\ref{Remark_Temporal_Bisection_Active_Triangulation_1}}) and \cref{corollary relationship diameter level - ATOMIC_SPLIT}, there are only finitely many $I'\times S'\in \calP$ with $\ell(I'\times S')=l$ such that $\textup{dist}_\xi(I'\times S', I\times S)\le A\,2^{-\frac{\min\left(\frac{s_2}{s_1}, 1\right)}{d}l}$, whereas the number is uniformly bounded by $d, s_1, s_2, \kappa_{\calP_0}$, and $\mu(\calP_0)$. Therefore,
    \begin{align*}
        \sum\limits_{I'\times S'\in \calP}^{}\lambda(I'&\times S', I\times S)= \sum\limits_{l=0}^{\ell(I\times S)+1}\sum\limits_{\substack{I'\times S'\in \calP,\\ \ell(I'\times S')=l}}^{}\lambda(I'\times S', I\times S)\lesssim_{\, d, s_1, s_2, \kappa_{\calP_0}, \mu(\calP_0)} \sum\limits_{l=-1}^{\ell(I\times S)} a(l)\le \sum\limits_{l=-1}^{\infty} a(l) \sim 1,
    \end{align*}
    where we have applied \ref{Proof of Thm:complexity results - prop initial sequences - item 1} from the beginning of the proof in the final step. This yields
    \begin{align}\label{Proof of Thm:complexity results - step 1}
       \sum\limits_{I\times S\in \calM}\sum\limits_{I'\times S'\in \calP}^{}\lambda(I'&\times S', I\times S)\lesssim_{\, d, s_1, s_2, \kappa_{\calP_0}, \mu(\calP_0)} \sum\limits_{I\times S\in \calM} 1 = \#\calM.
    \end{align}
    Next, we will consider $I'\times S'\in \calP\setminus(\calP\cap \calP_0)$ and the chain of prisms $(I_0\times S_0, I_1\times S_1, \dots , I_e\times S_e)$, $e\in \N$, $I_0\times S_0:=I'\times S'$, such that $I_e\times S_e\in \calP_0$, and $I_{i-1}\times S_{i-1}$ has been created during a call of the method \mbox{$\textup{PATCH}\_\textup{REFINE}(\tilde{\calP}, I_{i}\times S_{i}, d, s_1, s_2)$} with $I_{i}\times S_{i}\in \calM$, $i=1,\dots, e$. We know that \mbox{$\ell(I_{i-1}\times S_{i-1})\le \ell(I_{i}\times S_{i})+1$}, $i=1,\dots, e$, due to \cref{Theorem - conformity 1-reg Patch refine}. Therefore, it exists $\nu\in \{1, \dots, e\}$ minimal with $\ell(I_\nu\times S_\nu)= \ell(I_0\times S_0)-1$. In particular, this implies that $\ell(I_i\times S_i)>\ell(I_\nu\times S_\nu)=\ell(I_0\times S_0)-1$ for $i\in\{0,\dots, \nu - 1\}$.
    
    Now let $j\in \{1, \dots, \nu\}$. Due to \cref{corollary relationship diameter level - ATOMIC_SPLIT} and \cref{Lemma distance newly created element}, we obtain
    \begin{align*}
        \textup{dist}_\xi(I_0\times S_0, I_j \times S_j)&\le \sum\limits^{j}_{i=1} \textup{dist}_\xi(I_{i-1}\times S_{i-1}, I_i\times S_i) + \sum\limits^{j-1}_{i=1}\diam_{\xi}(I_i\times S_i)
        \\&\le C\left(\sum\limits_{i=1}^{j} 2^{\frac{1}{d}\min\left(\frac{s_2}{s_1},1\right)}\sum\limits_{k=\ell(I_{i-1}\times S_{i-1})}^{\ell(I_{i}\times S_{i})}2^{-\frac{k}{d}\min\left(\frac{s_2}{s_1}, 1\right)}  + \sum\limits_{i=1}^{j-1} 2^{-\frac{1}{d}\min\left(\frac{s_2}{s_1},1\right)\ell(I_i\times S_i)}  \right).
    \end{align*}
    Since for all $i\in \{1,...,j\}$,
    \begin{align*}
        \sum\limits_{k=\ell(I_{i-1}\times S_{i-1})}^{\ell(I_{i}\times S_{i})}2^{-\frac{k}{d}\min\left(\frac{s_2}{s_1}, 1\right)}  < 2^{-\frac{1}{d}\min\left(\frac{s_2}{s_1},1\right)\ell(I_{i-1}\times S_{i-1})}\sum\limits_{k=0}^{\infty} 2^{-\frac{k}{d}\min\left(\frac{s_2}{s_1},1\right)} = \frac{2^{-\frac{1}{d}\min\left(\frac{s_2}{s_1},1\right)\ell(I_{i-1}\times S_{i-1})}}{1-2^{-\frac{1}{d}\min\left(\frac{s_2}{s_1},1\right)}},
    \end{align*}
    due to the well-known formula for a geometric sum with respect to $2^{-\frac{1}{d}\min\left(\frac{s_2}{s_1},1\right)}<1$, we can further estimate
    \begin{align*}
        \textup{dist}_\xi(I_0\times S_0, I_j \times S_j) &< C\left(\frac{2^{\frac{1}{d}\min\left(\frac{s_2}{s_1}, 1\right)}}{1-2^{-\frac{1}{d}\min\left(\frac{s_2}{s_1}, 1\right)}}+1\right) \sum\limits_{i=0}^{j-1} 2^{-\frac{1}{d}\min\left(\frac{s_2}{s_1},1\right)\ell(I_i\times S_i)} 
        \\&= C\left(\frac{2^{\frac{1}{d}\min\left(\frac{s_2}{s_1}, 1\right)}}{1-2^{-\frac{1}{d}\min\left(\frac{s_2}{s_1}, 1\right)}}+1\right) \sum\limits_{p=0}^{\infty} m(p,j) 2^{-\frac{1}{d}\min\left(\frac{s_2}{s_1},1\right)\left(\ell(I_0\times S_0)+p\right)} ,
    \end{align*}
    where $m(p,j):=\#\{I_i\times S_i \mid \ell(I_i\times S_i) = \ell(I_0\times S_0) + p, \ i=0,\dots, j-1 \}$ for any $p\in \N_0$. Now we have to distinguish two cases. If $m(p,\nu)\le b(p)$ for all $p\in \N_0$, then using \ref{Proof of Thm:complexity results - prop initial sequences - item 2} and the definition of the constant $A$ yields
    \begin{align}\label{Proof of Thm:complexity results - intermediate step for step 3} 
        \textup{dist}_\xi(I_0\times S_0, I_\nu \times S_\nu) &< C\left(\frac{2^{\frac{1}{d}\min\left(\frac{s_2}{s_1}, 1\right)}}{1-2^{-\frac{1}{d}\min\left(\frac{s_2}{s_1}, 1\right)}}+1\right) \left(\sum\limits_{p=0}^{\infty} b(p) 2^{-\frac{1}{d}\min\left(\frac{s_2}{s_1},1\right)p}\right) 2^{-\frac{1}{d}\min\left(\frac{s_2}{s_1},1\right)\ell(I_0\times S_0)}\notag 
        \\&= C\left(\frac{2^{\frac{1}{d}\min\left(\frac{s_2}{s_1}, 1\right)}}{1-2^{-\frac{1}{d}\min\left(\frac{s_2}{s_1}, 1\right)}}+1\right)C_2 \, 2^{-\frac{1}{d}\min\left(\frac{s_2}{s_1},1\right)\ell(I_0\times S_0)} = A\,2^{-\frac{1}{d}\min\left(\frac{s_2}{s_1},1\right)\ell(I_0\times S_0)}. 
    \end{align}
    and therefore, it holds
    \begin{align}\label{Proof of Thm:complexity results - step 2}
        \sum\limits_{I\times S\in \calM}\lambda(I'\times S', I\times S) \ge \lambda(I_0\times S_0, I_\nu\times S_\nu)= a( \ell(I_\nu\times S_\nu) - \ell(I_0\times S_0) ) = a(-1)=1.
    \end{align}
    Otherwise, there exists one $p\in \N_0$ with $m(p,\nu)>b(p)$. Since the definition of $m(p, \cdot)$ implies that it is monotonically increasing in the second component, there is $j=j(p)\in \{1,\dots, \nu\}$ minimal such that \mbox{$m(p, j(p))> b(p)$}. Now let $p^*\in \N_0$ be chosen such that $j^*:=j(p^*)$ is the minimal value of the sequence $(j(p))_{p\in \N_0}$. In fact, it must hold $j^*-1\ge 1$, i.e., $j^*\ge 2$, since $j^*=1$ would lead to the contradiction $1 = m(p^*, 1) > b(p^*) \ge 1$. Therefore, for any $p\in \N_0$, $m(p, j^*-1)\le m(p, j(p)-1)\le b(p)$, due to the monotonicity of $m(p,\cdot)$ and the minimality condition imposed on $j(p)$. 
    
    This implies $\textup{dist}_\xi(I_0\times S_0, I_k\times S_k)< A\,2^{-\frac{1}{d}\min\left(\frac{s_2}{s_1},1\right)\ell(I_0\times S_0)}$ for any $k\in \{0,\dots, j^*-1\}$, as above in \eqref{Proof of Thm:complexity results - intermediate step for step 3}. In particular, for all such $k\ge 1$ with $\ell(I_k\times S_k)= \ell(I_0\times S_0) +p^*$, we can conclude that $\lambda(I_0\times S_0, I_k\times S_k)=a(p^*)$. Then, similar to before, we can estimate
    \begin{align}\label{Proof of Thm:complexity results - step 3}
        \sum\limits_{I\times S\in \calM}\lambda(I'\times S', I\times S) &\ge \sum\limits_{\substack{k\in \{1,\dots, j^*-1\}\text{ with} \\\ell(I_k\times S_k)=\ell(I_0\times S_0)+p^*}}\lambda(I_0\times S_0, I_k\times S_k) \ge \big( m(p^*, j^*)-1\big) a(p^*) \ge \frac{1}{2}m(p^*, j^*)a(p^*)
        \\&> \frac12 b(p^*)a(p^*) \ge \frac12 \inf\limits_{p\in \N_0} a(p)b(p)=\frac12 C_3\gtrsim_{\,d, s_1, s_2} 1 ,\notag
    \end{align}
    where we have used $m(p^*,j^*)\ge 2$, due to $m(p^*,j^*)>b(p^*)\ge 1$,  in the third step and employed \ref{Proof of Thm:complexity results - prop initial sequences - item 3} in the last one. Now we can finally prove the assertion, using first \eqref{Proof of Thm:complexity results - step 2} and \eqref{Proof of Thm:complexity results - step 3}, respectively, and then \eqref{Proof of Thm:complexity results - step 1}:
    \begin{align*}
        \# \calP - \#\calP_0 &\le \#(\calP\setminus(\calP_0\cap\calP))\lesssim_{\, d,s_1, s_2}\sum\limits_{I'\times S'\in (\calP\setminus(\calP_0\cap\calP))} \sum\limits_{I\times S\in \calM}\lambda(I'\times S', I\times S) 
        \\&\le \sum\limits_{I\times S\in \calM} \sum\limits_{I'\times S'\in \calP} \lambda(I'\times S', I\times S)  \lesssim_{\, d, s_1, s_2, \kappa_{\calP_0}, \mu(\calP_0)} \#\calM \le \sum\limits_{i=1}^k \#\calM_{i-1}.
        \qedhere
    \end{align*}
\end{proof}

\subsection{Overlay of regular space-time partitions}\label{subsect:Overlay_of_meshes}

The last questions, which we want to pose in this section about mesh refinement, are  if an overlay of two of our prismatic space-time partitions $\calP$ and $\calP'$ generated from $\calP_0$ by the $\textup{ATOMIC}\_\textup{SPLIT}$-method fulfills the regularity conditions from the beginning of \cref{subsect:Patch_refine} (if the meshes $\calP$ and $\calP'$ themselves fulfill them) and how we can estimate its number of elements.

These questions are of course of interest on their own, but also play a vital role in the proof of the fact that the approximation classes in \cref{Main_res:direct} and \cref{Main_res:invers} are themselves quasi-normed spaces, see \cref{cor:patch_refine_fulfills_necessary_condition_on_approximation_spaces} below   as well as the proof of \cref{lem:Generalized_Besov_K_functional_estimate}, which is fundamental for the final proof of the inverse estimates  (\cref{thm:Inverse_estimates} and thus \cref{Main_res:invers}).

First, we will have to properly define the overlay of the space-time partitions. We proceed very similar to the stationary setting described in \cite[Sect.~2.2 and Sect.~3.3]{CKNS08}.

\begin{defi}\label{def:Overlay_of_meshes}
    Let $\calP$ be derived from $\calP_0$ by finitely many applications of the method $\textup{ATOMIC}\_\textup{SPLIT}(\cdot, d, s_1, s_2)$. With $\calP$, we associate a \textbf{refinement graph} $\mathfrak{F}(\calP)$, which is indeed a tree: its vertices are the elements of $\calP_0$ and those prisms generated during the refinement process, with an edge connecting each element with its children created in its refinement via the $\textup{ATOMIC}\_\textup{SPLIT}$. Thus, the leaves of the tree $\mathfrak{F}(\calP)$ are precisely the elements of $\calP$. 
    
    Now, the \textbf{overlay} of two such meshes $\calP_1$ and $\calP_2$ is given as the space-time partition $\calP_1\oplus \calP_2$ corresponding to the tree $\mathfrak{F}(\calP_1)\cup \mathfrak{F}(\calP_2)$ such that
    \begin{align*}
        \mathfrak{F} (\calP_1\oplus \calP_2) = \mathfrak{F}(\calP_1)\cup \mathfrak{F}(\calP_2).
    \end{align*}
\end{defi}

Now we can prove the desired properties of the overlay.

\begin{thm}\label{thm:overlay_of_meshes_properties}
    Let $\calP_1$ and $\calP_2$ be as in \cref{def:Overlay_of_meshes} such that both fulfill the regularity conditions from the beginning of \cref{subsect:Patch_refine}. Then the overlay $\calP_1\oplus \calP_2$ also fulfills these conditions and it holds
    \begin{align*}
        \#(\calP_1\oplus \calP_2)\le \#\calP_1+ \#\calP_2-\#\calP_0.
    \end{align*}
\end{thm}
\begin{proof}
    Again, we will only provide an explicit proof in case of $d\ge 2$, the case of $d=1$ is simpler and works in the same manner with minor changes in the arguments, due to the fact that in this case the mesh is also 1-irregular in space.
    
    We will first exploit the fact that the overlay of two conforming, simplicial triangulations $\calT_1$ and $\calT_2$, defined as $\calT_1\oplus\calT_2$ correspondingly to \cref{def:Overlay_of_meshes}, is itself conforming. The latter has been shown in \cite[Lem.~3.7]{CKNS08}. By the construction of the $\textup{ATOMIC}\_\textup{SPLIT}$,  for every $t\in [0,T]$, it holds
    \begin{align*}
        \calT(t, \calP_1\oplus \calP_2)= \calT(t, \calP_1)\oplus \calT(t, \calP_2).
    \end{align*}
    Since $\calT(t, \calP_1)$ and $\calT(t, \calP_2)$ are conforming because $\calP_1$ and $\calP_2$ are conforming in space, their overlay is also conforming. Therefore, $\calP_1\oplus \calP_2$ is conforming in space. 
    
    Further, we have to show that $\calP_1\oplus \calP_2$ is 1-irregular in time. Assume it is not. Then, there are $I\times S,I'\times S' \in \calP_1\oplus \calP_2$ such that $I\times S$ and $I'\times S'$ touch in time (in the sense of \ref{Enumeration desired properties of the mesh 2} in the beginning of \cref{subsect:Patch_refine}), while $|\ell(I\times S)-\ell(I'\times S')|\ge 2$. Neither $I\times S, I'\times S'\in \calP_1$ nor $I\times S, I'\times S'\in \calP_2$ can hold, since then the partitions $\calP_1$ and $\calP_2$, respectively, would not be 1-irregular in time, which we assumed. Thus, without loss of generality, we may assume $I\times S\in \calP_1$ and $I'\times S'\in \calP_2$ and $\ell(I'\times S')\ge \ell(I\times S)+2$. Since $I\times S$ is in $\calP_1\oplus \calP_2$, this means that $I\times S$ is a leaf of the corresponding refinement graph $\mathfrak{F}(\calP_1\oplus \calP_2)$. Furthermore, since $\calP_2$ is a space-time partition, there is an element $I''\times S''\in \calP_2$ with non-empty intersection with $I\times S\subset \Omega_T$, i.e., $|(I\times S)\cap(I''\times S'')|>0$, which touches $I'\times S'$ in time. Due to the 1-irregularity of $\calP_2$, this yields $\ell(I'\times S')\le \ell(I''\times S'')+1$. Together with the above $\ell(I'\times S')\ge \ell(I\times S)+2$, this implies $\ell(I\times S)\le \ell(I'\times S')-2 \le \ell(I''\times S'')-1<\ell(I''\times S'')$. By the construction of $\calP_2$, this necessarily implies that $I''\times S''\in \textup{ATOMIC}\_\textup{SPLIT}(I\times S, d, s_1, s_2)^{\otimes k}$ for some $k\in \N$. But this is a contradiction to the fact that $I\times S$ is a leaf of $\mathfrak{F}(\calP_1\oplus \calP_2)$. So $\calP_1\oplus \calP_2$ has to be 1-irregular in time, too. 

    Therefore, the overlay indeed fulfills the asserted regularity conditions. Now it remains to show the estimate for $\#(\calP_1\oplus \calP_2)$. For $I_0\times S_0\in\calP_0$, let $\mathfrak{F}_i(I_0\times S_0)\subset \mathfrak{F}_i(\calP_i)$, $i=1,2$, be the tree which originates from the node corresponding to $I_0\times S_0$ and $\calP_i(I_0\times S_0)\subset \calP_i$, $i=1,2$, the corresponding space-time partition of $I_0\times S_0$. $\mathfrak{F}(I_0\times S_0)\subset \mathfrak{F}(\calP_1\oplus \calP_2)$ and $(\calP_1\oplus \calP_2)(I_0\times S_0)\subset \calP_1\oplus \calP_2$ shall further be understood correspondingly. By the definition of the overlay, $\mathfrak{F}(I_0\times S_0)\subset \mathfrak{F}_1(I_0\times S_0) \cup \mathfrak{F}_2(I_0\times S_0)$ and hence
    \begin{align*}
        \#(\calP_1\oplus \calP_2)(I_0\times S_0)\le \#\calP_1(I_0\times S_0)+ \#\calP_2(I_0\times S_0)-1,
    \end{align*}
    since the three trees share the common origin $I_0\times S_0$. Summing over all $I_0\times S_0\in\calP_0$ proves the assertion.
\end{proof}

This allows us, to show the aforementioned result on approximation classes for our adaptive space-time refinement with our method $\textup{PATCH}\_\textup{REFINE}$.

\begin{cor}\label{cor:patch_refine_fulfills_necessary_condition_on_approximation_spaces}
    Let $p,q\in (0,\infty]$, $s, s_1, s_2\in (0,\infty)$, $\alpha_1, \alpha_2,\beta_1, \beta_2\in \R^+(s_1, s_2)$, $r_1, r_2\in \N_{\ge 2}$, and $\alpha_i<\min\big(1+\frac1p,r_i\big)$, $i=1,2$. Then the approximation classes from \cref{Main_res:direct} and \cref{Main_res:invers}, i.e.,
    \begin{align*}
        \mathbb{A}_{s,q}(X, \V^{r_1, r_2}_\bullet,\textup{PATCH}\_\textup{REFINE}(\cdot, \cdot, d, s_1, s_2)), \quad X\in \left\{L_{p}(\Omega_T), B^{\alpha_1, \alpha_2}_{p,p}(\Omega_T),  \widehat{B}^{\beta_1, \beta_2}_{p,p}(\Omega_T)\right\}
    \end{align*}
    are quasi-normed spaces, which are complete for $X=L_{p}(\Omega_T)$.
\end{cor}
\begin{proof}
    First, we notice that $\V^{r_1, r_2}_\bullet\subset X$ for any of the above choices of $X$. For $X=L_p(\Omega_T)$ this is clear, for $X=B^{\alpha_1, \alpha_2}_{p,p}(\Omega_T)$ this follows from \cref{lem:FEM_elements_of_Aniso_Besov}, and for $X=\widehat{B}^{\beta_1, \beta_2}_{p,p}(\Omega_T)$ this is also true by \cref{rem:Generalized_Besov_Spaces}\ref{rem:Generalized_Besov_Spaces - item 1} further below. Second, we consider the corresponding sequences of sets $\mathfrak{A}_N:=\mathfrak{A}_N(\textup{PATCH}\_\textup{REFINE}(\cdot, \cdot, d, s_1, s_2), \V^{r_1, r_2}_{\bullet})$, $N\in \N_0$, as defined in \eqref{eq:necessary_condition_on_approximation_spaces} of \cref{subsubsect:Approximation classes}. As mentioned there, it remains to show that $(\mathfrak{A}_N)_{N\in \N_0}$ fulfills the following conditions from \cite[Sect.~1.2]{Pie81}.

    Let $N,M\in \N_0$. We observe that $0\in \mathfrak{A}_N\subset \mathfrak{A}_{N+1}$. Further, for every $\lambda \ne 0$, $\lambda\mathfrak{A}_N=\mathfrak{A}_N$, since for every admissible partition $\calP$ corresponding to $\mathfrak{A}_N$, $\mathbb{V}^{r_1,r_2}_\calP$ is a vector space. Additionally, $\mathfrak{A}_N + \mathfrak{A}_M\subset \mathfrak{A}_{N+M}$ due to \cref{thm:overlay_of_meshes_properties}: If two meshes are admissible with respect to $\mathfrak{A}_N$ and $\mathfrak{A}_M$, respectively, then their overlay is admissible with respect to $\mathfrak{A}_{N+M}$. These are all the properties from \cite[Sect.~1.2]{Pie81} which we had to show for $(\mathfrak{A}_N)_{N\in \N_0}$. Further, according to \cite[Prop.~1 from Sect.~1.4]{Pie81}, the approximation classes with respect to $X=L_p(\Omega_T)$ are complete, since $L_p(\Omega_T)$ is complete. This concludes the proof.
\end{proof}

\section{Support of basis functions and quasi-interpolation}\label{sect:investigation of mesh geometry and Quasi-interpolation}

\subsection{An investigation of the mesh geometry of hanging nodes}

Let $\hat{I}:=[0,1]$, $\hat{S}\subset \R^d$ be the $d$-dimensional standard simplex\footnote{By \textbf{standard simplex} we mean the simplex in $\R^d$ given as the convex hull of the origin and the canonical basis vectors.}, and $r_1, r_2 \in \N_{\ge 2}$. The \textbf{Lagrangian basis nodes} with respect to $\Pi^{r_1}(\hat{I})$ and $\Pi^{r_2}(\hat{S})$, i.e., polynomials on $\hat{I}$ and $\hat{S}$ of order $r_1$ and $r_2$, respectively, are given by
\begin{align*}
    \calL_{r_1}(\hat{I}):=\left\{\frac{n}{r_1-1}\bigg|\  n=0,\dots,r_1-1\right\}\qquad\text{and}\qquad \calL_{r_2}(\hat{S}):=\left\{\frac{\alpha}{r_2-1}\bigg|\  \alpha\in \N_0^d, \, |\alpha|\le r_2-1\right\}.
\end{align*}

Further, let $I$ be an interval and $S$ a $d$-dimensional simplex. Now consider the corresponding bijective, affine transformations $\Phi_I:\hat{I}\rightarrow \overline{I}$ and $\Phi_S:\hat{S}\rightarrow S$ that allow us to define the (isotropic) Lagrangian basis nodes with respect to $I$ and $S$ as follows:
\begin{align*}
    \calL(I):=\calL_{r_1}(I):=\Phi_I\left(\calL_{r_1}(\hat{I})\right) \qquad\text{and}\qquad \calL(S):=\calL_{r_2}(S):=\Phi_S\left(\calL_{r_2}(\hat{S})\right).
\end{align*}
Clearly, $\dim \Pi^r(I) = \# \calL_{r}(I)$ and $\dim \Pi^r(S) = \# \calL_{r}(S)$ which yields $\dim \Pi_{t,\bm{x}}^{r_1,r_2}(I\times S) = \# \calL(I\times S)$ for the \textbf{anisotropic Lagrangian nodes} $\calL(I\times S):=\calL(I)\times \calL(S) $. In order to recall the definition of these anisotropic polynomial spaces, see \cref{subsubsect:Prelim_General_Setting_FEM_Spaces}.  In particular:

\begin{lem}\label{lem:Conformity_Lagrange_nodes}
    Let $\calT$ be a conforming simplicial triangulation of $\Omega$ and $S, S'\in \calT$. Then 
    \begin{align*}
        S\cap \calL(S') = S'\cap\calL(S).
    \end{align*}
\end{lem}
\begin{proof}
    This is a direct consequence of fact that $S\cap S'$ is a common hyperface of $S$ and $S'$, due to the conformity of $\calT$, and the definition of the Lagrange nodes.
\end{proof}

Throughout the rest of this section we consider a space-time partition $\calP$ derived by finitely many applications of $\textup{PATCH}\_\textup{REFINE}(\cdot, \cdot, d, s_1, s_2)$, for fixed $s_1, s_2\in (0, \infty)$, starting with $\calP_0$ from \cref{sect:space_time_partition_and_refinement}. Thus, $\calP$ is a non-overlapping partition of $\Omega_T$ which fulfills the properties \ref{Enumeration desired properties of the mesh 1} and \ref{Enumeration desired properties of the mesh 2} from the beginning of \cref{subsect:Patch_refine}. Recall that on $\calP$  the space of \textbf{continuous anisotropic finite elements} is given by 
\begin{align*}
    \mathbb{V}^{r_1, r_2}_\calP=\left\{F\in C(\Omega_T, \R)\mid F_{|I\times S}\in \Pi^{r_1, r_2}_{t,\bm{x}}(I\times S), \, I\times S \in \calP\right\}.
\end{align*}
The set of all local degrees of freedom is given by $\calL(\calP):=\bigcup\limits_{I\times S\in \calP}\calL(I\times S)$. In particular, for the (possibly) \textbf{discontinuous} version of the above finite element space, i.e.,
\begin{align*}
    \mathbb{V}^{r_1, r_2}_{\calP, \textup{DC}}=\left\{F:\Omega_T\rightarrow \R\mid F_{|I\times S}\in \Pi^{r_1, r_2}_{t,\bm{x}}(I\times S), \, I\times S \in \calP\right\},
\end{align*}
$\dim \mathbb{V}^{r_1, r_2}_{\calP, \textup{DC}}= \sum\limits_{I\times S\in \calP} \#\calL(I\times S)$ holds true.

However, due to the continuity condition for the finite elements, the local degrees are not all global degrees of freedom. In the following theorem we will see that the set of \textbf{free nodes} $\calF(\calP):=\calL(\calP)\setminus \calH(\calP)$ indeed corresponds to the global degrees of freedom. Here, $\calH(\calP)$ is the set of \textbf{hanging nodes} in $\calP$, i.e.,
\begin{align*}
    \calH(\calP):=\{\nu \in \calL(\calP) \mid \exists \,I\times S\in \calP: \nu\in \overline{I\times S}\ \text{and}\ \nu \notin \calL(I\times S) \}.
\end{align*}

\begin{thm}\label{thm:basis free nodes}
    For $\nu \in \calF(\calP)$, $\phi_\nu\in \mathbb{V}^{r_1, r_2}_{\calP}$ is well-defined by $\phi_\nu(\nu')=\delta_{\nu,\nu'}$ for all $\nu'\in\calF(\calP)$. Further, $(\phi_\nu)_{\nu\in \calF(\calP)}$ forms a basis of $\mathbb{V}^{r_1, r_2}_{\calP}$.
\end{thm}
\begin{proof}
    In \cite[Thm.~3.1]{Grae11}, the corresponding result has been shown for simplicial meshes covering a polyhedral Lipschitz domain and fulfilling the analogue to the \enquote{hierarchical structure}-property from \cref{Enumeration further properties of the mesh}\ref{Enumeration further properties of the mesh - enum 1}. Since our meshes also fulfill this property, one can obtain the above corresponding result following the lines of the proof (and its prerequisites) there.
\end{proof}

\begin{center}
    \begin{figure}[H]
        \begin{center}
            \includegraphics[width=9cm]{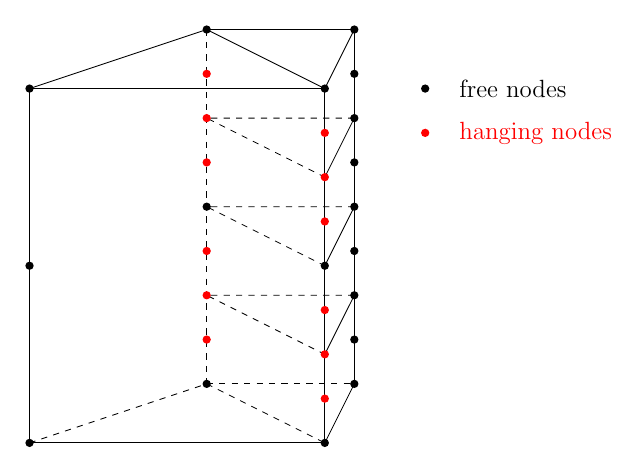}
            \captionof*{figure}{Illustrative example for $r_1=3$ and $r_2=2$}
        \end{center}
    \end{figure}
\end{center}

The main result of this section is the following characterization of $\supp \phi_\nu$, for $\nu\in \calF(\calP)$, which will be essential for the construction of an appropriate interpolation operator.

\begin{thm}\label{thm:support_in_power_omega_domain}
    Let $\nu \in \calF(\calP)$. If $I\times S\in \calP$ and $\nu\in \calL(I\times S)$, then $\supp \phi_\nu \subset \bigcup\limits_{\nu\in \overline{I'\times S'}}\omega_\calP^{j(d)-1}(I'\times S')\subset \omega^{j(d)}_\calP(I\times S)$, with $j(1):=2$ and $j(d):=3$ for $d\in \N_{\ge 2}$. Furthermore, 
    \[
    |I'\times S'| \sim_{\,d,s_1, s_2, \kappa_{\calP_0},\mu(\calP_0)} |\supp \phi_\nu|,
    \]
    for all $I'\times S' \in \calP$ such that $I'\times S' \subset \supp \phi_\nu$.
\end{thm}

In particular, it follows from \cref{thm:support_in_power_omega_domain} that the support of $\phi_{\nu}$ is contained in the extended neighborhood of degree 3 of the prism  $I\times S$ if $\dim(S)\geq 2$, and of degree 2 if $\dim(S)=1$. Moreover, the total measure of $\supp \phi_{\nu}$ is comparable to the measure of any prism contained in the support. The rest of this section is dedicated to the proof of this theorem, which we develop through a series of rather technical lemmas.

First, we consider the set of elements of $\calP$ that cause $\nu\in \calH(\calP)$ to \enquote{hang}, i.e.,
\begin{align*}
    \calP_\calH(\nu):=\{I\times S\in \calP\mid \nu\in \overline{I\times S}\ \text{and}\ \nu \notin \calL(I\times S) \}
\end{align*}

We observe the following:

\begin{rem}\label{Rem:Hanging_nodes}
    \begin{enumerate}[label=(\roman*)]
        \item If $I\times S\in \calP_\calH(\nu)$, clearly $\nu \in \partial(I\times S) = (\partial I \times S )\cup (\mathring{I}\times \partial S)$, since $\calP$ is a non-overlapping partition of $\Omega_T$. In particular, the above union representation is disjoint. \label{Rem:Hanging_nodes_i}
        \item Decomposing $\nu = (\nu_t, \nu_{\bm{x}})$ with $\nu_t\in [0,T]\subset \R$ and $\nu_{\bm{x}}\in \Omega\subset \R^d$ yields
        \begin{align*}
            \nu= (\nu_t, \nu_{\bm{x}})\in \calL(I\times S)\quad\iff \quad \nu_t\in \calL(I) \text{ and } \nu_{\bm{x}}\in \calL(S).
        \end{align*}\label{Rem:Hanging_nodes_ii}
    \end{enumerate}
\end{rem}

These observations can be used to investigate the geometry of hanging nodes more precisely. First, the following lemmata will shed some light on the mesh geometry around hanging nodes, if the nodes \enquote{hang in space}.

\begin{lem}\label{lem:Hanging_in_space_part_1}
    Let $\nu \in \calH(\calP)\cap \calL(I\times S)$ for some $I\times S\in \calP$ and $I'\times S'\in \calP_\calH(\nu)$ with $\nu\in \partial I'\times S'$. Then, 
    \begin{align*}
        \nu\in (\partial I\times S)\cap (\partial I'\times S') \quad\text{and}\quad \overline{I}\cap \overline{I'}=\partial I\cap \partial I' = \{\nu_t\}.
    \end{align*}
\end{lem}
\begin{proof}
    Since $\nu_t\in \partial I' \subset \calL(I')$, $\nu_{\bm{x}}\notin \calL(S')$ follows due to \cref{Rem:Hanging_nodes}\ref{Rem:Hanging_nodes_ii} and the assumption $I'\times S'\in \calP_\calH(\nu)$. Therefore, $\nu_t \in \calL(I)\cap \calL(I')\subset \overline{I}\cap \overline{I'}$ , since $\nu\in \calL(I\times S)$. $I$ and $I'$ are both results of classical bisections of elements in $\calI_0$, which is a disjoint partition of $[0,T]$. Therefore, either $\#(\overline{I}\cap \overline{I'})=1$, which would already imply the assertion, or $|I\cap I'|>0$. If the latter was true, there would be an $\varepsilon\in \R$ (with \enquote*{small} modulus), such that $\nu_t+\varepsilon\in I\cap I'$, i.e., $S, S'\in \calT(\nu_t+\varepsilon, \calP)$. Therefore, 
    \begin{align*}
        \nu_{\bm{x}} \in S' \cap \calL(S) = S\cap \calL(S') \subset \calL(S') ,
    \end{align*}
    due to \cref{lem:Conformity_Lagrange_nodes} and the conformity of $\calT(\nu_t+\varepsilon, \calP)$, which is a contradiction to $\nu_{\bm{x}}\notin \calL(S')$.
\end{proof}
This leads us directly to the following result.
\begin{lem}\label{lem:Hanging_in_space_part_2}
    We stay in the setting of \cref{lem:Hanging_in_space_part_1}. Then there exists $I''\times S'' \in \calP_\calH(\nu)$ such that $\nu \in \partial I'' \times S''$, $S\in \textup{BISECT}(S'', d)$, and the Lagrange nodes of $I''\times S''$ in the interior of the facet $\{\nu_t\}\times S''$ of $I'' \times S''$ are free, i.e.,
    \begin{align*}
        (\{\nu_t\}\times \mathring{S''}) \cap \calL(I''\times S'') \subset \calF(\calP).
    \end{align*}
    If $d=1$, also 
    \begin{align*}
        (\{\nu_t\}\times \partial S'') \cap \calL(I''\times S'') \subset \calF(\calP).
    \end{align*}
\end{lem}

    \begin{figure}[H]
	    \begin{center}
		    \includegraphics[height=7cm]{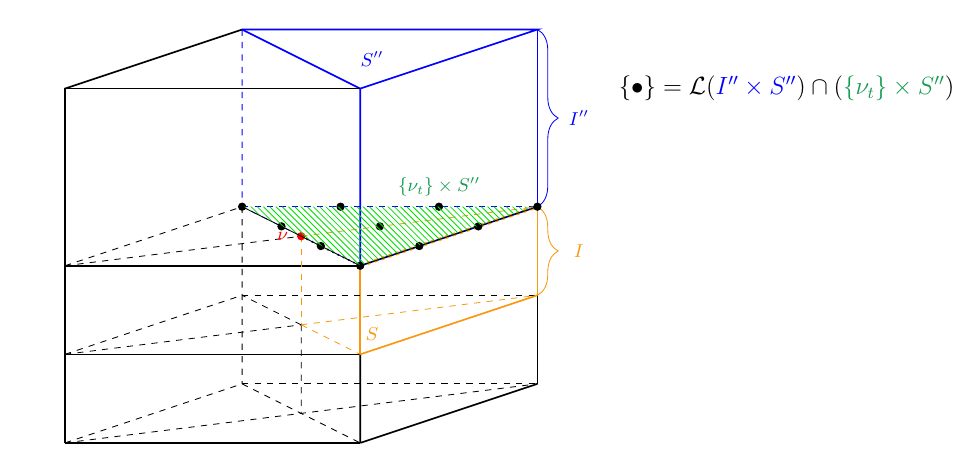}
	    \end{center}
        \captionof*{figure}{Illustration in the case of $r_2 = 4$, i.e., elements that are cubic in space}
    \end{figure}

\begin{proof}
    From \cref{lem:Hanging_in_space_part_1}, we obtain that $\nu \in (\partial I \times S)\cap (\partial I' \times S')$ and $\overline{I}\cap \overline{I'}=\partial I\cap \partial I' = \{\nu_t\}$. Without loss of generality we can therefore assume that $t<t'$ for all $(t,t')\in I \times I'$. Then $S'\in \calT(\nu_t, \calP)$ and setting $\varepsilon_0:=\min\limits_{J\times R\in \calP} \diam(J)$, yields $S\in \calT(\nu_t-\varepsilon, \calP)$ for all $\varepsilon\in (0,\varepsilon_0)$. Therefore, we define
    \begin{align*}
        \calT_{down} := \calT\left(\nu_t-\frac{\varepsilon_0}{2}, \calP\right) \quad\text{and}\quad \calT_{up} :=\calT\left(\nu_t, \calP\right).
    \end{align*}
    Now we consider the set $M$ of elements of $\calT_{up}$ that overlap with $S$ (and contain $\nu_x$), i.e.,
    \begin{align*}
        M:= \{S''\in \calT_{up}\mid |S\cap S''|>0\quad\text{and}\quad \nu_{\bm{x}}\in S''\}
    \end{align*}
    
    Let us now choose an arbitrary $S''\in M$. Due to the construction of $\calT_{up}$, there is a unique interval $I''$ such that $I''\times S''\in \calP$ and $\nu \in \overline{I''\times S''}$. 
    
    First, we will show that $I$ and $I''$ do not overlap. Due to $|S\cap S''|>0$, $|I\cap I''|>0$ would imply that \mbox{$|(I\times S)\cap (I''\times S'')|>0$}. Since $\calP$ is an non-overlapping covering of $\Omega_T$, this would yield $I\times S = I''\times S''$, in particular, $S=S''\in \calT_{up}$. Therefore, since $S', S''\in \calT_{up}$, \cref{lem:Conformity_Lagrange_nodes} implies that
    \begin{align}\label{lem:Hanging_in_space_part_2 - equation}
        \nu_{\bm{x}} \in S' \cap \calL(S) = S' \cap \calL(S'') = S'' \cap \calL(S') \subset \calL(S'),
    \end{align}
    due to the conformity of $\calT_{up}$. Now this would be a contradiction to our assumption $I'\times S'\in \calP_\calH(\nu)$, according to \cref{Rem:Hanging_nodes}\ref{Rem:Hanging_nodes_ii}, because $\nu_t\in \partial I'\subset \calL(I')$. 

    Therefore, we know that $\overline{I}\cap\overline{I''}= \partial I\cap \partial I'' = \{\nu_t\}$ (which yields $\nu\in \partial I'' \times S''$, since $S''\in M$) and $\dim(S\cap S'')=d$. So the $1$-irregular rule in time implies  
    \begin{align*}
        |\ell(I\times S) - \ell(I''\times S'')|\le 1\quad\iff \quad |\ell( S) - \ell( S'')|\le 1.
    \end{align*}
    Now we distinguish the three cases that this inequality allows for, regarding the relationship between $S$ and $S''$. Proceeding as in \eqref{lem:Hanging_in_space_part_2 - equation}, we obtain that $S=S''$ is not possible, since it would again imply $\nu_{\bm{x}}\in \calL(S') $. If $S''\in \textup{BISECT}(S,d)$ was true, this would yield $\calL(S)\subset \calL(S'') $, which would in turn lead to the same contradiction as for $S=S''$, if we replace the first \enquote*{$=$} in \eqref{lem:Hanging_in_space_part_2 - equation} by \enquote*{$\subset$}. Therefore, $S\in \textup{BISECT}(S'',d)$ is true.
    
    Furthermore,  $I''\times S''\in \calP_\calH(\nu)$ because $\nu_{\bm{x}}\notin \calL(S'')$ has to be true, as one can see as follows: If $\nu_{\bm{x}}\in \calL(S'')$ was true, then $\nu_{\bm{x}}\in \calL(S')$, since $S', S''\in \calT_{up}$, which is conforming. But this would be a contradiction to $I'\times S'\in \calP_\calH(\nu)$ as it has been used multiple times in the argumentation above.

    The assertion(s) about the facet $\{\nu_t\}\times S''$ of $I''\times S''$ remain(s) to be shown. In fact, this really is a facet, since $\nu_t\in \partial I''$ as shown above. Now assume $\nu'\in \calH(\calP)\cap (\{\nu_t\}\times \mathring{S''})$ and denote $\nu'= (\nu_t, \nu'_{\bm{x}})$ with $\nu'_{\bm{x}}\in \mathring{S''} \subset \Omega$. Since $S\in \textup{BISECT}(S'',d)$, $\{R\in \calT_{up}\cup \calT_{down}\mid \nu'_{\bm{x}}\in R \}=\{S, \tilde{S}, S''\}$, with $\textup{BISECT}(S'',d)=\{S, \tilde{S}\}$. Therefore, 
    \begin{align*}
    \{J\times R\in \calP\mid \nu' \in J\times R\}\subset\{I\times S, I\times\tilde{S} , I''\times S''\}.    
    \end{align*}
    Since $\nu_t\in \partial I \cap \partial I'' \subset \calL(I) \cap \calL(I'')$ and $\calL(S), \calL(\tilde{S})\supset \calL(S'')$, we conclude that $\nu'\in \calF(\calP)$.
    
    Lastly, consider the case $d=1$ and $\nu' \in \{\nu_t\}\times \partial S''$. If $\nu'\in \calH(\calP)$, there would have to exist $I'''\times S'''\in \calP_\calH(\nu')$. Thus, $S'''\in \calT_{up}\cup \calT_{down}$ with $\nu'_{\bm{x}}\in R$. Then either, as before $S'''\in \{S, \tilde{S}, S''\}$ and therefore $\calL(S'')\subset \calL(S''')$, or $\#(S''\cap S''')=1$, which also yields $\nu'_{\bm{x}}\in \partial S''\cap \partial S''' \subset \calL(S'')\cap \calL(S''')$. So $\nu_t \notin \partial I'''$ would have to hold in order to justify $I'''\times S'''\in \calP_\calH(\nu')$. But this in particular implies that $|I'''\cap I|>0$ and $|I''' \cap I''|>0$, since $I'''\times S'''$ then has to touch $I''\times S''$ and $I\times S$ or $I''\times S''$ and $I\times \tilde{S}$. In turn, this yields $\ell(I''')<\ell(I'')$, i.e., $\ell(I'''\times S''') = \ell(I''\times S'')-1 = \ell(I\times S)-2$, where we have used the 1-irregular rule in space in the first step. However, the latter conclusion is a contradiction to this very rule, since $I'''\times S'''$ and $I\times S$ touch in time. This ends the proof.
\end{proof}

Now, we will shed some further light regarding the nodes on the edges of the facet from \cref{lem:Hanging_in_space_part_2}.

\begin{lem}\label{lem:Hanging_in_space_part_3}
    Let $\nu \in \calH(\calP)\cap \calL(I\times S)$ and $I''\times S''\in \calP_\calH(\nu)$ be the element from \cref{lem:Hanging_in_space_part_2}. Further, consider a node $\nu'\in\{\nu_t\}\times \partial S''$. Then either, $\nu'\in \calF(\calP)$ or $\nu'$ \enquote{hangs in time}, i.e., it exists $I'''\times S'''\in \calP_\calH(\nu')$ with $\nu'\in \mathring{I'''} \times \partial S'''$.
\end{lem}
\begin{proof}
    If $\calP_\calH(\nu')=\emptyset$, then $\nu'\notin \calH(\calP)$ and thus $\nu'\in \calF(\calP)$. Otherwise, consider $I'''\times S'''\in \calP_\calH(\nu')$. In contrast to the assertion, assume $\nu'\in \partial I'''\times S'''$. Then, \cref{lem:Hanging_in_space_part_1} yields $\partial I \cap \partial I'' \cap \partial I''' = \{\nu_t\}$, $S\in \textup{BISECT}(S'',d) $, and \mbox{$S''\in \textup{BISECT}(S''',d) $} without loss of generality. Further, $S'''\in \calT_{down}$ or $S'''\in \calT_{up}$, where we use the terminology from the proof of \cref{lem:Hanging_in_space_part_2}. In any case, this is a contradiction to the fact that $\calT(t,\calP)$ is a non-overlapping triangulation of $\Omega$ for any $t\in [0,T]$ (see \cref{Remark_Temporal_Bisection_Active_Triangulation}\ref{Remark_Temporal_Bisection_Active_Triangulation_2}), since $|S\cap S'''|, |S''\cap S'''|>0$, $S\in \calT_{down}$, and $S''\in \calT_{up}$. 
\end{proof}

Now we consider nodes which \enquote{hang in time}.

\begin{lem}\label{lem:Hanging_in_time_part_1}
    Let $\nu \in \calH(\calP)\cap \calL(I\times S)$ for some $I\times S\in \calP$ and $I'\times S'\in \calP_\calH(\nu)$ with $\nu \in \mathring{I'}\times \partial S'$. Then, 
    \begin{align*}
        \nu\in (\overline{I}\times \partial S)\cap (\mathring{I'}\times \partial S'),\quad 0\le\dim(S\cap S')\le d-1, \quad\text{and}\quad \nu_{\bm{x}}\in \calL(S')\cap (\partial S\cap \partial S').
    \end{align*}
\end{lem}
\begin{proof}
    $\nu_t\in \overline{I}\cap \mathring{I'} $ implies $|I\cap I'|>0$. Due to $\nu_{\bm{x}}\in S\cap S'$, $\dim(S\cap S')\ge 0$. If $\dim(S\cap S')=d$ held, $|S\cap S'|>0$ would hold, too, which would imply $|(I\times S)\cap (I'\times S')|>0$. Since $\calP$ is a non-overlapping partition of $\Omega_T$ this would imply that $I\times S= I'\times S'$, in contradiction to $I'\times S'\in \calP_\calH(\nu)$. Therefore, we obtain the asserted $\dim(S\cap S')\in \{0,\dots, d-1\}$. This yields $\nu_{\bm{x}}\in \partial S\cap \partial S'$. Lastly, choose $t\in \mathring{I}\cap \mathring{I'}$. Since $S, S'\in \calT(t, \calP)$, $S'\cap \calL(S)=S\cap \calL(S')\subset \calL(S')$ due to \cref{lem:Conformity_Lagrange_nodes} and the conformity of $\calT(t, \calP)$. Since $\nu_{\bm{x}}\in S'\cap \calL(S)$, this concludes the proof.
\end{proof}

Together, \cref{lem:Hanging_in_space_part_1} and \cref{lem:Hanging_in_time_part_1} show that a hanging node can either \enquote{hang} in time or in space, but not in both.

\begin{cor}\label{Cor:Characterization_hanging_nodes}
    Let $\nu \in \calH(\calP)\cap \calL(I\times S)$ for some $I\times S\in \calP$. Then either $\nu \in \partial I'\times S'$ for all $I'\times S'\in \calP_\calH(\nu)$ or $\nu \in \mathring{I'}\times \partial S'$ for all $I'\times S'\in \calP_\calH(\nu)$.
\end{cor}
\begin{proof}
    Assume the existence of $I'\times S', I''\times S''\in \calP_\calH(\nu)$ with $\nu \in (\partial I'\times S')\cap(\mathring{I''}\times \partial S'')$. We can assume $\{\nu_t\}= \partial I\cap \partial I'\cap \mathring{I''}$ due to \cref{lem:Hanging_in_space_part_1} and $\nu_{\bm{x}}\in \calL(S'')$ due to \cref{lem:Hanging_in_time_part_1}, correspondingly, without loss of generality. The first property implies that $|I'\cap I''|>0$, thus one can choose $t\in\mathring{I'}\cap\mathring{I''}$. Since $\calT(t, \calP)$ is conforming and $S', S''\in \calT(t, \calP)$, we obtain
    \begin{align*}
        \nu_{\bm{x}}\in S'\cap\calL(S'')=S''\cap \calL(S')\subset \calL(S')
    \end{align*}
    due to \cref{lem:Conformity_Lagrange_nodes} and to the second property above. This is a contradiction to $I'\times S' \in \calP_\calH(\nu)$.
\end{proof}

Now we show the analogue to \cref{lem:Hanging_in_space_part_2} for nodes that \enquote{hang in time}.

\begin{lem}\label{lem:Hanging_in_time_part_2}
    We stay in the setting of \cref{lem:Hanging_in_time_part_1}. Then there exists $I''\times S'' \in \calP_\calH(\nu)$ such that $\nu \in \mathring{I''} \times \partial S''$, $I\in \textup{BISECT}(I'', 1)^{\otimes k}$, with $1\le k\le k_0=k_0(d,s_1, s_2, \kappa_{\calP})\in \N_0$, and the Lagrange nodes of $I''\times S''$ in the interior of the line segment $I''\times \{\nu_{\bm{x}}\}$ of $I'' \times S''$ are free, i.e.,
    \begin{align*}
        \mathring{I''}\times \{\nu_{\bm{x}}\} \cap \calL(I''\times S'') \subset \calF(\calP).
    \end{align*}
    If $d=1$, also 
    \begin{align*}
        \partial I''\times \{\nu_{\bm{x}}\} \cap \calL(I''\times S'') \subset \calF(\calP).
    \end{align*}
\end{lem}

\begin{figure}[H]
	\begin{center}
		\includegraphics[height=9cm]{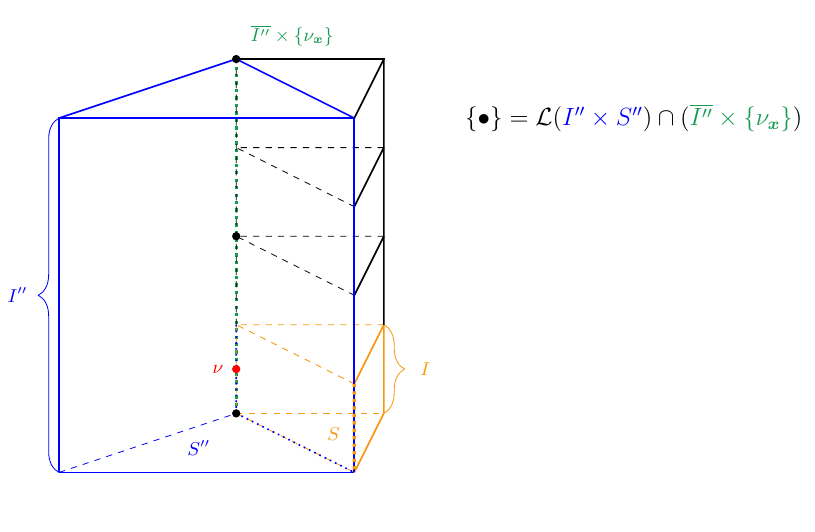}
	\end{center}
    \captionof*{figure}{Illustration in the case of $r_1 = 3$, i.e., elements that are quadratic in time}
\end{figure}

\begin{proof}
    In the case of $d=1$, the temporal and spatial directions are interchangeable, and therefore the second assertion already follows from \cref{lem:Hanging_in_space_part_2}.

    Now let us consider the case of $d\ge 2$. Due to \cref{Cor:Characterization_hanging_nodes}, we know that $\nu_t\in \mathring{I'}\setminus\calL(I')$ and $\nu_{\bm{x}}\in \calL(S')$ for any $I'\times S'\in \calP_\calH(\nu)$. In particular, $\nu_t\in \overline{I}\cap\mathring{I'}$ implies $|I\cap I'|>0$. Therefore, either $I\subset I'$ or $I'\subset I$ due to the construction of $\calP$. If the latter was true, $\nu_t \in \calL(I)\cap \overline{I'}\cap \overline{I}\subset \calL(I')$, which cannot hold, since $I'\times S'\notin \calP_\calH(\nu)$ and $\nu_{\bm{x}}\in \calL(S')$. Thus, $I\subsetneq I'$ and therefore $I\in \textup{BISECT}(I',1)^{\otimes k(I')}$ for some $k(I')\in \N$. This $k(I')$ is bounded by a constant $k_0= k_0( d, s_1, s_2, \kappa_\calP)\in \N$, since $I'\times S'\in \omega_\calP(I\times S)$ and 
    \begin{align*}
        |\ell(I_1\times S_1) - \ell(I_2\times S_2) |\lesssim_{\, d, \kappa_\calP} 1 \quad \text{for any}\quad  I_1\times S_1, I_2\times S_2\in \omega_\calP(I\times S).
    \end{align*}
    Now consider, $I''\times S''\in \calP_\calH(\nu)$ with $\ell(I''\times S'') = \argmin\limits_{I'\times S'\in \calP_\calH(\nu)} \ell(I'\times S')$, then $k(I')\le k(I'')$ and $I'\subset I''$. Now let $\nu':=(\nu'_t, \nu_{\bm{x}})\in (\mathring{I''}\times \{\nu_{\bm{x}}\})\cap \calL(I''\times S'')$ and assume that $\nu'\in \calH(\calP)$. Then there is $I'''\times S'''\in \calP_\calH(\nu')$. First, assume that $\nu'\in \mathring{I'''}\times \partial S'''$ for $I'''\times S'''\in \calP_\calH(\nu')$. Since $\nu_t'\in \mathring{I''}\cap \mathring{I'''}$, $|I''\cap I'''|>0$. As before, we can therefore conclude that $I''\subsetneq I'''$ and $\ell(I''\times S'')>\ell(I'''\times S''')$. But this would imply $I'''\times S'''\in \calP_\calH(\nu)$, in contrast of the minimality condition used in the choice of $I''\times S''$. Second, consider the case $\nu'\in \partial I''' \times S'''$. According to \cref{lem:Hanging_in_space_part_1}, $\nu_t'\in \partial I'' \cap \partial I'''$ has to hold, which is a contradiction to $\nu_t'\in \mathring{I''}$.
\end{proof}

In contrast to the edges on the boundary of the face in \cref{lem:Hanging_in_space_part_2} and \cref{lem:Hanging_in_space_part_3}, we can show that the edges of the line segment in \cref{lem:Hanging_in_time_part_2} are indeed also free.

\begin{lem}\label{lem:Hanging_in_time_part_3}
    Let $\nu \in \calH(\calP)\cap \calL(I\times S)$ and $I'\times S'\in \calP_\calH(\nu)$ be the element from \cref{lem:Hanging_in_time_part_2} that is there called $I''\times S''$. Additionally, let a node $\nu'\in\partial I'\times \{\nu_{\bm{x}}\}$. Then, $\nu'\in \calF(\calP)$.
\end{lem}
\begin{proof}
    According to \cref{lem:Hanging_in_space_part_2}, for $d=1$ there remains nothing to be shown, thus in the following we consider $d\ge 2$. As before, we know that $I'\supsetneq I$, i.e., $\ell(I'\times S')<\ell(I\times S)$, i.e., $\ell(I'\times S')\le\ell(I\times S)-1$. Now assume $\nu'=(\nu_t', \nu_{\bm{x}})\in \calH(\calP)$. As in \cref{lem:Hanging_in_time_part_2} above, one can show that $\nu'$ does not \enquote{hang in time}. Therefore, there is $I''\times S''\in \calP_\calH(\nu')$ such that $\nu'\in\partial I''\times S''$, i.e., $\nu_t'\in \calL(I'')$ but $\nu_{\bm{x}}\notin \calL(S'')$ according to \cref{lem:Hanging_in_space_part_1}. Without loss of generality, we may further assume $S'\in\textup{BISECT}(S'',d)$ due to \cref{lem:Hanging_in_space_part_2}.

    Furthermore, without loss of generality we may assume $t''>t'$ for $(t'', t')\in I''\times I'$. Similarly to the proof of \cref{lem:Hanging_in_space_part_2}, we then define $\calT_{up}:=\calT(v_t', \calP)$ and $\calT_{down}:=\calT\left(v_t'-\frac{\varepsilon_0}{2}, \calP\right)$, where $\varepsilon_0$ is chosen as in the aforementioned proof. Additionally, we choose a $S'''\in \calT_{down}$ with $|S\cap S'''|>0$ and $\nu_{\bm{x}}\in S'''$ as well as a corresponding interval $I'''$ such that $\nu'\in \overline{I'''\times S'''}$ and $I'''\times S'''\in \calP$. It holds that $\ell(I'''\times S''')>\ell(I'\times S')$, since otherwise, if $\ell(I'''\times S''')\le\ell(I'\times S')$, then $I'\subset I'''$ must hold, and therefore $I'''\times S''' = I\times S$. But this would be a contradiction to $I'\subset I''' = I \subsetneq I'$. Thus,
    \begin{align}\label{lem:Hanging_in_time_part_3 - equation}
        \ell(I'''\times S''')\ge \ell(I'\times S') + 1 = \ell(I''\times S'') + 2,
    \end{align}
    where we have applied $S'\in\textup{BISECT}(S'',d)$ in the second step.

    \begin{figure}[H]
	    \begin{center}
		    \includegraphics[height=11cm]{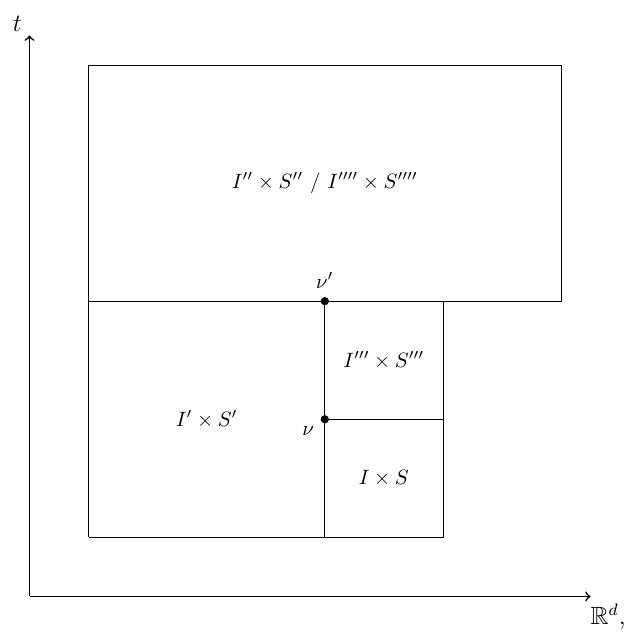}
	    \end{center}
    \captionof*{figure}{Illustration of the prisms in the proof of \cref{lem:Hanging_in_time_part_3}.}
    \end{figure}

Now if $|S'''\cap S''|>0$ was true, this would also yield a contradiction to the above equation for the levels due to the $1$-irregular rule in time. Thus, $\dim(S''\cap S''')\in \{0,\dots, d-1\}$. Now let us consider $S''''\in \calT_{up}$ with $|S'''\cap S''''|>0$ and $\nu_{\bm{x}}\in S''''$ together with an interval $I''''$ such that $I''''\times S''''\in \calP$ and $\nu'\in \overline{I''''\times S''''}$. Then, since $\nu_{\bm{x}}\in S''\cap S''''$ and $S'', S''''\in\calT_{up}$, $S''\cap S''''$ is a face of $S''$, due to the conformity of $\calT_{up}$, which does contain $\nu_{\bm{x}}$. The latter is a Lagrange node of the \enquote{child} $S'$ of $S''$, but not of $S''$, i.e., $\nu_{\bm{x}}\in \calL(S')\setminus \calL(S'')$. This can only hold if the refinement edge $\textup{RE}(S'')$ is contained in $S''''$, which implies $\ell(I''''\times S'''')\le \ell(I''\times S'') $ according to \mbox{\cref{Lemma_Properties_of_Neighboring_Space_Time_Elements}}. But again, together with \eqref{lem:Hanging_in_time_part_3 - equation} this yields another contradiction: Since $I'''\times S'''$ and $I''''\times S''''$ touch in time and have a level difference of at least $2$. This cannot hold due to the 1-irregularity rule in time. Therefore, this finally yields the assertion.
\end{proof}

We are now in a position to prove \cref{thm:support_in_power_omega_domain}.

\begin{proof}[Proof of \cref{thm:support_in_power_omega_domain}]
    We first consider the case $d=1$. Let $\nu'\in \calL(I'\times S')$ for $I'\times S'\in \calP\setminus\omega^{2}_\calP(I\times S)$. If $\nu'\in \calF(\calP)$, then $\nu\neq \nu'$, otherwise $I'\times S'\in \omega_\calP(I\times S)\subset \omega^{2}_\calP(I\times S)$ due to $\nu'\in (\overline{I\times S})\cap (\overline{I'\times S'})$, which would be a contradiction. Then $\phi_\nu(\nu') = 0$ according to \cref{thm:basis free nodes}. Now consider the case of $\nu'\in \calH(\calP)$. According to \cref{lem:Hanging_in_space_part_2} and \cref{lem:Hanging_in_time_part_2}, respectively, there is a $1$-dimensional face $F$ of some $I''\times S''\in \calP_\calH(\nu')$ with $\nu'\in F$ and $F\cap \calL(I''\times S'')\subset \calF(\calP)$. Let $\nu''\in F\cap \calL(I''\times S'')$. Then, $\nu''\neq \nu$, since the opposite would imply $I''\times S''\in\omega_\calP(I\times S)$ as before, and therefore $I'\times S'\in \omega_\calP^{2}$, which is a contradiction. Thus, $\phi_\nu(\nu'')=0$. Since $\phi_\nu\phantom{}_{|F}\in \Pi^{r_1}(I')$ with $\#(F\cap \calL(I''\times S''))=r_1 = \dim \Pi^{r_1}(I')$ or $\phi_\nu\phantom{}_{|F}\in \Pi^{r_2}(S')$ with $\#(F\cap \calL(I''\times S''))=r_2= \dim \Pi^{r_2}(S')$, correspondingly, this yields $\phi_\nu\phantom{}_{|F}= 0$, in particular, $\phi_\nu(\nu')=0$. Thus, $\phi_\nu(\nu')=0$ for all $\nu'\in \calL(I'\times S')$ and therefore $\phi_\nu\phantom{}_{|\overline{I'\times S'}}=0$, because $\#\calL(I'\times S')= \dim \Pi^{r_1,r_2}_{t, \bm{x}}(I'\times S') $. Since $I'\times S' \in \calP\setminus \omega^{2}_\calP(I\times S)$ has been arbitrarily chosen, this yields the assertion.

    Now we turn to the situation of $d\in \N_{\ge 2}$. The case of $\nu'\in \calF(\calP)$ is identical to the case $d=1$. So assume $\nu'\in \calH(\calP)$. Similarly, if $\nu'$ \enquote{hangs in time}, the proof works identically due to \cref{lem:Hanging_in_time_part_2} and \cref{lem:Hanging_in_time_part_3}. If $\nu'=(\nu_t', \nu_{\bm{x}}')$ \enquote{hangs in space}, it is a little more complicated. Then, there is $I''\times S''\in \calP_\calH(\nu')$ as in \cref{lem:Hanging_in_space_part_2} with $\{\nu_t'\}\times \mathring{S''}\subset \calF(\calP)$. Therefore, $\phi_\nu(\nu'')= 0$ for $\nu'' \in \{\nu_t'\}\times \mathring{S''}$ can be shown as above. Now consider $\nu''\in \{\nu_t'\}\times \partial S''$, i.e., $\nu''=(\nu_t', \nu_{\bm{x}}'')$. If $\nu''\in \calF(\calP)$ nothing changes, thus assume the case $\nu''\in \calH(\calP)$. Then, $\nu''$ \enquote{hangs in time} according to \cref{lem:Hanging_in_space_part_3}, i.e., there is $I'''\times S'''\in \calP_\calH(\nu'')$ such that $L(\nu''):=\overline{I'''}\times \{\nu_{\bm{x}}''\}\subset \calF(\calP)$ due to \cref{lem:Hanging_in_time_part_2} and \cref{lem:Hanging_in_time_part_3}. We now know that $I'''\times S'''\in \omega^{2}_\calP(I'\times S')$. Thus, if $\nu'''\in L(\nu'')\cap \calL(I'''\times S''')$ fulfilled $\nu''' = \nu\in \overline{I\times S}$, then $I'\times S' \in \omega^{3}_\calP(I\times S)$ would hold in contradiction to the assumption. Thus, $\phi_\nu(\nu''')= 0$, again due to \cref{thm:basis free nodes}. Since $\phi_\nu\phantom{}_{|L(\nu''')}\in \Pi^{r_1}(I''')$ and $ \calL(I'''\times S''') \cap L(\nu'')= \dim \Pi^{r_1}(I''')$, this implies $\phi_\nu\phantom{}_{|L(\nu'')}= 0$ and therefore $\phi_\nu(\nu'') = 0$. Since we have therefore shown that $\phi_\nu(\nu'') = 0$ for all $\nu''\in \calL(I''\times S'')\cap (\{\nu_t\}\times S'')$, and $\#(\calL(I''\times S'')\cap (\{\nu_t\}\times S''))=\dim \Pi^{r_2}(S'')$, this yields $\phi_\nu\phantom{}_{|\{\nu_t\}\times S''}=0$ and thus $\phi_\nu(\nu')=0$. Thus, the first assertion follows as in the case $d=1$.
    The second assertion is now an immediate consequence of \cref{Enumeration further properties of the mesh}\ref{Enumeration further properties of the mesh - enum 2}.
\end{proof}

\subsection{The (quasi-)interpolation operator}

The goal of this section is to construct an appropriate (quasi-)interpolation operator from $L_p(\Omega_T)$, $p\in (0,\infty]$, to our anisotropic finite element space $\mathbb{V}^{r_1, r_2}_\calP$. Quasi-interpolation operators of Scott--Zhang type for non-smooth functions on anisotropic simplicial meshes, with anisotropic local stability and error estimates in Sobolev norms, have been studied in \cite{Ape99}; here, the prismatic structure of the elements allows us to treat the temporal and spatial directions separately, and we need stability in $L_p$ for $p\in(0,\infty]$. We will proceed similar to \cite[Sect.~3]{GM14}. Also, in the following, we will use $j(d)$ as in \cref{thm:support_in_power_omega_domain}.

First, we will consider the Lagrange basis $(b^{I\times S}_{\nu})_{\nu\in \calL(I\times S)}$ of $\Pi^{r_1, r_2}_{t, \bm{x}}(I\times S)$ with respect to the nodes in $\calL(I\times S)$ for every $I\times S\in \calP$. Further, we will use the biorthogonal dual functions $(\zeta^{I\times S}_{\nu})_{\nu\in \calL(I\times S)}\subset \Pi^{r_1, r_2}_{t, \bm{x}}(I\times S)$ with respect to the $L_2(I\times S)$-inner product, i.e., 
\begin{align*}
    \int\limits_{I\times S}b^{I\times S}_\nu \zeta^{I\times S}_{\nu'} \,d(t,\bm{x}) = \delta_{\nu, \nu'}\quad \text{for all}\quad \nu, \nu'\in \calL(I\times S).
\end{align*}

\begin{rem}\label{rem:estimate_dual_elements}
    For every $I\times S\in \calP$, $\nu\in \calL(I\times S)$, and $p\in (0,\infty]$, we have $\|\zeta^{I\times S}_{\nu}\|_{L_p(I\times S)}\sim_{\,d,p,r_1, r_2} |I\times S|^{\frac{1}{p} - 1}$.
\end{rem}
\begin{proof}
    Let $\hat{I}=[0,1]$ and $\hat{S}$ be the $d$-dimensional standard simplex. Then, we know that $b^{\hat{I}\times \hat{S}}_{\hat{\nu}}=b^{I\times S}_{\nu} \circ \Phi_{I\times S}$ for $\Phi_{I\times S}(t,\bm{x}):=(\Phi_I(t), \Phi_{S}(\bm{x})) $ and $\hat{\nu}:=\Phi_{I\times S}^{-1}(\nu)$. Further, 
    \begin{align*}
        \int\limits_{I\times S} b^{I\times S}_\nu \left(\frac{|\hat{I}\times \hat{S}|}{|I\times S|}\cdot \zeta^{\hat{I}\times \hat{S}}_{\hat{\nu}'}\circ \Phi^{-1}_{I\times S}\right)\,d(t,\bm{x})= \int\limits_{\hat{I}\times \hat{S}} b^{\hat{I}\times \hat{S}}_{\hat{\nu}} \zeta^{\hat{I}\times \hat{S}}_{\hat{\nu}'} \,d(\hat{t}, \hat{\bm{x}})= \delta_{\hat{\nu}, \hat{\nu}'} = \delta_{\nu, \nu'}
    \end{align*}
    holds for any $\nu, \nu'\in \calL(I\times S)$ due to $|\det D\Phi_{I\times S}|=\frac{|I\times S|}{|\hat{I}\times \hat{S}|}$. This yields $\zeta^{I\times S}_\nu:= \frac{|\hat{I}\times \hat{S}|}{|I\times S|}\cdot\zeta^{\hat{I}\times \hat{S}}_{\hat{\nu}}\circ \Phi^{-1}_{I\times S}$. Additionally, $\|\zeta^{\hat{I}\times \hat{S}}_{\hat{\nu}}\|_{L_p(\hat{I}\times \hat{S})}\sim_{\, d, p, r_1, r_2}1$ for any such $\hat{\nu}\in \calL(\hat{I}\times \hat{S})$. Together, we obtain
    \begin{align*}
        \|\zeta^{I\times S}_\nu\|_{L_p(I\times S)}^p \sim_{\,d,p}\int\limits_{\hat{I}\times \hat{S}} \left(\frac{1}{|I\times S|^p} \left|\zeta^{\hat{I}\times \hat{S}}_{\hat{\nu}} \right|^p\right) |I\times S| \, d(\hat{t}, \hat{\bm{x}}) = |I\times S|^{1-p}\left\|\zeta^{\hat{I}\times \hat{S}}_{\hat{\nu}} \right\|^p_{L_p(\hat{I}\times \hat{S})}\sim_{\,d,p,r_1, r_2} |I\times S|^{1-p}.
    \end{align*}
    for $p<\infty$, due to $|\hat{I}\times \hat{S}|=|\hat{S}|\sim_{\,d}1$. The case $p=\infty$ works analogously with the usual modifications.
\end{proof}

\begin{defi}
    For given $\nu\in \calF(\calP)$, we choose an arbitrary, but fixed $J\times R:=(J\times R)(\nu)\in \calP$ with $\nu\in \overline{J\times R}$. Next, we set $\zeta_\nu:=\zeta^{(J\times R)(\nu)}_\nu$ in order to define $\mathcal{Q}_\calP:L_1(\Omega_T)\rightarrow \mathbb{V}^{r_1, r_2}_{\calP}$  via
    \begin{align*}
        \mathcal{Q}_\calP(f):=\sum\limits_{\nu\in \calF(\calP)}\left(\int\limits_{(J\times R)(\nu)}f\,\zeta_\nu\,d(t,\bm{x})\right)\phi_\nu = \sum\limits_{\nu\in \calF(\calP)}(f,\zeta_\nu)_{L_2((J\times R)(\nu))}\phi_\nu
    \end{align*}
\end{defi}

\begin{rem}\label{rem:linearity_and_projection_property_of_Q}
    It is easy to see that $\mathcal{Q}_\calP$ is linear and $\mathcal{Q}_\calP(f) = f$ if and only if $f\in \mathbb{V}^{r_1, r_2}_\calP$.
\end{rem}

    \begin{proof}
        The linearity is clear. Let $f\in \V^{r_1, r_2}_\calP$. Then, \cref{thm:basis free nodes} yields the expansion $f=\sum\limits_{\nu\in \calF(\calP)} f_\nu \phi_\nu$, $f_\nu \in \R$, with respect to the basis $(\phi_\nu)_{\nu\in \calF(\calP)} $ of $\V^{r_1, r_2}_\calP$. Due to the basis property, it is sufficient to show $(f,\zeta_\nu)_{L_2((J\times R)(\nu))}=f_\nu$ for every $\nu\in \calF(\calP)$ in order to obtain $\mathcal{Q}_\calP(f)=f$. Therefore, let $\nu\in \calF(\calP)$ and observe that $f_{|J\times R}$ is an anisotropic polynomial with a representation with respect to $(b^{J\times R}_{\nu'})_{\nu'\in \calL(J\times R)}$, which forms a basis of $\Pi^{r_1, r_2}_\calP(J\times R)$, i.e., $f_{|J\times R}=\sum\limits_{\nu'\in \calL(J\times R)} c_{\nu'}b_{\nu'}^{J\times R}$. First, we now obtain
        \begin{align*}
            f_\nu =\sum\limits_{\nu'\in \calF(\calP)}f_\nu \delta_{\nu,\nu'}(\nu)=\sum\limits_{\nu'\in \calF(\calP)} f_{\nu'}\phi_{\nu'}(\nu)= f(\nu) = \sum\limits_{\nu'\in \calL(J\times R)} c_{\nu'}b_{\nu'}^{J\times R}(\nu)= \sum\limits_{\nu'\in \calL(J\times R)} c_{\nu'}\delta_{\nu, \nu'}= c_\nu,
        \end{align*} 
        due to $\nu\in \overline{J\times R}$, which in turn allows us to derive
        \begin{align*}
            (f,\zeta_\nu)_{L_2(J\times R)} = \int\limits_{J\times R} f\,\zeta_\nu \,d(t,\bm{x})= \sum\limits_{\nu'\in \calL(J\times R)} c_{\nu'}\int\limits_{J\times R}b_{\nu'}^{J\times R}\zeta_\nu \, d(t,\bm{x})=\sum\limits_{\nu'\in \calL(J\times R)} c_{\nu'}\delta_{\nu, \nu'}=c_\nu = f_\nu.
        \end{align*}
        This concludes the proof.
    \end{proof}

\begin{rem}\label{rem:estimate_primal_elements}
For given $\nu \in \calF(\calP)\cap \calL(I\times S)$ and $p\in (0,\infty]$, we can estimate $\|\phi_\nu\|_{L_p(\Omega_T)}\sim_{\,d,p,r_1, r_2} |\supp \phi_\nu|^\frac{1}{p}$. 
\end{rem}
\begin{proof}

Let $M_\nu$ be the set of all elements of $\calP$ in the support of $\phi_\nu$, i.e., $M_\nu := \{I'\times S'\in \calP\mid I'\times S'\subset \supp \phi_\nu \}$. As above, we now use the affine mapping $\Phi_{I'\times S'}$ that transforms the reference prism $\hat{I}\times \hat{S}$ to $I'\times S'$ as well as $|\hat{I}\times \hat{S}|\sim_{\,d}1$ in order to obtain
\begin{align*}
    \|\phi_\nu\|_{L_p(\Omega_T)}^p &= \sum\limits_{I'\times S'\in M_\nu} \int\limits_{I'\times S'}|\phi_\nu|^p\, d(t, \bm{x}) = \sum\limits_{I'\times S'\in M_\nu} \int\limits_{\hat{I}\times \hat{S}} |\phi_\nu \circ \Phi_{I'\times S'}|^p\,\frac{|I'\times S'|}{|\hat{I}\times \hat{S}|}\, d(\hat{t}, \hat{\bm{x}}) 
    \\&\sim_{\,d} \sum\limits_{I'\times S'\in M_\nu} |I'\times S'| \int\limits_{\hat{I}\times \hat{S}} |\phi_\nu \circ \Phi_{I'\times S'}|^p\, d(\hat{t}, \hat{\bm{x}}) \sim_{\,d,p,r_1, r_2} \sum\limits_{I'\times S'\in M_\nu} |I'\times S'| = |\supp \phi_\nu|,
\end{align*}
if $p<\infty$. The case $p=\infty$ can be treated similarly.
\end{proof}
Together, this allows us to show the following result concerning the local boundedness of $\mathcal{Q}(f)$.

\begin{lem}\label{lem:local_boundedness_of_Q_projection}
    Let $p\in (0,\infty]$ and $I\times S\in \calP$. Then
    \begin{enumerate}[label=(\roman*)]
        \item $\|\mathcal{Q}_\calP(f)\|_{L_p(I\times S)} \lesssim_{\, d,p,s_1, s_2,r_1, r_2, \kappa_{\calP_0}, \mu(\calP_0)}\|f\|_{L_p(\omega^{j(d)}_\calP(I\times S))}\footnotemark \quad \text{for all}\quad f\in L_p(\Omega_T),\quad \text{if}\quad p\ge 1$. \label{lem:local_boundedness_of_Q_projection - item 1}
        \item $\|\mathcal{Q}_\calP(f)\|_{L_p(I\times S)} \lesssim_{\, d,p,s_1, s_2,r_1, r_2, \kappa_{\calP_0}, \mu(\calP_0)}\|f\|_{L_p(\omega^{j(d)}_\calP(I\times S))} \quad \text{for all}\quad f\in \V^{r_1, r_2}_{\calP,\textup{DC}}.$ \label{lem:local_boundedness_of_Q_projection - item 2}
    \end{enumerate}
    \footnotetext{With a little abuse of notation, we denote \\
    $\|f\|_{L_p(\omega^{j(d)}_\calP(I\times S))}:=\|f\|_{L_p\big( \bigcup\limits_{I'\times S'\in \omega^{j(d)}_\calP(I\times S)} I'\times S'\big)}$
    \ and \  $E\big(f, \cdot, \omega^{j(d)}_\calP(I\times S)\big)_p=E\bigg(f, \cdot, \bigcup\limits_{I'\times S'\in \omega^{j(d)}_\calP(I\times S)} I'\times S'\bigg)_p$.}
\end{lem}
\begin{proof}
    We will only provide the details of the proof for the case $p<\infty$, since $p=\infty$ is very similar. First, we assume $f\in L_p(\Omega_T)$ and notice that
    \begin{align}\label{lem:local_boundedness_of_Q_projection - proof - equation 0}
        \mathcal{Q}_\calP(f)_{|I\times S}= \sum\limits_{\nu\in \calF(\calP)} (f, \zeta_\nu)_{L_2((J\times R)(\nu))}\phi_\nu\phantom{}_{|I\times S} = \sum\limits_{\substack{\nu\in \calF(\calP),\\ I\times S\subset \supp \phi_\nu}} (f, \zeta_\nu)_{L_2((J\times R)(\nu))}\phi_\nu.
    \end{align}
    By \cref{thm:support_in_power_omega_domain} as well as the Minkowski and Hölder inequalities, this allows us to estimate
    \begin{align}\label{lem:local_boundedness_of_Q_projection - proof - equation 1}
        \|\mathcal{Q}_\calP(f)\|_{L_p(I\times S)} \le \sum\limits_{
        \substack{\nu\in \calF(\calP)\cap \calL(I'\times S'),\\ I'\times S'\in \omega^{j(d)-1}_\calP(I\times S)}
        } \|f\|_{L_{p}((J\times R)(\nu))} \|\zeta_\nu\|_{L_{p'}((J\times R)(\nu))}\|\phi_\nu\|_{L_p(I\times S)},
    \end{align}
    where $p'\in [1,\infty]$ is chosen such that $\frac{1}{p}+\frac{1}{p'}=1$. Additionally, we have 
    \begin{align}\label{lem:local_boundedness_of_Q_projection - proof - equation 2}
        \|\zeta_\nu\|_{L_{p'}((J\times R)(\nu))}\sim_{\,d,p,r_1, r_2} |(J\times R)(\nu)|^{\frac{1}{p'}-1}\sim_{\,d,s_1,s_2, \kappa_{\calP_0}, \mu(\calP_0)} |\omega^{j(d)}_\calP(I\times S)|^{\frac{1}{p'}-1}
    \end{align}
    due to \cref{rem:estimate_dual_elements} in the first step, and \cref{Enumeration further properties of the mesh}\ref{Enumeration further properties of the mesh - enum 2} in the second one, since $(J\times R)(\nu)\in \omega^{j(d)}_\calP(I\times S)$ for such $\nu$. Furthermore, 
    \begin{align}\label{lem:local_boundedness_of_Q_projection - proof - equation 3}
        \|\phi_\nu\|_{L_p(I\times S)}\le \|\phi_\nu\|_{L_p(\Omega_T)} \lesssim_{\,d,p, r_1, r_2} |\supp \phi_\nu|^\frac{1}{p} \lesssim_{\,d,p,s_1,s_2, \kappa_{\calP_0},\mu(\calP_0)} |\omega^{j(d)}_\calP(I\times S)|^{\frac{1}{p}}
    \end{align}
    holds true because of \cref{rem:estimate_primal_elements} and \cref{thm:support_in_power_omega_domain}. Inserting \eqref{lem:local_boundedness_of_Q_projection - proof - equation 2} and \eqref{lem:local_boundedness_of_Q_projection - proof - equation 3} into \eqref{lem:local_boundedness_of_Q_projection - proof - equation 1} now yields
    \begin{align*}
        \|\mathcal{Q}_\calP(f)\|_{L_p(I\times S)}&\lesssim_{\,d,p,s_1,s_2,r_1, r_2, \kappa_{\calP_0}, \mu(\calP_0)} \sum\limits_{\substack{\nu\in \calF(\calP)\cap \calL(I'\times S'),\\ I'\times S'\in \omega^{j(d)-1}_\calP(I\times S)}} \|f\|_{L_{p}((J\times R)(\nu))} 
        \\&\le \sum\limits_{\substack{\nu\in \calF(\calP)\cap \calL(I'\times S'),\\ I'\times S'\in \omega^{j(d)-1}_\calP(I\times S)}} \|f\|_{L_p(\omega^{j(d)}_\calP(I\times S))} \lesssim_{\,d,s_1,s_2,r_1,r_2 \kappa_{\calP_0}, \mu(\calP_0) } \|f\|_{L_p(\omega^{j(d)}_\calP(I\times S))},
    \end{align*}
    where we have used that $(J\times R)(\nu)\in \omega^{j(d)}_\calP(I\times S)$ for such $\nu$ in the penultimate step and the boundedness of $\omega^{j(d)-1}_\calP(I\times S)$ due to \cref{Enumeration further properties of the mesh}\ref{Enumeration further properties of the mesh - enum 2} in the last one. This concludes the proof for the case \ref{lem:local_boundedness_of_Q_projection - item 1}. Now, let us consider $f\in \V^{r_1, r_2}_{\calP, \textup{DC}}$, i.e., $f=\sum\limits_{I'\times S'\in \calP}\mathds{1}_{I'\times S'}f_{I'\times S'}$ with $f_{I'\times S'}\in \Pi^{r_1, r_2}_{t,\bm{x}}(I'\times S')$ for $I'\times S'\in \calP$. Since $\Omega_T$ is bounded, \ref{lem:local_boundedness_of_Q_projection - item 2} is a particular case of \ref{lem:local_boundedness_of_Q_projection - item 1}. So it is sufficient to show the assertion for $p\in (0,1)$ only. Similar to before, using the subadditivity of $\|\cdot\|_{L_p(\Omega_T)}^{p}$, we obtain
    \begin{align*}
        \|\mathcal{Q}_\calP(f)\|^p_{L_p(I\times S)}&\le \sum\limits_{
        \substack{\nu\in \calF(\calP)\cap \calL(I'\times S'),\\ I'\times S'\in \omega^{j(d)-1}_\calP(I\times S)}}|(f,\zeta_\nu)_{L_2((J\times R)(\nu))}|^p\, \|\phi_\nu\|_{L_p(I\times S)}^p 
        \\ &\lesssim_{\,d,p,r_1,r_2} \sum\limits_{\substack{\nu\in \calF(\calP)\cap \calL(I'\times S'),\\ I'\times S'\in \omega^{j(d)-1}_\calP(I\times S)}}\|f\|_{L_\infty((J\times R)(\nu))}^p \|\phi_\nu\|_{L_p(I\times S)}^p
        \\&\lesssim_{\, d,p, s_1, s_2, r_1, r_2,\kappa_{\calP_0}, \mu(\calP_0)} \sum\limits_{\substack{\nu\in \calF(\calP)\cap \calL(I'\times S'),\\ I'\times S'\in \omega^{j(d)-1}_\calP(I\times S)}} |\omega^{j(d)}_\calP(I\times S)|^{-1} \|f\|_{L_p((J\times R)(\nu))}^p |\omega^{j(d)}_\calP(I\times S)| 
        \\&\lesssim_{\,d,s_1,s_2, \kappa_{\calP_0}}  \|f\|_{L_p(\omega_\calP^{j(d)}(I\times S))}^p,
    \end{align*}
    where we have additionally used the (quasi-)norm equivalence from \cite[Lem.~3.5]{MSS26} together with \cref{Enumeration further properties of the mesh}\ref{Enumeration further properties of the mesh - enum 2} and \cref{rem:estimate_primal_elements} in the penultimate step. This shows \ref{lem:local_boundedness_of_Q_projection - item 2} and concludes the proof.
\end{proof}

Since $\mathcal{Q}_\calP$ cannot be applied to arbitrary functions in $L_p(\Omega_T)$, if $p\in (0,1)$, we will have to combine the projection $\mathcal{Q}_\calP$ appropriately with another operator. Therefore, for $D\subset \Omega_T$, consider the operator $\calB_{p, D}:L_p(D)\rightarrow \Pi^{r_1, r_2}_{t,\bm{x}}(D)$ that maps any function to one of its \textbf{best approximations}, i.e.,
\begin{align*}
    \|f-\calB_{p,D}(f)\|_{L_p(D)} = E\left(f, \Pi^{r_1, r_2}_{t,\bm{x}}(D), D\right)_p=:E\left(f, \Pi^{r_1, r_2}_{t,\bm{x}}, D\right)_p .
\end{align*}

\begin{rem}\label{rem:best_approximation_operator}
    For any $p\in (0,\infty]$, the space $L_p(\Omega_T)$ is a (quasi-)Banach space. Therefore, a best approximation of any $f\in L_p(\Omega_T)$ in the finite dimensional space $\Pi^{r_1, r_2}_{t,\bm{x}}$ exists, i.e., $\calB_{p,D}$ is well-defined. In the Banach space case, i.e., $p\in [1,\infty]$, this is a direct consequence of \cite[Thm.~1.1 of Ch.~3.1]{DL93}. For $p\in (0,1)$, this follows by minor modifications of the proof of the aforementioned theorem, adjusting it to the case of quasi-Banach spaces. 
    
    In particular, for $p\in (0,1]\cup \{\infty\}$, $\calB_{p,D}$ might not be continuous, since there can be multiple best approximations, because $\|\cdot\|_{L_p(D)}$ is not (strictly) convex.
\end{rem}

We will need the following property that tells us that best-approximations with respect to some small parameter are also quasi-best approximations for the larger parameters. The corresponding result for the stationary case can be found in \cite[Lem.~3.5]{GM14}.

\begin{lem}\label{lem:smaller_parameter_bestapprox_larger_parameter_quasi-bestapprox}
    Let $p\in (0,\infty]$, $\rho\in (0,p]$, and $f\in L_\rho(\Omega_T)$. Further, assume $D=I\times S$ or $D=\omega^{j(d)}_\calP(I\times S)$ for some $I\times S\in \calP$. Then,
    \begin{align*}
        \|f-\calB_{\rho,D}(f)\|_{L_p(D)}\lesssim E\left(f, \Pi^{r_1, r_2}_{t,\bm{x}}, D\right)_p,
    \end{align*}
    where the constant only depends on $d, \rho, p, r_1, r_2$ in the first case and on $d, \rho, p, s_1, s_2, r_1, r_2,\kappa_{\calP_0}$, and $\mu(\calP_0)$ in the latter.
\end{lem}
\begin{proof}
    This can be shown analogously to the corresponding proof in \cite[Lem.~3.2]{DP88}. Following the lines of the proof, one can see that the involved constant depends on $\rho$, $p$, and the constant $C$ corresponding to the (quasi-)norm equivalence estimate 
    \begin{align*}
        \|Q\|_{L_p(D)}\le C\, |D|^{\frac{1}{p}-\frac{1}{\rho}} \|Q\|_{L_\rho(D)}\quad \text{for all}\quad Q\in \V^{r_1, r_2}_\calP.
    \end{align*}
    The latter has to hold on any bounded $D$, since $\V^{r_1, r_2}_\calP$ is finite dimensional and thus all (quasi-)norms must be equivalent. Note that the estimate is invariant under affine-linear transformations. Therefore, for $D=I\times S$, the assertion follows directly by transforming $I\times S$ to the standard prism $\hat{I}\times \hat{S}$, where the estimate has to hold with a constant depending only on $d$, $r_1$, and $r_2$. Now, consider the case of $D=\omega^{j(d)}_\calP(I\times S)$ and $p<\infty$. Then,
    \begin{align*}
        \|Q\|^p_{L_p(\omega^{j(d)}_\calP(I\times S))} &= \sum\limits_{I'\times S'\in \omega^{j(d)}_\calP(I\times S)} \|Q\|^p_{L_p(I'\times S')} \lesssim_{\,d, \rho, p, r_1, r_2} \sum\limits_{I'\times S'\in \omega^{j(d)}_\calP(I\times S)} |I'\times S'|^{1-\frac{p}{\rho}} \|Q\|_{L_\rho(I'\times S')}^p
        \\&\lesssim_{\,d,\rho,p,s_1,s_2,\kappa_{\calP_0}, \mu(\calP_0)} |\omega^{j(d)}_\calP(I\times S)|^{1-\frac{p}{\rho}} \sum\limits_{I'\times S'\in \omega^{j(d)}_\calP(I\times S)} \|Q\|_{L_\rho(I'\times S')}^p,
        \\&\le |\omega^{j(d)}_\calP(I\times S)|^{1-\frac{p}{\rho}} \left(\sum\limits_{I'\times S'\in \omega^{j(d)}_\calP(I\times S)} \|Q\|_{L_\rho(I'\times S')}^\rho\right)^\frac{p}{\rho} = |\omega^{j(d)}_\calP(I\times S)|^{1-\frac{p}{\rho}} \|Q\|^p_{L_\rho(\omega^{j(d)}_\calP(I\times S))},
    \end{align*}
    where we have used the result for $D=I'\times S'$ in the second step, \cref{Enumeration further properties of the mesh}\ref{Enumeration further properties of the mesh - enum 2} in the third, and the embedding for the sequence spaces $\ell^\rho(\N_0)\rightarrow\ell^p(\N_0)$ in the penultimate one. This shows the assertion, since the case of $p=\infty$ works similar with the usual modifications.
\end{proof}

Consequently, for any $p\in (0,\infty]$, we can define $\calB_{p,\calP}:=\sum\limits_{I\times S\in \calP}\mathds{1}_{I\times S}\calB_{p,I\times S}\in \V^{r_1, r_2}_{\calP, \textup{DC}}$ and show a corresponding localized best approximation estimate.
\begin{lem}\label{lem:composition_of_Bestapproximations_operator}
    Let $p\in (0,\infty]$, $\rho\in (0,p]$, $f\in L_p(\Omega_T)$, and $I\times S\in \calP$. Then
    \begin{align*}
        \|f- \calB_{\rho, \calP}(f)\|_{L_p(\omega^{j(d)}_\calP(I\times S))} \lesssim_{\,d, \rho, p, r_1, r_2} E\left(f, \Pi^{r_1, r_2}_{t,\bm{x}}, \omega^{j(d)}_\calP(I\times S)\right)_p
    \end{align*}
\end{lem}
\begin{proof}
    For $p<\infty$, by using the best-approximation property of $\calB_{p,\cdot}$, we obtain
    \begin{align*}
        \|f- \calB_{\rho, \calP}(f)\|_{L_p(\omega^{j(d)}_\calP(I\times S))}^p &= \sum\limits_{I'\times S'\in \omega^{j(d)}_\calP(I\times S)} \|f-\calB_{\rho, I'\times S'}(f)\|_{L_p(I'\times S')}^p 
        \\&\lesssim_{\,d, \rho, p, r_1, r_2} \sum\limits_{I'\times S'\in \omega^{j(d)}_\calP(I\times S)} \|f-\calB_{p, \omega^{j(d)}_\calP(I\times S)}(f)\|_{L_p(I'\times S')}^p 
        \\&= \|f-\calB_{{p}, \omega^{j(d)}_\calP(I\times S)}(f)\|_{L_p(\omega_\calP^{j(d)}(I\times S))}^p = E\left(f, \Pi^{r_1, r_2}_{t,\bm{x}}, \omega^{j(d)}_\calP(I\times S)\right)^p_p,
    \end{align*}
    by using the result of \cref{lem:smaller_parameter_bestapprox_larger_parameter_quasi-bestapprox} in the second step, where $\calB_{\rho, \omega^{j(d)}_\calP(I\times S)}:=\calB_{\rho, \big(\bigcup\limits_{I'\times S'\in \omega^{j(d)}_\calP(I\times S)}I'\times S'\big)}$. The case $p=\infty$ works similarly, replacing the sums by suprema.
\end{proof}

This leads to the following definition of the \textbf{quasi-interpolation operator} on $\V^{r_1, r_2}_\calP$.
\begin{defi}\label{defi:quasi-interpolation operator}
    For $p\in (0,\infty]$, we define $\pi_{p,\calP}:L_p(\Omega_T)\rightarrow \V^{r_1, r_2}_\calP$ via $\pi_{p,\calP}:=\mathcal{Q}_\calP\circ \calB_{p,\calP}$.
\end{defi}

For this operator, we can show that $\pi_{p,\calP}$ also satisfies a local boundedness estimate, analogous to the stationary estimate in \cite[Lem.~3.13]{GM14}, which will be a key ingredient in the proof of the direct estimates.

\begin{thm}\label{thm:local_boundedness_of_pi_projection}
    Let $p\in (0,\infty]$, $\rho\in (0,p]$, $f\in L_\rho(\Omega_T)$, and $I\times S\in \calP$. Then
    \begin{align*}
        \|f-\pi_{\rho,\calP}(f)\|_{L_p(I\times S)}\lesssim_{\,d,\rho,p,s_1, s_2, r_1, r_2,\kappa_{\calP_0}, \mu(\calP_0)} E\left(f, \Pi^{r_1, r_2}_{t,\bm{x}},\omega^{j(d)}_\calP(I\times S)\right)_p.
    \end{align*}
\end{thm}
\begin{proof}
    Again, we will only present a detailed proof for the case $p\in (0,\infty)$, the case $p=\infty$ is similar. Before we start, let us point out that $\mathcal{Q}_\calP$ is formally defined for functions on $\Omega_T$, whereas $\calB_{\rho, \omega^{j(d)}_\calP(I\times S)}(f)$ is only defined on $\omega^{j(d)}_\calP(I\times S)$. However, $\mathcal{Q}_\calP$ is local in the following sense: by \eqref{lem:local_boundedness_of_Q_projection - proof - equation 0}, the restriction $\mathcal{Q}_\calP(g)_{|I\times S}$ only depends on the values of $g$ on the elements $(J\times R)(\nu)$  for $\nu$ such that $I\times S\subset\supp\phi_\nu$, all of which belong to $\omega^{j(d)}_\calP(I\times S)$ by \cref{thm:support_in_power_omega_domain}. Hence, in the second line below, $\mathcal{Q}_\calP$ may be applied to any extension of $\calB_{\rho, \omega^{j(d)}_\calP(I\times S)}(f)$ to $\Omega_T$ (e.g., by zero), as the result on $I\times S$ does not depend on the chosen extension. With this understanding, we estimate
    \begingroup
    \allowdisplaybreaks
    \begin{align*}
        \|f-&\pi_{\rho,\calP}(f)\|_{L_p(I\times S)}
        \\&\lesssim_{\, p} \|f-\calB_{\rho, \omega^{j(d)}_\calP(I\times S)}(f)\|_{L_p(I\times S)} + \|\calB_{\rho, \omega^{j(d)}_\calP(I\times S)}(f)-\pi_{\rho,\calP}(f)\|_{L_p(I\times S)}
        \\&= \|f-\calB_{\rho, \omega^{j(d)}_\calP(I\times S)}(f)\|_{L_p(I\times S)} + \|\mathcal{Q}_\calP(\calB_{\rho, \omega^{j(d)}_\calP(I\times S)}(f)-\calB_{\rho,\calP}(f))\|_{L_p(I\times S)}
        \\&\lesssim_{\, d,p,s_1, s_2,r_1, r_2, \kappa_{\calP_0}, \mu(\calP_0)} \|f-\calB_{\rho, \omega^{j(d)}_\calP(I\times S)}(f)\|_{L_p(I\times S)} + \|\calB_{\rho, \omega^{j(d)}_\calP(I\times S)}(f)-\calB_{\rho,\calP}(f)\|_{L_p(\omega^{j(d)}_\calP(I\times S))}
        \\&\lesssim_{\,p} \|f-\calB_{\rho, \omega^{j(d)}_\calP(I\times S)}(f)\|_{L_p(\omega^{j(d)}_\calP(I\times S))} + \|f-\calB_{\rho,\calP}(f)\|_{L_p(\omega^{j(d)}_\calP(I\times S))} 
        \\&\lesssim_{\, d, \rho, p, s_1, s_2,r_1, r_2, \kappa_{\calP_0}, \mu(\calP_0)} E\left(f,\Pi^{r_1, r_2}_{t,\bm{x}},  \omega^{j(d)}_\calP(I\times S)\right)_p,
    \end{align*}
    \endgroup
    where we have applied the result of \cref{rem:linearity_and_projection_property_of_Q} in the second step, \cref{lem:local_boundedness_of_Q_projection}\ref{lem:local_boundedness_of_Q_projection - item 2} in the third, and \cref{lem:smaller_parameter_bestapprox_larger_parameter_quasi-bestapprox} as well as \cref{lem:composition_of_Bestapproximations_operator} in the last one. 
\end{proof}

Also, in the next theorem, we show that the operator $\pi_{\rho, \calP}(f)$ is a \textbf{global quasi-best approximation} of $f$ with respect to $\|\cdot\|_{L_p(\Omega_T)}$, $0<\rho\le p$, in the finite element space $\V^{r_1, r_2}_{\calP}$. This will be crucial for obtaining inverse estimates.

\begin{thm}\label{thm:operater_pi_quasi-best_approx}
    Consider $p\in (0,\infty]$, $\rho\in (0,p]$, and $f\in L_\rho(\Omega_T)$. Then,
    \begin{align*}
        \|f-\pi_{\rho,\calP}(f)\|_{L_p(\Omega_T)}\lesssim_{\,d,\rho,p,s_1, s_2, r_1, r_2,\kappa_{\calP_0}, \mu(\calP_0)} E\left(f, \V^{r_1, r_2}_{\calP},\Omega_T\right)_p.
    \end{align*}
\end{thm}
\begin{proof}
        First, we observe that by \cref{lem:smaller_parameter_bestapprox_larger_parameter_quasi-bestapprox}, $\|f-\calB_{\rho,\calP}(f)\|_{L_p(\Omega_T)}\lesssim_{\,d,\rho, p, r_1, r_2}E\left(f, \V^{r_1, r_2}_{\calP, \textup{DC}},\Omega_T\right)_p$ holds true and additionally consider $G\in \V^{r_1, r_2}_{\calP}$ such that $\|f-G\|_{L_p(\Omega_T)}=E\left(f, \V^{r_1, r_2}_{\calP},\Omega_T\right)_p $, which has to exist, since $\V^{r_1, r_2}_{\calP}$ is finite dimensional. We proceed similar to the proof of \cref{thm:local_boundedness_of_pi_projection} and derive
        \begin{align*}
            \|f-\pi_{\rho, \calP}(f)\|_{L_p(\Omega_T)}&\lesssim_{\,p}\|f-G\|_{L_p(\Omega_T)} + \|G-\pi_{\rho, \calP}(f)\|_{L_p(\Omega_T)}
            \\&= E\left(f, \V^{r_1, r_2}_{\calP},\Omega_T\right)_p + \|\mathcal{Q}_\calP(G-\calB_{\rho, \calP}(f))\|_{L_p(\Omega_T)}
            \\&\lesssim_{\,d,p,s_1, s_2,r_1, r_2, \kappa_{\calP_0},\mu(\calP_0)} E\left(f, \V^{r_1, r_2}_{\calP},\Omega_T\right)_p + \|G-\calB_{\rho, \calP}(f)\|_{L_p(\Omega_T)}
            \\&\lesssim_{\,} E\left(f, \V^{r_1, r_2}_{\calP},\Omega_T\right)_p + \|f-G\|_{L_p(\Omega_T)} + \|f-\calB_{\rho, \calP}(f)\|_{L_p(\Omega_T)}
            \\&\lesssim_{\,d,\rho, p, r_1, r_2} E\left(f, \V^{r_1, r_2}_{\calP},\Omega_T\right)_p,
        \end{align*}
    where we have applied \cref{rem:linearity_and_projection_property_of_Q} in the second step, \cref{lem:local_boundedness_of_Q_projection}\ref{lem:local_boundedness_of_Q_projection - item 2} together with the finite overlap of the neighborhood domains due to \cref{Enumeration further properties of the mesh}\ref{Enumeration further properties of the mesh - enum 3} in the third, and $E\left(f, \V^{r_1, r_2}_{\calP, \textup{DC}},\Omega_T\right)_p\le E\left(f, \V^{r_1, r_2}_{\calP},\Omega_T\right)_p$ in the last one.
\end{proof}

Further, we will need an extension of the \textbf{Whitney estimate} from \cite[Thm.~1.2]{MSS26} to the cylindric closures of extended neighborhoods in the next section. Therefore, we proceed similar to \cite[Sect.~4.4]{GM14}.

\begin{lem}\label{lem:Extended_Whitney_Estimate}
    Let $p,q\in (0,\infty]$ as well as $s_1, s_2\in (0,\infty)$ such that $\frac{1}{\frac{1}{s_1}+\frac{d}{s_2}}-\frac{1}{q}+\frac{1}{p}>0$ and $s_i<r_i$ for some $r_i\in \N_{\ge 2}$, $i=1,2$. Additionally, let $I\times S\in \calP$ and $f\in B^{s_1, s_2}_{q,q}(\Omega_T)$. Then, 
    \begin{align*}
        E\Big(f, \Pi^{r_1, r_2}_{t,\bm{x}},\widetilde{\omega^{j(d)}_\calP}(I\times S)\Big)_p \lesssim_{\, d,p,q,s_1,s_2, \kappa_{\calP_0}, \mu(\calP_0), a(\calP_0)} |I\times S|^{\frac{1}{\frac{1}{s_1}+\frac{d}{s_2}}-\frac{1}{q}+\frac{1}{p}}|f|_{B^{s_1, s_2}_{q,q}(\widetilde{\omega_\calP^{j(d)}}(I\times S))}.
    \end{align*}
    The assertion still holds true, if we replace $|f|_{B^{s_1, s_2}_{q,q}(\widetilde{\omega_\calP^{j(d)}}(I\times S))}$ by $|f|_{B^{s_1, s_2}_{q,q}(\widetilde{\omega_\calP^{j(d)}}(I\times S))}^{*}$. Further, if we use the \mbox{(quasi-)}seminorms $|f|_{B^{s_1, s_2}_{q,q}(\widetilde{\omega_\calP^{j(d)}}(I\times S))}^{(r_1,r_2)}$ or $|f|_{B^{s_1, s_2}_{q,q}(\widetilde{\omega_\calP^{j(d)}}(I\times S))}^{*, (r_1, r_2)}$ instead,\footnote{Recall the definitions of the different Besov (quasi-)(semi)norms from \cref{subsubsect:Besov spaces}.} the assertion still holds, but the involved constants additionally depend on $r_1, r_2$. All of the above assertions also hold for the corresponding reference domain $[0,1]\times R^{ref}$ (cf. \cref{rem:Neighborhoods_generally_not_Lipschitz}\ref{rem:Neighborhoods_generally_not_Lipschitz - reference domains}) instead of $\widetilde{\omega^{j(d)}_\calP}(I\times S)$.
\end{lem}

\begin{proof}
If $\tilde r_i \ge r_i$ then $E(f,\Pi^{\tilde r_1, \tilde r_2}_{t,\bm{x}},D)_p \le E(f,\Pi^{r_1, r_2}_{t,\bm{x}},D)_p $ for any $D\subset \Omega_T$, whence it suffices to consider the case of $r_i:=\max(\lfloor s_i\rfloor +1,2)$, $i=1,2$, for the first part of the assertion.
Since $\widetilde{\omega_\calP^{j(d)}}(I\times S)$ is a cylindrical domain, we put $\widetilde{\omega_\calP^{j(d)}}(I\times S):=J\times R$ with an interval $J$ and a (spatial) domain $R\subset \R^d$, see \cref{rem:Neighborhoods_generally_not_Lipschitz}\ref{rem:Neighborhoods_generally_not_Lipschitz - General S_w''}.\footnote{Of course, technically this means $\bigcup\limits_{I'\times S' \in \widetilde{\omega_\calP^{j(d)}}(I\times S)}I'\times S'=J\times R$.} Further, consider the \textit{reference domain} $R^{ref}$ and the corresponding scaling operator $\Phi_{J\times R}$ from \cref{rem:Neighborhoods_generally_not_Lipschitz}\ref{rem:Neighborhoods_generally_not_Lipschitz - reference domains}. Using the bijective, affine transformation $\Phi_{J\times R}:J\times R \rightarrow [0,1]\times R^{ref}$ we have
    \begin{align}\label{lem:Extended_Whitney_Estimate - equation which is needed later}
        |J\times R|^{-\frac1q}\,\min\left(|J|^{s_1},|R|^\frac{s_2}{d}\right)|f|_{B^{s_1, s_2}_{q,q}(J\times R)}&\le|\hat{f}|_{B^{s_1, s_2}_{q,q}([0,1]\times R^{ref})}\notag
        \\&\le |J\times R|^{-\frac1q}\,\max\left(|J|^{s_1},|R|^\frac{s_2}{d}\right)|f|_{B^{s_1, s_2}_{q,q}(J\times R)}
    \end{align}
    for $\hat{f}:=f\circ \Phi_{J\times R}^{-1}$, as in the proof of \cite[Lem.~3.7]{MSS26}. Due to \eqref{rem:Neighborhoods_generally_not_Lipschitz - equation 3.5 - anisotropy of J x R} from \mbox{\cref{rem:Neighborhoods_generally_not_Lipschitz}\ref{rem:Neighborhoods_generally_not_Lipschitz - General S_w''}}, we can estimate
    \begin{align}\label{lem:Extended_Whitney_Estimate - equation 2}
        \min\left(|J|^{s_1}, |R|^\frac{s_2}{d}\right)\sim_{\,d, s_1, s_2, \kappa_{\calP_0}, \mu(\calP_0), a(\calP_0)} |J\times R|^{\frac{1}{\frac{1}{s_1}+\frac{d}{s_2}}}\sim_{\,d, s_1, s_2, \kappa_{\calP_0}, \mu(\calP_0), a(\calP_0)} \max\left(|J|^{s_1}, |R|^\frac{s_2}{d}\right)
    \end{align}
    as in the beginning of the proof of \cite[Thm.~1.2]{MSS26}. Also, by the Jacobi transformation theorem,
    \begin{align}\label{lem:Extended_Whitney_Estimate - equation 3}
        E\left(f, \Pi^{r_1, r_2}_{t,\bm{x}},J\times R\right)_p = |J\times R|^{\frac{1}{p}} \, E\left(\hat{f}, \Pi^{r_1, r_2}_{t,\bm{x}},[0,1]\times R^{ref}\right)_p.
    \end{align}
    Together, \eqref{lem:Extended_Whitney_Estimate - equation which is needed later}, \eqref{lem:Extended_Whitney_Estimate - equation 2}, \eqref{lem:Extended_Whitney_Estimate - equation 3}, and $|I\times S|\sim_{\,d, s_1, s_2, \kappa_{\calP_0}, \mu(\calP_0)} |J\times R|$, due to \cref{Enumeration further properties of the mesh}\ref{Enumeration further properties of the mesh - enum 2}, show that the assertion follows if 
    \begin{align*}
        E\left(\hat{f}, \Pi^{r_1, r_2}_{t,\bm{x}}, [0,1]\times R^{ref}\right)_p\lesssim_{\,\,d, p,q, s_1, s_2, \kappa_{\calP_0}, \mu(\calP_0),a(\calP_0)} |\hat{f}|_{B^{s_1, s_2}_{q,q}([0,1]\times R^{ref})}
    \end{align*}
    and the corresponding versions for the other (quasi-)seminorms can be shown for any reference domain $[0,1]\times R^{ref}$. 

    According to the anisotropic Jackson's estimate from \mbox{\cref{thm:Proofs_omega_tilde_domain}\ref{thm:Proofs_omega_tilde_domain - item Jackson}}, it exists $\hat{P}\in \Pi^{r_1, r_2}_{t,\bm{x}}([0,1]\times R^{ref})$ with
    \begin{align}\label{lem:Extended_Whitney_Estimate - equation 5}
        \|\hat{f}-\hat{P}\|_{L_q([0,1]\times R^{ref})} &\lesssim_{\, d,q,s_1,s_2, \kappa_{\calP_0}, \mu(\calP_0)} \omega_{r_1, t}(\hat{f}, [0,1]\times R^{ref}, 1)_q + \omega_{r_2, \bm{x}}(\hat{f}, [0,1]\times R^{ref}, \diam(R^{ref}))_q \notag
        \\&\lesssim_{\,d, q, s_2, \kappa_{\calP_0}, \mu(\calP_0)} \omega_{r_1, t}(\hat{f}, [0,1]\times R^{ref}, 1)_q + \omega_{r_2, \bm{x}}(\hat{f}, [0,1]\times R^{ref}, 1)_q \notag
        \\&\lesssim_{\, q} |\hat{f}|^{*, (r_1, r_2)}_{B^{s_1,s_2}_{q,q}([0,1]\times R^{ref})},
    \end{align}
    where the first step is the aforementioned Jackson's estimate, the second one follows from $ \diam(R^{ref})\lesssim_{\,d, \kappa_{\calP_0}, \mu(\calP_0)}1$ according to \cref{rem:Neighborhoods_generally_not_Lipschitz}\ref{rem:Neighborhoods_generally_not_Lipschitz - reference domains} together with \cref{Rem:Moduli of smoothness}\ref{Rem:Moduli of smoothness - monotonicity} and \ref{Rem:Moduli of smoothness - scaling item}, respectively, and the last one is a consequence of the definition of $|\cdot|^{*,(r_1,r_2)}_{B^{s_1,s_2}_{q,q}([0,1]\times R^{ref})}$. Further, we obtain
    \begin{align}\label{lem:Extended_Whitney_Estimate - equation 6}
        E\left(\hat{f}, \Pi^{r_1, r_2}_{t,\bm{x}},[0,1]\times R^{ref}\right)_p &\le \|\hat{f}-\hat{P}\|_{L_p([0,1]\times R^{ref})}\notag
        \\&\lesssim_{\, d,p,q,s_1, s_2, \kappa_{\calP_0},\mu(\calP_0), a(\calP_0)} \|\hat{f}-\hat{P}\|_{B^{s_1, s_2}_{q,q}([0,1]\times R^{ref})}^{(r_1, r_2)} \notag
        \\&= \|\hat{f}-\hat{P}\|_{L_q([0,1]\times R^{ref})} + |\hat{f}-\hat{P}|_{B^{s_1, s_2}_{q,q}([0,1]\times R^{ref})}^{(r_1, r_2)}
    \end{align}
    Since $|\cdot|_{B^{s_1, s_2}_{q,q}([0,1]\times R^{ref})}^{(r_1, r_2)}$ is a (quasi-)seminorm, we can further estimate 
    \begin{align}\label{lem:Extended_Whitney_Estimate - equation 7}
        |\hat{f}-\hat{P}|_{B^{s_1, s_2}_{q,q}([0,1]\times R^{ref})}^{(r_1, r_2)}\lesssim_{\,q}|\hat{f}|_{B^{s_1, s_2}_{q,q}([0,1]\times R^{ref})}^{(r_1, r_2)} + |\hat{P}|_{B^{s_1, s_2}_{q,q}([0,1]\times R^{ref})}^{(r_1, r_2)}= |\hat{f}|_{B^{s_1, s_2}_{q,q}([0,1]\times R^{ref})}^{(r_1, r_2)},
    \end{align}
    where we have additionally used the fact that $|\hat{P}|_{B^{s_1, s_2}_{q,q}([0,1]\times R^{ref})}^{(r_1, r_2)}=0$ because of \cref{thm:Proofs_omega_tilde_domain}\ref{thm:Proofs_omega_tilde_domain - item moduli of smoothness equal zero IFF anisotropic polynomials} and $\hat{P}\in \Pi^{r_1, r_2}_{t,\bm{x}}([0,1]\times R^{ref})$ in the second equality. Now, combining \eqref{lem:Extended_Whitney_Estimate - equation 5}, \eqref{lem:Extended_Whitney_Estimate - equation 6}, and \eqref{lem:Extended_Whitney_Estimate - equation 7} gives
    \begin{align}\label{lem:Extended_Whitney_Estimate - equation 8}
        E\left(\hat{f}, \Pi^{r_1, r_2}_{t,\bm{x}}, [0,1]\times R^{ref}\right)_p\lesssim_{\,\,d, p,q, s_1, s_2, \kappa_{\calP_0}, \mu(\calP_0),a(\calP_0)} |\hat{f}|^{*,(r_1,r_2)}_{B^{s_1,s_2}_{q,q}([0,1]\times R^{ref})}+ |\hat{f}|_{B^{s_1, s_2}_{q,q}([0,1]\times R^{ref})}^{(r_1, r_2)}.
    \end{align}

    Due to the equivalence of the \mbox{(quasi-)}seminorms $|\cdot|_{B^{s_1, s_2}_{q,q}([0,1]\times R^{ref})}^{(r_1, r_2)}$ and $|\cdot|_{B^{s_1, s_2}_{q,q}([0,1]\times R^{ref})}^{*, (r_1, r_2)}$ stated in \cref{subsubsect:Besov spaces} below \eqref{Eq:equivalence quasi-seminorms},\footnote{Technically, this equivalence has only been stated for $|\cdot|_{B^{s_1, s_2}_{q,q}([0,1]\times R^{ref})}$ and $|\cdot|_{B^{s_1, s_2}_{q,q}([0,1]\times R^{ref})}^{*}$ there, but it holds analogously for the versions with the additional \enquote{exponent} $(r_1, r_2)$.} and $\diam(R^{ref})\lesssim_{\,d,\kappa_{\calP_0}, \mu(\calP_0)}1$, the estimate \eqref{lem:Extended_Whitney_Estimate - equation 8} also holds with respect to only one of the \mbox{(quasi-)}seminorms. The asserted result with respect to $|\cdot|_{B^{s_1, s_2}_{q,q}([0,1]\times R^{ref})}$ and $|\cdot|_{B^{s_1, s_2}_{q,q}([0,1]\times R^{ref})}^{*}$ can now easily be deduced from \mbox{\cref{Rem:Moduli of smoothness}\ref{Rem:Moduli of smoothness - higher order lesser order}},\footnote{This step is necessary, since $|\cdot|_{B_{q,q}^{s_1, s_2}}\neq|\cdot|^{(r_1, r_2)}_{B_{q,q}^{s_1, s_2}} $ if $r_i\neq \lfloor s_i \rfloor +1$, i.e., $s_i\in (0,1)$ for $i=1,2$.} since it implies that 
    \begin{align*}
        |\cdot|_{B^{s_1, s_2}_{q,q}([0,1]\times R^{ref})}^{(r_1, r_2)}\lesssim_{\,q} |\cdot|_{B^{s_1, s_2}_{q,q}([0,1]\times R^{ref})}\text{ and }|\cdot|_{B^{s_1, s_2}_{q,q}([0,1]\times R^{ref})}^{*,(r_1,r_2)}\lesssim_{\,q} |\cdot|^*_{B^{s_1, s_2}_{q,q}([0,1]\times R^{ref})}.
    \end{align*}
    If we omit the initial choice of $r_i\in \N_{\ge 2}$ minimal with $r_i>s_i$, $i=1,2$, we can perform exactly the same calculations until \eqref{lem:Extended_Whitney_Estimate - equation 8} with additional dependencies of the constants on $r_1$ and $r_2$. Finally, the aforementioned equivalence of the seminorms shows the second part of the assertion and finishes the proof.
\end{proof}

\begin{rem}\label{rem:to_lem:Extended_Whitney_Estimate}
    Note that \cref{lem:Extended_Whitney_Estimate} extends to the enlarged versions of the domains (cf. \cref{rem:Neighborhoods_generally_not_Lipschitz}\ref{rem:Neighborhoods_generally_not_Lipschitz - enlargement}), where the proof is almost verbatim, since this case is also covered by \cref{thm:Proofs_omega_tilde_domain}\ref{thm:Proofs_omega_tilde_domain - item Jackson} and the enlargement does not change the anisotropic structure of the corresponding space-time partition.
\end{rem}

\section{Almost characterization of approximation classes with respect to space-time finite elements}\label{sect:Almost_characterization}

We are now in a position to apply the results of the previous sections to prove the main results of this article, i.e., an almost characterization of approximation classes for approximation with anisotropic finite elements. Throughout this section, we assume $\calP_0$ to be a tensor product structure mesh covering $\Omega_T$ as described in the beginning of \cref{sect:space_time_partition_and_refinement} and $j(d):=2$ for $d=1$, but $j(d):=3$ if $d\ge 2$ as in \cref{thm:support_in_power_omega_domain}.

\subsection{Multiscale decomposition}\label{subsect:Multiscale_decomp}

Let us assume that $\calP$ has been derived from $\calP_0$ by finitely many applications of $\textup{PATCH}\_\textup{REFINE}(\cdot, \cdot, d, s_1, s_2)$ for any given $s_1, s_2\in (0,\infty)$. Since we want to study how well functions can be approximated with respect to $\|\cdot\|_{B^{\alpha_1, \alpha_2}_{p,p}(\Omega_T)}$ for $p\in(0,\infty]$ and $\alpha_1, \alpha_2\in (0,\infty)$ by anisotropic polynomials, we first have to determine in which cases the approximants are in an anisotropic Besov space. This is done in the following lemma, similarly to \cite[Prop.~4.7]{GM14}.

\begin{lem}\label{lem:FEM_elements_of_Aniso_Besov}
    Let $r_1, r_2 \in \N_{\ge 2}$  be the temporal and spatial polynomial degrees with $\alpha_i < r_i$, $i=1,2$. Then $\V^{r_1, r_2}_\calP \subset B^{\alpha_1, \alpha_2}_{p,q}$ if and only if $\max(\alpha_1,\alpha_2) < 1+\frac1p$ or $\max(\alpha_1,\alpha_2) = 1+\frac1p$ and $q=\infty$.
\end{lem}
\begin{proof}
    Let $I\times S\in \calP$, $\nu\in \calN(I\times S)\cap\calF(\calP)$, and $\phi_\nu\in \V^{r_1, r_2}_\calP$ the basis function corresponding to $\nu$. Further, we set $\Theta_\nu:= \supp \phi_\nu\subset\omega^{j(d)}_\calP(I\times S)$, which is bounded due to \cref{thm:support_in_power_omega_domain}. Additionally, we consider the temporal and spatial diameter of $\Theta_\nu$, i.e.,
    \begin{align}\label{lem:FEM_elements_of_Aniso_Besov - Proof Eq 0}
        \diam_t\left(\Theta_\nu\right):=\sup_{(t,\bm{x}), (t', \bm{x}')\in \Theta_\nu} |t-t'|\textcolor{black}{\sim|I|}\quad \text{and}\quad \diam_{\bm{x}}\left(\Theta_\nu\right):=\sup_{(t,\bm{x}), (t', \bm{x}')\in \Theta_\nu} |\bm{x}-\bm{x}'|\textcolor{black}{\sim\diam(S)\sim|S|^{1/d},}
    \end{align}
    where the equivalence constants will depend on $d,s_1, s_2, \kappa_{\calP_0}$, and $\mu(\calP_0)$. Since $\phi_\nu(\nu)=1$ and $\phi_\nu \equiv 0$ outside of $\Theta_\nu$, $\phi_\nu(\cdot,\bm{x})$ and $\phi_\nu(t,\cdot)$ are Lipschitz continuous with respect to the constants $\Lip(\phi_\nu(\cdot, \bm{x}))$ and $\Lip(\phi_\nu(t,\cdot))$, where 
    \begin{align*}
        \Lip(\phi_\nu(\cdot,\bm{x}))\sim_{\,d, s_1, s_2, r_1, \kappa_{\calP_0}, \mu(\calP_0)} \diam_t\left(\Theta_\nu\right)^{-1}\quad \text{and} \quad \Lip(\phi_\nu(t,\cdot))\sim_{\,d, s_1, s_2, r_2, \kappa_{\calP_0}, \mu(\calP_0)} \diam_{\bm{x}}\left(\Theta_\nu\right)^{-1}
    \end{align*}
    for any $t\in [0,T]$ and $\bm{x}\in \overline{\Omega}$.\footnote{More precisely, $\Lip(\phi_\nu(t,\cdot))$ denotes the Lipschitz constant with respect to all $\bm{y},\bm{z}\in\overline\Omega$ with $[\bm{y},\bm{z}]\subset\overline\Omega$, where $[\bm{y},\bm{z}]$ means the line segment between $\bm{y}$ and $\bm{z}$. This constant suffices for the difference operators below and does not depend on the geometry of $\Omega$, whereas the Lipschitz constant with respect to all $\bm{y},\bm{z}\in\overline\Omega$ may.} Therefore, using the algebraic expressions for the difference operators from \cite[Rem.~2.2]{MSS26}\footnote{The results are only stated for Lipschitz domains in \cite{MSS26}, but this assumption is not necessary for them to hold true.}, we obtain \begin{align}\label{lem:FEM_elements_of_Aniso_Besov - Proof Eq 1}
        |\Delta^{r_1}_{h,t}\phi_\nu(t,\bm{x})| &\lesssim_{\,d, s_1, s_2, r_1, \kappa_{\calP_0}, \mu(\calP_0)} \min\left(1, |h|\diam_t(\Theta_\nu)^{-1}\right)\text{ as well as } \notag
        \\|\Delta^{r_2}_{\bm{h}, \bm{x}}\phi_\nu(t,\bm{x})|&\lesssim_{\,d, s_1, s_2, r_2, \kappa_{\calP_0}, \mu(\calP_0)}
        \min\left(1, |\bm{h}|\diam_{\bm{x}}(\Theta_\nu)^{-1}\right)
    \end{align}
    for $h\in \R$ and $(t,\bm{x})\in [0,T]_{r_1,h}\times \overline{\Omega}$ in the first case as well as for $\bm{h}\in \R^d$ and $(t,\bm{x})\in [0,T]\times \overline{\Omega}_{r_2,\bm{h}}$ in the second one. Note that for every $\delta\ll \min(\diam_t(\Theta_\nu), \diam_{\bm{x}}(\Theta_\nu))$, there exist such $h$ and $\bm{h}$ with $|h|= |\bm{h}|= \delta$ for which the above estimates hold with $\sim$ on a subset of $\supp(\Delta^{r_1}_{h,t}\phi_\nu)$ and $\supp(\Delta^{r_2}_{\bm{h},\bm{x}}\phi_\nu)$ with comparable measure, respectively. 
    
    Additionally, we know that $\Delta^{r_1}_{h,t}P = \Delta^{r_2}_{\bm{h},\bm{x}}P = 0$ for $P\in \Pi^{r_1, r_2}_{t,\bm{x}}$ due to \cref{thm:Proofs_omega_tilde_domain}\ref{thm:Proofs_omega_tilde_domain - item moduli of smoothness equal zero IFF anisotropic polynomials}. Therefore,\footnote{\textup{dist} refers to the Euclidean distance here.} $\Delta^{r_1}_{h,t}\phi_\nu = 0$ on $\bigcup\limits_{I'\times S'\in \Theta_\nu} \{t\in I'\mid \textup{dist}(t,\partial I')\ge r_1|h|\} \times S'$ and thus $\supp(\Delta_{h,t}^{r_1}\phi_\nu)$ is contained in a set of $(t,\bm x)$ points such that $t$ is close to the boundary of some $I'$, and $\bm x$ is in a set of size $\diam_{\bm x}(\Theta_\nu)$, whence
    \begin{align}\label{lem:FEM_elements_of_Aniso_Besov - Proof Eq 3}
        |\supp(\Delta^{r_1}_{h,t}\phi_\nu)| \lesssim_{\,d, r_1, \kappa_{\calP_0}, \mu(\calP_0)} \diam_{\bm{x}}(\Theta_\nu)^d \min(|h|, \diam_{t}(\Theta_\nu)).
    \end{align}
    Also, $\Delta^{r_2}_{\bm{h},\bm{x}}\phi_\nu = 0$ on $\bigcup\limits_{I'\times S'\in \Theta_\nu} I'\times \{\bm{x}\in {S'}\mid \textup{dist}(\bm{x},\partial S')\ge r_2|\bm{h}|\}$ and thus $\supp(\Delta_{\bm x,t}^{r_2}\phi_\nu)$ is contained in a set of $(t,\bm x)$ points such that $\bm x$ is close to the boundary of some $S'$, and $t$ is in a set of size $\diam_t(\Theta_\nu)$, so that
    \begin{align}\label{lem:FEM_elements_of_Aniso_Besov - Proof Eq 4}
        |\supp(\Delta^{r_2}_{\bm{h},\bm{x}}\phi_\nu)| \lesssim_{\,d, r_2, \kappa_{\calP_0}, \mu(\calP_0)} \diam_{t}(\Theta_\nu) \diam_{\bm{x}}(\Theta_\nu)^{d-1} \min(|\bm{h}|, \diam_{\bm{x}}(\Theta_\nu)).
    \end{align}
    We notice also that both these bounds hold with $\sim$ when $|h| \ll \diam_{t}(\Theta_\nu)$ and $|\bm{h}| \ll \diam_{\bm{x}}(\Theta_\nu)$, respectively. Recalling the definition of the moduli of smoothness from \cref{subsubsect:Besov spaces} we conclude that for $\delta\ll \min(\diam_t(\Theta_\nu), \diam_{\bm{x}}(\Theta_\nu))$, \eqref{lem:FEM_elements_of_Aniso_Besov - Proof Eq 1}, \eqref{lem:FEM_elements_of_Aniso_Besov - Proof Eq 3}, and \eqref{lem:FEM_elements_of_Aniso_Besov - Proof Eq 4} together yield
    \begin{align*}
        \omega_{r_1, t}(\phi_\nu, \Omega_T, \delta)_p^q + \omega_{r_2, \bm{x}}(\phi_\nu, \Omega_T, \delta)_p^q \sim_{\,d, s_1, s_2, r_1, r_2, \kappa_{\calP_0}, \mu(\calP_0)}\left(\delta^{1+\frac{1}{p}}\right)^q = \left(\delta^{q+\frac{q}{p}}\right),
    \end{align*}
    for $q<\infty$. Due to the definition of the Besov (quasi-)seminorms and their equivalences (also in \mbox{\cref{subsubsect:Besov spaces}}), $|\phi_\nu|^{(r_1, r_2)}_{B^{\alpha_1, \alpha_2}_{p,q}(\Omega_T)}<\infty$, i.e., $\phi_\nu\in B^{\alpha_1, \alpha_2}_{p,q}(\Omega_T)$, if and only if \mbox{$\max\{\alpha_1, \alpha_2\}<1+\frac1p$} for $q<\infty$. Similarly, for $q=\infty$, this quantity is finite (and thus $\phi_\nu\in B^{\alpha_1, \alpha_2}_{p,q}(\Omega_T)$) as long as \mbox{$\max\{\alpha_1, \alpha_2\}\le 1+\frac1p$}. This readily implies the assertion.
\end{proof}

Now since $F\in \V^{r_1, r_2}_\calP$ can be approximated up to an error zero with respect to the adaptive anisotropic finite element method and any error (quasi-)norm, but $F\notin B^{\alpha_1, \alpha_2}_{p,q}(\Omega_T)$ if $\max(\alpha_1, \alpha_2)\ge 1+\frac{1}{p}$, we additionally need a generalized notion of anisotropic Besov spaces for the derivation of inverse estimates. Namely, we want to show that if a function can be well approximated with the adaptive technique, it necessarily possesses some (generalized) Besov regularity of higher order. 

This generalized notion of Besov regularity, can also be interpreted as a different characterization of the original anisotropic Besov spaces (defined in \cref{subsubsect:Besov spaces}) for $\max(\alpha_1, \alpha_2)< 1+\frac{1}{p}$, which  is also of value when proving Sobolev-type embeddings in our anisotropic Besov scale and corresponding direct estimates. This is why we want to introduce it before the specific sections on direct and inverse estimates, i.e., \cref{subsect:direct estimates} and \cref{subsect:inverse estimates}, respectively. Its relevance is motivated by the result on multiscale decompositions stated in \cref{lem:Multiscale_decomposition}, which corresponds to \cite[Sect.~4.3 and Sect.~7]{GM14}.

\begin{assumptions}\label{assumptions:multiscale-decomposition}
Choose $s_1, s_2\in (0,\infty)$. For $n\in \N$, let $\calP_n$ be derived from $\calP_{n-1}$ via uniform refinement of all elements by the method $\textup{ATOMIC}\_\textup{SPLIT}(\cdot, d,s_1, s_2)$.\footnote{$\calP_n$ could also be interpreted as a result of $\textup{MARK}\_\textup{AND}\_\textup{REFINE}(\calP_0, d, s_1, s_2)$ with respect to a \textup{MARK}-routine which marks all prisms for refinement.} Further, assume $(\alpha_1, \alpha_2)\in \R^{+}(s_1, s_2)$ and $r_i\in \N_{\ge 2}$ for $i=1,2$, and set $\pi_n:=\pi_{\rho,\calP_n}$ (with respect to $r_1$ and $r_2$) for given $\rho\in (0,p]$ and $p,q\in (0,\infty]$. Now put $\Delta_n:=\pi_n-\pi_{n-1} $ for $n\in \N_0$ with $\pi_{-1}:=0$. 
\end{assumptions}

\begin{defi}\label{def:definition_of_multiscale_decomposition_Besov_"norm"}
    Under \cref{assumptions:multiscale-decomposition}, we define 
    \begin{align*}
        	\|f\|_{\widehat{B}^{\alpha_1,\alpha_2}_{p,q}(\Omega_T)}^\Delta:=	\left\{\begin{array}{ll} \left(\sum\limits_{n=0}^{\infty} 2^{n\frac{\alpha_2}{d}q}\|\Delta_n(f)\|_{L_p(\Omega_T)}^q\right)^\frac{1}{q},& q < \infty 
    		\\ \ \ \sup\limits_{n\in \N_0} 2^{n\frac{\alpha_2}{d}}\|\Delta_n(f)\|_{L_p(\Omega_T)} \ \ \ \, ,& q = \infty \end{array}\right\} 	
    \end{align*}
    for any $f\in L_{p}(\Omega_T)$. \textbf{Note} that there is a hidden dependence of both expressions on the choice of $\rho$ and $(r_1, r_2)$. We will not indicate this in the sequel in order to improve readability.
\end{defi}

\begin{rem}\label{rem:conformity_of_the_calP_n}
    Notice that for any $n\in \N_0$,
    \begin{align*}
        \calP_n = \bigcup\limits_{I\times S\in \calP_0} \textup{ATOMIC}\_\textup{SPLIT}(I\times S, d,s_1, s_2 )^{\otimes n} = \left(\bigcup\limits_{I\in \calI_0}  \textup{BISECT}(I,1)^{\otimes \left\lceil \frac{ns_2}{s_1d}\right\rceil}\right)\otimes \left(\bigcup_{S\in \calT_0} \textup{BISECT}(S,d)^{\otimes n} \right),
    \end{align*}
    i.e., $\calP_n$ is itself a tensor product mesh. Therefore, it is conforming, since its spatial triangulation is conforming due to \cite[Thm.~4.3]{Ste08}.
\end{rem}

\begin{lem}\label{lem:Multiscale_decomposition}
    Let \cref{assumptions:multiscale-decomposition} hold with $\alpha_i<r_i$, $i=1,2$, and $f\in B^{\alpha_1, \alpha_2}_{p,q}(\Omega_T)$. Then,
    \begin{align*}
		\|f\|_{\widehat{B}^{\alpha_1,\alpha_2}_{p,q}(\Omega_T)}^\Delta
		\lesssim_{\, d, \rho, p,q, \alpha_1, \alpha_2, r_1, r_2,\kappa_{\calP_0}, \mu(\calP_0), \LipProp(\Omega) } \|f\|_{B^{\alpha_1,\alpha_2}_{p,q}(\Omega_T)}.
	\end{align*}
\end{lem}
\begin{proof}
    We assume $q<\infty$, the proof for $q=\infty$ is very similar. Further, for any $n\in \N_0$ and $I\times S\in \calP_n$, we have
	\begin{align*}
		\|\Delta_n(f)\|_{L_p(I\times S)}\lesssim_{\,p}\|f-\pi_{n}(f)\|_{L_p(I\times S)} + \|f-\pi_{n-1}(f)\|_{L_p(I\times S)}.
	\end{align*}
    According to \cref{rem:Neighborhoods_generally_not_Lipschitz}\ref{rem:Neighborhoods_generally_not_Lipschitz - Uniform case S_w''}, we can denote any $j(d)$-th neighborhood as a cylindrical domain $J\times R=\omega^{j(d)}_{\calP_{n}}(I\times S)$. Further, for $n\in \N$, we denote $\hat{I}\times \hat{S}\in \calP_{n-1}$ for the parent of $I\times S$, i.e., $I\times S\in \textup{ATOMIC}\_\textup{SPLIT}(\hat{I}\times \hat{S}, d, s_1, s_2)$. Correspondingly, we define $\hat{J}\times \hat{R}:=\omega^{j(d)}_{\calP_{n-1}}(\hat{I}\times \hat{S})$ for $n\in \N$. For the sake of consistency, we also put $\hat{J}\times \hat{R}:=J\times R$ as well as $\hat{I}\times \hat{S}:=I\times S$ if $n=0$. Further, let $\hat{\hat{J}}\times \hat{\hat{R}}:= \omega_{\calP_{n-1}}^{j(d)+1}(\hat{I}\times \hat{S})$ for $n\in \N$ and $\hat{\hat{J}}\times \hat{\hat{R}}:=\omega_{\calP_{0}}^{j(d)+1}(I\times S)$ for $n=0$. For those sets, keep in mind that, as stated in \eqref{rem:Neighborhoods_generally_not_Lipschitz - equation 2} and \eqref{rem:Neighborhoods_generally_not_Lipschitz - equation 3}, $|\hat{J}|\sim_{\,d,\alpha_1,\alpha_2, \kappa_{\calP_0},\mu(\calP_0)}|I|\sim_{\,\mu(\calP_0)}2^{-\frac{n\alpha_2}{\alpha_1d}}$ and \mbox{$\diam(\hat{R})\sim_{\,d,\alpha_1,\alpha_2, \kappa_{\calP_0},\mu(\calP_0)}\diam(S)\sim_{\,\mu(\calP_0)}2^{-\frac{n}{d}}$}. This enables us to estimate 
	\begin{align}\label{lem:Multiscale_decomposition - proof eq 1}
		\|\Delta_n(f)\|_{L_p(I\times S)}&\lesssim_{\,p}\|f-\pi_{n}(f)\|_{L_p(I\times S)} + \|f-\pi_{n-1}(f)\|_{L_p(\hat{I}\times \hat{S})}\mathds{1}_{\N}(n) + \|f\|_{L_p(I\times S)}\,\mathds{1}_{\{0\}}(n)\notag
        \\&\lesssim_{\, \substack{d, \rho, p, \alpha_1, \alpha_2, r_1,\\ r_2, \kappa_{\calP_0}, \mu(\calP_0)}} E\left(f, \Pi^{r_1, r_2}_{t,\bm{x}}, J\times R\right)_p + E\left(f, \Pi^{r_1, r_2}_{t,\bm{x}}, \hat{J}\times \hat{R}\right)_p + \|f\|_{L_p(I\times S)}\,\mathds{1}_{\{0\}}(n) \notag
        \\&\lesssim E\left(f, \Pi^{r_1, r_2}_{t,\bm{x}}, \hat{J}\times \hat{R}\right)_p + \|f\|_{L_p(I\times S)}\,\mathds{1}_{\{0\}}(n)\notag
		\\&\lesssim_{\,\substack{d,p, r_1, r_2, \kappa_{\calP_0}, \mu(\calP_0)}} \omega_{r_1, t}\left(f,\hat{J}\times \hat{R}, |\hat{J}|\right)_p+ \omega_{r_2, \bm{x}}\left(f, \hat{J}\times \hat{R}, \diam(\hat{R})\right)_p + \|f\|_{L_p(I\times S)}\,\mathds{1}_{\{0\}}(n)\notag
        \\&\lesssim_{\,\substack{d,p,\alpha_1, \alpha_2, r_1, \\r_2, \kappa_{\calP_0}, \mu(\calP_0),\\\LipProp(\Omega)}} \mathrm{w}_{r_1, t}\left(f,\hat{\hat{J}}\times \hat{\hat{R}}, \delta_n^t\right)_p+ \mathrm{w}_{r_2, \bm{x}}\left(f, \hat{\hat{J}}\times \hat{\hat{R}}, \delta_n^{\bm{x}}\right)_p + \|f\|_{L_p(I\times S)}\,\mathds{1}_{\{0\}}(n),		
	\end{align}
	where we have subsequently applied the (quasi-)subadditivity of $\|\cdot\|_{L_p(I\times S)}$ together with $\pi_{-1}=0$ and $I\times S\subset \hat{I}\times \hat{S}$ in the first step, \cref{thm:local_boundedness_of_pi_projection} in the second, $J\times R\subset \hat{J}\times \hat{R}$ in the third, Jackson's estimate from \mbox{\cref{thm:Proofs_omega_tilde_domain}\ref{thm:Proofs_omega_tilde_domain - item Jackson}} in the penultimate, and \cref{Rem:Moduli of smoothness}\ref{Rem:Moduli of smoothness - scaling item} together with the \mbox{(quasi-)}equivalence between suprema and averaged moduli of smoothness from \mbox{\cref{thm:Proofs_omega_tilde_domain}\ref{thm:Proofs_omega_tilde_domain - Quasi-equivalence of suprema and averaged moduli of smoothness}} in the last one. Here, for every $n\in \N_0$, $0<\delta_n^{t}=\delta_n^{t}(d,\alpha_1, \alpha_2, r_1 ,\kappa_{\calP_0}, \mu(\calP_0) )\le 2^{-\frac{ns_2}{s_1d}}$ and $0<\delta_n^{\bm{x}}=\delta_n^{\bm{x}}( d,r_2, \kappa_{\calP_0}, \mu(\calP_0),\LipProp(\Omega))\le 2^{-\frac{n}{d}}$ are parameters which are chosen small enough, such that the estimate between suprema and averaged moduli of smoothness is applicable for every such $\hat{J}\times\hat{R}$, where we choose them as large as this requirement permits. The latter guarantees that $\delta_n^{t}\gtrsim_{\,d,\alpha_1, \alpha_2, r_1 ,\kappa_{\calP_0}, \mu(\calP_0)} 2^{-\frac{ns_2}{s_1d}}$ as well as $\delta_n^{\bm{x}}\gtrsim_{\,d,r_2, \kappa_{\calP_0}, \mu(\calP_0),\LipProp(\Omega)} 2^{-\frac{n}{d}}$ independent of $n$.\footnote{This is necessary to assure that the constant in the last estimate of \eqref{lem:Multiscale_decomposition - proof eq 1} is bounded (depending on the parameters) due to the application of \cref{Rem:Moduli of smoothness}\ref{Rem:Moduli of smoothness - scaling item}.} The observations from \cref{Rem:Moduli of smoothness}\ref{Rem:Moduli of smoothness - monotonicity} and \ref{Rem:Moduli of smoothness - additivity} together with the finite overlap of the neighborhood domains, due to \cref{Enumeration further properties of the mesh}\ref{Enumeration further properties of the mesh - enum 2}, now yield
    \begin{align*}
        \|\Delta_n(f)\|_{L_p(\Omega_T)}&\lesssim_{\, d, \rho, p, \alpha_1, \alpha_2,r_1, r_2, \kappa_{\calP_0}, \mu(\calP_0),\LipProp(\Omega)}\mathrm{w}_{r_1, t}\left(f, \Omega_T, \delta_n^t\right)_p+ \mathrm{w}_{r_2, \bm{x}}\left(f, \Omega_T, \delta_n^{\bm{x}}\right)_p + \|f\|_{L_p(\Omega_T)}\,\mathds{1}_{\{0\}}(n)	\notag	
        \\& \lesssim_{\,d,p, \alpha_1, \alpha_2,r_1, r_2} \omega_{r_1, t}\left(f, \Omega_T, 2^{-\frac{n\alpha_2}{\alpha_1d}}\right)_p+ \omega_{r_2, \bm{x}}\left(f, \Omega_T, 2^{-\frac{n}{d}}\right)_p + \|f\|_{L_p(\Omega_T)}\,\mathds{1}_{\{0\}}	(n)
    \end{align*}
    for $n\in \N_0$, by summing over all $I\times S\in \calP_n$. Recalling the different Besov (quasi-)norms and their notation from \cref{subsubsect:Besov spaces} and \cref{def:definition_of_multiscale_decomposition_Besov_"norm"}, respectively, this implies
    \begin{align*}
        \left(\|f\|^\Delta_{\widehat{B}^{\alpha_1, \alpha_2}_{p,q}(\Omega_T)}\right)&\phantom{}^q=\sum\limits_{n=0}^{\infty}2^{n\frac{\alpha_2}{d}q}\|\Delta_n(f)\|_{L_p(\Omega_T)}^q 
        \\\lesssim_{\, \substack{d, \rho, p,q, \alpha_1, \alpha_2, r_1, \\r_2,\kappa_{\calP_0}, \mu(\calP_0),\\\LipProp(\Omega)}}& \|f\|_{L_p(\Omega_T)}^q + \sum\limits_{n=0}^{\infty} 2^{n\frac{\alpha_2}{d}q}\left[\omega_{r_1, t}\left(f, \Omega_T, 2^{-\frac{n\alpha_2}{\alpha_1d}}\right)_p^q+ \omega_{r_2, \bm{x}}\left(f, \Omega_T, 2^{-\frac{n}{d}}\right)_p^q\right]\sim_{\,q} \left(\|f\|_{B^{\alpha_1, \alpha_2}_{p,q}(\Omega_T)}^{*, (r_1, r_2)}\right)^q.
    \end{align*}
    Thus, the assertion
    \begin{align*}
        \|f\|^\Delta_{\widehat{B}^{\alpha_1, \alpha_2}_{p,q}(\Omega_T)}\lesssim_{\, \substack{d, \rho, p,q, \alpha_1, \alpha_2, r_1, \\r_2,\kappa_{\calP_0}, \mu(\calP_0),\\\LipProp(\Omega)}}
        \|f\|_{B^{\alpha_1, \alpha_2}_{p,q}(\Omega_T)}^{*, (r_1, r_2)}
        \sim_{\,d ,p,q,\alpha_1, \alpha_2, r_1, r_2,\LipProp(\Omega)}\|f\|_{B^{\alpha_1, \alpha_2}_{p,q}(\Omega_T)}.
    \end{align*}
    follows, where we have applied the \mbox{(quasi-)}norm equivalence result mentioned below \eqref{Eq:equivalence quasi-seminorms} in the last step.
\end{proof}

This lemma leads to the following definition.

\begin{defi}\label{def:Generalized_Besov_Spaces}
	Given \cref{assumptions:multiscale-decomposition}, we define \textbf{generalized anisotropic Besov spaces} via multiscale decomposition, i.e.,
	\begin{align*}
		\widehat{B}^{\alpha_1,\alpha_2}_{p,q}(\Omega_T):=\left\{f\in L_{p}(\Omega_T) \ \bigg| \ 
        \|f\|^\Delta_{\widehat{B}^{\alpha_1,\alpha_2}_{p,q}(\Omega_T)}<\infty\right\}.
	\end{align*}
\end{defi}

\begin{rem}\label{rem:Generalized_Besov_Spaces}
    \begin{enumerate}[label=(\roman*)]
        \item \label{rem:Generalized_Besov_Spaces - item 1} These spaces indeed generalize the notion of the anisotropic Besov spaces, since \cref{lem:Multiscale_decomposition} implies $B^{\alpha_1,\alpha_2}_{p,q}(\Omega_T)\hookrightarrow\widehat{B}^{\alpha_1,\alpha_2}_{p,q}(\Omega_T)$ for $\alpha_i<r_i$, $i=1,2$. Further, they are suited for our analysis, since for every $\calP$ derived from $\calP_0$ through finitely many applications of $\textup{PATCH}\_\textup{REFINE}(\cdot, \cdot, d, s_1, s_2)$, there exists some $n\in \N$ such that $\calP_n$ is a sub-partition of $\calP$ and thus, $\V^{r_1, r_2}_\calP\subset \V^{r_1, r_2}_{\calP_m}$ for any $m\ge n$, $m\in \N$. The latter, in particular, yields $\Delta_m(F)=0$ for every $F\in \V^{r_1, r_2}_\calP$ and $m>n$, i.e., $F\in \widehat{B}^{\alpha_1,\alpha_2}_{p,q}(\Omega_T)$. Together with the result of \cref{lem:FEM_elements_of_Aniso_Besov}, this shows that $B^{\alpha_1,\alpha_2}_{p,q}(\Omega_T)\subsetneq\widehat{B}^{\alpha_1,\alpha_2}_{p,q}(\Omega_T)$, if $\max(\alpha_1, \alpha_2)\ge 1+\frac{1}{p}$ for $q<\infty$ and $\max(\alpha_1, \alpha_2)> 1+\frac{1}{p}$ for $q=\infty$ in the aforementioned case of $\alpha_i<r_i$, $i=1,2$. Lastly, further below in \cref{cor:to_thm:Embedding_generalized_into_classical_aniso_spaces}, we will see that in this case, i.e., $\alpha_i<r_i$, $i=1,2$, the classical and generalized Besov spaces indeed coincide in the sense of equivalent (quasi-)norms, if the smoothness parameters are small, i.e., $\max(\alpha_1, \alpha_2)<1+\frac{1}{p}$.

        \item \label{rem:Generalized_Besov_Spaces - item 2} These spaces are indeed quasi-normed with respect to $\|\cdot\|^\Delta_{\widehat{B}^{\alpha_1,\alpha_2}_{p,q}(\Omega_T)}$ as we show below in \cref{lem:norm equivalence_gen_besov}.
        
        \item Let $ f\in\widehat{B}^{\alpha_1, \alpha_2}_{p,q}(\Omega_T)\subset L_{p}(\Omega_T)$, then $\Delta_n(f)\in \V\phantom{}^{r_1,r_2}_{\calP_n} $ for any $n\in\N_0$. Further, consider the basis $\left(\phi_\nu^{(n)}\right)_{\nu\in \calF(\calP_n)}$ of $\V\phantom{}^{r_1,r_2}_{\calP_n}$, according to \cref{thm:basis free nodes}. Additionally, recall that $\calL(\calP_n)=\calF(\calP_n)$ because of the conformity of $\calP_n$. Then, there are real coefficients $b_\nu^{(n)}$, such that 
            \begin{align*}
                f = \sum\limits_{n=0}^{\infty} \Delta_n(f)= \sum\limits_{n=0}^\infty \sum\limits_{\nu\in\calL(\calP_n)}\left(b_\nu^{(n)}(f)\phi_\nu^{(n)}\right)\quad\text{in}\quad L_p(\Omega_T),
            \end{align*}
        which is a \textbf{multiscale decomposition} of $f$. Be attentive to the hidden dependence of $b_\nu^{(n)}(f)$ and $\phi_\nu^{(n)}$ on the polynomial orders $(r_1, r_2)$.\label{rem:Generalized_Besov_Spaces - item 3}       
        \item Furthermore, due to the conformity of $\calP_n$ according to \cref{rem:conformity_of_the_calP_n}, $\omega_{\calP_n}^{j}(I\times S)=\widetilde{\omega_{\calP_n}^{j}}(I\times S)$ for any $I\times S\in \calP_n$ and $j\in \N$, and 
        \begin{align*}
            \supp \phi_\nu^{(n)} = \{I\times S \in \calP_n\mid \nu\in \overline{I\times S}\}=:\omega_{\calP_n}(\nu)\subset \omega_{\calP_n}(I'\times S')
         \end{align*}
         for any $I'\times S'\in \calP$ with $\nu\in \overline{I'\times S'}$. Since  the proof from \cref{Enumeration further properties of the mesh}\ref{Enumeration further properties of the mesh - enum 3}  in this situation even yields $\#\omega_{\calP_n}(I\times S)\lesssim_{\, d, \kappa_{\calP_0}}1$ for any $I\times S\in \calP_n$, this implies 
         $\#\omega_{\calP_n}(\nu)\lesssim_{\, d, \kappa_{\calP_0}}1$. \label{rem:Generalized_Besov_Spaces - item 4} 
         Further, this also shows that $|I'\times S'|\sim_{\, d,\kappa_{\calP_0},\mu(\calP_0)}\left|\supp\phi_\nu^{(n)}\right|\sim_{\, d,\kappa_{\calP_0},\mu(\calP_0)} 2^{-n\left(1+\frac{s_2}{s_1d}\right)}$ for any $I'\times S'\in \calP_n$ with $I'\times S'\subset \supp\phi_\nu^{(n)}$.
         \item \label{rem:Generalized_Besov_Spaces - item 5} Additionally, we can estimate, very similarly to \cref{Enumeration further properties of the mesh}\ref{Enumeration further properties of the mesh - enum 2} that
         \begin{align*}
            \diam_t\left(\supp \phi_\nu^{(n)}\right)\sim_{\,d, \kappa_{\calP_0}, \mu(\calP_0)}2^{-\frac{ns_2}{s_1d}}\quad\text{and}\quad\diam_{\bm{x}}\left(\supp \phi_\nu^{(n)}\right)\sim_{\,d, \kappa_{\calP_0}, \mu(\calP_0)}2^{-\frac{n}{d}},
         \end{align*}
         where the notation has been borrowed from the proof of \cref{lem:FEM_elements_of_Aniso_Besov}, in particular, \eqref{lem:FEM_elements_of_Aniso_Besov - Proof Eq 0}.
    \end{enumerate}
\end{rem}

From now on, we will study and apply multiscale decompositions to investigate their relation to (generalized) anisotropic Besov spaces. In order to do so, we will need two lemmata, similar to \cite[Lem.~4.1 and Lem.~4.2]{DP88}.

\begin{lem}\label{lem:First_lem_for_gen=classical Besov spaces}
    Let \cref{assumptions:multiscale-decomposition} hold true, $n\in \N_0$, $F\in \V^{r_1, r_2}_{\calP_n}$, and $F = \sum\limits_{\nu\in\calL(\calP_n)} c_\nu \phi_\nu^{(n)}$ be the representation of $F$ with respect to the basis $\left(\phi_\nu^{(n)}\right)_{\nu\in \calL(\calP_n)}$ of $\V^{r_1, r_2}_{\calP_n}$. Then,
    \begin{align*}
        |c_\nu|\lesssim_{\,d,p,r_1, r_2, \kappa_{\calP_0}, \mu(\calP_0)} 2^{\frac{n}{p}\left(1+\frac{s_2}{s_1 d}\right)}\|F\|_{L_p\left(\supp \phi^{(n)}_\nu\right)}. 
    \end{align*}
\end{lem}
\begin{proof}
    First, we estimate $|c_\nu|\le \|F\|_{L_\infty\left(\supp \phi^{(n)}_{\nu}\right)}\le \sum\limits_{I\times S\subset \supp \phi_\nu^{(n)}}\|F\|_{L_\infty(I\times S)}$. Second, we apply \cite[Lem.~3.5]{MSS26}, in order to obtain
    \begin{align*}
        \|F\|_{L_\infty(I\times S)}\sim_{\,d,p,r_1,r_2} |I \times S|^{-\frac1p}\|F\|_{L_p(I\times S)}\sim_{\,d,\mu(\calP_0)} 2^{\frac{n}{p}\left(1+\frac{s_2}{s_1 d}\right)}\|F\|_{L_p(I\times S)}.
    \end{align*}
    Now the assertion follows using the (quasi-)norm equivalence in $\R^{\#\omega_{\calP_n}}$, which is finite dimensional with $\#\omega_{\calP_n}\lesssim_{\,d,\kappa_{\calP_0}} 1$ according to \cref{rem:Generalized_Besov_Spaces}\ref{rem:Generalized_Besov_Spaces - item 4}.
\end{proof}
\begin{lem}\label{lem:Second_lem_for_gen=classical Besov spaces}
    Under the assumptions from \cref{lem:First_lem_for_gen=classical Besov spaces}, it further holds
    \begin{align*}
        \left(\sum\limits_{\nu\in \calL(\calP_n)} |c_\nu|^p \, 2^{-n\left(1+\frac{s_2}{s_1 d}\right)}\right)^\frac1p \sim_{\,d,p,r_1, r_2, \kappa_{\calP_0}, \mu(\calP_0)} \|F\|_{L_p(\Omega_T)}
    \end{align*}
\end{lem}
\begin{proof}
    Rearranging the result from \cref{lem:First_lem_for_gen=classical Besov spaces} and summing up over all $\nu\in \calL(\calP_n)$ directly yields the asserted upper bound when we keep in mind that the supports of the basis functions have finite overlap depending only on $d$ and $\kappa_{\calP_0}$. Conversely,
    \begin{align*}
        \|F\|_{L_p(\Omega_T)}\sim_{\,d,p, \kappa_{\calP_0}, \mu(\calP_0)} \left(\sum\limits_{\nu\in \calL(\calP_n)} |c_\nu|^p \|\phi_\nu^{(n)}\|_{L_p(\Omega_T)}^p\right)^\frac{1}{p} \sim_{\,d,p, r_1, r_2, \kappa_{\calP_0}, \mu(\calP_0)}  \left(\sum\limits_{\nu\in \calL(\calP_n)} |c_\nu|^p \, 2^{-n\left(1+\frac{s_2}{s_1 d}\right)}\right)^\frac1p
    \end{align*}
    where we have used the (quasi-)norm equivalence for finite dimensional spaces together with the aforementioned finite overlap in the first step as well as \cref{rem:estimate_primal_elements} and \cref{rem:Generalized_Besov_Spaces}\ref{rem:Generalized_Besov_Spaces - item 4} in the second one.
\end{proof}

In the sequel, we employ the following (quasi-)norms on $\widehat{B}^{\alpha_1,\alpha_2}_{p,q}(\Omega_T)$.

\begin{lem}\label{lem:norm equivalence_gen_besov}
    Let \cref{assumptions:multiscale-decomposition} hold true. For $f\in \widehat{B}^{\alpha_1,\alpha_2}_{p,q}(\Omega_T)$, we define
    \begin{align*}
    \|f\|^E_{\widehat{B}^{\alpha_1,\alpha_2}_{p,q}(\Omega_T)}&:= |f|^E_{\widehat{B}^{\alpha_1,\alpha_2}_{p,q}(\Omega_T)}+ \|f\|_{L_{p}(\Omega_T)}:=\left(\sum\limits_{n=0}^{\infty} 2^{n\frac{\alpha_2}{d}q} E\left(f, \V^{r_1, r_2}_{\calP_n}, \Omega_T\right)_p^q \right)^\frac{1}{q} + \|f\|_{L_{p}(\Omega_T)}\text{ and}
    \\\|f\|^\pi_{\widehat{B}^{\alpha_1,\alpha_2}_{p,q}(\Omega_T)}&:=|f|^\pi_{\widehat{B}^{\alpha_1,\alpha_2}_{p,q}(\Omega_T)}+ \|f\|_{L_{p}(\Omega_T)}:=\left(\sum\limits_{n=0}^{\infty} 2^{n\frac{\alpha_2}{d}q} \|f-\pi_n(f)\|_{L_p(\Omega_T)}^q \right)^\frac{1}{q} + \|f\|_{L_p(\Omega_T)}, \text{ and}
    \\\|f\|_{\widehat{B}^{\alpha_1,\alpha_2}_{p,q} (\Omega_T)}&:=\left(\sum\limits_{n=0}^{\infty}\left(\sum\limits_{\nu\in\calL(\calP_n)} |\supp \phi_\nu^{(n)}|^{-\frac{p}{\frac{1}{\alpha_1}+\frac{d}{\alpha_2}}} \left\|b_{\nu}^{(n)}(f)\phi_{\nu}^{(n)}\right\|_{L_p(\Omega_T)}^p\right)^{\frac{q}{p}}\right)^\frac{1}{q}.
    \end{align*}
    with the usual modifications for $q=\infty$, where the components in the last line stem from the multiscale decomposition given in \cref{rem:Generalized_Besov_Spaces}\ref{rem:Generalized_Besov_Spaces - item 3}. Then, the above expressions are all equivalent to $\|\cdot\|^\Delta_{\widehat{B}^{\alpha_1,\alpha_2}_{p,q}(\Omega_T)}$ on $\widehat{B}^{\alpha_1,\alpha_2}_{p,q}(\Omega_T)$ with constants that depend (at most) on $d,\rho, p, q, s_1, s_2, \alpha_2, r_1, r_2, \kappa_{\calP_0}$, and $\mu(\calP_0)$. In particular, $\widehat{B}^{\alpha_1, \alpha_2}_{p,q}(\Omega_T)$ is quasi-normed with respect to any of the four expressions above. Further, for $p=q$, $\left(\|\cdot\|^E_{\widehat{B}^{\alpha_1,\alpha_2}_{p,p}(\Omega_T)}\right)^{\min(1,p)}$ is also subadditive.
\end{lem}

\begin{rem}\label{rem:dependenc_of_gen_besov_on_r_1_r_2_but_not_omega}
    In particular, we stress that $\|\cdot\|^E_{\widehat{B}^{\alpha_1,\alpha_2}_{p,q}(\Omega_T)}$ and due to the above lemma also $\widehat{B}^{\alpha_1,\alpha_2}_{p,q}(\Omega_T)$  do  depend on $(r_1, r_2)$ but not on the choice of $\rho$.
\end{rem}

\begin{proof}
    We will only show the assertion for $q<\infty$. The case $q=\infty$ works with the usual changes. First, we will show that $\|f\|^\Delta_{\widehat{B}^{\alpha_1,\alpha_2}_{p,q} (\Omega_T)}\sim_{\,d,p,r_1, r_2, \kappa_{\calP_0}, \mu(\calP_0)}\|f\|_{\widehat{B}^{\alpha_1,\alpha_2}_{p,q} (\Omega_T)}$. Therefore, using the representation from \cref{rem:Generalized_Besov_Spaces}\ref{rem:Generalized_Besov_Spaces - item 3} and \cref{lem:Second_lem_for_gen=classical Besov spaces}, we obtain
    \begin{align*}
        2^{n\frac{\alpha_2}{d}q}\|\Delta_n(f)\|^q_{L_p(\Omega_T)} \sim_{\,d,p,q,r_1, r_2, \kappa_{\calP_0}, \mu(\calP_0)} \left(\sum\limits_{\nu\in \calL(\calP_n)} 2^{n\frac{\alpha_2}{d}q} \, |b^{(n)}_\nu|^p \, 2^{-n\left(1+\frac{s_2}{s_1 d}\right)}\right)^\frac{q}{p}.
    \end{align*}
    The assertion now follows due to the fact that $2^{n\frac{\alpha_2}{d}q}\sim_{\, d,\kappa_{\calP_0},\mu(\calP_0)}|\supp\phi_\nu^{(n)}|^{-\frac{q}{\frac{1}{\alpha_1}+\frac{d}{\alpha_2}}}$ for $\nu\in \calL(\calP_n)$, which is a consequence of the last result in \cref{rem:Generalized_Besov_Spaces}\ref{rem:Generalized_Besov_Spaces - item 4} together with $2^{-n\left(1+\frac{s_2}{s_1 d}\right)}\sim_{\,d,p, r_1, r_2, \kappa_{\calP_0}, \mu(\calP_0)}\|\phi_\nu^{(n)}\|_{L_p(\Omega_T)}^p$ due to \cref{rem:estimate_primal_elements} and \cref{rem:Generalized_Besov_Spaces}\ref{rem:Generalized_Besov_Spaces - item 4}.

    Now we will estimate $\|f\|^E_{\widehat{B}^{\alpha_1,\alpha_2}_{p,q} (\Omega_T)} $, $\|f\|^\pi_{\widehat{B}^{\alpha_1,\alpha_2}_{p,q} (\Omega_T)} $, and $\|f\|^\Delta_{\widehat{B}^{\alpha_1,\alpha_2}_{p,q} (\Omega_T)} $ against each other, where we proceed similar to the proof of \cite[Thm.~3.3.3]{Lei00}. First, $\|f\|^E_{\widehat{B}^{\alpha_1,\alpha_2}_{p,q} (\Omega_T)}\le \|f\|^\pi_{\widehat{B}^{\alpha_1,\alpha_2}_{p,q} (\Omega_T)}$ is a direct consequence of \mbox{$\pi_n(f)\in \V^{r_1, r_2}_{\calP_n}$}. Second, $\|f\|^\pi_{\widehat{B}^{\alpha_1,\alpha_2}_{p,q} (\Omega_T)} \lesssim_{\,d,\rho, p, s_1, s_2, \alpha_2, r_1, r_2, \kappa_{\calP_0}, \mu(\calP_0)}\|f\|^E_{\widehat{B}^{\alpha_1,\alpha_2}_{p,q} (\Omega_T)}$, due to \cref{thm:operater_pi_quasi-best_approx}. Furthermore, following the calculations in \cref{lem:Multiscale_decomposition}, in particular, \eqref{lem:Multiscale_decomposition - proof eq 1}, we obtain the estimate \mbox{$\|f\|^\Delta_{\widehat{B}^{\alpha_1,\alpha_2}_{p,q} (\Omega_T)} \lesssim_{\,d,p,q,\alpha_2}\|f\|^\pi_{\widehat{B}^{\alpha_1,\alpha_2}_{p,q} (\Omega_T)}$}. Lastly, for $n\in \N_0\cup\{-1\}$, we obtain
    \begin{align*}
        \|f-\pi_n(f)\|_{L_p(\Omega_T)}^{\min(p,1)}\le \sum\limits_{j=n+1}^{\infty}\|\Delta_j(f)\|_{L_p(\Omega_T)}^{\min(p,1)},
    \end{align*}
    using the multiscale decomposition of $f$ and the subadditivity of $\|\cdot\|_{L_p(\Omega_T)}^{\min(p,1)}$. Now, $\|f\|^\pi_{\widehat{B}^{\alpha_1,\alpha_2}_{p,q} (\Omega_T)} \lesssim_{\,d, p, q, \alpha_2}\|f\|^\Delta_{\widehat{B}^{\alpha_1,\alpha_2}_{p,q} (\Omega_T)}$, recall $\alpha_i> 0$, $i=1,2$, is the consequence of an application of the discrete Hardy inequality (cf. \cite[Lem.~3.4 of Ch.~2.3]{DL93}). The quasi-normedness of $\widehat{B}^{\alpha_1,\alpha_2}_{p,q}(\Omega_T)$ with respect to the corresponding norms now directly follows from the proven equivalences and the fact that $\|\cdot\|^E_{\widehat{B}^{\alpha_1,\alpha_2}_{p,q}(\Omega_T)}$ is indeed a quasi-norm, due to $E\left(f+g, \V^{r_1, r_2}_{\calP_n}, \Omega_T\right)_p\lesssim_{\,p}E\left(f, \V^{r_1, r_2}_{\calP_n}, \Omega_T\right)_p+E\left(g, \V^{r_1, r_2}_{\calP_n}, \Omega_T\right)_p$ for any $n\in \N_0$ and $f,g \in \widehat{B}^{\alpha_1,\alpha_2}_{p,q}(\Omega_T)$. The subadditivity assertion for $\left(\|\cdot\|^E_{\widehat{B}^{\alpha_1,\alpha_2}_{p,p}(\Omega_T)}\right)^{\min(1,p)}$ lastly follows due to the very similar result
    \begin{align*}
        E\left(f+g, \V^{r_1, r_2}_{\calP_n}, \Omega_T\right)_p^{\min(1,p)}\le E\left(f, \V^{r_1, r_2}_{\calP_n}, \Omega_T\right)_p^{\min(1,p)}+E\left(g, \V^{r_1, r_2}_{\calP_n}, \Omega_T\right)_p^{\min(1,p)}.
    \end{align*}
    This concludes the proof.
\end{proof}

Furthermore, we want to show that an embedding from the generalized anisotropic Besov spaces into the classical ones is also possible for small smoothness parameters. This can be proven by making good use of an appropriate multiscale decomposition of Besov functions, following the approach from \cite[Thm.~4.8]{DP88}.  

\begin{thm}\label{thm:Embedding_generalized_into_classical_aniso_spaces}
    Let \cref{assumptions:multiscale-decomposition} hold, $0<\max(\alpha_1, \alpha_2)<1+\frac{1}{p}$, and $r_i>\alpha_i$, $i=1,2$. Then the embedding $\widehat{B}^{\alpha_1,\alpha_2}_{p,q} (\Omega_T)\hookrightarrow B^{\alpha_1,\alpha_2}_{p,q} (\Omega_T)$ holds true continuously with an embedding constant only depending on $d,p,q,\alpha_1, \alpha_2, r_1, r_2, \kappa_{\calP_0},\mu(\calP_0)$, and $\LipProp(\Omega)$.
\end{thm}
\begin{proof}
    Let us initially fix the notation $p^*:=\min(p,1)$ and choose $f\in \widehat{B}_{p,q}^{\alpha_1, \alpha_2}(\Omega_T)$. Further, for $n\in \N_0$, let $f_n, F_n\in \V^{r_1, r_2}_{\calP_n}$ such that 
    \begin{align*}
        \|f-F_n\|_{L_p(\Omega_T)} = E(f, \V^{r_1, r_2}_{\calP_n}, \Omega_T)_p\quad\text{and}\quad f_n:=F_{n}-F_{n-1},
    \end{align*}
    where $F_{-1}:=0$. We expand $f_n = \sum\limits_{\nu\in \calL(\calP_n)} c_\nu^{(n)}\phi_\nu^{(n)}$ with respect to the basis $\left(\phi_\nu^{(n)}\right)_{\nu\in \calL(\calP_n)}$ of $\V^{r_1, r_2}_{\calP_n}$. Now, choose $h\in \R$ and $\bm{h}\in \R^d$ with 
    \begin{align*}
        |h|\le r_1^{-1}\,2^{-\frac{n\alpha_2}{\alpha_1d}}\left(\min\limits_{I\in \calI_0} |I|\right)\lesssim_{\, r_1, \mu(\calP_0)} 2^{-\frac{n\alpha_2}{\alpha_1d}} \quad \text{and}\quad |\bm{h}|\le r_2^{-1}\,2^{-\frac{n}{d}}\left(\min\limits_{S\in \calT_0} |S|\right)\lesssim_{\,d, r_2, \kappa_{\calP_0}, \mu(\calP_0)}2^{-\frac{n}{d}},
    \end{align*}
    respectively. Now we estimate
    \begin{align}\label{thm:Embedding_generalized_into_classical_aniso_spaces - proof equation 1}
        \|\Delta^{r_1}_{h,t}f\|_{L_p([0,T]_{r_1,h}\times \Omega)}^{p^*}&\le \|\Delta^{r_1}_{h,t}(f-F_n)\|_{L_p([0,T]\times \Omega))}^{p^*} + \sum\limits_{k=0}^n \|\Delta^{r_1}_{h,t}f_k\|_{L_p([0,T]_{r_1,h}\times \Omega)}^{p^*}\notag
        \\&\lesssim_{\,r_1} \|f-F_n\|_{L_p([0,T]\times \Omega))}^{p^*} + \sum\limits_{k=0}^n \|\Delta^{r_1}_{h,t}f_k\|_{L_p([0,T]_{r_1,h}\times \Omega)}^{p^*}\notag
        \\&= E(f, \V^{r_1, r_2}_{\calP_n}, \Omega_T)_p^{p^*} + \sum\limits_{k=0}^n \|\Delta^{r_1}_{h,t}f_k\|_{L_p([0,T]_{r_1,h}\times \Omega)}^{p^*}, \quad n\in \N_0,
    \end{align}
    where, in the first step, we have used the identity $F_n=\sum\limits_{k=0}^{n}f_k$ and the subadditivity of $\|\cdot\|_{L_p([0,T]_{r_1,h}\times \Omega)}^{p^*}$, in the second step the estimate from the penultimate line of \cite[Rem.~2.2]{MSS26}, which is the analogue of \cref{Rem:Moduli of smoothness}\ref{Rem:Moduli of smoothness - higher order lesser order} for differences, and in the last step the definition of $F_n$. With the subadditivity, we also obtain
    \begin{align*}
        |\Delta_{h,t}^{r_1}f_k(t,\bm{x})|^{p^*} \le \sum\limits_{\nu\in \calL(\calP_k)} |c_\nu^{(k)}|^{p^*}|\Delta_{h,t}^{r_1}\phi_\nu^{(k)}(t,\bm{x})|^{p^*}\lesssim_{\, d,p,r_1, r_2, \kappa_{\calP_0}} \left(\sum\limits_{\nu\in \calL(\calP_k)} |c_\nu^{(k)}|^{p}|\Delta_{h,t}^{r_1}\phi_\nu^{(k)}(t,\bm{x})|^{p}\right)^\frac{p^*}{p}
    \end{align*}
    for every $k\in \{0,\dots,n\}$ and $(t,\bm{x})\in [0,T]_{r_1,h}\times \Omega$, by using the expansion of $f_k$ in the basis of $\V^{r_1, r_2}_{\calP_n}$. Additionally, in the second step, the (quasi-)norm equivalence in finite dimensional spaces has been used together with the fact that the basis functions have finite overlap depending only on $d$ and $\kappa_{\calP_0}$. Furthermore, we can estimate 
    \begin{align}\label{thm:Embedding_generalized_into_classical_aniso_spaces - proof equation 1.5}
        |\Delta_{h,t}^{r_1}f_k(t,\bm{x})|^{p} \lesssim_{\, d,p,\alpha_1, \alpha_2,r_1,r_2, \kappa_{\calP_0}, \mu(\calP_0)} |h|^p\sum\limits_{\nu\in \calL(\calP_k)} |c_\nu^{(k)}|^{p}\,\diam_t\left(\supp \phi^{(k)}_\nu\right)^{-p}\cdot \mathds{1}_{\supp \left(\Delta^{r_1}_{h,t}\phi_\nu^{(k)}\right)}(t,\bm{x})
    \end{align}
    using the temporal part of \eqref{lem:FEM_elements_of_Aniso_Besov - Proof Eq 1} as well as the fact that we have chosen $h$ to fulfill $|h|\le \diam_t\left(\supp\phi_\nu^{(n)}\right)\le \diam_t\left(\supp\phi_\nu^{(k)}\right)$. Employing the estimate for $\diam_t\left(\supp \phi^{(k)}_\nu\right)$ from \cref{rem:Generalized_Besov_Spaces}\ref{rem:Generalized_Besov_Spaces - item 5}, we now integrate over $[0,T]_{r_1, h}\times \Omega$ with respect to $(t,\bm{x})$ in order to derive
    \begin{align*}
        \|\Delta_{h,t}^{r_1}f_k\|^{p}_{L_p([0,T]_{r_1, h}\times \Omega)}&\lesssim_{\, d,p,\alpha_1, \alpha_2,r_1, r_2, \kappa_{\calP_0}, \mu(\calP_0)} |h|^p\,2^{\frac{k\alpha_2p}{\alpha_1d}}\sum\limits_{\nu\in \calL(\calP_k)}|c_\nu^{(k)}|^{p}\, \left|\supp\left(\Delta_{h,t}^{r_1}\phi_\nu^{(k)}\right)\right|
        \\&\lesssim_{\,d,p,r_1, \kappa_{\calP_0},\mu(\calP_0)} |h|^{1+p}\,2^{\frac{k\alpha_2p}{\alpha_1d}}\sum\limits_{\nu\in \calL(\calP_k)}|c_\nu^{(k)}|^{p}\,\diam_{\bm{x}}\left(\supp \phi_\nu^{(k)}\right)^d
        \\&\lesssim_{\,d, \kappa_{\calP_0},\mu(\calP_0)} |h|^{1+p}\,2^{\frac{k\alpha_2}{\alpha_1d}(1+p)}\sum\limits_{\nu\in \calL(\calP_k)}|c_\nu^{(k)}|^{p}\,2^{-k\left(1+\frac{\alpha_2}{\alpha_1d}\right)}.
    \end{align*}
    Here, we have additionally used \eqref{lem:FEM_elements_of_Aniso_Besov - Proof Eq 3} in the second step as well as the estimate for $\diam_{\bm{x}}\left(\supp \phi^{(k)}_\nu\right)$ from \cref{rem:Generalized_Besov_Spaces}\ref{rem:Generalized_Besov_Spaces - item 5} in the last one. Using, the result from \cref{lem:Second_lem_for_gen=classical Besov spaces}, this becomes
    \begin{align}\label{thm:Embedding_generalized_into_classical_aniso_spaces - proof equation 2}
        \|\Delta_{h,t}^{r_1}f_k\|^{p}_{L_p([0,T]_{r_1, h}\times \Omega)}&\lesssim_{\,d,p,\alpha_1, \alpha_2, r_1, r_2, \kappa_{\calP_0},\mu(\calP_0)} |h|^{1+p}\,2^{\frac{k\alpha_2}{\alpha_1d}(1+p)} \|f_k\|_{L_p(\Omega_T)}^p\notag
        \\&\lesssim_{\,p} |h|^{1+p}\,2^{\frac{k\alpha_2}{\alpha_1d}(1+p)}\left(E(f, \V^{r_1, r_2}_{\calP_k}, \Omega_T)_p^p + E(f, \V^{r_1, r_2}_{\calP_{k-1}}, \Omega_T)_p^p\right),
    \end{align}
    where we use $E(f, \V^{r_1, r_2}_{\calP_{-1}}, \Omega_T)_p:=\|f\|_{L_p(\Omega_T)}$. Inserting \eqref{thm:Embedding_generalized_into_classical_aniso_spaces - proof equation 2} into \eqref{thm:Embedding_generalized_into_classical_aniso_spaces - proof equation 1}, taking the supremum over all considered $h$ and applying the estimate from \cref{Rem:Moduli of smoothness}\ref{Rem:Moduli of smoothness - scaling item}, yields
    \begin{align*}
        \omega_{r_1, t}\left(f,\Omega_T, 2^{-\frac{n\alpha_2}{\alpha_1d}}\right)_p^{p^*}\lesssim_{\,d,p,\alpha_1, \alpha_2, r_1, r_2, \kappa_{\calP_0},\mu(\calP_0)} 2^{-\frac{n\alpha_2}{\alpha_1d}\left(1+\frac1p\right)p^*}\sum\limits_{k=-1}^{n}2^{\frac{k\alpha_2}{\alpha_1d}\left(1+\frac1p\right)p^*}E(f, \V^{r_1, r_2}_{\calP_k}, \Omega_T)_p^{p^*}.
    \end{align*}
    Since $\alpha_1<1+\frac1p$, $\frac{\alpha_2}{d}<\frac{\alpha_2}{\alpha_1d}\left(1+\frac1p\right)$, the discrete Hardy inequality yields 
    \begin{align}\label{thm:Embedding_generalized_into_classical_aniso_spaces - proof equation 3}
        \left(\sum\limits_{n=0}^\infty 2^{n\frac{\alpha_2}{d}q}\omega_{r_1, t}\left(f,\Omega_T, 2^{-\frac{n\alpha_2}{\alpha_1d}}\right)_p^q\right)^\frac1q\lesssim_{\,d,p,q,\alpha_1, \alpha_2, r_1, r_2, \kappa_{\calP_0},\mu(\calP_0)} \|f\|^E_{\widehat{B}_{\alpha_1, \alpha_2}^{p,q}(\Omega_T)}.
    \end{align}
    Here, we have used the version of the inequality, which one can derive from \cite[Lem.~3.4 of Ch.~2.3]{DL93}, considering what is pointed out in the corresponding remark, in particular, Eq. (3.13). The spatial part works similar.\footnote{In particular, in the spatial analogue to \eqref{thm:Embedding_generalized_into_classical_aniso_spaces - proof equation 1}, it is necessary to estimate $\|\Delta^{r_2}_{\bm{h}, \bm{x}}(f-F_n)\|_{L_p([0,T]\times \Omega))}^{p^*}\lesssim_{\,r_2}\|f-F_n\|_{L_p([0,T]\times \Omega))}^{p^*}$, which can be done, correspondingly, by the last estimate from \cite[Rem.~2.2]{MSS26}.} We obtain 
    \begin{align*}
        |\Delta_{\bm{h},\bm{x}}^{r_2}f_k(t,\bm{x})|^{p} \lesssim_{\, d,p,\alpha_1, \alpha_2,r_1,r_2, \kappa_{\calP_0}, \mu(\calP_0)} |\bm{h}|^p\sum\limits_{\nu\in \calL(\calP_k)} |c_\nu^{(k)}|^{p}\,\diam_{\bm{x}}\left(\supp \phi^{(k)}_\nu\right)^{-p}\cdot \mathds{1}_{\supp \left(\Delta^{r_2}_{\bm{h},\bm{x}}\phi_\nu^{(k)}\right)}(t,\bm{x})
    \end{align*}
    on $[0,T]\times \Omega_{r_2, h}$, correspondingly to the derivation of \eqref{thm:Embedding_generalized_into_classical_aniso_spaces - proof equation 1.5}. This then leads to
    \begin{align*}
        &\|\Delta_{\bm{h},\bm{x}}^{r_2}f_k\|^{p}_{L_p([0,T]\times \Omega_{r_2, h})}
        \lesssim_{\, \substack{d,p,\alpha_1, \alpha_2,r_1,r_2, \\\kappa_{\calP_0}, \mu(\calP_0)}} |\bm{h}|^{1+p} \sum\limits_{\nu\in \calL(\calP_k)} |c_\nu^{(k)}|^p \,\diam_t\left(\supp \phi^{(k)}_\nu\right) \diam_{\bm{x}}\left(\supp \phi^{(k)}_\nu\right)^{d-1-p}
        \\&\lesssim_{\,d,p,\kappa_{\calP_0}, \mu(\calP_0)}|\bm{h}|^{1+p} \,2^{\frac{k}{d}(1+p)} \sum\limits_{\nu\in \calL(\calP_k)} |c_\nu^{(k)}|^p  \,2^{-k\left(1+\frac{\alpha_2}{\alpha_1d}\right)}
    \end{align*}
    using \eqref{lem:FEM_elements_of_Aniso_Besov - Proof Eq 4} and $|\bm{h}|\le \diam_{\bm{x}}\left(\supp \phi_\nu^{(k)}\right)$ in the first step as well as \cref{rem:Generalized_Besov_Spaces}\ref{rem:Generalized_Besov_Spaces - item 5} in the second one. Further, combining the temporal analogues of \eqref{thm:Embedding_generalized_into_classical_aniso_spaces - proof equation 1} and \eqref{thm:Embedding_generalized_into_classical_aniso_spaces - proof equation 2} as well as going over to the supremum and using the scaling property of the moduli of smoothness, yields 
    \begin{align*}
        \omega_{r_2, \bm{x}}\left(f,\Omega_T, 2^{-\frac{n}{d}}\right)_p^{p^*}\lesssim_{\,d,p,\alpha_1, \alpha_2, r_1, r_2, \kappa_{\calP_0},\mu(\calP_0)} 2^{-\frac{n}{d}\left(1+\frac1p\right)p^*}\sum\limits_{k=-1}^{n}2^{\frac{k}{d}\left(1+\frac1p\right)p^*}E(f, \V^{r_1, r_2}_{\calP_k}, \Omega_T)_p^{p^*},
    \end{align*}
    which implies 
    \begin{align}\label{thm:Embedding_generalized_into_classical_aniso_spaces - proof equation 4}
        \left(\sum\limits_{n=0}^\infty 2^{n\frac{\alpha_2}{d}q}\omega_{r_2, \bm{x}}\left(f,\Omega_T, 2^{-\frac{n}{d}}\right)_p^q\right)^\frac1q\lesssim_{\,d,p,q,\alpha_1, \alpha_2, r_1, r_2, \kappa_{\calP_0},\mu(\calP_0)} \|f\|^E_{\widehat{B}^{\alpha_1, \alpha_2}_{p,q}(\Omega_T)}
    \end{align}
    by Hardy's inequality, where we have exploited that $\frac{\alpha_2}{d}<\frac{1}{d}\left(1+\frac1p\right)$, since $\alpha_2<1+\frac1p$. Adding \eqref{thm:Embedding_generalized_into_classical_aniso_spaces - proof equation 3}, \eqref{thm:Embedding_generalized_into_classical_aniso_spaces - proof equation 4} and $\|f\|_{L_p(\Omega_T)}$ now yields $\|f\|^{*,(r_1,r_2)}_{B^{\alpha_1, \alpha_2}_{p,q}(\Omega_T)} \lesssim_{\,d,p,q,\alpha_1, \alpha_2, r_1, r_2, \kappa_{\calP_0},\mu(\calP_0)} \|f\|^E_{\widehat{B}^{\alpha_1, \alpha_2}_{p,q}(\Omega_T)}$. With the (quasi-)norm equivalence result from the end of \cref{subsubsect:Besov spaces} below \eqref{Eq:equivalence quasi-seminorms}, the assertion finally follows, since $r_i>\alpha_i$, $i=1,2$. In particular, this last step is where the Lipschitz properties $\LipProp(\Omega)$ of $\Omega$ in the assertion come from.
\end{proof}

\begin{cor}\label{cor:to_thm:Embedding_generalized_into_classical_aniso_spaces}
    Under \cref{assumptions:multiscale-decomposition}, $B^{\alpha_1, \alpha_2}_{p,q}(\Omega_T)= \widehat{B}^{\alpha_1, \alpha_2}_{p,q}(\Omega_T)$ in the sense of equivalent (quasi-)norms for $r_i>\alpha_i>0$, $i=1,2$, if $\alpha_i< 1+\frac1p$, $i=1,2$, where the embedding constant only depends on $d,\rho, p, q, \alpha_1, \alpha_2, r_1, r_2$, $\kappa_{\calP_0},\mu(\calP_0)$, and $\LipProp(\Omega)$. In particular, this implies that all the spaces $\widehat{B}^{\alpha_1, \alpha_2}_{p,q}(\Omega_T)$ coincide for different choices of the polynomial degrees as long as $r_i>\alpha_i$, $i=1,2$, if $\max(\alpha_1,\alpha_2)<1+\frac1p$. In contrast, $B^{\alpha_1, \alpha_2}_{p,q}(\Omega_T)$ and $\widehat{B}^{\alpha_1, \alpha_2}_{p,q}(\Omega_T)$ do not coincide if $\max(\alpha_1, \alpha_2)\ge1+\frac1p$ and $q<\infty$ or $\max(\alpha_1, \alpha_2)>1+\frac1p$ and $q=\infty$.
\end{cor}
\begin{proof}
    That the result holds true for $\max(\alpha_1, \alpha_2)<1+\frac1p$ is a consequence of \cref{lem:Multiscale_decomposition} together with \cref{thm:Embedding_generalized_into_classical_aniso_spaces}. In turn, for any space-time partition $\calP$ created from $\calP_0$ by finitely many applications of $\textup{PATCH}\_\textup{REFINE}(\cdot, \cdot, d, s_1, s_2)$, $\V^{r_1, r_2}_\calP\not\subseteq B^{\alpha_1, \alpha_2}_{p,q}(\Omega_T)$ and $\V^{r_1, r_2}_\calP\subset \widehat{B}^{\alpha_1, \alpha_2}_{p,q}(\Omega_T)$ due to \cref{lem:FEM_elements_of_Aniso_Besov} and \cref{rem:Generalized_Besov_Spaces}\ref{rem:Generalized_Besov_Spaces - item 1}, correspondingly,  yielding it for $\max(\alpha_1, \alpha_2)\ge1+\frac1p$ and $q<\infty$ as well as $\max(\alpha_1, \alpha_2)>1+\frac1p$ and $q=\infty$.
\end{proof}

The above results, that can be interpreted as anisotropic analogues to \cite[Thm.~3.5]{Gan17}, allow us to show a Sobolev-type embedding between anisotropic Besov spaces of different smoothness, which can be seen as an extension of \cite[Thm.~1.3]{MSS26}.

\begin{thm}\label{thm:Sobolev-type-embedding}
    Let \cref{assumptions:multiscale-decomposition} hold, $p,q\in (0,\infty]$, $s_1, s_2\in \R^+$ with \mbox{$\frac{1}{\frac{1}{s_1}+\frac{d}{s_2}}-\frac1q+\frac1p>0 $}, and $\alpha_i<1+\frac{1}{p}$, $i=1,2$. Then, the embedding $B^{\alpha_1+s_1, \alpha_2 + s_2}_{q,q}(\Omega_T)\hookrightarrow B^{\alpha_1, \alpha_2}_{p,p}(\Omega_T)$ is continuous with a constant depending only on $d,p,q,s_1, s_2,\alpha_1, \alpha_2, \kappa_{\calP_0}, \mu(\calP_0),a(\calP_0), |\Omega_T|$, and $\LipProp(\Omega)$.
\end{thm}
\begin{rem}
    We point out that the restriction on $\alpha_1$ and $\alpha_2$ is purely due to the fact that the anisotropic Besov regularity of our approximants does not become larger for increasing polynomial degrees. If smoother approximants are available, where the latter is the case  for example for anisotropic splines on simpler domains  like rectangular ones, it should be possible to show the above theorem without the restriction on the smoothness parameters $\alpha_1$ and $\alpha_2$.
\end{rem}
\begin{proof}
    First, we will consider the case $q\ge p$. Then, the embeddings $B^{\alpha_1+s_1, \alpha_2 + s_2}_{q,q}(\Omega_T) \hookrightarrow B^{\alpha_1, \alpha_2}_{q,p}(\Omega_T)\hookrightarrow B^{\alpha_1, \alpha_2}_{p,p}(\Omega_T)$ are continuous due to \cite[Rem.~3.2(ii) and (iii)]{MSS26}.\footnote{The second embedding is, where the dependence on $|\Omega_T|$ in the assertion comes from.} Now we can consider the opposite case, i.e., $q<p$. We will only present the proof for $p<\infty$, the other case works similarly with the usual changes. Now let $f\in B^{\alpha_1+s_1, \alpha_2 + s_2}_{q,q}(\Omega_T)$ and put $r_i:=\max(\lfloor \alpha_i + s_i \rfloor+1,2) $, $i=1,2$. Now setting $\rho:=q$  and employing \mbox{\cref{cor:to_thm:Embedding_generalized_into_classical_aniso_spaces}} yields $\|f\|_{B^{\alpha_1, \alpha_2}_{p,p}(\Omega_T)}\lesssim_{\,d,p,s_1, s_2,\alpha_1, \alpha_2, \kappa_{\calP_0}, \mu(\calP_0), \LipProp(\Omega)}\|f\|_{\widehat{B}^{\alpha_1, \alpha_2}_{p,p}(\Omega_T)}^\pi$, so it will suffice to bound $\|f\|_{L_p(\Omega_T)}$ and $|f|_{\widehat{B}^{\alpha_1, \alpha_2}_{p,p}(\Omega_T)}^\pi$. Since $\frac{1}{\frac{1}{s_1}+\frac{d}{s_2}}-\frac1q+\frac1p>0 $, also $\frac{1}{\frac{1}{\alpha_1+s_1}+\frac{d}{\alpha_2+s_2}}-\frac1q+\frac1p>0 $, and therefore $B^{\alpha_1+s_1, \alpha_2 + s_2}_{q,q}(\Omega_T)\hookrightarrow L_{p}(\Omega_T)$, i.e., $\|f\|_{L_p(\Omega_T)}\lesssim_{\,d,p,q, s_1, s_2,\alpha_1, \alpha_2, \kappa_{\calP_0}, \mu(\calP_0),a(\calP_0)}\|f\|_{B^{\alpha_1+s_1, \alpha_2 + s_2}_{q,q}(\Omega_T)}$ according to \cite[Thm.~1.3]{MSS26}. Furthermore,
    \begingroup
    \allowdisplaybreaks
    \begin{align}\label{thm:Sobolev-type-embedding - proof equation}
        &|f|_{\widehat{B}^{\alpha_1, \alpha_2}_{p,p}(\Omega_T)}^\pi = \left(\sum\limits_{n=0}^{\infty} 2^{n\frac{\alpha_2}{d}p} \sum\limits_{I\times S\in \calP_n} \|f-\pi_n(f)\|^p_{L_p(I\times S)}\right)^\frac1p\notag
        \\&\lesssim_{\substack{\, d,p,q,s_1, s_2,\alpha_1, \alpha_2,\\ \kappa_{\calP_0},\mu(\calP_0), a(\calP_0)}} \left(\sum\limits_{n=0}^{\infty} 2^{n\frac{\alpha_2}{d}p} \sum\limits_{I\times S\in \calP_n} |I\times S|^{\left(\frac{1}{\frac{1}{\alpha_1+s_1}+\frac{d}{\alpha_2+s_2}}-\frac1q+\frac1p\right)p } |f|^p_{B^{\alpha_1+s_1, \alpha_2 + s_2}_{q,q}\left(\omega_{\calP_n}^{j(d)}(I\times S)\right)}\right)^\frac1p\notag
        \\&\le \left(\sum\limits_{n=0}^{\infty} 2^{n\frac{\alpha_2}{d}p} \left(\sum\limits_{I\times S\in \calP_n} |I\times S|^{\left(\frac{1}{\frac{1}{\alpha_1+s_1}+\frac{d}{\alpha_2+s_2}}-\frac1q+\frac1p\right)q} |f|^q_{B^{\alpha_1+s_1, \alpha_2 + s_2}_{q,q}\left(\omega_{\calP_n}^{j(d)}(I\times S)\right)}\right)^\frac{p}{q}\right)^\frac1p\notag
        \\&\lesssim_{\, \substack{d, q,s_1, s_2,\alpha_1, \alpha_2, \kappa_{\calP_0},\\\mu(\calP_0),\LipProp(\Omega)}} \left(\sum\limits_{n=0}^{\infty} 2^{n\frac{\alpha_2}{d}p}\,2^{-n\left(1+\frac{\alpha_2}{\alpha_1d}\right)\left(\frac{1}{\frac{1}{\alpha_1+s_1}+\frac{d}{\alpha_2+s_2}}-\frac1q+\frac1p\right)p} \right)^\frac1p |f|_{B^{\alpha_1+s_1, \alpha_2+s_2}_{q,q}(\Omega_T)}\notag
        \\&= \left(\sum\limits_{n=0}^{\infty} \left(2^{-\left(1+\frac{s_2}{s_1d}\right)\left(\frac{1}{\frac{1}{s_1}+\frac{d}{s_2}}-\frac{1}{q}+\frac{1}{p}\right)p}\right)^n \right)^\frac1p |f|_{B^{\alpha_1+s_1, \alpha_2+s_2}_{q,q}(\Omega_T)}\sim_{\,d,p,q, s_1, s_2} |f|_{B^{\alpha_1+s_1, \alpha_2+s_2}_{q,q}(\Omega_T)}
    \end{align}
    \endgroup
    by \cref{thm:local_boundedness_of_pi_projection} and \cref{lem:Extended_Whitney_Estimate} in the second step\footnote{Note that $a(\calP_0)=a(\calP_0, \alpha_1, \alpha_2)=a(\calP_0, \alpha_1+s_1, \alpha_2+s_2)$ due to the collinearity of $(\alpha_1, \alpha_2)$ and $(s_1, s_2)$.}, the embedding $\ell^q(\N_0)\hookrightarrow \ell^p(\N_0)$, due to $q<p$,
    in the third one as well as \cref{thm:Proofs_omega_tilde_domain}\ref{thm:Proofs_omega_tilde_domain - summing of Besvo norms}\ref{thm:Proofs_omega_tilde_domain - summing of Besvo norms - uniform part} in the fourth. The penultimate step is the result of a lengthy algebraic calculation exploiting the linear dependence of $(\alpha_1, \alpha_2)$ and $(s_1, s_2)$, provided in \cref{Appendix:Tedious_Calculation}, whereas the last one is due to the fact that the geometric sum converges, since $\frac{1}{\frac{1}{s_1}+\frac{d}{s_2}}-\frac{1}{q}+\frac{1}{p}>0$. Thus, the assertion is shown.
\end{proof}

We conclude this section with an analogue of the above embedding for the reference domain $[0,1]\times R^{ref}$, where we make use of the notation from \cref{rem:Neighborhoods_generally_not_Lipschitz}\ref{rem:Neighborhoods_generally_not_Lipschitz - reference domains} and \ref{rem:Neighborhoods_generally_not_Lipschitz - enlargement}.\footnote{\cref{thm:Sobolev-type-embedding} is not generally applicable for $[0,1]\times R^{ref}$, since $R^{ref}$ might not be Lipschitz (cf. \cref{rem:Neighborhoods_generally_not_Lipschitz}\ref{rem:Neighborhoods_generally_not_Lipschitz - counterexample}).}

\begin{cor}\label{cor_to:thm:Sobolev-type-embedding}
    Let the assumptions of \cref{thm:Sobolev-type-embedding} hold true and let $r_1, r_2\in \N_{\ge 2}$ with $r_i > \alpha_i +s_i$, $i=1,2$. Then, the embedding 
    \begin{align*}
        \left(B^{\alpha_1+s_1, \alpha_2 + s_2}_{q,q}(\Enl([0,1]\times R^{ref})),\|\cdot\|_{B^{\alpha_1+s_1, \alpha_2 + s_2}_{q,q}(\Enl([0,1]\times R^{ref}))}^{(r_1, r_2)}\right)\hookrightarrow \left(B^{\alpha_1, \alpha_2}_{p,p}([0,1]\times R^{ref}),\|\cdot\|_{B^{\alpha_1, \alpha_2}_{p,p}([0,1]\times R^{ref})}^{(r_1, r_2)}\right)
    \end{align*}
    is continuous. The corresponding constant only depends on $d,p,q,s_1, s_2,\alpha_1, \alpha_2, r_1, r_2, \kappa_{\calP_0}, \mu(\calP_0),a(\calP_0)$, and $\LipProp(\Omega)$.
\end{cor}
\begin{proof}
    In case of $q\ge p$, the chain of embeddings from the beginning of the proof of \cref{thm:Sobolev-type-embedding} with $[0,1]\times R^{ref}$ instead of $\Omega_T$ might not hold, since $R^{ref}$ could be non-Lipschitz. Nevertheless, we can embed for fixed order of the moduli of smoothness, i.e.,
    \begin{align*}
       &\left(B^{\alpha_1+s_1, \alpha_2 + s_2}_{q,q}(\Enl([0,1]\times R^{ref})),\|\cdot\|_{B^{\alpha_1+s_1, \alpha_2 + s_2}_{q,q}(\Enl([0,1]\times R^{ref}))}^{(r_1, r_2)}\right) 
       \\\hookrightarrow&\left(B^{\alpha_1+s_1, \alpha_2 + s_2}_{q,q}([0,1]\times R^{ref}),\|\cdot\|_{B^{\alpha_1+s_1, \alpha_2 + s_2}_{q,q}([0,1]\times R^{ref})}^{(r_1,r_2)}\right) 
       \\\hookrightarrow &\left(B^{\alpha_1, \alpha_2}_{q,p}([0,1]\times R^{ref}), \|\cdot\|_{B^{\alpha_1, \alpha_2}_{q,p}([0,1]\times R^{ref})}^{(r_1,r_2)}\right)\hookrightarrow \left(B^{\alpha_1, \alpha_2}_{p,p}([0,1]\times R^{ref}), \|\cdot\|_{B^{\alpha_1, \alpha_2}_{p,p}([0,1]\times R^{ref})}^{(r_1, r_2)}\right),
    \end{align*}
    which can be shown with minor changes in the classical proof technique behind \cite[Rem.~3.2(ii) and (iii)]{MSS26}. Now we can turn to the case of $q<p$. 
    First, we notice that \cref{def:Generalized_Besov_Spaces} does not itself require $\Omega$ to be Lipschitz, so the generalized Besov spaces can also be defined with respect to $[0,1]\times R^{ref}$ and the corresponding (natural) initial partition from \cref{rem:Neighborhoods_generally_not_Lipschitz}\ref{rem:Neighborhoods_generally_not_Lipschitz - reference domains}. Second, $R^{ref}$ is not required to be Lipschitz for \cref{lem:norm equivalence_gen_besov} to hold in a corresponding sense. Third, the Lipschitz properties $\LipProp(\Omega)$ of $\Omega$ are used only in the final step of the proof of \cref{thm:Embedding_generalized_into_classical_aniso_spaces}. Besides this step, everything is valid if we replace $\Omega_T$ and its initial space-time partition $\calP_0$ by $[0,1]\times R^{ref}$ and $\tilde{\calP}_{[0,1]\times R^{ref}}$, respectively. This yields the intermediate result
    \begin{align*}
        \|f\|_{B^{\alpha_1, \alpha_2}_{p,p}([0,1]\times R^{ref})}^{(r_1, r_2)}\lesssim_{\,d,p,s_1, s_2,\alpha_1, \alpha_2,r_1, r_2,\kappa_{\calP_0},\mu(\calP_0)}\|f\|_{\widehat{B}^{\alpha_1, \alpha_2}_{p,p}([0,1]\times R^{ref})}^\pi.
    \end{align*}
    Therefore, it now suffices to bound $\|f\|_{L_p([0,1]\times R^{ref})}$ and $|f|^\pi_{B_{p,p}^{\alpha_1+s_1, \alpha_2+s_2}([0,1]\times R^{ref})}$ as in the proof of the above theorem. The former part is possible due to \cref{thm:Proofs_omega_tilde_domain}\ref{thm:Proofs_omega_tilde_domain - Embeddings between Besov and Lebesgue scale} and the latter works as in \eqref{thm:Sobolev-type-embedding - proof equation}, using \cref{thm:Proofs_omega_tilde_domain}\ref{thm:Proofs_omega_tilde_domain - summing of Besvo norms}\ref{thm:Proofs_omega_tilde_domain - summing of Besvo norms - R^ref part} and $|f|^{(r_1, r_2)}_{B^{\alpha_1+s_1, \alpha_2+s_2}_{q,q}(\Enl([0,1]\times R^{ref}))}$ instead of \cref{thm:Proofs_omega_tilde_domain}\ref{thm:Proofs_omega_tilde_domain - summing of Besvo norms}\ref{thm:Proofs_omega_tilde_domain - summing of Besvo norms - uniform part} and $|f|_{B^{\alpha_1+s_1, \alpha_2+s_2}_{q,q}(\Omega_T)}$, respectively, with an additional dependency on $r_1$ and $r_2$. 
\end{proof}

\subsection{Direct estimates}\label{subsect:direct estimates}

This section is now devoted to the proof of \cref{Main_res:direct}. We will first have to localize the (quasi-)norm in $B^{\alpha_1, \alpha_2}_{p,p}(\Omega_T)$. This corresponds to the spatial result from \cite[Prop.~2.1, Eq.~(3) and (4)]{GM14}.

\begin{lem}\label{lem:Besovnorm behaves a little like L_p}
    We assume \cref{assumptions:multiscale-decomposition} to hold true with $\alpha_i<\min\left(1+\frac1p, r_i\right)$, $i=1,2$, and $\calP$ to be obtained from $\calP_0$ by finitely many applications of $\textup{PATCH}\_\textup{REFINE}(\cdot, \cdot, d, s_1, s_2)$. Then, 
    \begin{align*}
        \|f-\pi_{\rho, \calP}(f)\|_{B_{p,p}^{\alpha_1, \alpha_2}(\Omega_T)}^p\lesssim_{\, d, \rho,p, \alpha_1, \alpha_2, r_1, r_2, \kappa_{\calP_0}, \mu(\calP_0), a(\calP_0), \LipProp(\Omega)} \sum\limits_{I\times S\in \calP} \left(|f|^{*,(r_1,r_2)}_{B_{p,p}^{\alpha_1, \alpha_2}\left(\widetilde{\omega^{j(d)}_\calP}(I\times S)\right)}\right)^p
    \end{align*}
    for $p<\infty$ with the classical modification for $p=\infty$.
\end{lem}

\begin{proof}
    We will only show the case of $p<\infty$ and first employ the result of \cref{thm:Embedding_generalized_into_classical_aniso_spaces},
    \begin{align}\label{lem:Besovnorm behaves a little like L_p - proof 1}
        \|f-\pi_{\rho, \calP}(f)\|_{B_{p,p}^{\alpha_1, \alpha_2}(\Omega_T)}^p \lesssim_{\,\substack{ d,p, \alpha_1, \alpha_2, r_1,\\ r_2, \kappa_{\calP_0}, \mu(\calP_0), \\ \LipProp(\Omega)}} \left(\sum\limits_{n=0}^\infty 2^{n\frac{\alpha_2}{d}p}E(f-\pi_{\rho, \calP}(f), \V^{r_1, r_2}_{\calP_{n}}, \Omega_T)_p^p\right) + \|f-\pi_{\rho, \calP}(f)\|_{L_p(\Omega_T)}^p.
    \end{align}
    We already know from \cref{thm:local_boundedness_of_pi_projection} together with \cref{lem:Extended_Whitney_Estimate} that 
    \begin{align}\label{lem:Besovnorm behaves a little like L_p - proof 2}
        \|f-\pi_{\rho, \calP}(f)\|_{L_p(\Omega_T)}^p &= \sum\limits_{I\times S\in \calP} \|f-\pi_{\rho, \calP}(f)\|_{L_p(I\times S)}^p\notag
        \\&\lesssim_{\,\substack{ d,\rho,p, \alpha_1, \alpha_2, r_1, r_2,\\ \kappa_{\calP_0}, \mu(\calP_0), a(\calP_0)}} \sum\limits_{I\times S\in \calP} |I\times S|^{\frac{p}{\frac{1}{\alpha_1}+\frac{d}{\alpha_2}}} \left(|f|^{*,(r_1,r_2)}_{B_{p,p}^{\alpha_1, \alpha_2}\left(\widetilde{\omega^{j(d)}_\calP}(I\times S)\right)}\right)^p \notag
        \\&\lesssim_{\,\substack{d, p, \alpha_1, \alpha_2, \mu(\calP_0)}} \sum\limits_{I\times S\in \calP} \left(|f|^{*,(r_1,r_2)}_{B_{p,p}^{\alpha_1, \alpha_2}\left(\widetilde{\omega^{j(d)}_\calP}(I\times S)\right)}\right)^p.
    \end{align}
    To justify the last step, the additional rough estimate $|I\times S|\lesssim_{\,\mu(\calP_0)}1$, due to $\calP$ being derived from $\calP_0$, is sufficient. Now we have to estimate the sum in \eqref{lem:Besovnorm behaves a little like L_p - proof 1}. Therefore,
    \begin{align}\label{lem:Besovnorm behaves a little like L_p - proof 2.5}
        \sum\limits_{n=0}^\infty 2^{n\frac{\alpha_2}{d}p}E(f-\pi_{\rho, \calP}(f), \V^{r_1, r_2}_{\calP_{n}}, \Omega_T)_p^p &\le \sum\limits_{n=0}^\infty 2^{n\frac{\alpha_2}{d}p}\sum\limits_{I_n\times S_n\in \calP_n}\|f-\pi_{\rho, \calP}(f)-\pi_n(f-\pi_{\rho, \calP}(f))\|_{L_p(I_n\times S_n)}^p
        \\&\lesssim_{\, \substack{d,\rho,p,\alpha_1,\\ \alpha_2, r_1, r_2,\\ \kappa_{\calP_0}, \mu(\calP_0)}}\sum\limits_{n=0}^\infty 2^{n\frac{\alpha_2}{d}p}\sum\limits_{I_n\times S_n\in \calP_n} E\left(f-\pi_{\rho, \calP}(f), \Pi^{r_1, r_2}_{t,\bm{x}}, \omega_{\calP_n}^{j(d)}(I_n\times S_n)\right)_p^p\notag
    \end{align}
    according to \mbox{\cref{thm:local_boundedness_of_pi_projection}}. The estimation of the above term on the right hand side will be done in the following in three steps.

    \textit{Step 1: Decomposition of $\calP$.} Our goal is to decompose $\calP_n$, $n\in \N_0$, into two disjoint sets, i.e., $\calP_n=\calP_n^A\mathring{\cup}\calP_n^B$. In Case A, $\calP$ is supposed to be locally at least as fine as $\calP_n$ (up to a universally bounded number of levels), in Case B it should be locally coarser.

    To derive such a decomposition, we first observe the following: Due to \mbox{\cref{Enumeration further properties of the mesh}\ref{Enumeration further properties of the mesh - enum 2}} and the level bound on the elements of the extended neighborhood domains given below \eqref{Enumeration further properties of the mesh - proof eq 1}, there is a constant\footnote{Recall that $(\alpha_1, \alpha_2)$ and $(s_1, s_2)$ are linearly dependent.} $c=c(d,\alpha_1, \alpha_2, \kappa_{\calP_0}, \mu(\calP_0))$ such that
    \begin{align*}
        \{(t,\bm{x})\in \Omega_T\mid \textup{dist}(t, I)\le c\,|I|,~\textup{dist}(\bm{x}, S)\le c\,\diam(S)\}\subset\omega^{j(d)}_\calP(P)
    \end{align*}
    for any $P=I\times S\in \calP$.\footnote{\textup{dist} refers to the Euclidean distance here.} We recall $\diam_t\left(\omega_{\calP_n}^{j(d)+1}(P_n)\right)\sim_{\,d,\alpha_1, \alpha_2,\kappa_{\calP_0},\mu(\calP_0)}2^{-\frac{n\alpha_2}{\alpha_1d}}$ and $\diam_{\bm{x}}\left(\omega_{\calP_n}^{j(d)+1}(P_n)\right)\sim_{\,d,\kappa_{\calP_0},\mu(\calP_0)}2^{-\frac{n}{d}}$ for $P_n\in \calP_n$ according to \cref{rem:Neighborhoods_generally_not_Lipschitz}\ref{rem:Neighborhoods_generally_not_Lipschitz - General S_w''}, where we borrow the notation for temporal and spatial diameters from \cref{lem:FEM_elements_of_Aniso_Besov}. Further, define 
    \begin{align*}
        \calN(P_n):=\left\{P\in \calP\,\Big|\, P\cap\omega_{\calP_n}^{j(d)}(P_n)\neq \emptyset\right\}.
    \end{align*}
    The above behavior of temporal and spatial diameter of $\omega_{\calP_n}^{j(d)+1}(P_n)$ now shows that there is another constant $c_0$, depending on $d, \alpha_1, \alpha_2, \kappa_{\calP_0}$, and $\mu(\calP_0)$ but not on $n\in \N_0$, such that 
    \begin{align}\label{lem:Besovnorm behaves a little like L_p - step 1}
        P\in \calN(P_n)\quad\text{implies}\quad\omega_{\calP_n}^{j(d)+1}(P_n)\subset \omega_{\calP}^{j(d)}(P)
    \end{align}
    if $n\ge \ell(P)+c_0$, $P_n\in \calP_n$, $n\in \N_0$. Defining 
    \begin{align*}
        \calP_n^{A}:=\{P_n\in\calP_n\mid \ell(P)>n-c_0\ \text{for all}\ P\in\calN(P_n)\},\qquad \calP_n^{B}:=\calP_n\setminus\calP_n^{A},
    \end{align*}
    now yields the desired decomposition of $\calP_n$. 
    
    \textit{Step 2: Case of $P_n\in\calP^A_n$.} Here, we use the trivial bound 
    \begin{align*}
        E\left(f-\pi_{\rho, \calP}(f),\Pi^{r_1,r_2}_{t,\bm{x}},\omega_{\calP_n}^{j(d)}(P_n)\right)_p^p \le\|f-\pi_{\rho, \calP}(f)\|^p_{L_p\left(\omega_{\calP_n}^{j(d)}(P_n)\right)}\le\sum_{P\in\calN(P_n)}\|f-\pi_{\rho, \calP}(f)\|^p_{L_p(P)}.
    \end{align*}
    Every $P\in\calN(P_n)$ satisfies $\ell(P)>n-c_0$, hence $\diam_t(P)\lesssim 2^{-\frac{n\alpha_2}{\alpha_1d}}$ and \mbox{$\diam_{\bm{x}}(P)\lesssim2^{-\frac{n}{d}}$} with constants only depending on $d,\alpha_1, \alpha_2, \kappa_{\calP_0}$, and $\mu(\calP_0)$, and therefore a fixed $P\in\calP$ with $\ell(P)>n-c_0$ can meet $\omega_{\calP_n}^{j(d)}(P_n)$ only for a number of $P_n\in\calP_n$ that is bounded in terms of the same parameters as above; all such $P_n$ lie within a bounded number of layers of $\calP_n$ around $P$. Consequently, 
    \begin{align*}
        &\sum_{n=0}^\infty2^{n\frac{\alpha_2}{d}p}\sum_{P_n\in\calP^A_n}E\left(f-\pi_{\rho, \calP}(f),\Pi^{r_1,r_2}_{t,\bm{x}},\omega_{\calP_n}^{j(d)}(P_n)\right)_p^p
        \le \sum_{n=0}^\infty2^{n\frac{\alpha_2}{d}p}\sum_{P_n\in\calP^A_n}\sum_{P\in\calN(P_n)}\|f-\pi_{\rho, \calP}(f)\|^p_{L_p(P)}
        \\&\lesssim_{\,\substack{d,p,\alpha_1, \alpha_2,\\ \kappa_{\calP_0},\mu(\calP_0)}}\sum_{n=0}^\infty2^{n\frac{\alpha_2}{d}p}\sum_{\substack{P\in\calP\\ \ell(P)>n-c_0}}\|f-\pi_{\rho, \calP}(f)\|^p_{L_p(P)} = \sum\limits_{P\in \calP} \sum\limits_{n=0}^{\ell(P)+c_0-1} 2^{n\frac{\alpha_2}{d}p}\|f-\pi_{\rho, \calP}(f)\|^p_{L_p(P)}
        \\&\lesssim_{\,\substack{ d,\rho, p, \alpha_1, \alpha_2, \\r_1, r_2, \kappa_{\calP_0}, \\\mu(\calP_0), a(\calP_0)}}\sum\limits_{P\in \calP} \sum\limits_{n=0}^{\ell(P)+c_0-1} 2^{n\frac{\alpha_2}{d}p}|P|^{\frac{p}{\frac{1}{\alpha_1}+\frac{d}{\alpha_2}}}\left(|f|_{B_{p,p}^{\alpha_1, \alpha_2}\left(\widetilde{\omega_{\calP}^{j(d)}}(P)\right)}^{*,(r_1,r_2)}\right)^p 
        \\&\sim_{\,\substack{d,p, \alpha_1, \\\alpha_2,\mu(\calP_0)}} \sum\limits_{P\in\calP}\left(|f|_{B_{p,p}^{\alpha_1, \alpha_2}\left(\widetilde{\omega_{\calP}^{j(d)}}(P)\right)}^{*,(r_1,r_2)}\right)^p\sum\limits_{n=0}^{\ell(P)+c_0-1} 2^{(n-\ell(P))\frac{\alpha_2}{d}p}\lesssim_{\,\substack{d,p,\alpha_1, \alpha_2,\\ \kappa_{\calP_0},\mu(\calP_0)}}\sum\limits_{P\in\calP}\left(|f|_{B_{p,p}^{\alpha_1, \alpha_2}\left(\widetilde{\omega_{\calP}^{j(d)}}(P)\right)}^{*,(r_1,r_2)}\right)^p
        ,
    \end{align*}
    where we have additionally applied \cref{thm:local_boundedness_of_pi_projection} together with \cref{lem:Extended_Whitney_Estimate} in the fourth step, $|P|\sim_{\,\mu(\calP_0 )}2^{-\ell(P)\left(1+\frac{\alpha_2}{\alpha_1d}\right)} $ as well as $\left(1+\frac{\alpha_2}{\alpha_1d}\right)\frac{1}{\frac{1}{\alpha_1}+\frac{d}{\alpha_2}}=\frac{\alpha_2}{d}$ in the penultimate, and the fact that the geometric sum is bounded independent of $\ell(P)$ in the last one.
    
    \textit{Step 3: Case of $P_n\in\calP^B_n$.}  First, for $P_n\in\calP^B_n$, we use the subadditivity of the $L_p$-norm, to obtain
    \begin{align}\label{lem:Besovnorm behaves a little like L_p - case 3}
        E\left(f-\pi_{\rho, \calP}(f),\Pi^{r_1,r_2}_{t,\bm{x}},\omega_{\calP_n}^{j(d)}(P_n)\right)_p^{p}\lesssim_{\,p}E\left(f,\Pi^{r_1,r_2}_{t,\bm{x}},\omega_{\calP_n}^{j(d)}(P_n)\right)_p^{p}+E\left(\pi_{\rho, \calP}(f),\Pi^{r_1,r_2}_{t,\bm{x}},\omega_{\calP_n}^{j(d)}(P_n)\right)_p^{p}.
    \end{align}
    We will now treat the two terms on the right side separately. Before doing so, we clarify some necessary notation. By definition of $\calP^B_n$, for every
    $P_n\in\calP^B_n$ there is at least one $P\in\calN(P_n)$ with $\ell(P)\le n-c_0$. We fix one such element, $P^*=P^*(P_n)\in\calN(P_n)$ with $ \ell\big(P^*(P_n)\big)=\min\{\ell(P)\mid P\in\calN(P_n)\}\le n-c_0$; choosing an arbitrary one among those of minimal level if it is not unique.

    \textit{Step 3.1: Case of $P_n\in\calP^B_n$ -- part of $f$.} We proceed similarly to \cref{lem:Multiscale_decomposition}, with $0<\delta_n^t\le 2^{-\frac{n\alpha_2}{\alpha_1d}}$ and $0<\delta_n^{\bm{x}}\le 2^{-\frac{n}{d}}$ from there, to obtain
    \begin{align}\label{lem:Besovnorm behaves a little like L_p - step 3.1 - first equation}
        E\left(f,\Pi^{r_1,r_2}_{t,\bm{x}},\omega_{\calP_n}^{j(d)}(P_n)\right)_p^{p}\lesssim_{\,\substack{d, p, \alpha_1, \alpha_2, r_1,\\r_2, \kappa_{\calP_0}, \mu(\calP_0)}}\mathrm{w}_{r_1, t}\left(f,\omega_{\calP_n}^{j(d)+1}(P_n), \delta_n^t\right)_p^p+ \mathrm{w}_{r_2, \bm{x}}\left(f, \omega_{\calP_n}^{j(d)+1}(P_n), \delta_n^{\bm{x}}\right)_p^p,
    \end{align}
    for $n\in \N_0$. Exploiting the definition of $P^*(P_n)$, we can now calculate
    \begin{align*}
        &\sum\limits_{P_n\in\calP_n^B}E\left(f,\Pi^{r_1,r_2}_{t,\bm{x}},\omega_{\calP_n}^{j(d)}(P_n)\right)_p^{p} = \sum\limits_{P\in \calP}\sum\limits_{\substack{P_n\in\calP_n^B,\\P=P^*(P_n)}}E\left(f,\Pi^{r_1,r_2}_{t,\bm{x}},\omega_{\calP_n}^{j(d)}(P_n)\right)_p^{p}
        \\&\lesssim_{\,\substack{d, p, \alpha_1, \alpha_2, r_1,\\r_2, \kappa_{\calP_0}, \mu(\calP_0)}}\sum\limits_{P\in \calP}\sum\limits_{\substack{P_n\in\calP_n^B,\\P=P^*(P_n)}}\mathrm{w}_{r_1, t}\left(f,\omega_{\calP_n}^{j(d)+1}(P_n), \delta_n^t\right)_p^p+ \mathrm{w}_{r_2, \bm{x}}\left(f, \omega_{\calP_n}^{j(d)+1}(P_n), \delta_n^{\bm{x}}\right)_p^p 
        \\&\lesssim_{\,d,p,\kappa_{\calP_0}, \mu(\calP_0)}\sum\limits_{P\in\calP} \mathrm{w}_{r_1, t}\left(f, \widetilde{\omega_{\calP}^{j(d)}}(P), \delta_n^t\right)_p^p+ \mathrm{w}_{r_2, \bm{x}}\left(f, \widetilde{\omega_{\calP}^{j(d)}}(P), \delta_n^{\bm{x}}\right)_p^p
        \\&\le \sum\limits_{P\in\calP} \omega_{r_1, t}\left(f, \widetilde{\omega_{\calP}^{j(d)}}(P), 2^{-\frac{n\alpha_2}{\alpha_1d}}\right)_p^p+ \omega_{r_2, \bm{x}}\left(f, \widetilde{\omega_{\calP}^{j(d)}}(P), 2^{-\frac{n}{d}}\right)_p^p,
    \end{align*}
    where the second step is justified by \eqref{lem:Besovnorm behaves a little like L_p - step 3.1 - first equation}. Further, the fact that all the neighborhoods $\omega_{\calP_n}^{j(d)+1}(P_n)$, $P_n\in\calP_n$, have finite overlap due to $\eqref{rem:Neighborhoods_generally_not_Lipschitz - equation 5}$ together with the property \eqref{lem:Besovnorm behaves a little like L_p - step 1} from above implies the third because of the superadditivity of the averaged moduli (cf. \cref{Rem:Moduli of smoothness}\ref{Rem:Moduli of smoothness - additivity}), while the last is possible due to the fact that the averaged moduli are always less than the supremum versions and \cref{Rem:Moduli of smoothness}\ref{Rem:Moduli of smoothness - monotonicity}. Recalling the definitions of the Besov (quasi-)seminorms from \cref{subsubsect:Besov spaces}, this shows 
    \begin{align}\label{lem:Besovnorm behaves a little like L_p - case 3.1}
        \sum\limits_{n=0}^\infty 2^{n\frac{\alpha_2}{d}p}\sum\limits_{P_n\in\calP_n^B}E\left(f,\Pi^{r_1,r_2}_{t,\bm{x}},\omega_{\calP_n}^{j(d)}(P_n)\right)_p^{p} \lesssim_{\,\substack{d, p, \alpha_1, \alpha_2, r_1,\\r_2, \kappa_{\calP_0}, \mu(\calP_0)}}\sum\limits_{P\in \calP} \left(|f|_{B_{p,p}^{\alpha_1, \alpha_2}\left(\widetilde{\omega_{\calP}^{j(d)}}(P)\right)}^{*,(r_1,r_2)}\right)^p.
    \end{align}

    \textit{Step 3.2: Case of $P_n\in\calP^B_n$ -- part of $\pi_{\rho, \calP}(f)$.} Let $P^*=I^*\times S^*:=P^*(P_n)$. If $\calN(P_n)=\{P^*\}$, then $\omega_{\calP_n}^{j(d)}(P_n)\subset P^*$, $\pi_{\rho,\calP}(f)|_{\omega_{\calP_n}^{j(d)}(P_n)}$ is a polynomial, and $E(\pi_{\rho, \calP}(f),\Pi^{r_1,r_2}_{t,\bm{x}},\omega_{\calP_n}^{j(d)}(P_n))_p=0$. Otherwise, $\omega_{\calP_n}^{j(d)}(P_n)$ meets at least two elements of $\calP$, and we denote by $\calP^{B,2}_n$ the set of these $P_n$. 
    
    For $P\in \omega_\calP^{j(d)}(P^*)$, let $F_P\in \Pi^{r_1, r_2}_{t,\bm{x}}$ be the polynomial with $\pi_{\rho,\calP}(f)_{|P}=F_{P}$, and $Q$ the polynomial best approximation of $f$ in $L_p\left(\widetilde{\omega_\calP^{j(d)}}(P^*)\right)$.\footnote{With the notation from above \cref{rem:best_approximation_operator}, this corresponds to $Q=\calB_{p,\widetilde{\omega_\calP^{j(d)}}(P^*)}(f)$.} Now for every $P\in \omega_\calP^{j(d)}(P^*)$, we can define $D_P:=F_P-Q\in \Pi^{r_1, r_2}_{t,\bm{x}}$ and estimate
    \begin{align}\label{lem:Besovnorm behaves a little like L_p - D_P}
        &\|D_P\|_{L_p(P)}=\|F_P-Q\|_{L_p(P)}\lesssim_{\,p}\|f-F_P\|_{L_p(P)}+\|f-Q\|_{L_p(P)}\notag
        \\&\lesssim_{\,\substack{d,\rho,p,\alpha_1, \alpha_2, r_1,\\ r_2, \kappa_{\calP_0},\mu(\calP_0)}} 
        \sum\limits_{P'\in \omega_\calP^{j(d)}(P^*)}E\left(f, \Pi^{r_1, r_2}_{t,\bm{x}}, \widetilde{\omega_\calP^{j(d)}}(P')\right)_p
        \lesssim_{\,\substack{d,p,\alpha_1, \alpha_2,r_1, r_2,\\ \kappa_{\calP_0},\mu(\calP_0),a(\calP_0)}} |P^*|^{\frac{1}{\frac{1}{\alpha_1}+\frac{d}{\alpha_2}}}\sum\limits_{P'\in \omega_\calP^{j(d)}(P^*)} |f|_{B_{p,p}^{\alpha_1, \alpha_2}\left(\widetilde{\omega_{\calP}^{j(d)}}(P')\right)}^{*,(r_1,r_2)}\notag
        \\&\lesssim_{\,d,\alpha_1, \alpha_2, \mu(\calP_0)}2^{-\ell(P^*)\frac{\alpha_2}{d}}\sum\limits_{P'\in \omega_\calP^{j(d)}(P^*)} |f|_{B_{p,p}^{\alpha_1, \alpha_2}\left(\widetilde{\omega_{\calP}^{j(d)}}(P')\right)}^{*,(r_1,r_2)}=:\mathfrak{E}(P^*),
    \end{align}
    where we have applied \cref{thm:local_boundedness_of_pi_projection} for $F_P$ and the best approximation property for $Q$ (note that $P, P^*\in\omega_\calP^{j(d)}(P^*)$) in the third step, used \cref{lem:Extended_Whitney_Estimate} together with $|P|\sim_{\,d,\alpha_1, \alpha_2, \kappa_{\calP_0},\mu(\calP_0)} |P^*|$ due to \cref{Enumeration further properties of the mesh}\ref{Enumeration further properties of the mesh - enum 2} in the penultimate step as well as $|P^*|\sim_{\mu(\calP_0)}2^{-\ell(P^*)\left(1+\frac{\alpha_2}{\alpha_1d}\right)}$ and  $\left(1+\frac{\alpha_2}{\alpha_1d}\right)\frac{1}{\frac{1}{\alpha_1}+\frac{d}{\alpha_2}}=\frac{\alpha_2}{d}$ in the last one. Since $D_P$ is a polynomial on the prism $P=I\times S$, an affine transformation to the standard prism and the fact that on finite dimensional spaces every seminorm is dominated by every (quasi-)norm give the inverse estimates\footnote{By standard prism we mean the prism $[0,1]\times \hat{S}$, with $\hat{S}$ as the standard simplex from the beginning of \cref{sect:investigation of mesh geometry and Quasi-interpolation}. Indeed, the proof works as follows: as in \cite[Lem.~3.5]{MSS26}, we transform $P=I\times S$ to $[0,1]\times \hat{S}$, where $\|\partial_t\cdot\|_{L_\infty([0,1]\times \hat{S})}$ and $\|\nabla_{\bm{x}}\cdot\|_{L_\infty([0,1]\times \hat{S})}$ are seminorms and $\|\cdot\|_{L_p([0,1]\times \hat{S})}$ is a (quasi-)norm on the finite dimensional space $\Pi^{r_1,r_2}_{t,\bm{x}}$. Transforming back, the Jacobi transformation theorem gives $|P|^{-\frac1p}$, while the chain rule yields $|I|^{-1}$ in time and, by spatial shape regularity, $\diam(S)^{-1}$ in space.}
    \begin{align}\label{lem:Besovnorm behaves a little like L_p - inverse}
        \|\partial_tD_P\|_{L_\infty(P)}\lesssim_{\,\substack{d,p,\\r_1,r_2}}|I|^{-1}|P|^{-\frac1p}\|D_P\|_{L_p(P)}\text{ and }\|\nabla_{\bm{x}}D_P\|_{L_\infty(P)}\lesssim_{\,\substack{d,p,r_1,\\r_2,\kappa_{\calP_0}}}\diam(S)^{-1}|P|^{-\frac1p}\|D_P\|_{L_p(P)}.
    \end{align}
    Moreover, since $D_{P^*}$ is a polynomial, \eqref{lem:Besovnorm behaves a little like L_p - inverse} with $P=P^*$ remains valid if the $L_\infty(P^*)$-norms on the left-hand sides are replaced by $L_\infty\left(\omega_\calP^{j(d)}(P^*)\right)$-norms, because this set has temporal and spatial diameters depending on those of $P^*$ according to \eqref{rem:Neighborhoods_generally_not_Lipschitz - equation 2} and \eqref{rem:Neighborhoods_generally_not_Lipschitz - equation 3}; the constants then additionally depend on $\kappa_{\calP_0}$ and $\mu(\calP_0)$. 

    Now let $(t,\bm{x})\in\omega_{\calP_n}^{j(d)}(P_n)$ and choose $(\tau,\bm{y})\in\omega_{\calP_n}^{j(d)}(P_n)\cap P^*$. Since $\omega_{\calP_n}^{j(d)}(P_n)=J_n\times R_n$ is a cylinder according to \cref{rem:Neighborhoods_generally_not_Lipschitz}\ref{rem:Neighborhoods_generally_not_Lipschitz - Uniform case S_w''}, the two points can be joined inside $\omega_{\calP_n}^{j(d)}(P_n)$ by a path consisting of a temporal path from $(t,\bm{x})$ to $(\tau, \bm{x})$ of length $\le |J_n|\lesssim_{\,d,\alpha_1, \alpha_2, \kappa_{\calP_0}, \mu(\calP_0)}2^{-\frac{n\alpha_2}{\alpha_1d}}$ and a spatial one from $(\tau, \bm{x})$ to $(\tau, \bm{y})$ of length $\lesssim_{\,d,\kappa_{\calP_0}}\diam(R_n)\lesssim_{\,d,\kappa_{\calP_0}, \mu(\calP_0)} 2^{-\frac{n}{d}}$ according to \cref{rem:Neighborhoods_generally_not_Lipschitz}\ref{rem:Neighborhoods_generally_not_Lipschitz - General S_w''}. By \eqref{lem:Besovnorm behaves a little like L_p - step 1},
    applied to $P=P^*$, we have $\omega_{\calP_n}^{j(d)}(P_n)\subset\omega_\calP^{j(d)}(P^*)$, and therefore both paths decompose into finitely many pieces, each of which is contained in a single element of $\omega_\calP^{j(d)}(P^*)$. This number of pieces is bounded by $\#\omega_\calP^{j(d)}(P^*) \lesssim_{\,d,\alpha_1, \alpha_2, \kappa_{\calP_0}}1$, due to \cref{Enumeration further properties of the mesh}\ref{Enumeration further properties of the mesh - enum 3}. 
    
    Since $\pi_{\rho,\calP}(f)-Q$ is continuous and $\left(\pi_{\rho,\calP}(f)-Q\right)_{|P}=D_P$ for every $P\in\omega_\calP^{j(d)}(P^*)$, its increments along the pieces of the paths add up, and each of them is bounded by the corresponding derivative of $D_P$ times the length of that piece. Together with $\left(\pi_{\rho,\calP}(f)-Q\right)(\tau,\bm{y})=D_{P^*}(\tau,\bm{y})$ due to $(\tau,\bm{y})\in P^*$, this gives
    \begin{align*}
        &|\pi_{\rho,\calP}(f)(t,\bm{x})-F_{P^*}(t,\bm{x})|
        =\left|\left(\pi_{\rho,\calP}(f)-Q\right)(t,\bm{x})-D_{P^*}(t,\bm{x})\right|
        \\&\le\left|\left(\pi_{\rho,\calP}(f)-Q\right)(t,\bm{x})-\left(\pi_{\rho,\calP}(f)-Q\right)(\tau,\bm{y})\right|
        +\left|D_{P^*}(\tau,\bm{y})-D_{P^*}(t,\bm{x})\right|
        \\&\lesssim_{\,d,\alpha_1, \alpha_2,\kappa_{\calP_0},\mu(\calP_0)}\ 2^{-\frac{n\alpha_2}{\alpha_1d}}\max_{P\in\omega_\calP^{j(d)}(P^*)}\|\partial_tD_P\|_{L_\infty(P)}
        +2^{-\frac{n}{d}}\max_{P\in\omega_\calP^{j(d)}(P^*)}\|\nabla_{\bm{x}}D_P\|_{L_\infty(P)}
        \\&\lesssim_{\,\substack{d,\rho,p,\alpha_1,\alpha_2,r_1, r_2,\\\kappa_{\calP_0},\mu(\calP_0),a(\calP_0)}} \left(2^{-\frac{(n-\ell(P^*))\alpha_2}{\alpha_1d}}+2^{-\frac{n-\ell(P^*)}{d}}\right)|P^*|^{-\frac1p}\mathfrak{E}(P^*),
    \end{align*}
    where we have exploited \eqref{lem:Besovnorm behaves a little like L_p - inverse} (and the comment below), \eqref{lem:Besovnorm behaves a little like L_p - D_P}, and the comparability of
    the elements of $\omega_\calP^{j(d)}(P^*)$ due to \eqref{rem:Neighborhoods_generally_not_Lipschitz - equation 3} and \eqref{rem:Neighborhoods_generally_not_Lipschitz - equation 2} in the last step.
    Since $F_{P^*}\in\Pi^{r_1,r_2}_{t,\bm{x}}$, this yields
    \begin{align}\label{lem:Besovnorm behaves a little like L_p - case 3.2 - final equation 1}
        E\left(\pi_{\rho,\calP}(f),\Pi^{r_1,r_2}_{t,\bm{x}},\omega_{\calP_n}^{j(d)}(P_n)\right)_p^p&\le\|\pi_{\rho,\calP}(f)-F_{P^*}\|^p_{L_p\left(\omega_{\calP_n}^{j(d)}(P_n)\right)}\notag
        \\&\lesssim_{\,\substack{d,\rho,p,\alpha_1,\alpha_2,r_1, r_2,\\\kappa_{\calP_0},\mu(\calP_0),a(\calP_0)}}\left(2^{-\frac{(n-\ell(P^*))\alpha_2}{\alpha_1d}}+2^{-\frac{n-\ell(P^*)}{d}}\right)^p\frac{\left|\omega_{\calP_n}^{j(d)}(P_n)\right|}{|P^*|}\,\mathfrak{E}(P^*)^p.
    \end{align}
    Now the factor $|\omega_{\calP_n}^{j(d)}(P_n)|$ is summable over all $P_n\in\calP^{B,2}_n$ with $P^*(P_n)=P^*$ due to \cref{rem:Generalized_Besov_Spaces}\ref{rem:Generalized_Besov_Spaces - item 4}. Further, each such $\omega_{\calP_n}^{j(d)}(P_n)$ intersects the boundary of $P^*=I^*\times S^*$ and, as above, has temporal and spatial diameter $\lesssim 2^{-\frac{n\alpha_2}{\alpha_1d}}$ and $\lesssim 2^{-\frac{n}{d}}$, respectively (with dependencies as above), so we can estimate
    \begin{align}\label{lem:Besovnorm behaves a little like L_p - case 3.2 - final equation 2}
        \sum\limits_{\substack{P_n\in\calP^{B,2}_n,\\P^*(P_n)=P^*}}|\omega_{\calP_n}^{j(d)}(P_n)|&\lesssim_{\,\substack{d,\alpha_1,\alpha_2,\\\kappa_{\calP_0},\mu(\calP_0)}} 2^{-\frac{n\alpha_2}{\alpha_1d}}\diam(S^*)^d+|I^*|2^{-\frac{n}{d}}\diam(S^*)^{d-1}\notag
        \\&\sim_{\,d,\kappa_{\calP_0}}\left(\frac{2^{-\frac{n\alpha_2}{\alpha_1d}}}{|I^*|}+\frac{2^{-\frac{n}{d}}}{\diam(S^*)}\right)|P^*|\sim_{\,d, \kappa_{\calP_0},\mu(\calP_0)}\left(2^{-\frac{(n-\ell(P^*))\alpha_2}{\alpha_1d}}+2^{-\frac{n-\ell(P^*)}{d}}\right)|P^*|,
    \end{align}
    where we have additionally exploited $|I^*|\sim_{\mu(\calP_0)}2^{-\frac{\ell(P^*)\alpha_2}{\alpha_1d}} $ and $\diam(S^*)\sim_{\,d, \kappa_{\calP_0},\mu(\calP_0)}2^{-\frac{\ell(P^*)}{d}}$ in the last step. The combination of \eqref{lem:Besovnorm behaves a little like L_p - case 3.2 - final equation 1} and \eqref{lem:Besovnorm behaves a little like L_p - case 3.2 - final equation 2} shows us that  
    \begin{align*}
        2^{n\frac{\alpha_2}{d}p}&\sum\limits_{\substack{P_n\in\calP^{B,2}_n,\\P^*(P_n)=P^*}}E\left(\pi_{\rho,\calP}(f),\Pi^{r_1,r_2}_{t,\bm{x}},\omega_{\calP_n}^{j(d)}(P_n)\right)_p^p
        \\&\lesssim_{\,\substack{d,\rho,p,\alpha_1,\alpha_2,r_1, r_2,\kappa_{\calP_0},\mu(\calP_0),a(\calP_0)}}2^{n\frac{\alpha_2}{d}p}\left(2^{-\frac{(n-\ell(P^*))\alpha_2(p+1)}{\alpha_1d}}+2^{-\frac{(n-\ell(P^*))(p+1)}{d}}\right)\mathfrak{E}(P^*)^p
        \\&=2^{\ell(P^*)\frac{\alpha_2}{d}p}\mathfrak{E}(P^*)^p\left(2^{(n-\ell(P^*))\frac{\alpha_2}{d}\left(p-\frac{p+1}{\alpha_1}\right)}+2^{(n-\ell(P^*))\left(\frac{\alpha_2p}{d}-\frac{(p+1)}{d}\right)}\right),
    \end{align*}
    for every $n\in \N_0$. Further, due to $\alpha_i<1+\frac1p$, $i=1,2$, both terms on the right hand side are summable over $n\in \N_0$, since the corresponding geometric sums converge. In particular, these sums can be estimated independent of $\ell(P^*)$, since $n-\ell(P^*)\ge c_0$. This yields
    \begin{align*}
        \sum\limits_{n=0}^\infty 2^{n\frac{\alpha_2}{d}p}\sum\limits_{\substack{P_n\in\calP^{B,2}_n,\\P^*(P_n)=P^*}}E\Big(\pi_{\rho,\calP}(f),\Pi^{r_1,r_2}_{t,\bm{x}}&,\omega_{\calP_n}^{j(d)}(P_n)\Big)_p^p
        \\&\lesssim_{\,\substack{d,\rho,p,\alpha_1,\alpha_2,r_1, r_2,\\\kappa_{\calP_0},\mu(\calP_0),a(\calP_0)}}2^{\ell(P^*)\frac{\alpha_2}{d}p}\mathfrak{E}(P^*)^p=\left(\sum\limits_{P\in \omega_\calP^{j(d)}(P^*)} |f|_{B_{p,p}^{\alpha_1, \alpha_2}\left(\widetilde{\omega_{\calP}^{j(d)}}(P)\right)}^{*,(r_1,r_2)}\right)^p,
    \end{align*}
    if we notice that the inner left sum is empty for $n-\ell(P^*)<c_0$, due to the definition of $\calP^B_n$. Since $\#\omega_\calP^{j(d)}(P^*)\lesssim_{\,d,\alpha_1,\alpha_2,\kappa_{\calP_0}}1$, due to \cref{Enumeration further properties of the mesh}\ref{Enumeration further properties of the mesh - enum 3}, and the fact that the inclusion of all $P_n\in \calP_n^B\setminus \calP_n^{B,2}$ does not change the estimate due to the observations at the beginning of Step 3.2, we finally obtain the following estimate by summing over all $P^*\in \calP$,
    \begin{align}\label{lem:Besovnorm behaves a little like L_p - case 3.2 - final equation 3}
        \sum\limits_{n=0}^\infty 2^{n\frac{\alpha_2}{d}p}\sum\limits_{P_n\in \calP_n^B}E\left(\pi_{\rho,\calP}(f),\Pi^{r_1,r_2}_{t,\bm{x}},\omega_{\calP_n}^{j(d)}(P_n)\right)_p^p\lesssim_{\,\substack{d,\rho,p,\alpha_1,\alpha_2,r_1, r_2,\\\kappa_{\calP_0},\mu(\calP_0),a(\calP_0)}}\sum\limits_{P\in \calP} \left(|f|_{B_{p,p}^{\alpha_1, \alpha_2}\left(\widetilde{\omega_{\calP}^{j(d)}}(P)\right)}^{*,(r_1,r_2)}\right)^p.
    \end{align}
    Combining \eqref{lem:Besovnorm behaves a little like L_p - case 3} with \eqref{lem:Besovnorm behaves a little like L_p - case 3.1} and \eqref{lem:Besovnorm behaves a little like L_p - case 3.2 - final equation 3} concludes Step 3, which together with Step 2 and \eqref{lem:Besovnorm behaves a little like L_p - proof 2} as well as \eqref{lem:Besovnorm behaves a little like L_p - proof 2.5} finishes the proof.
\end{proof}        

We can further estimate the local Besov (quasi-)seminorm of lower order with one of higher order. This result corresponds to \cite[Lem.~4.17]{GM14} in the stationary setting.

\begin{lem}\label{lem:Whitney for besov-seminorms}
    Under the assumptions of \cref{lem:Besovnorm behaves a little like L_p}, while additionally requiring \mbox{$\frac{1}{\frac{1}{s_1}+\frac{d}{s_2}}-\frac1q+\frac1p>0$} as well as $r_i>\alpha_i+s_i$, $i=1,2$,
        \begin{align*}
            |f|^{*,(r_1,r_2)}_{B_{p,p}^{\alpha_1, \alpha_2}\left(\widetilde{\omega^{j(d)}_\calP}(I\times S)\right)} \lesssim_{\,\substack{d,p,q,s_1, s_2,\alpha_1, \alpha_2, r_1,r_2, \kappa_{\calP_0},\\\mu(\calP_0), a(\calP_0),\LipProp(\Omega)}} |I\times S|^{\frac{1}{\frac{1}{s_1}+\frac{d}{s_2}}-\frac1q+\frac1p}|f|_{B_{q,q}^{\alpha_1+s_1, \alpha_2+s_2}\left(\Enl(\widetilde{\omega^{j(d)}_\calP}(I\times S))\right)}
        \end{align*}
    holds true for any $I\times S\in \calP$.
\end{lem}
\begin{proof}
    We employ the notation from \cref{lem:Extended_Whitney_Estimate}, i.e., $J\times R:=\widetilde{\omega^{j(d)}_\calP}(I\times S)$, $R^{ref}$ as the corresponding spatial reference domain, and $\hat{f}$ as the rescaled version of $f$ on $[0,1]\times R^{ref}$. Further, we use the enlargement operator from \cref{rem:Neighborhoods_generally_not_Lipschitz}\ref{rem:Neighborhoods_generally_not_Lipschitz - enlargement}. Due to \eqref{lem:Extended_Whitney_Estimate - equation which is needed later} together with \eqref{lem:Extended_Whitney_Estimate - equation 2} (and their analogue for the discrete (quasi-)seminorm and the enlarged domains) and $|J\times R|\sim_{\,d,s_1, s_2, \kappa_{\calP_0}, \mu(\calP_0)}|\Enl(J\times R)|$, we have 
    \begin{align*}
        &|\hat{f}|^{*,(r_1,r_2)}_{B^{\alpha_1, \alpha_2}_{p,p}([0,1]\times R^{ref})} \sim_{\,d, \alpha_1, \alpha_2, \kappa_{\calP_0}, \mu(\calP_0), a(\calP_0)} |J\times R|^{\frac{1}{\frac{1}{\alpha_1}+\frac{d}{\alpha_2}}-\frac1p} |f|_{B^{\alpha_1, \alpha_2}_{p,p}(J\times R)}^{*,(r_1,r_2)} 
        \\\text{as well as }&|\hat{f}|_{B^{\alpha_1+s_1, \alpha_2+s_2}_{q,q}(\Enl([0,1]\times R^{ref}))} \sim_{\,\substack{d, s_1, s_2,\alpha_1, \alpha_2, \\\kappa_{\calP_0}, \mu(\calP_0), a(\calP_0)}} |J\times R|^{\frac{1}{\frac{1}{\alpha_1+s_1}+\frac{d}{\alpha_2+s_2}}-\frac1q} |f|_{B^{\alpha_1+s_1, \alpha_2+s_2}_{q,q}(\Enl(J\times R))}.
    \end{align*}
    Here, we keep in mind that $a(\calP_0)=a(\calP_0, \alpha_1, \alpha_2)= a(\calP_0, \alpha_1+s_1, \alpha_2+s_2)$ due to the linear dependence of $(s_1, s_2)$ and $(\alpha_1, \alpha_2)$. The latter now additionally implies that it suffices to show the result on $[0,1]\times R^{ref}$ with respect to $\hat{f}$. So, we set $\tilde{P}:=\calB_{q, \Enl([0,1]\times R^{ref})}(\hat{f})$ with respect to $(r_1, r_2)$.\footnote{This notation has been defined in \cref{sect:investigation of mesh geometry and Quasi-interpolation}; directly above \cref{rem:best_approximation_operator}.} Then,
    \begin{align*}
        |\hat{f}&|^{*, (r_1, r_2)}_{B^{\alpha_1, \alpha_2}_{p,p}([0,1]\times R^{ref})} = |\hat{f} - \tilde{P}|^{*, (r_1, r_2)}_{B^{\alpha_1, \alpha_2}_{p,p}([0,1]\times R^{ref})} \le 
        \|\hat{f} - \tilde{P}\|_{B^{\alpha_1, \alpha_2}_{p,p}([0,1]\times R^{ref})}^{*, (r_1, r_2)} 
        \\ &\lesssim_{\, \substack{d,p,q,s_1, s_2,\alpha_1, \alpha_2, \\r_1, r_2, \kappa_{\calP_0},\mu(\calP_0),\\a(\calP_0),\LipProp(\Omega)}} \|\hat{f} - \tilde{P}\|_{B^{\alpha_1+s_1, \alpha_2+s_2}_{q,q}(\Enl([0,1]\times R^{ref}))}^{(r_1, r_2)}
        \\ &=\|\hat{f} - \tilde{P}\|_{L_{q}(\Enl([0,1]\times R^{ref}))}+|\hat{f}|_{B^{\alpha_1+s_1, \alpha_2+s_2}_{q,q}(\Enl([0,1]\times R^{ref}))}^{(r_1, r_2)}
        \lesssim_{\,\substack{d,p,q,s_1, s_2,\alpha_1, \alpha_2, r_1,\\r_2, \kappa_{\calP_0},\mu(\calP_0), a(\calP_0)}} |\hat{f}|_{B^{\alpha_1+s_1, \alpha_2+s_2}_{q,q}(\Enl([0,1]\times R^{ref}))}.
    \end{align*}
    Above, we have applied the result of \cref{thm:Proofs_omega_tilde_domain}\ref{thm:Proofs_omega_tilde_domain - item moduli of smoothness equal zero IFF anisotropic polynomials} in the first and penultimate step. Additionally, in the third one, we have used the embedding result from \mbox{\cref{cor_to:thm:Sobolev-type-embedding}} together with \eqref{thm:Proofs_omega_tilde_domain - Embeddings between Besov and Lebesgue scale - equation 2} from \cref{thm:Proofs_omega_tilde_domain}\ref{thm:Proofs_omega_tilde_domain - Embeddings between Besov and Lebesgue scale} as well as the analogue to \cref{lem:Extended_Whitney_Estimate} for the enlarged reference domains (cf. \cref{rem:to_lem:Extended_Whitney_Estimate}) and \cref{Rem:Moduli of smoothness}\ref{Rem:Moduli of smoothness - higher order lesser order} in the last one.\footnote{The latter directly implies $|\cdot|_{B^{\alpha_1+s_1, \alpha_2+s_2}_{q,q}(\Enl([0,1]\times R^{ref}))}^{(r_1, r_2)}\lesssim_{\,q, r_1, r_2} |\cdot|_{B^{\alpha_1+s_1, \alpha_2+s_2}_{q,q}(\Enl([0,1]\times R^{ref}))}$.} This shows the assertion.
\end{proof}

Let us now consider the algorithm $\textup{MARK}\_\textup{AND}\_\textup{REFINE}(\calP_0, d, s_1, s_2)$ from \cref{subsect:Algorithmic_complexity} with respect to the marking routine 
\begin{align*}
    \textup{MARK}(\calP)\textcolor{black}{:=}\textup{MARK}_{f,\delta}(\calP)=\left\{I\times S\in \calP \,\bigg|\, 
    \calE_\calP(f,I\times S) \textcolor{black}{:=}|I\times S|^{\frac{1}{\frac{1}{s_1}+\frac{d}{s_2}}-\frac{1}{q}+\frac{1}{p}}|f|_{B^{\alpha_1+s_1, \alpha_2+s_2}_{q,q}\left(\Enl(\widetilde{\omega_\calP^{j(d)}}(I\times S))\right)}>\delta\right\}
\end{align*}    
for any corresponding space-time partition $\calP$ of $\Omega_T$, where $f\in B^{\alpha_1+s_1, \alpha_2+s_2}_{q,q}(\Omega_T)$ is the function to be approximated and $\delta\in(0,\infty)$ is a given tolerance. In other words, starting from $\calP_0$, in each loop of the \textbf{while}-clause all elements $I\times S$ of the current partition $\calP$ whose local, suitably scaled Besov (quasi-)seminorm exceeds $\delta$ are marked, and then refined by $\textup{PATCH}\_\textup{REFINE}(\calP, I\times S, d, s_1, s_2)$; the algorithm stops as soon as no element of the current partition exceeds the tolerance, and returns this final partition. The resulting algorithm is a type of greedy algorithm and is very similar to corresponding ones in the literature, for example in \cite{BDDP02, GM14, AMS23, MSS26}.

As in \cref{subsect:Algorithmic_complexity}, we will denote the marked elements in the space-time partition at the end of the $k$-th loop of the \textbf{while}-clause of the routine by $\calM_k$, for $k\in \N_0$. Since the algorithm is a priori potentially non-terminating, we first show that, for this particular marking routine, it does terminate, so that the final partition is well-defined.

\begin{lem}\label{lem:algorithm_terminates}
    Let \cref{assumptions:multiscale-decomposition} hold with $\frac{1}{\frac{1}{s_1}+\frac{d}{s_2}}-\frac{1}{q}+\frac{1}{p}>0$ as well as $0\le \alpha_i <1+ \frac1p$, $i=1,2$. The algorithm $\textup{MARK}\_\textup{AND}\_\textup{REFINE}(\calP_0, d, s_1, s_2)$ with respect to the above marking method terminates for every $\delta\in (0,\infty)$ and $f\in B^{\alpha_1+s_1, \alpha_2+s_2}_{q,q}(\Omega_T)$.
\end{lem}
\begin{proof}
    The proof is similar to the one of \cite[Lem.~4.3]{MSS26}. If the algorithm did not terminate, there would be $I_k\times S_k\in \calM_k$ for every $k\in \N$. For those prisms, we know that $|I_k\times S_k|\xrightarrow{k\rightarrow \infty}0$, since the atomic refinement method gets called to $I_k\times S_k$ during the procedure, which at least halves the volume of prisms to which it is called. Since $I_k\times S_k\in \calM_k$, we can therefore estimate
    \begin{align*}
        \delta < |I_k\times S_k|^{\frac{1}{\frac{1}{s_1}+\frac{d}{s_2}}-\frac{1}{q}+\frac{1}{p}}|f|_{B^{\alpha_1+s_1, \alpha_2+s_2}_{q,q}(\Omega_T)}\xrightarrow{k\rightarrow\infty}0,
    \end{align*}
    which is a contradiction to $\delta >0$.
\end{proof}

By carefully choosing an appropriate $\delta$, we can proceed to the following theorem, which will be the first part of the almost characterization. This results together with the embeddings from \cite[Thm.~1.3]{MSS26} and \cref{thm:Sobolev-type-embedding}, respectively, shows \cref{Main_res:direct}. In it, we will exploit the notation $B^{0,0}_{p,p}:=L_p$ for simplicity.

A version of this theorem for the adaptive approximation with discontinuous, anisotropic polynomials was already proven in \cite[Thm.~1.5]{MSS26}. 

\begin{thm}\label{thm:direct_estimates}
    Let $p,q ,s_1,s_2, f$ as in \cref{lem:algorithm_terminates}, $(\alpha_1, \alpha_2)\in \R^+_0(s_1, s_2)$ with $\alpha_i<1+\frac1p$, and \mbox{$\rho\in (0,\min(p,q)]$}. Further, consider $r_i\in\N_{\ge 2}$ with $r_i>\alpha_i+s_i$, $i=1,2$. Additionally, choose an arbitrary $\varepsilon>0$. Then, there is $\delta = \delta (\varepsilon, d, s_1, s_2,\alpha_1, \alpha_2, p,f)>0$ such that $\textup{MARK}\_\textup{AND}\_\textup{REFINE}(\calP_0, d, s_1, s_2)$ with the above marking technique terminates with a space-time partition $\calP$ which satisfies
    \begin{align}\label{thm:direct_estimates - equation 1}
        \#\calP - \#\calP_0 \lesssim_{\, d,p,q, s_1,s_2, \alpha_1, \alpha_2, \kappa_{\calP_0}, \mu(\calP_0), |\Omega_T|,\LipProp(\Omega)} \varepsilon^{-\left(\frac{1}{s_1}+\frac{d}{s_2}\right)} \quad
    \end{align}
        and admits the existence of some $F\in \V^{r_1, r_2}_\calP$ such that
    \begin{align*}
        \quad \|f-F\|_{B_{p,p}^{\alpha_1, \alpha_2}(\Omega_T)}\lesssim_{\,d,\rho,p,q,s_1,s_2,\alpha_1, \alpha_2, \kappa_{\calP_0}, \mu(\calP_0), a(\calP_0),|\Omega_T|, \LipProp(\Omega), \#\calP_0} \varepsilon |f|_{B^{\alpha_1+s_1, \alpha_2+s_2}_{q,q}(\Omega_T)}.
    \end{align*}
\end{thm}
\begin{proof}
    Since $\V^{r_1, r_2}_\calP\subset \V^{r_1', r_2'}_\calP$ for $r_i\le r_i'\in \N_{\ge 2}$, $i=1,2$, it is sufficient to show the above result for the minimal choice of $r_i\in\N_{\ge 2}$ with $r_i>\alpha_i+s_i$. This means, we may assume \mbox{$r_i:=\max(\lfloor \alpha_i + s_i\rfloor + 1,2)$}, $i=1,2$, without loss of generality. Below, we will show  that $\pi_{\rho, \calP}$ with respect to this choice of $(r_1, r_2)$  on the asserted partition $\calP$, is indeed a valid choice for $F$.
    
    If $|f|_{B^{\alpha_1+s_1, \alpha_2+s_2}_{q,q}(\Omega_T)}=0$, then the error bound $\calE_{\calP_0}(f,I\times S)$ used in the marking procedure vanishes for all $I\times S \in \calP_0$ and thus $\calM_0=\emptyset$. Therefore, the algorithm terminates with $\calP=\calP_0$ for arbitrary $\delta>0$. $\calP_0$ of course, fulfills the first assertion. Since $f\in \Pi^{r_1, r_2}_{t,\bm{x}}(\Omega_T)$, which follows from \cite[Lem.~2.16]{MSS26}, we have  $\calB_{\rho, \calP}(f)=f$ because $f$ itself is the best approximation of $f$ on any domain with respect to elements of $\Pi^{r_1, r_2}_{t,\bm{x}}(\Omega_T)$. Furthermore, \cref{rem:linearity_and_projection_property_of_Q} implies $\mathcal{Q}_\calP(f)=f$, i.e., $\pi_{\rho,\calP}(f)=f$ and therefore the second assertion is true.

    Now let us consider the case $|f|_{B^{\alpha_1+s_1, \alpha_2+s_2}_{q,q}(\Omega_T)}>0$ and set $\delta:=\varepsilon^{1+\frac{1}{s_1p}+\frac{d}{s_2p}}|f|_{B^{\alpha_1+s_1, \alpha_2+s_2}_{q,q}(\Omega_T)}>0$. According to \cref{lem:algorithm_terminates}, the algorithm terminates with $\calP= \calP_k$ for some $k\in \N_0$. Therefore, $\calM_j\neq \emptyset$ if and only if $j\le k-1$. We define
    \begin{align*}
        \overline{\calM}:=\bigcup\limits_{j=0}^{k-1}\calM_j\quad\text{and}\quad \Gamma_j:=\{I\times S\in \overline{\calM}\mid \ell(I\times S)=j\}\quad\text{for}\quad j\in \N_0.
    \end{align*}
    Notice that \cref{Theorem - conformity 1-reg Patch refine} implies that also $\Gamma_j=\emptyset$ for $j\ge k$. Our goal will be to bound the cardinality of $\Gamma_j$ appropriately. On the one hand, for every $I\times S\in \Gamma_j$ and $j\in \N_0$, $|I\times S|\sim_{\,\mu(\calP_0)}2^{-j\left(\frac{s_2}{s_1d}+1\right)}$ and therefore $\#\Gamma_j\lesssim_{\, \mu(\calP_0), |\Omega_T|}2^{j\left(\frac{s_2}{s_1d}+1\right)}$. On the other hand, consider the following calculation
    \begin{align*}
        \#\Gamma_j \delta^q= \sum\limits_{I\times S\in \Gamma_j}\delta^q &\le \sum\limits_{I\times S\in \Gamma_j} |I\times S|^{\frac{q}{\frac{1}{s_1}+\frac{d}{s_2}}-1 + \frac{q}{p}}|f|_{B^{\alpha_1+s_1,\alpha_2+s_2}_{q,q}(\Enl(\widetilde{\omega_\calP^{j(d)}}(I\times S)))}^q
        \\&\lesssim_{\,q,\mu(\calP_0), |\Omega_T|} 2^{-jq\left(1+\frac{s_2}{s_1d}\right)\left(\frac{1}{\frac{1}{s_1}+\frac{d}{s_2}}-\frac{1}{q}+ \frac{1}{p}\right)} \sum\limits_{I\times S\in \Gamma_j}|f|_{B^{\alpha_1+s_1,\alpha_2+s_2}_{q,q}(\Enl(\widetilde{\omega_\calP^{j(d)}}(I\times S)))}^q 
        \\&\lesssim_{\, d, q, s_1, s_2, \alpha_1, \alpha_2, \kappa_{\calP_0}, \mu(\calP_0),\LipProp(\Omega)} 2^{-jq\left(1+\frac{s_2}{s_1d}\right)\left(\frac{1}{\frac{1}{s_1}+\frac{d}{s_2}}-\frac{1}{q}+ \frac{1}{p}\right)} |f|_{B^{\alpha_1+s_1,\alpha_2+s_2}_{q,q}(\Omega_T)}^q,
    \end{align*}
    where we have used that $\Gamma_j\subset \overline{\calM}$ in the second step and $2^{-j\left(1+\frac{s_2}{s_1d}\right)}\sim_{\,\mu(\calP_0), |\Omega_T|}|I\times S|$ for $I\times S\in \Gamma_j$ as well as $\frac{1}{\frac{1}{s_1}+\frac{d}{s_2}}-\frac1q+ \frac{1}{p}>0$ in the third one. In the last step, the result from \cref{thm:Proofs_omega_tilde_domain}\ref{thm:Proofs_omega_tilde_domain - summing of Besvo norms}\ref{thm:Proofs_omega_tilde_domain - summing of Besvo norms - general part} has been applied. Thus, 
    \begin{align*}
        \#\Gamma_j\lesssim_{\,d,q, s_1,s_2,\alpha_1, \alpha_2, \kappa_{\calP_0}, \mu(\calP_0), |\Omega_T|,\LipProp(\Omega)}\min\left(2^{j\left(1+\frac{s_2}{s_1d}\right)}, \frac{1}{\delta^q }2^{-jq\left(1+\frac{s_2}{s_1d}\right)\left(\frac{1}{\frac{1}{s_1}+\frac{d}{s_2}}-\frac{1}{q}+ \frac{1}{p}\right)} |f|_{B^{\alpha_1+s_1,\alpha_2+s_2}_{q,q}(\Omega_T)}^q \right).
    \end{align*}
    Now the identical calculations to those from \cite[Eq.~(57)--(60)]{MSS26} show the asserted estimate in \eqref{thm:direct_estimates - equation 1} for $\#\overline{\calM}$ instead of $\#\calP-\#\calP_0$, for $q<\infty$. The desired result itself is now a consequence of the complexity result from \cref{Thm:complexity results}. The case $q=\infty$ can be achieved similarly with the usual modifications. Now we show the second assertion for $p<\infty$, the case $p=\infty$ works, again, correspondingly. Let $C$ be constant from \eqref{thm:direct_estimates - equation 1}. First, we consider the case, where $\varepsilon^{\frac{1}{s_1}+\frac{d}{s_2}}> C$, in particular, $\varepsilon>1$ due to $C\ge 1$. Then, \eqref{thm:direct_estimates - equation 1} implies that $\#\calP - \#\calP_0=0$, i.e., $\calP=\calP_0$. Therefore, for $(\alpha_1, \alpha_2)\neq (0,0)$,
    \begin{align*}
        \|f-\pi_{\rho,\calP_{0}}(f)\|_{B^{\alpha_1, \alpha_2}_{p,p}(\Omega_T)}^p&\lesssim_{\,\substack{d,\rho, p,s_1, s_2, \alpha_1, \alpha_2, \kappa_{\calP_0} , \\\mu(\calP_0), a(\calP_0),\LipProp(\Omega)}}\sum\limits_{I\times S\in \calP_0}\left(|f|_{B^{\alpha_1, \alpha_2}_{p,p}\left(\widetilde{\omega_{\calP_0}^{j(d)}}(I\times S)\right)}^{*, (r_1, r_2)}\right)^p
        \\&\lesssim_{\,\substack{d,p,q,s_1, s_2,\alpha_1, \alpha_2, \kappa_{\calP_0}, \\\mu(\calP_0), a(\calP_0), \LipProp(\Omega)}} \sum\limits_{I\times S\in \calP_0} |I\times S|^{p\left(\frac{1}{\frac{1}{s_1}+\frac{d}{s_2}}-\frac{1}{q}+\frac{1}{p}\right)}|f|_{B^{\alpha_1+s_1, \alpha_2+s_2}_{q,q}(\Enl(\widetilde{\omega_{\calP_{0}}^{j(d)}}(I\times S)))}^p
        \\&\lesssim_{\,\substack{d,p,q,s_1, s_2,\alpha_1, \alpha_2, \kappa_{\calP_0},\\\mu(\calP_0),|\Omega_T|, \LipProp(\Omega),\#\calP_0}} |f|_{B^{\alpha_1+s_1, \alpha_2+s_2}_{q,q}(\Omega_T)}^p\le \varepsilon^p |f|_{B^{\alpha_1+s_1, \alpha_2+s_2}_{q,q}(\Omega_T)}^p,
    \end{align*}
   due to \cref{lem:Besovnorm behaves a little like L_p} in the first step, \cref{lem:Whitney for besov-seminorms} in the second, \mbox{\cref{thm:Proofs_omega_tilde_domain}\ref{thm:Proofs_omega_tilde_domain - summing of Besvo norms}\ref{thm:Proofs_omega_tilde_domain - summing of Besvo norms - general part}} as before above in the third one, and $\varepsilon>1$ in the last one. Now we can turn to the case, where $\varepsilon^{\frac{1}{s_1}+\frac{d}{s_2}}\le C$. Then, similarly
    \begin{align*}
        \|f-\pi_{\rho,\calP}(f)\|_{B^{\alpha_1, \alpha_2}_{p,p}(\Omega_T)}^p&\lesssim_{\,\substack{d,\rho, p,q, s_1, s_2, \alpha_1, \alpha_2, \kappa_{\calP_0},\\ \mu(\calP_0), a(\calP_0), \LipProp(\Omega)}} \sum\limits_{I\times S\in \calP} |I\times S|^{p\left(\frac{1}{\frac{1}{s_1}+\frac{d}{s_2}}-\frac{1}{q}+\frac{1}{p}\right)}|f|_{B^{\alpha_1+s_1, \alpha_2+s_2}_{q,q}(\Enl(\widetilde{\omega_{\calP}^{j(d)}}(I\times S)))}^p
        \\&\le \#\calP\, \delta^p \le \left(\#\calP_0 + C\varepsilon^{-\left(\frac{1}{s_1}+\frac{d}{s_2}\right)}\right)  \delta^p
        \le C(\#\calP_0 +1)\varepsilon^{-\left(\frac{1}{s_1}+\frac{d}{s_2}\right)}\delta^p
        \\&\lesssim_{\,\substack{d,p,q, s_1,s_2, \alpha_1, \alpha_2, \kappa_{\calP_0}, \mu(\calP_0),\\ |\Omega_T|,\LipProp(\Omega), \#\calP_0}} \varepsilon^{-\left(\frac{1}{s_1}+\frac{d}{s_2}\right)}\varepsilon^{p+\frac{1}{s_1}+\frac{d}{s_2}}|f|_{B^{\alpha_1+s_1, \alpha_2+s_2}_{q,q}(\Omega_T)}^p 
        \\&= \varepsilon^p |f|_{B^{\alpha_1+s_1, \alpha_2+s_2}_{q,q}(\Omega_T)}^p,
    \end{align*}
    where, the fact that the algorithm has terminated with $\calP$ has been additionally applied in the second step, as well as \eqref{thm:direct_estimates - equation 1} in the third, $\varepsilon^{\frac{1}{s_1}+\frac{d}{s_2}}\le C$ in the fourth, and the definition of $\delta$ in the penultimate one. For $(\alpha_1, \alpha_2)= (0,0)$, we can obtain corresponding results, using $\|f-\pi_{\rho,\calP}(f)\|_{L_{p}(\Omega_T)}^p=\sum\limits_{I\times S\in \calP} \|f-\pi_{\rho,\calP}(f)\|^p_{L_{p}(I\times S)}$, \cref{thm:local_boundedness_of_pi_projection} and \cref{lem:Extended_Whitney_Estimate} instead of \cref{lem:Besovnorm behaves a little like L_p} and \cref{lem:Whitney for besov-seminorms}.
\end{proof}

\begin{dis}\label{disc:comparison_continuous_discontinuous}
It is noteworthy that for $(\alpha_1, \alpha_2)=(0,0)$, this result for continuous finite elements yields the same qualitative dependence on the parameters as the result for the discontinuous case from \cite[Thm.~1.5]{MSS26}. Therefore, from an approximation theoretical viewpoint, there is no significant advantage in using discontinuous anisotropic space-time finite elements over continuous ones. This result  is consistent to the findings in \cite{Vee16}, where it was shown that for the approximation of stationary functions with respect to the (broken) Sobolev-norm on a conforming triangulation, there is no qualitative advantage to  using discontinuous FEM over continuous ones. A local version of this comparison, for the Scott--Zhang projection onto continuous finite elements on the Dirichlet boundary, can already be found in \cite[Prop.~3.1]{AFKPP13}, where it is used to control the Dirichlet data oscillations.
\end{dis}

\subsection{Inverse estimates}\label{subsect:inverse estimates}

The goal of this section is to show inverse estimates, i.e., \cref{Main_res:invers}. From now on we will always work under the following assumptions on the parameters. 

\begin{assumptions}\label{assumptions:inverse_estimates_2}
Let \cref{assumptions:multiscale-decomposition} hold with \mbox{$\frac{1}{p}+\frac{1}{\frac{1}{s_1}+\frac{d}{s_2}}=\frac{1}{q}$}, which, in particular, implies $q<p$. Additionally, we require $\rho\le q$  and use the notation $q^*:=\min(1,q)$. Further, $\calP$ will always represent a space-time partition created from $\calP_0$ by finitely many applications of $\textup{PATCH}\_\textup{REFINE}(\cdot, \cdot, d, s_1, s_2)$. We will denote $\calP\in \,\mathbb{P}_N:=\mathbb{P}_N(\calP_0)$, if at most $N\in \N_0$ applications of this method were used to create $\calP$ from $\calP_0$.
\end{assumptions}
\begin{rem}
    If $\calP\in \mathbb{P}_N$, then $\#\calP - \#\calP_0 \lesssim_{\,d, s_1, s_2, \kappa_{\calP_0}, \mu(\calP_0)} N$, due to \cref{Thm:complexity results}.
\end{rem}

For anisotropic finite elements on $\calP\in \mathbb{P}_N$, we can estimate higher order generalized (anisotropic) Besov \mbox{(quasi-)}norms appropriately by those of lesser order. Further, generalized (anisotropic) Besov (quasi-)norms with small enough regularity can even be estimated in terms of the Lebesgue (quasi-)norm, as can be  seen in the following lemma. It corresponds to \cite[Thm.~7.3]{GM14} (in the stationary case) and is an estimate of Bernstein type.

\begin{lem}[Bernstein-type inequality]\label{lem:Generalized_Besov_norm_Bernstein_estimates}
Let \cref{assumptions:inverse_estimates_2} hold true. Then, for every $N\in \N$, $\calP\in \mathbb{P}_N$, and $F\in \V^{r_1, r_2}_{\calP}$, we can estimate
    \begin{align*}
        \|F\|_{\widehat{B}^{\alpha_1+s_1, \alpha_2+s_2}_{q,q}(\Omega_T)} \lesssim_{\,d, s_1, s_2,r_1, r_2,\kappa_{\calP_0}, \mu(\calP_0), \#\calP_0} N^\frac{1}{\frac{1}{s_1}+\frac{d}{s_2}}\|F\|_{\widehat{B}^{\alpha_1,\alpha_2}_{p,p}(\Omega_T)}.
    \end{align*}
    Additionally, if $(s_1', s_2')\in \R^+(s_1, s_2)$ with $s_i'<s_i$, $i=1,2$, then 
    \begin{align*}
        \|F\|_{\widehat{B}^{s_1', s_2'}_{q,q}(\Omega_T)} \lesssim_{\,d, \rho, p, s_1, s_2, s_2',r_1, r_2,\kappa_{\calP_0}, \mu(\calP_0), \#\calP_0} N^\frac{1}{\frac{1}{s_1}+\frac{d}{s_2}}\|F\|_{L_{p}(\Omega_T)}.
    \end{align*}
\end{lem}

\begin{proof}
    Consider the multiscale approximation of $F$ from \cref{rem:Generalized_Besov_Spaces}\ref{rem:Generalized_Besov_Spaces - item 3} and define 
    \begin{align*}
    \Lambda:=\{(n,\nu)\mid n\in \N_0, \nu \in \calL(\calP_n)\text{ with } b_{\nu}^{(n)}\neq 0\},~i.e.,~F = \sum\limits_{\lambda \in \Lambda} b_\lambda(F)\phi_\lambda,
    \end{align*}
    where $b_\lambda(F) = b_{\nu}^{(n)}(F)$ and $\phi_\lambda = \phi_{\nu}^{(n)}$ for $\lambda = (n, \nu) \in \Lambda$. Further, consider the tree $\mathbb{T}_\calP$ of all the prisms which have been created during the sequence of applications of the $\textup{PATCH}\_\textup{REFINE}(\cdot,\cdot, d, s_1, s_2)$ procedure corresponding to $\calP$. Then, surely $\Lambda\subset \bigcup\limits_{I\times S\in \mathbb{T}_\calP} \left\{(\ell(I\times S), \nu') \mid \nu'\in \calL(I\times S)\right\}$, which yields $\#\Lambda\lesssim_{\, d, r_1, r_2} \#\mathbb{T}_\calP$. Since every application of $\textup{PATCH}\_\textup{REFINE}(\cdot, \cdot, d, s_1, s_2)$ produces at most $\sim_{\,d, s_1, s_2, \kappa_{\calP_0}, \mu(\calP_0)}\,2^{\left\lceil 1+\frac{s_2}{s_1d} \right\rceil}$ new elements, we obtain
    \begin{align}\label{lem:Generalized_Besov_norm_Bernstein_estimates - proof 1}
        \#\Lambda\lesssim_{\, d, r_1, r_2} \#\mathbb{T}_\calP \lesssim_{\,d, s_1, s_2, \kappa_{\calP_0}, \mu(\calP_0)} \#\calP_0 + N\lesssim_{\, \#\calP_0} N.
    \end{align}
    Now,
    \begin{align*}
        \|F\|_{\widehat{B}^{\alpha_1+s_1, \alpha_2+s_2}_{q, q}} 
        &= \left(\sum\limits_{\lambda\in \Lambda} |\supp \phi_\lambda|^{-\frac{q}{\frac{1}{\alpha_1+s_1}+\frac{d}{\alpha_2+s_2}}} \left\|b_\lambda(f)\phi_{\lambda}\right\|_{L_q(\Omega_T)}^q\right)^\frac{1}{q}
        \\
        &\le (\#\Lambda)^\frac{1}{\frac{1}{s_1}+\frac{d}{s_2}}\left(\sum\limits_{\lambda\in \Lambda} |\supp \phi_\lambda|^{-\frac{p}{\frac{1}{\alpha_1+s_1}+\frac{d}{\alpha_2+s_2}}} \left\|b_\lambda(f)\phi_{\lambda}\right\|_{L_q(\Omega_T)}^p\right)^\frac{1}{p}
        \\
        &\le (\#\Lambda)^\frac{1}{\frac{1}{s_1}+\frac{d}{s_2}}\left(\sum\limits_{\lambda\in \Lambda} |\supp \phi_\lambda|^{-\frac{p}{\frac{1}{\alpha_1+s_1}+\frac{d}{\alpha_2+s_2}}} |\supp \phi_\lambda|^{\frac{p}{\frac{1}{s_1}+\frac{d}{s_2}}} \left\|b_\lambda(f)\phi_{\lambda}\right\|_{L_p(\Omega_T)}^p\right)^\frac{1}{p}
    \end{align*}
    where in the first inequality we have used H\"older's inequality for sums and in the last step we have used H\"older's inequality for integrals and the fact that $\frac{q}{p}+\frac{q}{\frac{1}{s_1}+\frac{d}{s_2}}=1$. 
    Finally, since there is a constant $C$ such that  $s_i = C \alpha_i$, for $i=1,2$ (collinearity of $(\alpha_1+s_1, \alpha_2+s_2)$ and $(s_1, s_2)$) we have that $-\frac{p}{\frac{1}{\alpha_1+s_1}+\frac{d}{\alpha_2+s_2}}+ \frac{p}{\frac{1}{s_1}+\frac{d}{s_2}} = -\frac{p}{\frac{1}{\alpha_1}+\frac{d}{\alpha_2}}$,  whence
\begin{align*}
    \|F\|_{\widehat{B}^{\alpha_1+s_1, \alpha_2+s_2}_{q, q}} 
        &\le (\#\Lambda)^\frac{1}{\frac{1}{s_1}+\frac{d}{s_2}}\left(\sum\limits_{\lambda\in \Lambda} |\supp \phi_\lambda|^{-\frac{p}{\frac{1}{\alpha_1}+\frac{d}{\alpha_2}}} \left\|b_\lambda(f)\phi_{\lambda}\right\|_{L_p(\Omega_T)}^p\right)^\frac{1}{p}
        \\ 
        &\lesssim_{\,d, s_1, s_2,r_1, r_2,\kappa_{\calP_0}, \mu(\calP_0), \#\calP_0} N^\frac{1}{\frac{1}{s_1}+\frac{d}{s_2}}\left(\sum\limits_{\lambda\in \Lambda} |\supp \phi_\lambda|^{-\frac{p}{\frac{1}{\alpha_1}+\frac{d}{\alpha_2}}} \left\|b_\lambda(f)\phi_{\lambda}\right\|_{L_p(\Omega_T)}^p\right)^\frac{1}{p}
        \\&= N^\frac{1}{\frac{1}{s_1}+\frac{d}{s_2}} \|f\|_{\widehat{B}^{\alpha_1, \alpha_2}_{p,p}(\Omega_T)},
    \end{align*}
    where we have used \eqref{lem:Generalized_Besov_norm_Bernstein_estimates - proof 1} in the penultimate step. This shows the first asserted estimate. Now choosing $\alpha_i:=s_i'-s_i<0$, $i=1,2$, we obtain 
    \begin{align}\label{lem:Generalized_Besov_norm_Bernstein_estimates - proof 2}
        \|F\|_{\widehat{B}^{s_1', s_2'}_{q,q}(\Omega_T)}=\|F\|_{\widehat{B}^{\alpha_1+s_1, \alpha_2+s_2}_{q,q}(\Omega_T)} \lesssim_{\,d, s_1, s_2,r_1, r_2,\kappa_{\calP_0}, \mu(\calP_0), \#\calP_0} N^\frac{1}{\frac{1}{s_1}+\frac{d}{s_2}}\|F\|_{\widehat{B}^{\alpha_1,\alpha_2}_{p,p}(\Omega_T)}
    \end{align}
    by applying the estimate which we have just shown. Note that this notation is slightly abusive due to $\alpha_i<0$, \mbox{$i=1,2$}, since generalized Besov spaces with negative smoothness were not formally defined. Having a look at the beginning of the proof of \mbox{\cref{lem:norm equivalence_gen_besov}}, we observe that $\|F\|_{\widehat{B}^{\alpha_1,\alpha_2}_{p,p}(\Omega_T)}\sim_{} \|F\|_{\widehat{B}^{\alpha_1,\alpha_2}_{p,p}(\Omega_T)}^\Delta$ also holds for negative smoothness parameters with a constant only depending on $d,p, r_1, r_2, \kappa_{\calP_0}$, and $\mu(\calP_0)$. Further, due to $\Delta_n(f)=\pi_n(f)-\pi_{n-1}(f)$ and \cref{thm:operater_pi_quasi-best_approx}, \mbox{$\|\Delta_n(F)\|_{L_p(\Omega_T)}\lesssim_{\,d,\rho, p, s_1, s_2 , r_1, r_2, \kappa_{\calP_0}, \mu(\calP_0)}\|F\|_{L_p(\Omega_T)}$} holds true with a uniform bound independent of $n\in \N_0$. Thus,
    \begin{align*}
        \|F\|_{\widehat{B}_{p,p}^{\alpha_1, \alpha_2}(\Omega_T)}\sim_{\,\substack{d,p,r_1,r_2,\\ \kappa_{\calP_0},\mu(\calP_0)}}\|F\|^\Delta_{\widehat{B}_{p,p}^{\alpha_1, \alpha_2}(\Omega_T)}\lesssim_{\,\substack{d,\rho, p, s_1, s_2 , r_1,\\ r_2, \kappa_{\calP_0}, \mu(\calP_0)}}\|F\|_{L_p(\Omega_T)} \left(\sum\limits_{n=0}^\infty\left(2^{\frac{\alpha_2}{d}p}\right)^n\right)^\frac1p\lesssim_{\,d,p,s_2, s_2'}\|F\|_{L_p(\Omega_T)},
    \end{align*}
    where we have exploited that the last sum converges since $\alpha_2=s_2'-s_2<0$. Combining this with \mbox{\eqref{lem:Generalized_Besov_norm_Bernstein_estimates - proof 2}} yields the second assertion.
\end{proof}

\begin{rem}
    Exploiting the result from \mbox{\cref{cor:to_thm:Embedding_generalized_into_classical_aniso_spaces}}, we even obtain a Bernstein-type estimate for the classical anisotropic Besov and Lebesgue norms, if $s_i'<1+\frac1q$, $i=1,2$, in the second version of the above lemma. This is an anisotropic version of \cite[Thm.~6.1(b)]{BDDP02}.
\end{rem}

Further, we need a similar result for the Lebesgue norm of the finite elements on the partitions of the multiscale decomposition $\calP_n$ that corresponds to the one given in \cite[Eq.~(7.2)]{DP88}.

\begin{lem}\label{lem:Lebesgue_norm_estimate_for_finite_elements}
    Let \cref{assumptions:inverse_estimates_2} hold and consider $F\in \V^{r_1, r_2}_{\calP_n}$ for some $n\in \N_0$. Then
    \begin{align*}
        \|F\|_{L_p(\Omega_T)}\lesssim_{\,d, p, s_1, s_2, r_1, r_2, \mu(\calP_0)} 2^{n\frac{s_2}{d}} \|F\|_{L_q(\Omega_T)}
    \end{align*}
\end{lem}
\begin{proof}
    Assume $p<\infty$, the case $p=\infty$ is very similar. We use \cite[Lem.~3.5]{MSS26} in the second step below, and, in the third one, $|I\times S|\sim_{\mu(\calP_0)} 2^{-n\left(1+\frac{s_2}{s_1d}\right)}$, the identity for $q, p, s_1$, and $s_2$, and the embedding $\ell^\frac{q}{p}(\N_0)\hookrightarrow\ell^1(\N_0) $ due to $q<p$; the last step is a direct computation. Altogether this yields
    \begin{align*}
        \|F\|_{L_p(\Omega_T)}^p&= \sum\limits_{I\times S\in \calP_n} \|F\|_{L_p(I\times S)}^p\sim_{\,d, p, r_1, r_2} \sum\limits_{I\times S\in \calP_n} |I\times S|^{1-\frac{p}{q}} \|F\|_{L_q(I\times S)}^p
        \\&\lesssim_{\,d,p, s_1, s_2, \mu(\calP_0)}2^{-n\left(1+\frac{s_2}{s_1d}\right)\left(-\frac{1}{\frac{1}{s_1}+\frac{d}{s_2}}\right)p}\left(\sum\limits_{I\times S\in \calP_n}\|F\|_{L_q(I\times S)}^q\right)^\frac{p}{q}=2^{n\frac{s_2}{d}p}\|F\|_{L_q(\Omega_T)}^p.
        \qedhere
    \end{align*}
\end{proof}

In particular, this result allows us to show embeddings between generalized anisotropic Besov spaces of different smoothness, similar to the results from \cite[Thm.~1.3]{MSS26} and \cref{thm:Sobolev-type-embedding} for the classical anisotropic spaces.

\begin{lem}\label{lem:Embedding between anisotropic spaces}
    The embedding $\widehat{B}^{\alpha_1+s_1, \alpha_2+s_2}_{q, q}(\Omega_T) \hookrightarrow \widehat{B}^{\alpha_1,\alpha_2}_{p,p}(\Omega_T)$ is continuous with an embedding constant only depending on $d, p, s_1, s_2, r_1, r_2$, and $\mu(\calP_0)$. 
\end{lem}
\begin{proof}
    We estimate
    \begin{align*}
        \|f\|^\Delta_{\widehat{B}^{\alpha_1,\alpha_2}_{p,p}(\Omega_T)} &=\left(\sum\limits_{n=0}^{\infty} 2^{n\frac{\alpha_2}{d}p}\|\Delta_n(f)\|^p_{L_p(\Omega_T)}\right)^\frac{1}{p}\lesssim_{\,d, p, s_1, s_2, r_1, r_2, \mu(\calP_0)} \left(\sum\limits_{n=0}^{\infty} 2^{n\frac{\alpha_2+s_2}{d}p}\|\Delta_n(f)\|^p_{L_q(\Omega_T)}\right)^\frac{1}{p}
        \\ &\le \left(\sum\limits_{n=0}^{\infty} 2^{n\frac{\alpha_2+s_2}{d}q}\|\Delta_n(f)\|^q_{L_q(\Omega_T)}\right)^\frac{1}{q} = \|f\|_{\widehat{B}^{\alpha_1+s_1, \alpha_2+s_2}_{q,q}(\Omega_T)}^\Delta, \quad f\in \widehat{B}^{\alpha_1+s_1, \alpha_2+s_2}_{q,q}(\Omega_T),
    \end{align*}
    where we have used \cref{lem:Lebesgue_norm_estimate_for_finite_elements} in the second step and the embedding $\ell^{q}(\N_0)\hookrightarrow \ell^{p}(\N_0)$, due to $q < p$, in the third one. 
\end{proof}

The results that we show in the following are closely related to sequence spaces and (their) real interpolation. 
We now introduce the corresponding definitions and results for the sake of completeness.

\begin{defi}\label{Def:Generalized_Sequence_Spaces}
For $\gamma\in \R$, $\eta\in (0,\infty)$, and a (quasi-)Banach space $X$, we define the \textbf{generalized sequence spaces} via
\begin{align*}
	\ell^{\gamma, \eta}(X):=\bigg\{a:\N_0\rightarrow X \ \bigg| \ \|a\|_{\ell^{\gamma, \eta}(X)}:=\Big(\sum\limits_{k=0}^{\infty}\left(2^{\gamma k} \|a_k\|_X\right)^\eta\Big)^\frac{1}{\eta}<\infty \bigg\}.
\end{align*}
The space $\ell^{\gamma, \infty}(X)$ and its norm are defined in the same manner with the usual modification. These spaces are themselves \mbox{(quasi-)}Banach spaces with (quasi-)norm $\|\cdot\|_{\ell^{\gamma, t}(X)}$. For more information on these spaces, the reader is referred to \cite[Ch.~1.18.1]{Tri78}.
\end{defi}

\begin{rem}
    For us, it will be crucial to exploit the identity $\|f\|_{\widehat{B}^{s_1,s_2}_{q,q}(\Omega_T)}=\left\|(\Delta_n(f))_{n\in \N_0}\right\|_{\ell^{\frac{s_2}{d},q}(L_q(\Omega_T))}$.
\end{rem}

\begin{defi}\label{defi:real interpolation}
    Let $X_1$ and $X_2$ be (quasi-)Banach spaces with corresponding (quasi-)norms $\|\cdot\|_{X_i}$, $i=1,2$. Then, the \textbf{$K$-functional} of $X_1$ and $X_2$ is defined for $x\in X_1+X_2$ as
    \begin{align*}
        K(x,\delta, X_1, X_2):=\inf\limits_{\substack{x_1\in X_1,\ x_2\in X_2\\ x_1+x_2=x}} \|x_1\|_{X_1} + \delta\|x_2\|_{X_2}, \quad \delta > 0.
    \end{align*}
    Using this functional, we can define the \textbf{(real) interpolation space of $X_1$ and $X_2$} for any $\theta \in (0,1)$ and $\eta\in (0,\infty)$ via
    \begin{align*}
        (X_1, X_2)_{(\theta, \eta)}:=\bigg\{x\in X_1+X_2\ \Big| \ \|x\|_{(X_1, X_2)_{(\theta, \eta)}}:=\left(\int_0^\infty \Big[\delta^{-\theta}K(x,\delta, X_1, X_2)\Big]^{\eta}\frac{d\delta}{\delta}\right)^\frac{1}{\eta} <\infty\bigg\}.
    \end{align*}
\end{defi}

For the general theory regarding the interpolation of quasi-Banach spaces with respect to the \enquote{$K$-method}, we refer to \cite[Ch.~1.1--1.3 and Ch.~1.7]{Tri78}\footnote{The theory and the proofs in this source are generally only given for the Banach space case. Nevertheless, in several remarks the author explains that the theory can be extended easily to the quasi-Banach space setting.}. For our further investigations, we only mention some important results related to these interpolation spaces.

\begin{rem}\label{rem:on_interpolation_spaces}
    \begin{enumerate}[label=(\roman*)]
        \item For any $a\in (1,\infty)$, the discrete interpolation (quasi-)norm given by 
            \begin{align*}
                \|x\|^*_{(X_1, X_2)_{(\theta, \eta)}}:=\left(\sum\limits_{k=-\infty}^{\infty} \left[a^{k\theta} K(x, a^{-k}, X_1, X_2)\right]^\eta \right)^\frac{1}{\eta}
            \end{align*}
            is equivalent to $\|\cdot \|_{(X_1, X_2)_{(\theta, \eta)}}$ on $(X_1, X_2)_{(\theta, \eta)}$ with equivalence constants depending only on $a$, $\theta$, and $\eta$. This can be shown according to   part (a) of the theorem in \cite[Ch.~1.7]{Tri78}. \label{rem:on_interpolation_spaces - discrete norm}
        \item \label{rem:on_interpolation_spaces - bound on discrete norm if X_2 embedded in X_1}Assume $a>1$ as above in (i), then
        \begin{align*}
            \|x\|^*_{(X_1, X_2)_{(\theta, \eta)}}\lesssim_{\, a, \theta, \eta} \|x\|^{**}_{(X_1, X_2)_{(\theta, \eta)}}:=\|x\|_{X_1}+\left(\sum\limits_{k=0}^{\infty} \left[a^{k\theta} K(x, a^{-k}, X_1, X_2)\right]^\eta \right)^\frac{1}{\eta} \quad \text{for} \quad x\in X_1.
        \end{align*}
        \begin{proof}
            Since any $x\in X_1$ can be decomposed into $x + 0\in X_1+X_2$, we can estimate $K(x, \delta, X_1, X_2)\le \|x_1\|_{X_1}$ for any $\delta\in [0,\infty)$. Accordingly,
            \begin{align*}
                \left(\sum\limits_{k=-\infty}^{-1} \left[a^{k\theta} K(x, a^{-k}, X_1, X_2)\right]^\eta \right)^\frac{1}{\eta} \le \|x_1\|_{X_1} \left(\sum\limits_{k=1}^{\infty} \left(a^{-\theta\eta}\right)^k \right)^\frac{1}{\eta}\sim_{\, a, \theta, \eta} \|x_1\|_{X_1},
            \end{align*}
            since the last sum is geometric with $a^{-\theta\eta}<1$.
        \end{proof}
        \item An important result that we will need below, is a result on the interpolation of generalized sequence spaces which can be found, for example, in \cite[Thm.~4 of Ch.~5]{Pee76} and \cite[Thm.~5.6.2]{BL76}\footnote{Again, technically, in the second source, the result is only stated directly in the case where $Y_1$ and $Y_2$ are Banach spaces. But the proof also works in the quasi-Banach space setting.}. Let $Y_1$, $Y_2$ be (quasi-)Banach spaces, $\gamma, \gamma_1, \gamma_2 \in \R$, $\eta, \eta_1, \eta_2\in (0,\infty]$, and $\theta \in (0,1)$ with $\gamma=(1-\theta) \gamma_1 + \theta\gamma_2$ as well as $\frac{1}{\eta}=\frac{1-\theta}{\eta_1} + \frac{\theta}{\eta_2}$, then
        \begin{align*}
            \left(\ell^{\gamma_1, \eta_1}(Y_1), \ell^{\gamma_2, \eta_2}(Y_2)\right)_{(\theta, \eta)}= \ell^{\gamma, \eta}\left((Y_1, Y_2)_{(\theta, \eta)}\right)
        \end{align*}
        in the sense of equivalent (quasi-)norms. The corresponding constants depend on $\gamma_i$, $\eta_i$, $i=1,2$, $\theta$, and the (quasi-)triangle inequality constants of $Y_i$, $i=1,2$. \label{rem:on_interpolation_spaces - sequence spaces}
        \item\label{rem:on_interpolation_spaces - Lebesgue spaces special case} A similar result holds true for Lebesgue spaces, as proven in \cite[Ch.~1.18.4]{Tri78}. Let $\eta, \eta_1, \eta_2,$ and $\theta\in (0,1)$ as above in \ref{rem:on_interpolation_spaces - sequence spaces}. Then, 
        \begin{align*}
            \left(L_{\eta_0}(\Omega_T), L_{\eta_1}(\Omega_T)\right)_{(\theta, \eta)}= L_{\eta}(\Omega_T)
        \end{align*}
        holds true in the sense of equivalent (quasi-)norms with involved constants depending on $\eta_1$, $\eta_2$, and $\theta$.
    \end{enumerate}
\end{rem}

These results lead to the following lemma that we will need at the end of the section when we prove the inverse estimate result. 

\begin{lem}\label{lem:interpolation_of_certain_gen_sequence_spaces}
    Let the parameters be chosen as in \cref{rem:on_interpolation_spaces}\ref{rem:on_interpolation_spaces - sequence spaces}. Then the following equality holds in the sense of equivalent (quasi-)norms: 
    \begin{align*}
        \left(\ell^{\gamma_1, \eta_1}\left(L_{\eta_1}(\Omega_T)\right), \ell^{\gamma_2, \eta_2}\left(L_{\eta_2}(\Omega_T)\right)\right)_{(\theta, \eta)} = \ell^{\gamma, \eta}\left(L_\eta(\Omega_T)\right).
    \end{align*}
    The equivalence constants depend on the parameters $\gamma_i$, $\eta_i$, $i=1,2$, and $\theta$.
\end{lem}
\begin{proof}
    This is a direct consequence of the iterative application of \cref{rem:on_interpolation_spaces}\ref{rem:on_interpolation_spaces - sequence spaces} and \ref{rem:on_interpolation_spaces - Lebesgue spaces special case}.
\end{proof}

Further, we can relate the $K$-functional of a function $f\in \widehat{B}^{\alpha_1, \alpha_2}_{p,p}(\Omega_T)$ with respect to the generalized anisotropic Besov spaces to that of its multiscale decomposition with respect to generalized sequence spaces. A corresponding result has been proven in \cite[Thm.~6.1]{DP88}.

\begin{lem}\label{lem:estimate between different K functionals}
    Let \cref{assumptions:inverse_estimates_2} hold. Choose $f\in \widehat{B}^{\alpha_1, \alpha_2}_{p,p}(\Omega_T)$ and $\delta\in [0,\infty)$. Then,
    \begin{align*}
        K\left((\Delta_n(f))_{n\in \N_0}, \delta, \ell^{\frac{\alpha_2}{d}, p}(L_p(\Omega_T)), \ell^{\frac{\alpha_2+s_2}{d}, q}(L_q(\Omega_T))\right) \lesssim K\left(f,\delta,\widehat{B}^{\alpha_1, \alpha_2}_{p,p}(\Omega_T), \widehat{B}^{\alpha_1+s_1, \alpha_2+s_2}_{q, q}(\Omega_T)\right) .
    \end{align*}
    with a constant depending only on $d, \rho, p, s_1, s_2, \alpha_2, r_1, r_2, \kappa_{\calP_0}$, and $\mu(\calP_0)$.
\end{lem}

\begin{proof}
    Let $g\in \widehat{B}^{\alpha_1+s_1, \alpha_2+s_2}_{q, q}(\Omega_T)$, then $f-g\in \widehat{B}^{\alpha_1, \alpha_2}_{p,p}(\Omega_T)$ according to \cref{lem:Embedding between anisotropic spaces}, and recall that 
    \begin{align*}
        \pi_n(f)= \mathcal{Q}_{\calP_n}\left(\calB_{\rho, \calP_n}(f)\right)=\mathcal{Q}_{\calP_n}\left(\sum\limits_{I\times S\in \calP_n} \mathds{1}_{I\times S} \calB_{\rho, I\times S}(f)\right), 
    \end{align*}
    where $\calB_{\rho, I\times S}(f)$ is a best approximation of $f$ in \mbox{$\Pi^{r_1, r_2}_{t,\bm{x}}(I\times S)$} with respect to \mbox{$L_\rho(I\times S)$}. According to \cite[Lem.~6.2]{DP88}, for every $n\in \N_0$ and $I\times S\in \calP_n$, there are (quasi-)best approximations $\overline{\calB}_{\rho, I\times S}(g)$ of $g$ in $\Pi^{r_1, r_2}_{t,\bm{x}}(I\times S)$ with respect to \mbox{$L_\rho(I\times S)$} such that $\overline{\calB}_{\rho, I\times S}(f-g):=\calB_{\rho, I\times S}(f)-\overline{\calB}_{\rho, I\times S}(g)$ is a \mbox{(quasi-)}best approximation of $f-g$. Further, the corresponding constants involved in the (quasi-)best approximation property only depend on the constant involved in the (quasi-)triangle inequality of the space $L_\rho(I\times S)$ which only depends on $\rho$. Now, for $h\in \{g, f-g\}$ and $n\in \N_0$, set
    \begin{align*}
        \overline{\pi}_n(h):=\mathcal{Q}_{\calP_n}\left(\overline{\calB}_{\rho, \calP_n}(h)\right):=\mathcal{Q}_{\calP_n}\left(\sum\limits_{I\times S\in \calP_n} \mathds{1}_{I\times S} \overline{\calB}_{\rho, I\times S}(h)\right).
    \end{align*}
    Since we have chosen $\rho\le q<p$, we can show 
    \begin{align*}
        \left\|g-\overline{\pi}_n(g)\right\|_{L_q(\Omega_T)}&\lesssim_{\, d,\rho, p, s_1, s_2, r_1, r_2, \kappa_{\calP_0}, \mu(\calP_0)} E(g,\V^{r_1, r_2}_\calP,\Omega_T)_q\quad\text{and}\quad
        \\\left\|(f-g)-\overline{\pi}_n(f-g)\right\|_{L_p(\Omega_T)}&\lesssim_{\, d,\rho, p, s_1, s_2, r_1, r_2, \kappa_{\calP_0}, \mu(\calP_0)} E(f-g,\V^{r_1, r_2}_\calP,\Omega_T)_p
    \end{align*}
    as in \cref{thm:operater_pi_quasi-best_approx}. This allows us to establish 
    \begin{align*}
        &\left\|\left(\overline{\Delta}_n(g)\right)_{n\in \N_0}\right\|_{\ell^{\frac{\alpha_2+s_2}{d}, q}(L_q(\Omega_T))}
        \sim_{\, d, \rho, p, s_1, s_2, \alpha_2,r_1, r_2, \kappa_{\calP_0}, \mu(\calP_0)} \|g\|_{\widehat{B}^{\alpha_1+s_1, \alpha_2+s_2}_{q, q}(\Omega_T)}^{\Delta}\quad\text{and}\quad
        \\&\left\|\left(\overline{\Delta}_n(f-g)\right)_{n\in \N_0}\right\|_{\ell^{\frac{\alpha_2}{d}, p}(L_p(\Omega_T))}
        \sim_{\, d, \rho, p, s_1, s_2, \alpha_2,r_1, r_2, \kappa_{\calP_0}, \mu(\calP_0)} \|f-g\|_{\widehat{B}^{\alpha_1, \alpha_2}_{p, p}(\Omega_T)}^{\Delta}
    \end{align*}
    analogously to \cref{lem:norm equivalence_gen_besov}. Now 
    \begin{align*}
        K\Big((\Delta_n(f))_{n\in \N_0}, \delta, \ell^{\frac{\alpha_2}{d}, p}(&L_p(\Omega_T)), \ell^{\frac{\alpha_2+s_2}{d}, q}(L_q(\Omega_T))\Big)
        \\&\le \left\|\left(\overline{\Delta}_n(f-g)\right)_{n\in \N_0}\right\|_{\ell^{\frac{\alpha_2}{d}, p}(L_p(\Omega_T))} + \delta \left\|\left(\overline{\Delta}_n(g)\right)_{n\in \N_0}\right\|_{\ell^{\frac{\alpha_2+s_2}{d}, q}(L_q(\Omega_T))}
        \\& \sim_{\, d, \rho, p, s_1, s_2, \alpha_2,r_1, r_2, \kappa_{\calP_0}, \mu(\calP_0)}\|f-g\|^{\Delta}_{\widehat{B}^{\alpha_1, \alpha_2}_{p, p}(\Omega_T)}+ \delta \|g\|_{\widehat{B}^{\alpha_1+s_1, \alpha_2+s_2}_{q, q}(\Omega_T)}^{\Delta}
    \end{align*}
    due to $\left(\Delta_n(f)\right)_{n\in \N_0}= \left(\overline{\Delta}_n(f-g)\right)_{n\in \N_0} + \left(\overline{\Delta}_n(g)\right)_{n\in \N_0}$, which is a consequence of the linearity of $\mathcal{Q}_{\calP_n}$. Taking the infimum over all $g\in \widehat{B}^{\alpha_1, \alpha_2}_{p,p}(\Omega_T)$ now yields the assertion.    
\end{proof}

\begin{rem}
    Note that we did not use the classical retraction-coretraction method as in \cite[Proof of Thm.~2.5]{GM14}. This is due to the fact that the method requires the operators to be linear, which might not be the case when $\rho\in (0,1]$. We circumvent this by the above lemma using a corresponding idea from \cite[Sect.~6]{DP88}.
\end{rem}

Additionally, we can bound the $K$-functional of a function by a suitable sum of its $N$-term approximation errors using anisotropic finite elements.

\begin{lem}\label{lem:Generalized_Besov_K_functional_estimate}
Let \cref{assumptions:inverse_estimates_2} hold, $f\in \widehat{B}^{\alpha_1, \alpha_2}_{p,p}(\Omega_T)$, and $n\in \N_0$. Then we can estimate
    \begin{align*}
    K\bigg(f, 2^{-\frac{n}{\frac{1}{s_1}+\frac{d}{s_2}}}, &\widehat{B}^{\alpha_1, \alpha_2}_{p,p}(\Omega_T),  \widehat{B}^{\alpha_1+s_1, \alpha_2+s_2}_{q,q}(\Omega_T)\bigg)
    	\\ &\lesssim_{\, \substack{d,\rho, p, s_1, s_2,\alpha_2, r_1, r_2, \\\kappa_{\calP_0}, \mu(\calP_0),\#\calP_0}} 2^{-\frac{n}{\frac{1}{s_1}+\frac{d}{s_2}}}\left[\left(\sum\limits_{k=0}^n \left(2^{\frac{k}{\frac{1}{s_1}+\frac{d}{s_2}}}\sigma_{2^k}(f)_{\widehat{B}^{\alpha_1, \alpha_2}_{p,p}(\Omega_T)}\right)^{q^*} \right)^\frac{1}{q^*} + \|f\|_{\widehat{B}^{\alpha_1, \alpha_2}_{p,p}(\Omega_T)}\right],
    \end{align*}
	where $\sigma_N(f)_{\widehat{B}^{\alpha_1, \alpha_2}_{p,p}(\Omega_T)}$, $N\in \N_0$, is an abbreviated notation for the corresponding $N$-term approximation error defined in \cref{subsect:main_results}, i.e.,
	\begin{align*}
		\sigma_N(f)_{\widehat{B}^{\alpha_1, \alpha_2}_{p,p}(\Omega_T)}:=\sigma_N(f, \widehat{B}^{\alpha_1, \alpha_2}_{p,p}(\Omega_T), \V^{r_1, r_2}_\bullet, \textup{PATCH}\_\textup{REFINE}(\cdot, \cdot, d, s_1, s_2)).
	\end{align*}
\end{lem}	

\begin{proof}
	For given $N\in \N$, $\mathbb{P}_N$ is finite and $\V^{r_1, r_2}_{\calP}$ is finite dimensional, therefore $\sigma_N(f)_{\widehat{B}^{\alpha_1, \alpha_2}_{p,p}(\Omega_T)}$ is actually attained by some $\calP\in\mathbb{P}_N$ and $F\in \V^{r_1, r_2}_{\calP}$. Let now $F_k\in \V^{r_1, r_2}_{\hat{\calP}_k}$ with $\hat{\calP}_k\in \mathbb{P}_{2^k}$ such that $\|f-F_k\|_{\widehat{B}^{\alpha_1, \alpha_2}_{p, p}}=\sigma_{2^k}(f)_{\widehat{B}_{p,p}^{\alpha_1, \alpha_2}(\Omega_T)} $ for every $ k\in\{0,\dots,n\}$ and $\hat{\calP}_{-1}:=\calP_0$. Additionally define $F_{-1}:=0$ and $R_k:=F_k - F_{k-1}$ for $k\in\{0,\dots,n\}$. Thus, since $F_k\in \V^{r_1, r_2}_{\hat{\calP}_k}\subset \widehat{B}^{\alpha_1+s_1, \alpha_2+s_2}_{q, q}(\Omega_T)$ and $f-F_k\in \widehat{B}^{\alpha_1, \alpha_2}_{p,p}(\Omega_T)$, due to $F_k\in \V^{r_1, r_2}_{\hat{\calP}_k} \subset \widehat{B}^{\alpha_1, \alpha_2}_{p,p}(\Omega_T)$, we can estimate 
    \begingroup
    \allowdisplaybreaks
	\begin{align}\label{lem:Generalized_Besov_K_functional_estimate - proof eq 1}
		2^{\frac{n}{\frac{1}{s_1}+\frac{d}{s_2}}} & K\bigg(f, 2^{-\frac{n}{\frac{1}{s_1}+\frac{d}{s_2}}}, \widehat{B}^{\alpha_1, \alpha_2}_{p, p}, \widehat{B}^{\alpha_1+s_1, \alpha_2+s_2}_{q, q}(\Omega_T)\bigg)\notag
		\\&\le 2^{\frac{n}{\frac{1}{s_1}+\frac{d}{s_2}}}\left(2^{-\frac{n}{\frac{1}{s_1}+\frac{d}{s_2}}}\|F_n\|_{\widehat{B}^{\alpha_1+s_1, \alpha_2+s_2}_{q, q}(\Omega_T)}^{\Delta} + \|f-F_n\|_{\widehat{B}^{\alpha_1, \alpha_2}_{p, p}(\Omega_T)}^{\Delta}\right)\notag
		\\&\lesssim_{\,d,\rho, p, s_1, s_2, \alpha_2, r_1, r_2, \kappa_{\calP_0},\mu(\calP_0)}\left\|\sum\limits_{k=0}^{n} R_k\right\|^{E}_{\widehat{B}^{\alpha_1+s_1, \alpha_2+s_2}_{q, q}(\Omega_T)} + 2^{\frac{n}{\frac{1}{s_1}+\frac{d}{s_2}}}\|f-F_n\|^{\Delta}_{\widehat{B}^{\alpha_1, \alpha_2}_{p, p}(\Omega_T)}\notag
		\\&\le \left(\sum\limits_{k=0}^{n} \left(\left\|R_k\right\|^{E}_{\widehat{B}^{\alpha_1+s_1, \alpha_2+s_2}_{q, q}(\Omega_T)}\right)^{q^*} + 2^{\frac{nq^*}{\frac{1}{s_1}+\frac{d}{s_2}}}\left(\|f-F_n\|^{\Delta}_{\widehat{B}^{\alpha_1, \alpha_2}_{p, p}(\Omega_T)}\right)^{q^*}\right)^\frac{1}{q^*}\notag
        \\&\lesssim_{\,d, p, s_1, s_2, \alpha_2, r_1, r_2, \kappa_{\calP_0},\mu(\calP_0)} \left(\sum\limits_{k=0}^{n} \left\|R_k\right\|_{\widehat{B}^{\alpha_1+s_1, \alpha_2+s_2}_{q, q}(\Omega_T)}^{q^*} + 2^{\frac{nq^*}{\frac{1}{s_1}+\frac{d}{s_2}}}\left(\|f-F_n\|^{\Delta}_{\widehat{B}^{\alpha_1, \alpha_2}_{p, p}(\Omega_T)}\right)^{q^*}\right)^\frac{1}{q^*}
	\end{align}
    \endgroup
	where we have applied a quasi-norm equivalence due to \mbox{\cref{lem:norm equivalence_gen_besov}} and have expanded $F_n$ in a telescoping sum in the second step, used the embedding $\ell^{q^*}(\N_0)\hookrightarrow \ell^{1}(\N_0)$, recall $q*=\min(1,q)\le 1$, and the subadditivity of $\left(\|\cdot\|^{E}_{\widehat{B}^{\alpha_1+s_1, \alpha_2+s_2}_{q, q}(\Omega_T)}\right)^{q^*}$, also due to \cref{lem:norm equivalence_gen_besov}, in the third, and another quasi-norm equivalence due to \mbox{\cref{lem:norm equivalence_gen_besov}} in the fourth one.
    
    Now note that $R_k$ is an anisotropic finite element on the overlay of the meshes $\hat{\calP}_{k-1}$ and $\hat{\calP}_k$ (for a definition, see \mbox{\cref{def:Overlay_of_meshes}}), denoted as $\hat{\calP}_{k-1}\oplus\hat{\calP}_k$. It fulfills the regularity conditions required for meshes according to \cref{thm:overlay_of_meshes_properties} such that, in particular, we are in the position to directly employ the (first) Bernstein-type estimate from \mbox{\cref{lem:Generalized_Besov_norm_Bernstein_estimates}} to $R_k$ with respect to the mesh size
    \begin{align}\label{lem:Generalized_Besov_K_functional_estimate - proof eq 2}
        \#(\hat{\calP}_{k-1}\oplus\hat{\calP}_k)\le\# \hat{\calP}_{k-1}+ \#\hat{\calP}_k\lesssim 2^k,
    \end{align}
    also due to \cref{thm:overlay_of_meshes_properties}. We obtain
    \begingroup
    \allowdisplaybreaks
    \begin{align*}
        2&\phantom{}^{\frac{n}{\frac{1}{s_1}+\frac{d}{s_2}}}  K\bigg(f, 2^{-\frac{n}{\frac{1}{s_1}+\frac{d}{s_2}}}, \widehat{B}^{\alpha_1, \alpha_2}_{p, p}, \widehat{B}^{\alpha_1+s_1, \alpha_2+s_2}_{q, q}(\Omega_T)\bigg)
        \\&\lesssim_{\, d, \rho, s_1, s_2, \alpha_2, r_1, r_2, \kappa_{\calP_0}, \mu(\calP_0),\#\calP_0} \left(\sum\limits_{k=0}^{n} 2^{\frac{kq^*}{\frac{1}{s_1}+\frac{d}{s_2}}} \left\|R_k\right\|^{q^*}_{\widehat{B}^{\alpha_1, \alpha_2}_{p, p}(\Omega_T)} + 2^{\frac{nq^*}{\frac{1}{s_1}+\frac{d}{s_2}}}\left(\|f-F_n\|_{\widehat{B}^{\alpha_1, \alpha_2}_{p, p}(\Omega_T)}^{\Delta}\right)^{q^*}\right)^\frac{1}{q^*}
        \\&\lesssim_{\, d, p, s_1, s_2,\alpha_2, r_1, r_2, \kappa_{\calP_0}, \mu(\calP_0)} \left(\sum\limits_{k=0}^{n} 2^{\frac{kq^*}{\frac{1}{s_1}+\frac{d}{s_2}}} \left(\left\|R_k\right\|^\Delta_{\widehat{B}^{\alpha_1, \alpha_2}_{p, p}(\Omega_T)}\right)^{q^*} + 2^{\frac{nq^*}{\frac{1}{s_1}+\frac{d}{s_2}}}\left(\|f-F_n\|_{\widehat{B}^{\alpha_1, \alpha_2}_{p, p}(\Omega_T)}^{\Delta}\right)^{q^*}\right)^\frac{1}{q^*}
		\\&\lesssim_{\,\substack{d,\rho, p,s_1,s_2,\alpha_2, r_1, r_2, \kappa_{\calP_0},\mu(\calP_0)}}
        \\&\:\:\left(\sum\limits_{k=0}^{n} 2^{\frac{kq^*}{\frac{1}{s_1}+\frac{d}{s_2}}} \left(\left(\left\|f-F_k\right\|^{\Delta}_{\widehat{B}^{\alpha_1, \alpha_2}_{p, p}(\Omega_T)}\right)^{q^*}+\left(\left\|f-F_{k-1}\right\|^{\Delta}_{\widehat{B}^{\alpha_1, \alpha_2}_{p, p}(\Omega_T)}\right)^{q^*}\right) + 2^{\frac{nq^*}{\frac{1}{s_1}+\frac{d}{s_2}}}\left(\|f-F_n\|_{\widehat{B}^{\alpha_1, \alpha_2}_{p, p}(\Omega_T)}^{\Delta}\right)^{q^*}\right)^\frac{1}{q^*}
		\\&\lesssim_{\,d,p,s_1,s_2} \left(\sum\limits_{k=0}^n \left(2^{\frac{k}{\frac{1}{s_1}+\frac{d}{s_2}}}\sigma_{2^k}(f)_{\widehat{B}^{\alpha_1, \alpha_2}_{p, p}(\Omega_T)}\right)^{q^*} \right)^\frac{1}{q^*} + \|f\|_{\widehat{B}^{\alpha_1, \alpha_2}_{p, p}(\Omega_T)}^{\Delta},
    \end{align*}
    \endgroup
    using the aforementioned \cref{lem:Generalized_Besov_norm_Bernstein_estimates} in the first, another quasi-norm equivalence from \cref{lem:norm equivalence_gen_besov} in the second, the (quasi-)subadditivity of $\left\|\cdot\right\|^\Delta_{\widehat{B}^{\alpha_1, \alpha_2}_{p, p}(\Omega_T)}$ in the third, and the choice of $F_k$, $k\in\N_0\cup\{-1\}$ in the last step.
\end{proof}

The tools which we have developed up to now would suffice to prove the part of \cref{Main_res:invers} which is concerned with approximation classes with respect to generalized anisotropic Besov spaces. But we also want to extend this result to such classes with respect to the pure Lebesgue norm. This is why we need the following lemma, which relates both aforementioned equivalence classes.

\begin{lem}\label{lem:Relationship_Aquiv_classes}
    Let \mbox{\cref{assumptions:inverse_estimates_2}} hold and $(\varepsilon_1, \varepsilon_2),(\varepsilon_1', \varepsilon_2')\in \R^+(s_1, s_2)$ with $0<\varepsilon_i'<\varepsilon_i<s_i$, $i=1,2$. Further, let $r_1, r_2\in \N_{\ge 2}$ and $p_0\in (0,\infty)$ such that 
    \begin{align*}
       \frac{1}{p_0} = \frac1p+\frac{1}{\frac{1}{\varepsilon_1}+\frac{d}{\varepsilon_2}}.
    \end{align*}
    Then, with the abbreviation $\textup{REF}:=\textup{PATCH}\_\textup{REFINE}(\cdot, \cdot, d, s_1, s_2)$, we have the continuous embedding
    \begin{align*}
        \mathbb{A}_{\frac{1}{\frac{1}{s_1}+\frac{d}{s_2}},q}(L_p(\Omega_T),\V^{r_1, r_2}_{\bullet},\mathrm{REF}) \hookrightarrow \mathbb{A}_{\frac{1}{\frac{1}{s_1-\varepsilon_1}+\frac{d}{s_2-\varepsilon_2}},q}(\widehat{B}^{\varepsilon_1',\varepsilon_2'}_{p_0,p_0}\left(\Omega_T\right),\V^{r_1, r_2}_{\bullet},\mathrm{REF})
    \end{align*}
     with an embedding constant only dependent on $d,\rho, p, s_1, s_2, \varepsilon_1, \varepsilon_2, \varepsilon_2', r_1, r_2, \kappa_{\calP_0},\mu(\calP_0)$, and $\#\calP_0$. 
\end{lem}
\begin{proof}
    In the proof we will abbreviate $\mathbb{A}_{\cdot, \cdot}(\cdot):=\mathbb{A}_{\cdot, \cdot}(\cdot, \V^{r_1, r_2}_{\bullet}, \textup{PATCH}\_\textup{REFINE}(\cdot, \cdot, d, s_1, s_2))$ to improve the readability. First, let $\mathbb{A}_{\frac{1}{\frac{1}{s_1}+\frac{d}{s_2}},q}(L_p(\Omega_T))$. We borrow the notation $\sigma_N(\cdot)_{\widehat{B}^{\varepsilon_1',\varepsilon_2'}_{p_0,p_0}\left(\Omega_T\right)}$ for the optimal $N$-term approximation error with respect to $\|\cdot\|_{\widehat{B}^{\varepsilon_1',\varepsilon_2'}_{p_0,p_0}\left(\Omega_T\right)}$, $N\in \N_0$, from \mbox{\cref{lem:Generalized_Besov_K_functional_estimate}} above. Similarly, we put $\sigma_N(\cdot)_{L_p(\Omega_T)}$ for the  error in $L_p(\Omega_T)$. Further, for $n=-1$ we put $\sigma_{2^n}(f)_{\widehat{B}^{\varepsilon_1',\varepsilon_2'}_{p_0,p_0}\left(\Omega_T\right)}:=\|f\|_{\widehat{B}^{\varepsilon_1',\varepsilon_2'}_{p_0,p_0}\left(\Omega_T\right)} $ and $\sigma_{2^n}(f)_{L_{p}\left(\Omega_T\right)}:=\|f\|_{L_{p}\left(\Omega_T\right)} $, respectively.
    
    Also similar to \mbox{\cref{lem:Generalized_Besov_K_functional_estimate}}, for every $k\in \N_0$ there is $\hat{\calP}_k\in \mathbb{P}_{2^k}$ and $F_k\in \V^{r_1, r_2}_{\hat{\calP}_k}$ such that $\sigma_{2^k}(f)_{L_p(\Omega_T)}=\|f-F_k\|_{L_p(\Omega_T)}$. Again, we define $R_k:=F_k-F_{k-1}$ with $F_{-1}:=0$. Setting $p_0^*:=\min(1, p_0)$ and applying the same procedure to $f-F_n=\sum^{\infty}_{k=n+1}R_k$, $n\in \{-1\}\cup\N_0$, which was applied to $F_n=\sum^{n}_{k=0}R_k$ in \mbox{\eqref{lem:Generalized_Besov_K_functional_estimate - proof eq 1}} from the second to the last line, we obtain 
    \begin{align*}
        \sigma_{2^n}(f)_{\widehat{B}^{\varepsilon_1',\varepsilon_2'}_{p_0,p_0}\left(\Omega_T\right)}&\le \|f-F_n\|_{\widehat{B}^{\varepsilon_1',\varepsilon_2'}_{p_0,p_0}\left(\Omega_T\right)}\lesssim_{\,d,\rho, p, \varepsilon_1, \varepsilon_2, \varepsilon_2', r_1, r_2, \kappa_{\calP_0},\mu(\calP_0)} \left(\sum\limits_{k=n+1}^{\infty} \|R_k\|_{\widehat{B}^{\varepsilon_1',\varepsilon_2'}_{p_0,p_0}\left(\Omega_T\right)}^{p_0^*}\right)^{\frac{1}{p_0^*}}
        \\&\lesssim_{\,\substack{d,\rho, p, \varepsilon_1, \varepsilon_2, \\\varepsilon_2', r_1, r_2, \kappa_{\calP_0},\\\mu(\calP_0),\#\calP_0}} \left(\sum\limits_{k=n+1}^{\infty} 2^{\frac{k}{\frac{1}{\varepsilon_1}+\frac{d}{\varepsilon_2}}\,p_0^*}\|R_k\|_{L_{p}\left(\Omega_T\right)}^{p_0^*}\right)^{\frac{1}{p_0^*}}\lesssim_{\,d,p,\varepsilon_1, \varepsilon_2} \left(\sum\limits_{k=n}^{\infty} 2^{\frac{k}{\frac{1}{\varepsilon_1}+\frac{d}{\varepsilon_2}}\,p_0^*}\sigma_{2^k}(f)_{L_p(\Omega_T)}^{p_0^*}\right)^{\frac{1}{p_0^*}},
    \end{align*}
    where we have additionally used the second estimate from \cref{lem:Generalized_Besov_norm_Bernstein_estimates} in the penultimate step (the justification why this is possible with respect to $R_k$ is exactly the same as above in \eqref{lem:Generalized_Besov_K_functional_estimate - proof eq 2}) as well as the subadditivity of $\|\cdot\|_{L_p(\Omega_T)}$ and the choice of $(F_k)_{k\in \{-1\}\cup\N_0}$ in the last one.
    
    Using the discrete Hardy's inequality (cf. \cite[Lem.~3.4 of Ch.~2.3]{DL93}) with respect to $q$ and weight $\frac{1}{\frac{1}{s_1-\varepsilon_1}+\frac{d}{s_2-\varepsilon_2}}$, this yields
    \begin{align*}
        \|f\|^q_{\mathbb{A}_{\frac{1}{\frac{1}{s_1-\varepsilon_1}+\frac{d}{s_2-\varepsilon_2}},q}(\widehat{B}^{\varepsilon_1',\varepsilon_2'}_{p_0,p_0}\left(\Omega_T\right))}&\lesssim_{\,\substack{d,\rho, p, s_1, s_2,\varepsilon_1,\\ \varepsilon_2, \varepsilon_2', r_1, r_2, \kappa_{\calP_0},\\\mu(\calP_0),\#\calP_0}} \|f\|_{L_p(\Omega_T)}^q+\sum\limits_{k=0}^\infty 2^{\frac{k}{\frac{1}{s_1-\varepsilon_1}+\frac{d}{s_2-\varepsilon_2}}q}2^{\frac{k}{\frac{1}{\varepsilon_1}+\frac{d}{\varepsilon_2}}\,q}\sigma_{2^k}(f)_{L_p(\Omega_T)}^{q}
        \\&=\|f\|_{L_p(\Omega_T)}^q+\sum\limits_{k=0}^\infty 2^{\frac{k}{\frac{1}{s_1}+\frac{d}{s_2}}\,q}\sigma_{2^k}(f)_{L_p(\Omega_T)}^{q}\sim_{\,q} \|f\|^q_{\mathbb{A}_{\frac{1}{\frac{1}{s_1}+\frac{d}{s_2}},q}(L_p(\Omega_T))}
    \end{align*}
    for $q<\infty$, with the usual modification if $q=\infty$. The second step is indeed valid, since $(s_1,s_2)$ and $(\varepsilon_1, \varepsilon_2) $ are linearly dependent, i.e., $\varepsilon_i=cs_i$, $i=1,2$, for some $c\in (0,1)$.
\end{proof}

Now we can finally prove the inverse estimates from \cref{Main_res:invers}.

\begin{thm}\label{thm:Inverse_estimates}
	Under \cref{assumptions:inverse_estimates_2} with $\rho<q$, the embedding
    \begin{align*}
        \mathbb{A}_{\frac{1}{\frac{1}{s_1}+\frac{d}{s_2}}, q}\Big(\widehat{B}^{\alpha_1, \alpha_2}_{p,p}(\Omega_T), \V^{r_1, r_2}_{\bullet}, \textup{PATCH}\_\textup{REFINE}(\cdot, \cdot, d, s_1, s_2)\Big)\hookrightarrow \widehat{B}^{\alpha_1+s_1, \alpha_2+s_2}_{q, q}(\Omega_T)
    \end{align*}
    is continuous with an embedding constant depending only on $d, \rho, p, s_1, s_2, \alpha_2, r_1, r_2, \kappa_{\calP_0}, \mu(\calP_0)$, and $\#\calP_0$. If additionally $(s_1', s_2')\in \R^+(s_1, s_2)$ with $s_i'<s_i$, $i=1,2$, then also 
    \begin{align*}
        \mathbb{A}_{\frac{1}{\frac{1}{s_1}+\frac{d}{s_2}}, q}\Big(L_{p}(\Omega_T), \V^{r_1, r_2}_{\bullet}, \textup{PATCH}\_\textup{REFINE}(\cdot, \cdot, d, s_1, s_2)\Big)\hookrightarrow \widehat{B}^{s_1', s_2'}_{q, q}(\Omega_T)
    \end{align*}
    continuously with an embedding constant depending on $d, \rho,p,s_1, s_2, s_1', s_2', r_1, r_2, \kappa_{\calP_0}$, $\mu(\calP_0)$, and $\#\calP_0$. 
\end{thm}

\begin{proof}
    In the proof we will abbreviate $\mathbb{A}_{\cdot, \cdot}(\cdot):=\mathbb{A}_{\cdot, \cdot}(\cdot, \V^{r_1, r_2}_{\bullet}, \textup{PATCH}\_\textup{REFINE}(\cdot, \cdot, d, s_1, s_2))$ to improve the readability. Let $C\in (1,\infty)$ with $\tilde{s}_i= Cs_i$, $i=1,2$, such that $\tilde{q}\in (0,\infty)$ with $\frac{1}{\tilde{q}}=\frac{1}{p}+\frac{1}{\frac{1}{\tilde{s}_1}+\frac{d}{\tilde{s}_2}}$ fulfills $\rho\le \tilde{q}$, and set $\tilde{q}^*:=\min(1,\tilde{q})$. Then, 
    \begin{align*}
        \left(1-\frac{s_2}{\tilde{s}_2}\right) \frac{\alpha_2}{d}+ \frac{s_2}{\tilde{s}_2}\, \frac{\tilde{s}_2+\alpha_2}{d} = \frac{\alpha_2+s_2}{d}\quad\text{and}\quad 
        \left(1-\frac{s_2}{\tilde{s}_2}\right) \frac{1}{p}+ \frac{s_2}{\tilde{s}_2}\, \frac{1}{\tilde{q}} = \frac{1}{q},
    \end{align*}
    which implies 
    \begin{align*}
        \left(\ell^{\frac{\alpha_2}{d},p}(L_p(\Omega_T)), \ell^{\frac{\tilde{s}_2+\alpha_2}{d}, \tilde{q}}(L_{\tilde{q}}(\Omega_T))\right)_{\frac{s_2}{\tilde{s}_2},q} = \ell^{\frac{\alpha_2+s_2}{d}, q}(L_{q}(\Omega_T))
    \end{align*}
    according to \cref{lem:interpolation_of_certain_gen_sequence_spaces}, with involved constants depending on $d, \rho, p, s_1, s_2$, and $\alpha_2$. Therefore, we can calculate
\begin{align}\label{thm:Inverse_estimates - proof equation 1}
	&\left(\|f\|^\Delta\right)^q_{\widehat{B}^{\alpha_1+s_1, \alpha_2+s_2}_{q, q}(\Omega_T)} = \left\|\left(\Delta_n(f)\right)_{n\in \N_0}\right\|^q_{\ell^{\frac{\alpha_2+s_2}{d},q}(L_q(\Omega_T))}\notag
    \\&\lesssim_{\, d, \rho, p, s_1, s_2, \alpha_2} \left\|\left(\Delta_n(f)\right)_{n\in \N_0}\right\|^q_{\left(\ell^{\frac{\alpha_2}{d},p}(L_p(\Omega_T)), \ell^{\frac{\tilde{s}_2+\alpha_2}{d}, \tilde{q}}(L_{\tilde{q}}(\Omega_T))\right)_{\frac{s_2}{\tilde{s}_2},q}}\notag	
	\\&\lesssim_{\,d,\rho, p,s_1, s_2,\alpha_2} \left\|\left(\Delta_n(f)\right)_{n\in \N_0}\right\|^q_{\ell^{\frac{\alpha_2}{d},p}(L_p(\Omega_T))}\notag
    \\& \quad \quad \quad \quad + \sum\limits_{n=0}^\infty \left[2^{\frac{n\, \frac{s_2}{\tilde{s}_2}}{\frac{1}{\tilde{s}_1}+\frac{d}{\tilde{s}_2}}}K\left(\left(\Delta_n(f)\right)_{n\in \N_0},2^{-\frac{n}{\frac{1}{\tilde{s}_1}+\frac{d}{\tilde{s}_2}}},\ell^{\frac{\alpha_2}{d},p}(L_p(\Omega_T)), \ell^{\frac{\tilde{s}_2 + \alpha_2}{d},\tilde{q}}(L_{\tilde{q}}(\Omega_T))\right) \right]^q\notag 
    \\&\lesssim_{\, \substack{d, \rho, p, s_1, s_2, \alpha_2, \\ r_1, r_2, \kappa_{\calP_0}, \mu(\calP_0)}}  \quad \|f\|_{\widehat{B}^{\alpha_1, \alpha_2}_{p,p}(\Omega_T)}^{q} + \sum\limits_{n=0}^{\infty}\left[2^{\frac{n}{\frac{1}{s_1}+\frac{d}{s_2}}} K\left(f, 2^{-\frac{n}{\frac{1}{\tilde{s}_1}+\frac{d}{\tilde{s}_2}}}, \widehat{B}^{\alpha_1, \alpha_2}_{p,p}(\Omega_T), \widehat{B}^{\tilde{s}_1+\alpha_1, \tilde{s}_2 + \alpha_2}_{\tilde{q}, \tilde{q}}(\Omega_T)\right)\right]^q\notag
    \\&\lesssim_{\, d, \rho, p, s_1, s_2, r_1, r_2, \#\calP_0}\notag
    \\& \quad \|f\|_{\widehat{B}^{\alpha_1, \alpha_2}_{p,p}(\Omega_T)}^q + \sum\limits_{n=0}^{\infty} 2^{n\left(\frac{1}{\frac{1}{s_1}+\frac{d}{s_2}}-\frac{1}{\frac{1}{\tilde{s}_1}+\frac{d}{\tilde{s}_2}}\right)q}\left[\left(\sum\limits_{k=0}^n \left(2^{\frac{k}{\frac{1}{\tilde{s}_1}+\frac{d}{\tilde{s}_2}}}\sigma_{2^k}(f)_{\widehat{B}^{\alpha_1, \alpha_2}_{p,p}(\Omega_T)}\right)^{\tilde{q}^*} \right)^\frac{q}{\tilde{q}^*} + \|f\|_{\widehat{B}^{\alpha_1, \alpha_2}_{p,p}(\Omega_T)}^q\right]
\end{align}
where we have used the discrete estimates for the interpolation space norm from \cref{rem:on_interpolation_spaces}\ref{rem:on_interpolation_spaces - discrete norm} and \ref{rem:on_interpolation_spaces - bound on discrete norm if X_2 embedded in X_1} in the third step, \cref{lem:estimate between different K functionals} in the fourth, and \cref{lem:Generalized_Besov_K_functional_estimate} in the last step. On the one hand, 
\begin{align*}
    \left(\frac{1}{\frac{1}{s_1}+\frac{d}{s_2}}-\frac{1}{\frac{1}{\tilde{s}_1}+\frac{d}{\tilde{s}_2}}\right)q = (1-C)\frac{q}{\frac{1}{s_1}+\frac{d}{s_2}} < 0,\quad\text{since}\quad  C>1. 
\end{align*}
Therefore, the geometric sum 
\begin{align}\label{thm:Inverse_estimates - proof equation 2}
    \sum\limits_{n=0}^{\infty} 2^{n\left(\frac{1}{\frac{1}{s_1}+\frac{d}{s_2}}-\frac{1}{\frac{1}{\tilde{s}_1}+\frac{d}{\tilde{s}_2}}\right)q} \|f\|_{\widehat{B}^{\alpha_1, \alpha_2}_{p,p}(\Omega_T)}^q \sim_{\, d, \rho,p, s_1, s_2} \|f\|_{\widehat{B}^{\alpha_1, \alpha_2}_{p,p}(\Omega_T)}^q
\end{align}
converges. On the other hand, 
\begin{align}\label{thm:Inverse_estimates - proof equation 3}
    \sum\limits_{n=0}^{\infty} 2^{n\left(\frac{1}{\frac{1}{s_1}+\frac{d}{s_2}}-\frac{1}{\frac{1}{\tilde{s}_1}+\frac{d}{\tilde{s}_2}}\right)q}\left(\sum\limits_{k=0}^n \left(2^{\frac{k}{\frac{1}{\tilde{s}_1}+\frac{d}{\tilde{s}_2}}}\sigma_{2^k}(f)_{\widehat{B}^{\alpha_1, \alpha_2}_{p,p}(\Omega_T)}\right)^{\tilde{q}^*} \right)^\frac{q}{\tilde{q}^*} = \sum\limits_{n\in \mathbb{Z}} 2^{(-n)\Theta q} b_{-n}^q 
\end{align}
with $\displaystyle{b_m:=\mathds{1}_{\mathbb{Z}^-_0}(\{m\})\left(\sum\limits_{k=0}^{-m}\left(2^{\frac{k}{\frac{1}{\tilde{s}_1}+\frac{d}{\tilde{s}_2}}}\sigma_{2^k}(f)_{\widehat{B}^{\alpha_1, \alpha_2}_{p,p}(\Omega_T)}\right)^{\tilde{q}^*}\right)^\frac{1}{\tilde{q}^*}}$, $m\in \mathbb{Z}$, $\Theta := \left(\frac{1}{\frac{1}{\tilde{s}_1}+\frac{d}{\tilde{s}_2}}-\frac{1}{\frac{1}{s_1}+\frac{d}{s_2}}\right)>0$, and $\mathbb{Z}^-_0:=\{k\in \mathbb{Z}\mid k\le 0\}$. Since $b_m=\left(\sum\limits_{k=m}^{\infty} a_k^{\tilde{q}^*} \right)^\frac{1}{\tilde{q}^*}$, $m\in \mathbb{Z}$, for $a_k:=\mathds{1}_{\mathbb{Z}^-_0}(\{k\})2^{-\frac{k}{\frac{1}{\tilde{s}_1} + \frac{d}{\tilde{s}_2}}} \sigma_{2^{-k}}(f)_{\widehat{B}^{\alpha_1,\alpha_2}_{p,p}(\Omega_T)}$ using the discrete Hardy inequality as in  \cite[Lem.~3.4 of Ch.~2.3]{DL93} yields 
\begin{align}\label{thm:Inverse_estimates - proof equation 4}
    &\sum\limits_{n\in \mathbb{Z}} 2^{(-n)\Theta q} b_{-n}^q \lesssim_{\,\substack{d,\rho, p,\\s_1, s_2}} \sum\limits_{n\in \mathbb{Z}} 2^{n\Theta q} a_{n}^q 
    = \sum\limits_{n\in \mathbb{Z}} \mathds{1}_{\mathbb{Z}^-_0}(\{n\})\, 2^{n\left(\frac{1}{\frac{1}{\tilde{s}_1}+\frac{d}{\tilde{s}_2}}-\frac{1}{\frac{1}{s_1}+\frac{d}{s_2}}\right)q} 2^{n\left(-\frac{1}{\frac{1}{\tilde{s}_1} + \frac{d}{\tilde{s}_2}}\right) q} \sigma_{2^{-n}}(f)_{\widehat{B}^{\alpha_1,\alpha_2}_{p,p}(\Omega_T)}^q
    \\&= \sum\limits_{n\in \mathbb{Z}} \mathds{1}_{\mathbb{Z}^-_0}(\{n\})\, 2^{-n\left(\frac{1}{\frac{1}{s_1}+\frac{d}{s_2}}\right)q}\sigma_{2^{-n}}(f)_{\widehat{B}^{\alpha_1,\alpha_2}_{p,p}(\Omega_T)}^q = \sum\limits_{n=0}^\infty 2^{n\left(\frac{1}{\frac{1}{s_1}+\frac{d}{s_2}}\right)q}\sigma_{2^{n}}(f)_{\widehat{B}^{\alpha_1,\alpha_2}_{p,p}(\Omega_T)}^q = |f|_{\mathbb{A}_{\frac{1}{\frac{1}{s_1}+\frac{d}{s_2}},q}\left(\widehat{B}^{\alpha_1, \alpha_2}_{p,p}(\Omega_T)\right)}^q.\notag
\end{align}
Now inserting \eqref{thm:Inverse_estimates - proof equation 2} and \eqref{thm:Inverse_estimates - proof equation 3} together with \eqref{thm:Inverse_estimates - proof equation 4} into \eqref{thm:Inverse_estimates - proof equation 1} yields the first asserted embedding.

Now we may turn towards the second embedding. We recall some notation from \mbox{\cref{lem:Relationship_Aquiv_classes}}, i.e. $(\varepsilon_1, \varepsilon_2),(\varepsilon_1', \varepsilon_2')\in \R^+(s_1, s_2)$ with $0<\varepsilon_i'<\varepsilon_i<s_i$, $i=1,2$, and $p_0\in (0,\infty)$ with
    \begin{align*}
       \frac{1}{p_0} = \frac1p+\frac{1}{\frac{1}{\varepsilon_1}+\frac{d}{\varepsilon_2}}.
    \end{align*}
    Since there is $c\in (0,1)$ such that $\varepsilon_i=cs_i$, $i=1,2$, we can calculate
    \begin{align}\label{thm:Inverse_estimates - proof equation 5}
        \frac{1}{p_0}+\frac{1}{\frac{1}{s_1-\varepsilon_1}+\frac{d}{s_2-\varepsilon_2}}=\frac{1}{p_0}+\frac{1}{\frac{1}{s_1}+\frac{d}{s_2}}-\frac{1}{\frac{1}{\varepsilon_1}+\frac{d}{\varepsilon_2}}=\frac1p+\frac{1}{\frac{1}{s_1}+\frac{d}{s_2}}=\frac1q.
    \end{align}
    Therefore, we can exploit the first assertion of the theorem, which has been proven to hold before, 
    with respect to $p_0$ instead of $p$ and $s_i-\varepsilon_i$ instead of $s_i$, $i=1,2$, which is possible due to \eqref{thm:Inverse_estimates - proof equation 5}. Thus, we can also deduce that the embedding $\mathbb{A}_{\frac{1}{\frac{1}{s_1-\varepsilon_1}+\frac{d}{s_2-\varepsilon_2}},q}(\widehat{B}^{\varepsilon_1',\varepsilon_2'}_{p_0,p_0}\left(\Omega_T\right))\hookrightarrow \widehat{B}^{s_1-\varepsilon_1+\varepsilon_1',s_2-\varepsilon_2+\varepsilon_2'}_{q,q}(\Omega_T)$ holds continuously. Together with \mbox{\cref{lem:Relationship_Aquiv_classes}}, this finally yields 
    \begin{align*}
        \mathbb{A}_{\frac{1}{\frac{1}{s_1}+\frac{d}{s_2}},q}(L_p(\Omega_T)) \hookrightarrow \mathbb{A}_{\frac{1}{\frac{1}{s_1-\varepsilon_1}+\frac{d}{s_2-\varepsilon_2}},q}(\widehat{B}^{\varepsilon_1',\varepsilon_2'}_{p_0,p_0}\left(\Omega_T\right))\hookrightarrow \widehat{B}^{s_1-\varepsilon_1+\varepsilon_1',s_2-\varepsilon_2+\varepsilon_2'}_{q,q}(\Omega_T).
    \end{align*}
    Now choosing $0<\varepsilon_i:=s_i-\frac{s_i'}{2}<s_i$ and $0<\varepsilon_i':=\frac{s_i'}{2}<\varepsilon_i$, for $i=1,2$, proves the assertion, because $s_i-\varepsilon_i+\varepsilon_i'=s_i-\left(s_i-\frac{s_i'}{2}\right)+\frac{s_i'}{2}= s_i'$.
\end{proof}

\appendix

\section{\texorpdfstring{Proofs of statements of \cref{thm:Proofs_omega_tilde_domain}}{Appendix \getrefnumber{Appendix:Proofs_for_omega_tilde_domains}: Proofs of statements of Theorem \getrefnumber{thm:Proofs_omega_tilde_domain}}}

\label{Appendix:Proofs_for_omega_tilde_domains}

\begin{proof}[Proof of \ref{thm:Proofs_omega_tilde_domain - item Jackson} - Jackson inequality.]
    Recall the covering of $R$ by non-overlapping simplices $S_1'', \dots, S_k''$ from \cref{rem:Neighborhoods_generally_not_Lipschitz}\ref{rem:Neighborhoods_generally_not_Lipschitz - General S_w''}, which we will just call $S_1, \dots, S_k$ here. Further, we define the set of two-element patches $\Sigma(R):=\{S_w\cup S_{w'}\mid \dim(S_w\cap S_{w'})=d-1, \:\; w,\,w'=1,\dots, k\}$. Clearly, due to the boundedness of the minimal head angle of the simplices $(S_w)_{w=1,\dots,k}$, every $E\in \Sigma(R)$ is a Lipschitz domain and the Lipschitz properties $\LipProp(E)$ are determined by $d$ and $\kappa_{\calP_0}$. Thus, applying \cite[Thm.~1.1]{MSS26} on such $E$ yields $P_E\in \Pi^{r_1, r_2}_{t,\bm{x}}$ with
    \begin{align}\label{eq:Appendix:Proofs_for_omega_tilde_domains - Jackson inequ. 1}
        \|f-P_E\|_{L_p\left(J\times E\right)}&\lesssim_{\,d,p,r_1, r_2,\kappa_{\calP_0}}\omega_{r_1, t}\left(f,J\times E, |J|\right)_p +\omega_{r_2, \bm{x}}\left(f,J\times E, \diam(E)\right)_p\notag
         \\&\le \omega_{r_1, t}\left(f,J\times R, |J|\right)_p +\omega_{r_2, \bm{x}}\left(f,J\times R, \diam(R)\right)_p,
    \end{align}
    where we have additionally exploited $E\subset R$ as well as the monotonicity of the moduli from \cref{Rem:Moduli of smoothness}\ref{Rem:Moduli of smoothness - monotonicity} in the second step. Furthermore, if $E, E'\in \Sigma(R)$ with $E\cap E'=S_w$ for some $w=1,\dots, k$, then by the (quasi-triangle) inequality (we put $p^*:=\min(1,p)$) and the result from above, we obtain
    \begin{align}\label{eq:Appendix:Proofs_for_omega_tilde_domains - Jackson inequ. 2}
        \|P_E-P_{E'}\|_{L_p\left(J\times S_w\right)}&\le\left( \|f-P_E\|_{L_p\left(J\times E\right)}^{p^*} + \|f-P_{E'}\|_{L_p\left(J\times E'\right)}^{p^*}\right)^\frac{1}{p^*}\notag
        \\&\lesssim_{\,d,p,r_1, r_2,\kappa_{\calP_0}}\omega_{r_1, t}\left(f,J\times R, |J|\right)_p +\omega_{r_2, \bm{x}}\left(f,J\times R, \diam(R)\right)_p.
    \end{align}
    Additionally, due to \eqref{rem:Neighborhoods_generally_not_Lipschitz - equation 2}, it holds
    \begin{align}\label{eq:Appendix:Proofs_for_omega_tilde_domains - Jackson inequ. 3}
        \|Q\|_{L_p(J\times S_w)}\lesssim_{\,j,d,p, r_2, \kappa_{\calP_0}, \mu(\calP_0)}\|Q\|_{L_p(J\times S_{w'})}
    \end{align}
    for any $Q\in \Pi^{r_1, r_2}_{t,\bm{x}}$ and $w, \,w'=1,\dots,k$. Now we fix an arbitrary $E_0\in \Sigma(R)$. By the explanations given in \cref{rem:Neighborhoods_generally_not_Lipschitz}, it is clear that for every $S_w$, $w=1,\dots, k$, there is a quasi-minimal path from $E_0$ to $S_w$ in terms of two-element patches, i.e., a sequence $(E_i)_{i=0,\dots, n}\subset \Sigma(R)$ with $n\le k$ such that $E_i\cap E_{i-1}=S_{e_i}$ with $1\le e_i\le k$ for $i=1, \dots, n$, and $S_w\subset E_n$. Now we can exploit the (quasi-triangle) inequality in the first step, \eqref{eq:Appendix:Proofs_for_omega_tilde_domains - Jackson inequ. 3} in the second, and \eqref{eq:Appendix:Proofs_for_omega_tilde_domains - Jackson inequ. 1} as well as \eqref{eq:Appendix:Proofs_for_omega_tilde_domains - Jackson inequ. 2} in the third one, to obtain
    \begin{align*}
        \|f-P_{E_0}\|_{L_p(J\times S_w)}&\le \left(\|f-P_{E_n}\|_{L_p(J\times S_w)}^{p^*} + \sum\limits_{i=1}^n \|P_{E_i}-P_{E_{i-1}}\|_{L_p(J\times S_w)}^{p^*}\right)^\frac{1}{p^*}
        \\&\lesssim_{\,j,d,p,r_2, \kappa_{\calP_0}, \mu(\calP_0)} \left(\|f-P_{E_n}\|_{L_p(J\times S_w)}^{p^*} + \sum\limits_{i=1}^n \|P_{E_i}-P_{E_{i-1}}\|_{L_p(J\times S_{e_i})}^{p^*}\right)^\frac{1}{p^*}
        \\&\lesssim_{\,j,d,p,r_1, r_2, \kappa_{\calP_0}}\omega_{r_1, t}\left(f,J\times R, |J|\right)_p +\omega_{r_2, \bm{x}}\left(f,J\times R, \diam(R)\right)_p.
    \end{align*}
    Appropriate summing over all $w=1,\dots,k$ (if $p<\infty$) or taking the corresponding supremum (if $p=\infty$) yields the assertion for $J\times R$ due to $k\lesssim_{\,j,d, \kappa_{\calP_0}}1$. The version of the proof for $[0,1]\times R^{ref}$ follows by scaling using the identities from \cite[Lem.~2.5]{MSS26} (which also hold if the spatial domain is not Lipschitz).

    For the enlarged versions, the only difference is the number of simplices in the corresponding covering, which is still bounded. Hence, the above proof directly extends to them.
\end{proof}

\begin{proof}[Proof of \ref{thm:Proofs_omega_tilde_domain - item moduli of smoothness equal zero IFF anisotropic polynomials} - Relationship between temporal and spatial moduli and anisotropic polynomials.]

First, we will prove \ref{thm:Proofs_omega_tilde_domain - item moduli of smoothness equal zero IFF anisotropic polynomials - item 1}. If $P\in \Pi^{r_1, r_2}_{t,\bm{x}}(J\times R)$ this means that $P_t:=P(t,\cdot)\in \Pi^{r_2}(R)$ and $P_{\bm{x}}:=P(\cdot, \bm{x})\in \Pi^{r_1}(J)$ for any $(t,\bm{x})\in J\times R$. According to \cite[Lem.~1.26]{Dek22}, this means that 
\begin{align*}    
    \Delta^{r_1}_{h,t}P=\Delta^{r_1}_{h}P_{\bm{x}}=0=\Delta^{r_2}_{\bm{h}}P_{t}=\Delta^{r_2}_{\bm{h},\bm{x}}P,
\end{align*}
since $J$ and $R$ are indeed connected. This is the assertion. Now let us turn to \ref{thm:Proofs_omega_tilde_domain - item moduli of smoothness equal zero IFF anisotropic polynomials - item 2}. Let $\Sigma(R)$ be the set from the above proof of the Jackson inequality. Then $\omega_{r_i, \mathrm{dir}}(f,J\times E, \delta_i )= 0$ for every $E\in \Sigma(R)$ and $(r_i, \mathrm{dir}, \delta_i)\in\{(r_1, t, \delta_1), (r_2, \bm{x}, \delta_2)\}$. Since $E$ is a Lipschitz domain, \cite[Lem.~2.16]{MSS26} implies that $f_{|E}\in \Pi^{r_1, r_2}_{t,\bm{x}}(J\times E)$. Since this holds for every such $E\in \Sigma(R)$, the fact that $\Sigma(R)$ is overlapping covering of $R$ implies that $f\in \Pi^{r_1, r_2}_{t,\bm{x}}(J\times R)$. This shows the result for $J\times R$. Replacing it by $[0,1]\times R^{ref}$ in the above arguments shows the corresponding result.
\end{proof}

\begin{proof}[Proof of \ref{thm:Proofs_omega_tilde_domain - Quasi-equivalence of suprema and averaged moduli of smoothness} - Quasi-equivalence of suprema and averaged moduli of smoothness.]
    First, we consider the setting of \ref{thm:Proofs_omega_tilde_domain - Quasi-equivalence of suprema and averaged moduli of smoothness - uniform part}. The second inequality of both estimates directly follows from $J\times \omega^{j+1}_{\calT_n}(S)\subset \omega^{j+1}_{\calP_n}(I\times S)$, which holds since $\omega^{j+1}_{\calP_n}(I\times S)=J'\times\omega^{j+1}_{\calT_n}(S)$ for an interval $J'\supset J$, see \cref{rem:Neighborhoods_generally_not_Lipschitz}\ref{rem:Neighborhoods_generally_not_Lipschitz - Uniform case S_w''}.
    
    We will only provide the proof for $p<\infty$, the case $p=\infty$ is similar with the usual modifications. The temporal part is not affected by the fact that the spatial domain might not be a Lipschitz domain, so this works identically to the isotropic case in \cite[Lem.~5.1 of Ch.~6.5]{DL93}. The upper bound for $\delta$ given there is $\frac{|J|}{4r_1}\sim_{\,j,d,s_1, s_2, \kappa_{\calP_0}, \mu(\calP_0)}|I|$ according to \eqref{rem:Neighborhoods_generally_not_Lipschitz - equation 3}. 
    
    Now we will focus on the spatial part. Let $\mathcal{V}$ be the set of vertices of elements of $\omega_{\calT_n}^j(S)$ and $\omega_{\calT_n}(\nu):=\{S'\in \calT_n\mid \nu \in S'\}$ and observe that for $\bm{x} \in \omega^j_{\calT_n}(S)$, there is $S'\in \omega^j_{\calT_n}(S)$ with $\bm{x}\in S'$. Further, let $F$ be the facet of $S'$ which has the maximal distance to $\bm{x}$. Additionally, let $\nu\in \mathcal{V}$ be the vertex opposed to $F$ in $S'$. By the shape regularity, we know that there is a constant $C$ depending on $d$, $j$, $\kappa_{\calP_0},\mu(\calP_0)$, and $\LipProp(\Omega)$ such that $B_{C\,\diam(S)}(\bm{x})\cap\overline{\Omega}\subset \omega_{\calT_n}(\nu)$, where the former means the closed ball around $\bm{x}$ with radius $C\,\diam(S)$. Now choosing $\bm{h}\in \R^d$ with $r_2|\bm{h}|\le C\,\diam(S)$ yields that $\bm{x}+i\bm{h}\in \omega_{\calT_n}(\nu)$ if $\bm{x}\in \omega^j_{\calT_n}(S)_{r_2, \bm{h}}$ and $i\in[0, r_2]$. Therefore, this implies
    \begin{align*}
        \omega^j_{\calT_n}(S)_{r_2, \bm{h}}\subset \bigcup_{\nu\in \mathcal{V}}\omega_{\calT_n}(\nu)_{r_2, \bm{h}}
    \end{align*}
    for such $\bm{h}$. The definition of the moduli of smoothness from \cref{subsubsect:Besov spaces} together with the finite overlap of neighborhood domains $(\omega_{\calT_n}(\nu))_{\nu\in \mathcal{V}}$ due to the boundedness of the minimal head angle of the involved simplices, now yields
    \begin{align}\label{thm:Proofs_omega_tilde_domain - Quasi-equivalence of suprema and averaged moduli of smoothness - proof eq 1}
        \omega_{r_2,\bm{x}}\left(f, \omega^j_{\calP_n}(I\times S), \delta\right)_p^p = \omega_{r_2,\bm{x}}\left(f, J\times \omega^j_{\calT_n}(S), \delta\right)_p^p\lesssim_{\,d,p, \kappa_{\calP_0}}\sum\limits_{\nu\in \mathcal{V}} \omega_{r_2,\bm{x}}\left(f, J\times \omega_{\calT_n}(\nu), \delta\right)_p^p
    \end{align}
    for $0\le\delta\le\frac{C\,\diam(S)}{r_2}$. For every $\nu \in \mathcal{V}$ with $\nu\in \mathring{\Omega}$, $\omega_{\calT_n}(\nu)$ is star-shaped with respect to a ball. Thus, $\omega_{\calT_n}(\nu)$ is a Lipschitz domain in those cases, in contrast to $\omega_{\calT_n}(S)$, with Lipschitz properties $\LipProp\left(\omega_{\calT_n}(\nu)\right)$ depending only on $d,\kappa_{\calP_0}$, and $\mu(\calP_0)$. Since $\Omega$ is a Lipschitz domain, this extends to those $\nu\in \partial\Omega$ with an additional dependency on the Lipschitz properties $\LipProp(\Omega)$ of $\Omega$. Therefore, \cite[Lem.~2.8]{MSS26} is applicable to $\omega_{\calT_n}(\nu)$, $\nu\in \mathcal{V}$, and there is a constant $\tilde{C}=\tilde{C}(j,d,r_2, \kappa_{\calP_0},\mu(\calP_0),\LipProp(\Omega))\le\frac{C}{r_2}$ such that for all $0\le\delta\le \tilde{C}\,\diam(S)$, the following calculation is valid:
    \begin{align*}
        \sum\limits_{\nu\in \mathcal{V}} \omega_{r_2,\bm{x}}\left(f, J\times \omega_{\calT_n}(\nu), \delta\right)_p^p\lesssim_{\,\substack{d,p,r_2, \kappa_{\calP_0},\mu(\calP_0),\\\LipProp(\Omega)}}\sum\limits_{\nu\in \mathcal{V}} \mathrm{w}_{r_2,\bm{x}}\left(f, J\times \omega_{\calT_n}(\nu), \delta\right)_p^p\lesssim_{\,d,p, \kappa_{\calP_0}}\mathrm{w}_{r_2,\bm{x}}\left(f, J\times \omega^{j+1}_{\calT_n}(S), \delta\right)_p^p.
    \end{align*}
    Here, we have additionally used the aforementioned, bounded overlap and $J\times\omega_{\calT_n}(\nu)\subset J\times \omega^{j+1}_{\calT_n}(S)$ in the second step. Combining this with \eqref{thm:Proofs_omega_tilde_domain - Quasi-equivalence of suprema and averaged moduli of smoothness - proof eq 1} now yields the missing spatial part of the assertion in \ref{thm:Proofs_omega_tilde_domain - Quasi-equivalence of suprema and averaged moduli of smoothness - uniform part}. Now we can turn to \ref{thm:Proofs_omega_tilde_domain - Quasi-equivalence of suprema and averaged moduli of smoothness - other part}. In this case, the proof works analogously, when $\calT_n$ is replaced by $\tilde{\calT}$ from \cref{rem:Neighborhoods_generally_not_Lipschitz}\ref{rem:Neighborhoods_generally_not_Lipschitz - General S_w''}.
\end{proof}

\begin{proof}[Proof of \ref{thm:Proofs_omega_tilde_domain - summing of Besvo norms} - Summation of Besov (quasi-)seminorms.]
    We will only prove the results for $p<\infty$, the case $p=\infty$ works accordingly. Generally, the versions for general $(r_1, r_2)$ with $r_i>\alpha_i$ also hold true with the (besides this choice of $r_i$, $i=1,2$) unchanged proof, so we will only explicit treat the case of $r_i:=\lfloor \alpha_i\rfloor +1$, $i=1,2$.
    
    First, we will consider \ref{thm:Proofs_omega_tilde_domain - summing of Besvo norms - uniform part}. Proceeding as in \cite[Lem.~3.3]{MSS26} with \cref{thm:Proofs_omega_tilde_domain}\ref{thm:Proofs_omega_tilde_domain - Quasi-equivalence of suprema and averaged moduli of smoothness}\ref{thm:Proofs_omega_tilde_domain - Quasi-equivalence of suprema and averaged moduli of smoothness - uniform part} instead of \cite[Lem.~2.6 and Lem.~2.8]{MSS26}, we can show that 
    \begin{align*}
        \sum\limits_{I\times S\in \calP_n}|f|^p_{B^{\alpha_1, \alpha_2}_{p,p}\left(\omega_\calP^j(I\times S)\right)}\lesssim_{\,\substack{j,d,p,s_1, s_2,\alpha_1, \alpha_2, \\\kappa_{\calP_0},\mu(\calP_0), \LipProp(\Omega)}}\sum\limits_{I\times S\in \calP_n}\sum\limits_{i=1,2}\int\limits_{0}^{\infty} \delta^{-\alpha_ip-1}\mathrm{w}_{r_i, \mathrm{dir}_i}(f, \omega_\calP^{j+1}(I\times S), \delta)_p^p\,d\delta
    \end{align*}
    with $\mathrm{dir}_1=t$ and $\mathrm{dir}_2=\bm{x}$. With the superadditivity of the moduli of smoothness from \cref{Rem:Moduli of smoothness}\ref{Rem:Moduli of smoothness - additivity} and the finite overlap from \eqref{rem:Neighborhoods_generally_not_Lipschitz - equation 5}, the assertion follows. \ref{thm:Proofs_omega_tilde_domain - summing of Besvo norms - general part} works analogously with \cref{thm:Proofs_omega_tilde_domain}\ref{thm:Proofs_omega_tilde_domain - Quasi-equivalence of suprema and averaged moduli of smoothness}\ref{thm:Proofs_omega_tilde_domain - Quasi-equivalence of suprema and averaged moduli of smoothness - other part} and \eqref{rem:Neighborhoods_generally_not_Lipschitz - equation 4} (together with their analogues for the enlarged spatial domain from \cref{rem:Neighborhoods_generally_not_Lipschitz}\ref{rem:Neighborhoods_generally_not_Lipschitz - enlargement}) instead of \cref{thm:Proofs_omega_tilde_domain}\ref{thm:Proofs_omega_tilde_domain - Quasi-equivalence of suprema and averaged moduli of smoothness}\ref{thm:Proofs_omega_tilde_domain - Quasi-equivalence of suprema and averaged moduli of smoothness - uniform part} and \eqref{rem:Neighborhoods_generally_not_Lipschitz - equation 5}, respectively. Now we can turn to \ref{thm:Proofs_omega_tilde_domain - summing of Besvo norms - R^ref part}. We consider the space-time partition $\tilde{\calP}_{\Enl([0,1]\times R^{ref})}$ which is defined correspondingly to $\tilde{\calP}_{[0,1]\times R^{ref}}$ from \cref{rem:Neighborhoods_generally_not_Lipschitz}\ref{rem:Neighborhoods_generally_not_Lipschitz - reference domains}, which contains the latter. Let now $\mathcal{C}_n^{\Enl}$ denote the $n$-th, $n\in\N_0$, uniform refinement of $\tilde{\calP}_{\Enl([0,1]\times R^{ref})}$ with respect to $\textup{ATOMIC}\_\textup{SPLIT}(\cdot, d, s_1, s_2)$ such that $\mathcal{C}_n\subset\mathcal{C}_n^{\Enl}$.  Then the proof works in the same manner as the one of \ref{thm:Proofs_omega_tilde_domain - summing of Besvo norms - uniform part}, where the spatial neighborhood domains are taken with respect to $\mathcal{C}_n^{\Enl}$ and the corresponding space-time domains are thus contained in $\Enl([0,1]\times R^{ref})$.
\end{proof}

\begin{proof}[Proof of \ref{thm:Proofs_omega_tilde_domain - Embeddings between Besov and Lebesgue scale} - Embeddings between Besov and Lebesgue scale.]
    
A close observation of the proof of the corresponding result from \cite[Thm.~1.3]{MSS26} reveals that $[0,1]\times R^{ref}$ (there called $I\times D$) just has to be covered by finitely many non-overlapping prisms, but $R^{ref}$ does not necessarily have to be Lipschitz for the proof to work in exactly the same manner. $[0,1]\times R^{ref}$ is such a space-time domain according to \cref{rem:Neighborhoods_generally_not_Lipschitz}\ref{rem:Neighborhoods_generally_not_Lipschitz - reference domains}. Further, the proof can be performed for any polynomial order $(r_1, r_2)\in \N_{\ge 2}$ with $r_i>s_i$, $i=1,2$, thus giving the corresponding result (with an additional dependency on $(r_1, r_2$) for the norm $\|\cdot\|^{(r_1, r_2)}_{B^{s_1, s_2}_{p,q}([0,1]\times R^{ref})}$ instead of $\|\cdot\|_{B^{s_1, s_2}_{p,q}([0,1]\times R^{ref})}$. This yields the embeddings from \eqref{thm:Proofs_omega_tilde_domain - Embeddings between Besov and Lebesgue scale - equation 1}. In particular, the exact dependence on the constants is substantiated by the further comments on the corresponding properties of $[0,1]\times R^{ref}$ and its space-time partitioning $\widetilde{\calP}_{[0,1]\times R^{ref}}$ from \cref{rem:Neighborhoods_generally_not_Lipschitz}\ref{rem:Neighborhoods_generally_not_Lipschitz - reference domains}. Lastly, the equivalence from \eqref{thm:Proofs_omega_tilde_domain - Embeddings between Besov and Lebesgue scale - equation 2} is a consequence of the estimate $\diam(R^{ref})\sim_{\,d,\kappa_{\calP_0}, \mu(\calP_0)}1$ from \cref{rem:Neighborhoods_generally_not_Lipschitz}\ref{rem:Neighborhoods_generally_not_Lipschitz - reference domains} and the equivalence of the discrete and continuous Besov \mbox{(quasi-)}seminorms given below \eqref{Eq:equivalence quasi-seminorms} in \cref{subsubsect:Besov spaces}, which is only stated for the (quasi-)seminorms without the \enquote{exponent} $(r_1, r_2)$ there, but still valid in the same manner for the versions with respect to $(r_1, r_2)$.
\end{proof}

\section{\texorpdfstring{Algebraic calculation in \cref{thm:Sobolev-type-embedding}}{Appendix \getrefnumber{Appendix:Tedious_Calculation}: Algebraic calculation in Theorem \getrefnumber{thm:Sobolev-type-embedding}}}

\label{Appendix:Tedious_Calculation}
We observe that due to the collinearity of $(\alpha_1, \alpha_2)$ and $(s_1, s_2)$, there is a constant $C>0$ such that $(s_1, s_2)= C(\alpha_1, \alpha_2)$. This implies 
\begin{align*}
    \frac{1}{\frac{1}{\alpha_1 + s_1} + \frac{d}{\alpha_2+s_2}}= \frac{1}{\frac{1}{(1+C)\alpha_1} + \frac{d}{(1+C)\alpha_2}}= (1+C)\frac{1}{\frac{1}{\alpha_1} + \frac{d}{\alpha_2}} = (1+C)\frac{\alpha_1\alpha_2}{\alpha_2+d\alpha_1} = \frac{\alpha_1(\alpha_2+s_2)}{\alpha_2 + d\alpha_1}.
\end{align*}
Therefore,
\begin{align*}
    \left(1+\frac{\alpha_2}{\alpha_1d}\right)\frac{1}{\frac{1}{\alpha_1 + s_1} + \frac{d}{\alpha_2+s_2}}=\frac{\alpha_1(\alpha_2+s_2)}{\alpha_1d} = \frac{\alpha_2+s_2}{d},
\end{align*}
which yields
\begin{align*}
    \frac{\alpha_2}{d}-\left(1+\frac{\alpha_2}{\alpha_1d}\right)\left(\frac{1}{\frac{1}{\alpha_1 + s_1} + \frac{d}{\alpha_2+s_2}}-\frac{1}{q}+\frac1p\right) = -\frac{s_2}{d} - \left(1+\frac{\alpha_2}{\alpha_1d}\right)\left(-\frac{1}{q}+\frac1p\right).
\end{align*}
Moreover, 
\begin{align*}
    \frac{s_2}{d} = \frac{s_1d+s_2}{s_1d}\frac{s_1s_2}{s_2+s_1d} = \left(1+\frac{s_2}{s_1d}\right)\frac{1}{\frac{1}{s_1}+\frac{d}{s_2}}, 
\end{align*}
which  together with $\left(1+\frac{\alpha_2}{\alpha_1d}\right)=\left(1+\frac{s_2}{s_1d}\right)$, due to the collinearity, gives 
\begin{align*}
    \frac{\alpha_2}{d}-\left(1+\frac{\alpha_2}{\alpha_1d}\right)\left(\frac{1}{\frac{1}{\alpha_1 + s_1} + \frac{d}{\alpha_2+s_2}}-\frac{1}{q}+\frac1p\right) = -\left(1+\frac{s_2}{s_1d}\right)\left(\frac{1}{\frac{1}{s_1}+\frac{d}{s_2}}-\frac1q+\frac1p\right).
\end{align*}

\paragraph{Acknowledgements.} \phantom{ }We want to thank the anonymous referees for their very valuable suggestions which allowed us to improve the quality of this paper; they have also pointed out some mathematical gaps in our proofs that have been fixed thanks to their thorough review.  The first author was partially supported by Universidad Nacional del
Litoral through grant CAI+D-2024 85520240100018LI, and by Agencia Nacional de Promoción Científica
y Tecnológica through grant PICT-2020-SERIE A-03820. This work was partially developed during the third author's visits to Santa Fe in Argentina which were made possible by a fellowship for doctoral candidates of the German Academic Exchange Service (DAAD).

\paragraph{Declaration of generative AI in the manuscript preparation process.} The preparation of the initial version of this article, did not involve the use of generative AI. This applies in particular to the main results and to the conception and structure of the article. During the revision process of this work the authors have selectively used Claude Opus 5.5 and Fable 5.1 to help overcome mistakes in the proofs of \cref{lem:Besovnorm behaves a little like L_p} and \cref{Main_res:invers} for case $(\alpha_1, \alpha_2) = (0,0)$ (the latter essentially lead to the inclusion of \cref{lem:Relationship_Aquiv_classes}). Further, the discovery of the error in the literature from \cref{rem:Neighborhoods_generally_not_Lipschitz}\ref{rem:Neighborhoods_generally_not_Lipschitz - counterexample} and \ref{rem:Neighborhoods_generally_not_Lipschitz - remark on literature} was made using these tools. The same applies to one of the solutions to this problem, i.e., the proof of \cref{thm:Proofs_omega_tilde_domain}\ref{thm:Proofs_omega_tilde_domain - item Jackson}. Further, these tools have been used for consistency checks throughout the manuscript, e.g., regarding notations and references. After using these tools, the authors reviewed and edited the content as needed and take full responsibility for the content of the published article.

\bibliographystyle{alpha}

\end{document}